\documentclass[letterpaper,11pt]{article} 
\usepackage{amssymb,amsmath,amsthm}
\usepackage{fullpage}
\usepackage{rotating}
\usepackage{graphicx}
\usepackage{afterpage}
\usepackage{natbib,natbibspacing}
\usepackage{times}
\usepackage{setspace}
\usepackage{caption}
\usepackage{amssymb}
\usepackage{amsmath}

\newcommand{\bfzero}{\mathbf{0}}
\newcommand{\bfone}{\mathbf{1}}
\newcommand{\bfA}{\mathbf{A}}

\newcommand{\bfC}{\mathbf{C}}

\newcommand{\bfE}{\mathbf{E}}

\newcommand{\bfI}{\mathbf{I}}

\newcommand{\bfK}{\mathbf{K}}

\newcommand{\bfM}{\mathbf{M}}

\newcommand{\bfQ}{\mathbf{Q}}

\newcommand{\bfW}{\mathbf{W}}
\newcommand{\bfX}{\mathbf{X}}
\newcommand{\bfY}{\mathbf{Y}}
\newcommand{\bfZ}{\mathbf{Z}}

\newcommand{\bfd}{\mathbf{d}}

\newcommand{\bfs}{\mathbf{s}}

\newcommand{\bfv}{\mathbf{v}}
\newcommand{\bfw}{\mathbf{w}}
\newcommand{\bfx}{\mathbf{x}}

\newcommand{\bbE}{\mathbb{E}}

\newcommand{\bbI}{\mathbb{I}}

\newcommand{\bbP}{\mathbb{P}}

\newcommand{\bbR}{\mathbb{R}}

\newcommand{\bbone}{{\ifmmode\mathrm{1\!l}\else\mbox{\(\mathrm{1\!l}\)}\fi}}

\newcommand{\mcalA}{\mathcal{A}}

\newcommand{\mcalH}{\mathcal{H}}

\newcommand{\mcalO}{\mathcal{O}}

\newcommand{\mcalS}{\mathcal{S}}

\newcommand{\mcalY}{\mathcal{Y}}
\newcommand{\mcalZ}{\mathcal{Z}}

\newcommand{\sign}{\mbox{sign}}

\newcommand{\convinprob}{\stackrel{p}{\rightarrow}}
\newcommand{\convinlaw}{\stackrel{d}{\rightarrow}}

\DeclareMathOperator*{\argmin}{arg\,min}

\DeclareMathOperator{\var}{var}
\DeclareMathOperator{\cov}{cov}

\newcommand{\normal}{\mathcal{N}}

\newcommand{\figurewidth}{0.45}

\newtheorem{theorem}{Theorem}
\newtheorem{lemma}[theorem]{Lemma}
\newtheorem{corollary}[theorem]{Corollary}

\newenvironment{assumption}[1]
{\par\medskip\noindent{\bf #1}\hspace{0.5em}\it}
{\par\medskip}

\begin{document}

\title{Asymptotic distribution and sparsistency \\
for $\ell_{1}$ penalized parametric M-estimators, \\
with applications to linear SVM and logistic regression}
\author {
Guilherme Rocha
\footnote{Indiana University, gvrocha@indiana.edu, gvrocha.stat@gmail.com}, 
Xing Wang 
\footnote{Renmin University of China, wangxingscy@gmail.com}
and 
Bin Yu
\footnote{University of California, Berkeley, binyu@stat.berkeley.edu}
}
\maketitle

\onehalfspacing

\begin{abstract}
	Since its early use in least squares regression problems, the $\ell_{1}$-penalization framework for variable selection has been employed in conjunction with a wide range of loss functions encompassing regression, classification and survival analysis.
	While a well developed theory exists for the $\ell_{1}$-penalized least squares estimates, few results concern the behavior of $\ell_{1}$-penalized estimates for general loss functions.
	In this paper, we derive two results concerning penalized estimates for a wide array of penalty and loss functions.
	Our first result characterizes the asymptotic distribution of penalized parametric M-estimators under mild conditions on the loss and penalty functions in the classical setting (fixed-$p$-large-$n$).
	Our second result explicits necessary and sufficient generalized irrepresentability (GI) conditions for $\ell_{1}$-penalized parametric M-estimates to consistently select the components of a model (sparsistency) as well as their sign (sign consistency).
	In general, the GI conditions depend on the Hessian of the risk function at the true value of the unknown parameter.
	Under Gaussian predictors, we obtain a set of conditions under which the GI conditions can be re-expressed solely in terms of the second moment of the predictors.
	We apply our theory to contrast $\ell_{1}$-penalized SVM and logistic regression classifiers and find conditions under which they have the same behavior in terms of their model selection consistency (sparsistency and sign consistency).
	Finally, we provide simulation evidence for the theory based on these classification examples.

\end{abstract}

%
%
\newpage


\doublespacing

\section{Introduction}
\label{section:introduction}

When modeling the a response variable $\bfY\in \mcalY$ as a function of a set of predictors $\bfX\in \bbR^{p}$, statisticians often rely on M-estimators for linear models defined as
\begin{eqnarray}
	\label{equation:definition_unpenalized_estimate_linear_model}
	\begin{array}{ccccc}
	\left(\hat{\alpha}_{n}, \hat{\beta}_{n} \right)
	& := & 
	\argmin\limits_{a\in\bbR, b\in\bbR^{p}}
	\left[\frac{1}{n}\cdot\sum_{i=1}^{n}L(Y_{i}, a+X_{i}^{T}b, t)\vphantom{\int_{0}^{1}}\right],  
	\end{array}
\end{eqnarray}
where $Z_{i} = (Y_{i}, X_{i})$, $i=1, \ldots, n$,  are independent observations of $\bfZ = (\bfY, \bfX)$ and the loss function $L:\mcalY\times \bbR\rightarrow\bbR_{+}$ measures the lack of quality of $a+X_{i} b$ in representing $Y_{i}$.
For a given problem, many alternative loss functions can be used.
Some recent results are aimed at comparing the properties of estimates obtained from alternative loss functions \citep{zhang:2004:statistical-behavior-and-consistency-of-classification-methods-based-on-convex-risk-minimization, bartlett:2005:convexity}.

The choice of an appropriate loss function must take the goal of the analysis into account.
Often, the estimates in \eqref{equation:definition_unpenalized_estimate_linear_model} are used as a tool in understanding the effects of $\bfX$ on $\bfY$.
In that case, sparse estimates $\hat{\beta}_{n}$ are desirable as they select which predictors in $\bfX$ have an effect on the response $\bfY$.
Sparse estimates are often achieved by a penalized estimate
\begin{eqnarray}
	\label{equation:definition_penalized_estimate_linear_model}
	\left(
	\hat{\alpha}_{n}(\lambda_{n}),
	\hat{\beta}_{n}(\lambda_{n}) 
	\right)
	& := & 
	\argmin_{a\in\bbR, b\in\bbR^{p}}
	\left[
	\frac{1}{n}\sum_{i=1}^{n}L(Y_{i}, a+X_{i}b) + \lambda_{n} \cdot T(b)\vphantom{\int}
	\right],
\end{eqnarray}
where $\lambda_{n}\ge0$ is a regularization parameter and $T:\bbR^{p}\rightarrow \bbR_{+}$ is a function penalizing non-sparse models.
Many alternative sparsity inducing penalties exist and a popular family of such penalties is the set of $\ell_{\gamma}$ norms with $\gamma\in(0,1]$ used in bridge estimates \citep{frank:1993:bridge-regression}. The $\ell_{\gamma}$ norm function given by $\|b\|_{\gamma} := \left(\sum_{j=1}^{p}\|b_{j}\|^{\gamma}\right)^{\frac{1}{\gamma}}$.
Two important particular cases are the $\ell_{0}$-penalty -- defined as a penalty on the number of non-zero terms in the estimate used in \citep{akaike:1973:information-criterion, akaike:1974:a-new-look-at-the-statistical-model-identification, schwarz:1978:bic, rissanen:1978:modeling-by-shortest-data-description, hansen:2001:gmdl}, and
the $\ell_{1}$-penalty used in the LASSO \citep{tibshirani:1996:lasso} and basis pursuit \citep{chen:2001:basis-pursuit}.

Recently, a large number of $\ell_{1}$-penalized estimates based on different loss functions have been proposed in the literature.
Some examples are the logistic regression and Cox's proportional hazards model loss \citep{tibshirani:1997:cox_lasso, park:2006:l1-regularization-path-algorithms-for-generalized-linear-models}, 
the hinge loss function for classification \citep{zhu:2004:l1_penalized_svms}, 
the quantile regression loss \citep{li:2006:the-l1-norm-quantile-regression}, and 
the log-determinant Bregman divergence of covariance matrices \citep{banerjee:2005:sparse-covariance-selection-via-robust-maximum-likelihood-estimation, 
ravikumar:2008:high-dimensional-covariance-estimation-minimizing-ell1-penalized-log-determinant-divergence}.
Simultaneously, many families of sparse inducing penalties have been introduced such as the SCAD penalty \citep{fan:1999:scad_penalty} and the generalized elastic net \citep{friedman:2008:fast-sparse-regression-and-classification}.
In this paper, we present theoretical results allowing the behavior of estimates based on different loss and penalty functions to be compared.

Our first main result is a characterization of the asymptotic distribution of the penalized estimates in \eqref{equation:definition_penalized_estimate_linear_model} for a wide class of penalty and convex loss functions.
Our result extends previous results for the squared error loss and $\ell_{\gamma}$ norms by  \citet{knight:2000:asymptotics} and applies to the classical asymptotic setup (large $n$, fixed $p$).
We state our results in a modular fashion so they encompass several combinations of loss and penalty functions.
We provide sufficient conditions on the loss and on the penalty functions for our results to apply.
On the loss side, our results depend on convexity of the loss function and on the risk function defined as
\begin{eqnarray}
	\label{equation:definition_risk_function_linear_model}
	R(t) := \bbE_{\bfX, \bfY}\left[L(\bfY, a + b^{T}\bfX)\right], \mbox{ for } (a, b)\in\bbR^{1+p},
\end{eqnarray}
to be twice continuously differentiable at the ``true'' value of the parameters $(\alpha, \beta)$
\begin{eqnarray}
	\label{equation:definition_risk_minimizer_linear_model}
	\left(
	\alpha, 
	\beta
	\right)
	& := & 
	\argmin_{t\in\bbR^{p}} R(a, b).
\end{eqnarray}

Our second result obtains necessary and sufficient conditions for the $\ell_{1}$-penalized estimate in \eqref{equation:definition_penalized_estimate_linear_model} to consistently select the zeroes (sparsistency) and signs (sign consistency) in the parameter $\beta$.
Previous results for $\ell_{1}$-penalized least squares linear regression show that the set of active and inactive predictors must be sufficiently disentangled for sparsistency to hold.
This requirement is embodied in ``incoherence'' or ``irrepresentability'' conditions  \citep{meinshausen:2004:consistent-neighborhood-selection-for-sparse-high-dimensional-graphs-with-the-lasso, zhao:2006:lasso_model_selection_consistency,  zou:2006:adaptive_lasso, wainwright:2006:sharp-thresholds-for-high-dimensional-and-noisy-recovery-of-sparsity}. 
We call the condition for sparsistency and sign-consistency of general $\ell_{1}$-penalized M-estimators the generalized irrepresentability (GI) condition.
Intuitively, the GI condition can be interpreted as a requirement that the the effects of active and inactive predictors on the loss are distinguishable enough (after ``controlling'' for the intercept term).
This second result relies on the quadratic approximation developed on the first result and is thus only applicable on the classical small-$p$-large-$n$ case.

Our third result shows that, if the predictors are zero-meaned Gaussian and the response variable only depends on $\bfX$ through an affine function of the predictors, the conditions for $\ell_{1}$-penalized estimates as in \eqref{equation:definition_penalized_estimate_linear_model} do not depend on the loss function.
In that case, the GI condition reduces to the ``irrepresentable'' condition in \citet{zhao:2006:lasso_model_selection_consistency}.
This surprising result stems from the properties of the multivariate Gaussian distribution, namely on its  linear mean and constant variance when conditioned on one of its linear combinations.

We apply the theory to contrast and compare linear classifiers based on the hinge loss (parametric SVM) and logistic regression. 
We obtain expressions for the Hessians of the SVM and logistic regression risks and characterize them as weighted averages of the second moment matrices of the predictors conditional on a properly defined ``linear predictor'' variable $\bfM = \alpha+\beta^{T}\bfX$.
Based on this characterization and using our third result, we show that, for a given joint distribution $(\bfY, \bfX)$ where the predictors are Gaussian and the response variable only depends on the predictors through an affine transformation, the two classifiers are either both sparsistent or not.
For more general joint distributions, one of the classifiers can be sparsistent while the other is not.
Over a set of cases where the predictors are mixed Gaussian, we observed logistic regression to be sparsistent more often than SVM classifiers but also observed mixed results in finite samples.
The conditionally weighted second moment characterization of the Hessians also evidences that the Hessians of both SVM and logistic regression risk functions emphasize the second moment of the predictors closer to the optimal separating hyperplane.
This emphasis on the region close to the margin echoes previous results in the non-parametric works of \citet{audibert:2007:fast-learning-rates-for-plug-in-classifiers} and \citet{steinwart:2007:fast-rates-for-support-vector-machines-using-gaussian-kernels} and help explain the similarities between SVM and logistic regression classifiers.

The remainder of this paper is organized as follows.
Section \ref{section:theory} presents our asymptotic results for penalized empirical risk minimizers for general loss and penalty functions.
Section \ref{section:application} presents necessary and sufficient conditions for model selection consistency of $\ell_{1}$-norm penalized empirical risk minimizers for general loss functions.
Section \ref{section:linear_classifiers} applies the results in the previous two sections to the study and comparison of SVM and logistic regression classifiers satisfy the requirements for our results from the previous two sections to apply.
Section \ref{section:simulations} shows a series of simulations providing empirical support for the model selection consistency theory we developed as well as comparisons between SVM and logistic regression classifiers.
Finally, Section \ref{section:discussion} concludes with a brief discussion.

\section{Asymptotic distribution of penalized parametric M-estimators}
\label{section:theory}

In this section, we present the first main result of this paper (Theorem \ref{result:main_result}) which characterizes the asymptotic distribution of penalized empirical risk minimizers for a broad range of penalty and loss functions for a fixed number of predictors.
Theorem \ref{result:main_result} extends previous results by \citet{knight:2000:asymptotics} regarding norm-penalized least squares estimates.
In essence, the steps in the proof of \ref{result:main_result} closely parallel the ones used by \citet{knight:2000:asymptotics}, but we keep the study of the convergence of loss and penalty functions separate so our results can be applied to any combination of loss and penalty functions satisfying the conditions detailed below.

Before proceeding, we introduce some notation.
Our results apply to penalized estimates defined as
\begin{eqnarray}
	\label{equation:definition_penalized_estimate}
	\hat{\theta}_{n}(\lambda_{n}) 
	& := & 
	\argmin_{t\in\Theta\subset\bbR^{p+1}}
	\left[
	\frac{1}{n}\sum_{i=1}^{n}L(Z_{i}, t) + \lambda_{n} \cdot T(t)\vphantom{\int}
	\right].
\end{eqnarray}
The definition in \eqref{equation:definition_penalized_estimate_linear_model} is a particular case that encompasses linear models by setting $Z_{i} = (Y_{i}, X_{i})\in \mcalY\times \bbR^{p}$, $t=(a,b)$ and $L(Z_{i}, t) = L(Y_{i}, a+b^{T}X_{i})$.
In this extended case, the best model we can select is parameterized by
\begin{eqnarray}
	\label{equation:definition_risk_minimizer}
	\theta
	& := & 
	\argmin_{t\in\Theta\subset\bbR^{p+1}} R(t),
\end{eqnarray}
where the risk function has the usual definition
\begin{eqnarray}
	\label{equation:definition_risk_function}
	R(t) := \bbE_{\bfZ}\left[L(\bfZ, t)\right], \mbox{ for } t\in\Theta\subset\bbR^{p+1}.
\end{eqnarray}

Let $u\in \bbR^{p+1}$ and $q_{n}$ be a sequence of non-negative numbers such that $q_{n}\rightarrow\infty$ as $n\rightarrow\infty$. Define
\begin{eqnarray*}
	\begin{array}{cclc}
	C^{(n)}_{\theta}(\bfZ, u)
	& := & 
	\sum_{i=1}^{n}
	\left[
	L\left(\bfZ, \theta+\frac{u}{q_{n}}\right)
	-
	L\left(\bfZ, \theta\right)
	\right]
	,
	\\
	G^{(n)}_{\theta}(u)
	& := & 
	\left[
	T\left(\theta+\frac{u}{q_{n}}\right) 
	-
	T\left(\theta\right)
	\right]
	, & \mbox{ and }
	\\
	V^{(n)}_{\theta}(\bfZ, \lambda_{n}, u) 
	& := &
	C^{(n)}_{\theta}(\bfZ, u)
	+
	\lambda_{n}
	\cdot
	G^{(n)}_{\theta}(u).	
	\end{array} 
\end{eqnarray*}
The $V^{(n)}_{\theta}$ function corresponds to a recentered and rescaled version of the objective function in \eqref{equation:definition_penalized_estimate} so
\begin{eqnarray*}
	q_{n}\cdot\left(\hat{\theta}_{n}(\lambda_{n})-\theta\right)
	& = & 
	\argmin_{u\in\bbR^{p}} V^{(n)}_{\theta}(Z, \lambda_{n}, u).
\end{eqnarray*}
The asymptotic behavior of $\hat{\theta}_{n}(\lambda_{n})$ can be characterized in terms of asymptotic results for $V_{\theta}^{(n)}$ and its minimizer.
A close study of the proof used by \citet{knight:2000:asymptotics} shows that for the most part, the convergences of the loss $\left(C_{\theta}^{(n)}\right)$ and the penalty $\left(G_{\theta}^{(n)}\right)$ functions are studied separately.
This is reflected in our Theorem \ref{result:assembling_theorem}: a versatile and an important ``assembling'' tool.
Any set of assumptions made on the loss and penalty functions that ensures the conditions required by Theorem \ref{result:assembling_theorem} can be used to obtain a characterization of the distribution of penalized estimates.

\begin{theorem}
\label{result:assembling_theorem}
Let $\lambda_{n}\ge 0$ be a sequence of positive (potentially random) real numbers, $Z_{i}$, $i=1, \ldots, n,$ be a sequence of i.i.d. realizations from a distribution $\bbP_{\bfZ}$, $L:\mcalZ\times \Theta\rightarrow\bbR$ be a loss function and $T:\Theta\rightarrow \bbR$ be a penalty function.
Let $\hat{\theta}(\lambda_{n})$ be as defined in \eqref{equation:definition_penalized_estimate}.

Suppose there exist functions $C_{\theta}, G_{\theta}$, a constant $\lambda$, a random vector $\bfW$ and a sequence $q_{n}$ of deterministic positive real numbers 
with $q_{n}\rightarrow \infty$ as $n\rightarrow\infty$ such that, for any compact set $K\subset \bbR^{p}$:
\begin{enumerate}
	\item[i)]  
	$\sup\limits_{u\in K}\left|
	\sum_{i=1}^{n}
	\left[
	L\left(Z_{i}, \theta+\frac{u}{q_{n}}\right)
	-
	L\left(Z_{i}, \theta\right)
	\right]
	-
	C_{\theta}(\bfW, u)
	\right| \stackrel{p}{\rightarrow}0$;
	\item[ii)] 
	$\sup\limits_{u\in K}\left|
	\lambda_{n}
	\left[T\left(\theta+\frac{u}{q_{n}}\right) 
	-
	T\left(\theta\right)\right]
	-
	\lambda\cdot
	G_{\theta}(u)
	\right| \stackrel{p}{\rightarrow}0$;
	\item [iii)] $\hat{\theta}_{n}(\lambda_{n})$ is $O_{p}(q_{n}^{-1})$.
\end{enumerate}
Let $V_{\theta}(\bfW, u) = C_{\theta}(\bfW, u) + \lambda\cdot G_{\theta}(u)$.\\
If i) and ii) hold, then: 
\begin{enumerate}
  \item [a)] 
  $\sup\limits_{u\in K}
  \left|V^{(n)}_{\theta}(\bfZ, \lambda_{n}, u) 
  - 
  V_{\theta}(\bfW, u)\right|\stackrel{p}{\rightarrow}0$;
\end{enumerate}
If i), ii) and iii) hold, then: 
\begin{enumerate}
  \item [b)] 
  $q_{n}\left(\hat{\theta}_{n}(\lambda_{n})-\theta\right)
  \stackrel{d}{\rightarrow}
  \arg\min\limits_{u} V_{\theta}(\bfW, u)$.
\end{enumerate}
\end{theorem}
Roughly speaking, we can prove Theorem \ref{result:assembling_theorem} by observing that boundedness in probability of the sequence $\hat{\theta}_{n}(\lambda_{n})$ implies that $\hat{\theta}_{n}(\lambda_{n}) \in K$, for some compact set $K$ with probability approaching $1$.
Given this condition, it follows that the uniform convergence in probability over compact sets is sufficient to ensure that the minimizer of $V^{(n)}_{\theta}$ converges in probability to the minimizer of $V_{\theta}$.
A detailed proof is given in Appendix \ref{section:proof}.

Based on Theorem \ref{result:assembling_theorem}, we now proceed to study the loss and penalty functions separately.

\subsection{Loss functions}
\label{section:theory_loss}

We now establish sufficient conditions for the loss function to display the convergence required in Theorem \ref{result:assembling_theorem}.
Our results use standard approximations for the loss function in terms of the risk function combined with the Convexity Lemma by \citet{pollard:1991:asymptotics-for-least-absolute-deviation-regression-estimators}, which is used as a tool to upgrade pointwise convergence results to uniform convergence over compact sets.

\begin{assumption}{Loss Assumptions (LA)}
\begin{enumerate}
	\item [L1.] The parameter $\theta = \argmin\limits_{t\in \Theta}\bbE\left[L(\bfZ, t)\right]$ is bounded and unique; 

	\item [L2.] $\bbE\left| L(\bfZ, t)\right|< \infty$ for each $t$; 
	
	\item [L3.] The loss function $L(Z, t)$ is such that:
    \begin{enumerate}
		\item [a)] $L(Z, t)$ is differentiable with respect to $t$ at $t=\theta$ for $\bbP_{\bfZ}$-almost every $Z$ with derivative $\nabla_{t}L(Z,\theta)$ and
		\begin{eqnarray}
		\label{equation:covariance_of_scores}
		J(\theta) & := & \bbE\left[\nabla_{t}L(\bfZ,\theta)\nabla_{t}L(\bfZ,\theta)^{T}\right]<\infty;	
		\end{eqnarray}

		\item [b)] the risk function $R(t)=\bbE\left[L(\bfZ, t)\right]$ is twice differentiable with respect to $t$ at $t=\theta$ with positive definite Hessian matrix
		\begin{eqnarray}
		\label{equation:risk_hessian}
			\left[H(\theta)\right]_{ij} 
			:= 
			\left.
			\frac{\partial^{2}R(t)}{\partial t_{i}\partial t_{j}}
			\right|_{t=\theta}
			= 
			\left.
			\frac
			{\partial^{2}\left(\bbE\left[ L(\bfZ, t) \right]\right)}
			{\partial t_{i} \partial t_{j}}
			\right|_{t=\theta};
		\end{eqnarray}

    \end{enumerate}

	\item [L4.] The loss function $L(Z, t)$ is convex with respect to its argument $t$ for $\bbP_{\bfZ}$-almost every $Z$.
\end{enumerate}
\end{assumption}

Assumptions L1-L4 -- the L being a mnemonic for the loss function -- are relatively mild.
The first assumption on the loss function (L1) ensures that the parameter $\theta$ in \eqref{equation:definition_risk_minimizer} is well defined and is thus a minimal requirement.
Assumption L2 yields that a  law of large numbers is valid for each value of $t$, and thus that the risk function equals the pointwise limit of the empirical risk.
In our proofs, assumption L3 is used extensively to obtain local quadratic asymptotic approximations to the risk function around the parameter $\theta$ that are pointwise valid around $\theta$ (i.e., for each $\theta+\frac{u}{q_{n}}$ for a sequence $0<q_{n}\rightarrow \infty$ as $n\rightarrow\infty$). 
The requirement that the risk function is twice differentiable does not require differentiability of the loss function itself, as will become evident in our analysis of the hinge loss in Section \ref{section:linear_classifiers}.
Finally, assumption L4 is used to upgrade the local approximation for the risk function from pointwise to uniform over compact sets by means of Pollard's convexity lemma \citep{pollard:1991:asymptotics-for-least-absolute-deviation-regression-estimators}.
Alternative assumptions can replace L4: any set of conditions yielding uniform convergence over compact sets will do.
One could, for instance, replace it by conditions on the local complexity/entropy of the loss function \citep[see, for instance, ][]{dudley:1999:uniform-central-limit-theorems}.
We stick to convexity here given its computational convenience and widespread use in statistics and machine learning \citep{bartlett:2005:convexity}.

\begin{lemma}
  \label{result:uniform_convergence_for_loss}
  Under the LA assumptions L1, L2, and L3:
  \begin{enumerate}
	\item [a)] There exists a $p$-dimensional random vector $\bfW\sim N\left(0, J(\theta)\right)$ such that
	\begin{eqnarray*}
      \frac{1}{n}\sum_{i=1}^{n}
	  \left[
	  L\left(Z_{i}, \theta+\frac{u}{\sqrt{n}}\right)
      -
      L\left(Z_{i}, \theta\right) 
	  \right]
      -
      \left[u^{T}\cdot H(\theta)\cdot u + \bfW^{T}\cdot u\right]
      \stackrel{p}{\rightarrow}0, \mbox{ for each } u\in \bbR^{p}.
	\end{eqnarray*}
	
  \end{enumerate}

  \begin{enumerate}
	\item [b)]
	If, in addition, LA assumption L4 holds, then:
	\begin{enumerate}
		\item [b.1)]
		for every compact subset $K\subset \bbR^{p}$,
		\begin{eqnarray*}
		\sup\limits_{u\in K}
		 \left\|
      	 \frac{1}{n}\sum_{i=1}^{n}
	  	 \left[
	  	 L\left(Z_{i}, \theta+\frac{u}{\sqrt{n}}\right)
      	 -
      	 L\left(Z_{i}, \theta\right) 
	  	 \right]
      	 -
      	 \left[u^{T}\cdot H(\theta)\cdot u + \bfW^{T}\cdot u\right]
      	 \right\|
      	 \stackrel{p}{\rightarrow}0, \mbox{ and }
		\end{eqnarray*}
		\item [b.2)]
		$\sqrt{n}\cdot\hat{\theta}_{n}(0) = O_{p}(1)$.
	\end{enumerate}
  \end{enumerate}
\end{lemma}

Our proof of the pointwise convergence (a) and of boundedness of the M-estimator (b.2) is offered in the Appendix \ref{section:proof}.
It can be seen as an extension of the results for the absolute error loss due to \citet{pollard:1991:asymptotics-for-least-absolute-deviation-regression-estimators}.
The upgrade from pointwise convergence to uniform convergence over compact sets is a direct application of the Convexity Lemma in  \citet{pollard:1991:asymptotics-for-least-absolute-deviation-regression-estimators}.

\subsection{Penalty functions}
\label{section:theory_penalty}

Lemma \ref{result:uniform_convergence_for_penalty} establish conditions for non-adaptive penalties to satisfy the conditions required by Theorem \ref{result:assembling_theorem}.

\begin{assumption}{Penalty Assumptions (PA)}
	\begin{enumerate}
		\item [P1.] $T:\Theta\rightarrow \bbR$ is non-random and $T(t)\ge 0$ for all $t\in\Theta$;
		
		\item [P2.] $T$ is continuous in $t\in\Theta$;
		
		\item [P3.] The function 
		\begin{eqnarray}
			\label{equation:definition_penalty_pseudogradient}
			G_{\theta}( u) := \lim_{h\downarrow 0}\frac{T(\theta+ u\cdot h)-T(\theta)}{h}
		\end{eqnarray}
		is well defined and continuous for all $ u\in \bbR^{p}$;

		\item [P4.] The set $\{t\in \Theta: T(t)\le c\}$ is compact for all $c<T(\theta)$.

	\end{enumerate}
\end{assumption}

The set of assumptions P1 through P4 on the penalties -- P is a mnemonic for penalty function -- 
is broad enough to encompass  all $\ell_{\gamma}$ norms with $\gamma>0$ and the set of generalized elastic net penalties in \citet{friedman:2008:fast-sparse-regression-and-classification}.
With minor adjustments, the SCAD penalty \citep{fan:1999:scad_penalty} can also be treated by our theory.
We emphasize that convexity is not a requirement.
Non-randomness and continuity (assumptions P1 and P2) make it easy to obtain uniform convergence over compact sets.
Condition P3 is similar but milder than a differentiability requirement. 
We prove that the penalty function converges uniformly over compact sets by using conditions P1 through P3.
Condition P4 is useful in ensuring that the penalized estimates are bounded in probability.
It amounts to a requirement that the penalty function $T$ constrains the penalized estimates to be within a compact set for all $\lambda>0$.

\begin{lemma}
  \label{result:uniform_convergence_for_penalty}
	Let $\theta$ be as defined in \eqref{equation:definition_risk_minimizer}, $q_{n}$ be a sequence of non-random positive real numbers satisfying  $q_{n}\rightarrow\infty$ and $\lambda_{n}$ be a sequence on non-negative (potentially random) real numbers with $\lambda_{n}\cdot q_{n}^{-1}\convinprob \lambda$ as $n\rightarrow\infty$.
	Suppose that the $T$ is a penalty function satisfying the PA conditions P1 through P3.
	Then, for all compact subsets $K\in \bbR^{p}$:
	\begin{eqnarray*}
	\sup_{ u\in K}\left|\lambda_{n}\cdot \left[T\left(\theta + \frac{ u}{q_{n}}\right) - T(\theta)\right] - \lambda\cdot G_{\theta}( u)\right|\rightarrow 0, 
	 \mbox{ as } n\rightarrow\infty.	
	\end{eqnarray*}
\end{lemma}

A proof for Lemma \ref{result:uniform_convergence_for_penalty} is offered in Appendix \ref{section:proof}.


\subsection{Convergence of penalized empirical risk minimizers}
\label{section:theory_convergence}

We now state our first main result, which characterizes the asymptotic distribution of penalized parametric M-estimators.

\begin{theorem}
	\label{result:main_result}
	Assume $\lambda_{n}$ be a sequence of non-negative (potentially random) real numbers such that $\lambda_{n}\cdot n^{-\frac{1}{2}}\convinprob \lambda\ge0$ as $n\rightarrow\infty$.
	Let $\theta$, $\hat{\theta}_{n}(\lambda_{n})$, $J(\theta)$, $H(\theta)$, and $G_{\theta}(u)$ be as defined in \eqref{equation:definition_risk_minimizer}, \eqref{equation:definition_penalized_estimate},\eqref{equation:covariance_of_scores}, \eqref{equation:risk_hessian}, and \eqref{equation:definition_penalty_pseudogradient} respectively.
	Define:	
	\begin{eqnarray*}
		V_{\theta}(w,  u)
		& = &
		 u^{T}\cdot H(\theta)\cdot u + w^{T}\cdot  u + \lambda\cdot G_{\theta}( u), \mbox{ for } w\in \bbR^{p}.
	\end{eqnarray*}
	If the loss function satisfies the LA assumptions and the penalty function satisfies the PA assumptions, then there exists a $p$-dimensional random vector $\bfW\sim N\left(0, J(\theta)\right)$ such that:
	\begin{eqnarray*}
		\sqrt{n}\left(\hat{\theta}_{n}(\lambda_{n})-\theta\right)
		\convinlaw 
		\argmin_{ u} V_{\theta}(\bfW,  u).
	\end{eqnarray*}
\end{theorem}
\begin{proof} \textbf{Theorem \ref{result:main_result}.}\\
In the appendix, we prove that $\hat{\theta}_{n}(0)=O_{p}(1)$ implies $\hat{\theta}_{n}(\lambda_{n})=O_{p}(1)$ for all $\lambda_{n}>0$ (Lemma \ref{result:bounded_unregularized_estimate_implies_bounded_penalized_estimate}).
Thus, under the assumptions made, Lemma \ref{result:uniform_convergence_for_loss} along with Lemma \ref{result:bounded_unregularized_estimate_implies_bounded_penalized_estimate} ensures that conditions (i) and (iii) in Theorem \ref{result:assembling_theorem} are satisfied.
Additionally, Lemma \ref{result:uniform_convergence_for_penalty} ensures that condition (ii) in Theorem \ref{result:assembling_theorem} is met.
The result then follows directly from Theorem \ref{result:assembling_theorem}.
\end{proof}

We emphasize that the approximation afforded by Theorem \ref{result:main_result} is valid for the unique minimizer of the risk function $\theta$ as defined in \eqref{equation:definition_risk_minimizer}.
As the penalty function is not assumed to be convex, local minima may exist in finite samples.
However, the conditions in Theorem \ref{result:main_result} ensure that asymptotically the penalty component of the $V_{\theta}^{(n)}$ function is negligible in comparison to the risk component and asymptotically the minimizer is unique.

In the next section, we use the asymptotic characterization of the distribution of $\ell_{1}$-norm penalized empirical risk minimizers in Theorem \ref{result:main_result} to obtain necessary and sufficient conditions for the existence of a sequence of tuning parameters $\lambda_{n}$ for which $\hat{\theta}_{n}(\lambda_{n})$ is model selection consistent.

\section{Model selection consistency of $\ell_{1}$-penalized for M-estimators}
\label{section:application}

Our main result concerning $\ell_{1}$-norm penalized estimates gives necessary and sufficient conditions ensuring the existence of a sequence of regularization parameters $\lambda_{n}$ such that $\hat{\theta}_{n}(\lambda_{n})$ correctly identify the signs of the entries in the optimal vector of coefficients $\theta$ as defined in \eqref{equation:definition_risk_minimizer} as the sample size increases.
Before we can state this result, we must introduce some notation and terminology.
To allow the usual practice of including non-penalized intercepts to linear models, we write the risk minimizer as $\theta=\left(\alpha, \beta\right)\in \bbR^{p+1}$, where only the coefficients in $\beta\in\bbR^{p}$ are included in the $\ell_{1}$-penalty.
We define a partition of $\beta$ in terms of its sparsity pattern:
\begin{eqnarray*}
	\mcalA = \left\{j\in \{1, \ldots, p\}: \beta_{j}\neq0\right\}, & \mbox{ and } & \mcalA^{c} = \left\{j\in \{1, \ldots, p\}: \beta_{j}=0\right\}.
\end{eqnarray*}
We let $q$ denote the number of indices in the set $\mcalA$.
We will say that an estimate $\hat{\theta}_{n}(\lambda)$ is \textit{sign-correct} if $\sign(\hat{\beta}_{n}(\lambda)) = \sign(\beta)$, where $\sign(t)$ for a vector $t\in \bbR^{p}$ is a $p$-dimensional vector with:
\begin{eqnarray}
	\left[\sign(t)\right]_{j} 
	& = & 
	\left\{
	\begin{array}{ll}
	1,  \mbox{ if } t_{j}>0, \\
	0,  \mbox{ if } t_{j}=0, \mbox{ and } \\
	-1, \mbox{ if } t_{j}<0.
	\end{array}
	\right.
\end{eqnarray}
We will say that a sequence of estimates of regularization paths $\hat{\theta}(.):\bbR\rightarrow\Theta$ is \textit{sparsistent and sign consistent} if there exists a sequence $\lambda_{n}$ of (potentially random) non-negative values of the regularization parameters such that
\begin{eqnarray*}
	\lim_{n\rightarrow\infty}\bbP\left(\hat{\beta}_{n}(\lambda_{n}) \mbox{ is sign correct}\right)=1.
\end{eqnarray*}
We emphasize that the definition requires only the penalized components of $\theta$ to be asymptotically sign-correct.

For a risk function satisfying assumption L2 above, rearrange and partition the $\left(1+q+(p-q)\right)\times\left(1+q+(p-q)\right)$ Hessian:
\begin{eqnarray}
	H(\theta) 
	& = & 
	\left[
	\begin{array}{ccc}
	H_{\alpha,     \alpha    }(\theta)  & H_{\alpha,     \mcalA    }(\theta)  & H_{\alpha    , \mcalA^{c}}(\theta)
	\\
	H_{\mcalA,     \alpha    }(\theta)  & H_{\mcalA,     \mcalA    }(\theta)  & H_{\mcalA    , \mcalA^{c}}(\theta)
	\\
	H_{\mcalA^{c}, \alpha    }(\theta)  & H_{\mcalA^{c}, \mcalA    }(\theta)  & H_{\mcalA^{c}, \mcalA^{c}}(\theta)
	\end{array}
	\right].
\end{eqnarray}

\begin{theorem}
	\label{result:generalized_irrepresentability_condition}
	Let $\hat{\theta}_{n}(\lambda_{n}) = \left(\hat{\alpha}_{n}(\lambda_{n}), \hat{\beta}_{n}(\lambda_{n})\right)$ be as defined in \eqref{equation:definition_penalized_estimate} above with an $\ell_{1}$-penalty applied only to the terms in $\hat{\beta}_{n}(\lambda_{n})$.	
	Suppose the loss function satisfy the conditions in Assumption Set 1 and define
	\begin{eqnarray}
	\label{equation:generalized_irrepresentability_condition}
	\eta(\theta) 
	:= 
	1
	-
	\left\|H_{\mcalA^{c}, \mcalA    }(\theta)
	\left[ 
	H_{\mcalA,\mcalA}(\theta)
	- 
	H_{\mcalA,\alpha}(\theta)H_{\alpha,\alpha}(\theta)^{-1}H_{\alpha,\mcalA}(\theta)
	\right]^{-1}\sign(\beta_{\mcalA})\right\|_{\infty}\ge 0.
	\end{eqnarray}

\begin{enumerate}
	\item [a)]
	Let $\lambda_{n}$ is a sequence of non-negative (potentially random) real numbers such that such that $\lambda_{n}\cdot n^{-1}\convinprob 0$, and $\lambda_{n}\cdot n^{-\frac{1+c}{2}}\convinprob \lambda > 0 $ for some $0<c<\frac{1}{2}$ as $n\rightarrow\infty$.
	If $\eta(\theta)>0$, then:
		\begin{eqnarray*}
		\bbP\left[\sign\left(\hat{\beta}_{n}(\lambda_{n})\right)=\sign(\beta)\right] \ge 1 - \exp[-n^{c}].
		\end{eqnarray*}	
		
		\item [b)] Conversely, if $\eta(\theta)<0$, then, for any sequence of non-negative numbers  $\lambda_{n}$
		\begin{eqnarray*}
		\lim_{n\rightarrow\infty}\bbP\left[\sign\left(\hat{\beta}_{n}(\lambda_{n})\right)=\sign(\beta)\right] < 1.
		\end{eqnarray*}
\end{enumerate}
\end{theorem}

The result in Theorem \ref{result:generalized_irrepresentability_condition} extends  the model selection consistency results in  \citet{zhao:2006:lasso_model_selection_consistency} concerning LASSO estimates (based on $L_{2}$-loss) to more general parametric estimates defined as $\ell_{1}$-norm penalized M-estimators based on loss functions satisfying the conditions in Assumption Set 1.
We will call the condition in \eqref{equation:generalized_irrepresentability_condition} the \textit{generalized irrepresentability condition} (GI condition) which in the case of the $L_{2}$-loss with zero-mean predictors recovers Zhao and Yu's irrepresentable condition.
Accordingly, we call $\eta(\theta)$ the \textit{GI index},
which can be interpreted as a measure of incoherence between the active and inactive predictors.
Positive values of $\eta(\theta)$ imply the effects of active and inactive predictors are distinguishable enough so the $\ell_{1}$-penalized estimate can correctly identify the signs of all coefficients in the optimal model given a sufficiently large sample size.

A condition similar to the generalized irrepresentability condition \eqref{equation:generalized_irrepresentability_condition} appears in \citet{ravikumar:2008:high-dimensional-covariance-estimation-minimizing-ell1-penalized-log-determinant-divergence}.
There, the GI condition is used to obtain sufficient conditions for the consistent selection of the terms of an infinite dimensional precision matrix estimate defined as the $\ell_{1}$-norm penalized minimizer of the log-likelihood loss for Gaussian distributions.
This suggests it is possible to extend Theorem \ref{result:generalized_irrepresentability_condition} to the non-parametric setting where the number of regressors $p$ grows with the sample size $n$ (i.e., $p=p_{n}\rightarrow\infty$ as $n\rightarrow\infty$).
Such extension will be the subject of future research.

Finally, we would like to emphasize that, even if $\eta(\theta)<0$, it may be possible to correctly recover the signs of $\beta$ with a relatively high probability.
What the converse in Theorem \ref{result:generalized_irrepresentability_condition} says is that this probability is bounded away from $1$ in the limit.

\subsection{Simplification of the GI condition under linear models and Gaussian predictors}

Our next result gives sufficient conditions for the $\eta(\alpha, \beta)$ to be computable directly from the covariance of the predictors.
This result is limited to $\ell_{1}$-penalized linear models as defined in \eqref{equation:definition_penalized_estimate_linear_model}.
Since the loss function only depends on $\bfX$ through an affine transformation, the Hessian $H(a, b)$ of the risk function $R(a, b)$ as well as the covariance matrix of scores $J(a, b)$ involves the expected value of an expression involving the second order cross products in the matrix $Q(\bfX)$ defined as
\begin{eqnarray}
	\label{equation:demi_hessian}
	Q(\bfX)
	& := &
	\left[
	\begin{array}{cc}
	1 & \bfX^{T} \\
	\bfX& \bfX\bfX^{T}
	\end{array}
	\right].
\end{eqnarray}

\begin{theorem}
	\label{result:simplified_gi_condition}
	Let the coefficients of a linear model $(\alpha, \beta)$ be as defined in \eqref{equation:definition_risk_minimizer_linear_model}.
	If
	\begin{enumerate}
		\item [a)] $\bfX\sim N(0, \Sigma)$, and 
		\item [b)] the Hessian of the risk function in \eqref{equation:definition_risk_function_linear_model} can be written in the form
		\begin{eqnarray}
			\label{equation:linear_model_risk_hessian}
			H\left(\alpha,\beta\right)
			& = &
			\bbE
			\left[
			\bbE
			\left[
			Q(\bfX)
			\left|
			\vphantom{\int}
			\alpha+\bfX^{T}\beta
			\right.
			\right]
			\cdot
			w(\alpha+\bfX^{T}\beta)
			\right], \mbox{ for some function } w:\bbR\rightarrow\bbR, 
		\end{eqnarray}
	\end{enumerate}
	then
$
		\eta(\alpha, \beta)
		=
		1- 
		\left[\bbE\left(\bfX_{\mcalA^{c}}\bfX_{\mcalA    }^{T}\right)\right]
		\left[\bbE\left(\bfX_{\mcalA    }\bfX_{\mcalA    }^{T}\right)\right]^{-1}
		\sign(\beta_{\mcalA}).
$
\end{theorem}

A proof is given in Appendix \ref{section:proof}.
Theorem~\ref{result:simplified_gi_condition} tells us that, for zero-mean Gaussian predictors and loss functions whose Hessian can be expressed as a weighted ``average'' of second moments of $\bfX$ conditional on the linear predictor variable $\bfM_{\alpha, \beta}(\bfX) := \alpha+\bfX^{T}\cdot\beta$, the GI condition can be computed directly from the matrix of second moments $\bfE\left[\bfX\bfX^{T}\right]$.
In Section \ref{section:linear_classifiers}, we will see that Theorem~\ref{result:simplified_gi_condition} holds for linear SVM and logistic regression classifiers.s
Besides the particular cases studied in Section \ref{section:linear_classifiers}, we notice that Theorem~\ref{result:simplified_gi_condition} can find ample use for $\ell_{1}$ penalized estimates in view of our next result.
\begin{corollary}
	\label{result:simplified_gi_condition_smooth_linear_models}
	Suppose that:
	\begin{itemize}
		\item [a)] $\bfX\sim N(0, \Sigma)$, 
		\item [b)] $L(\bfZ, t) = L(\bfY, a + b^{T}\bfX)$, 
		\item [c)] $L(\bfY, a+\bfX^{T}b)$ is twice differentiable in  its second argument for almost every $\bfY$, and 
		\item [d)] $\bfY \perp \bfX | \alpha+\beta^{T}\bfX$.
	\end{itemize} 
	Then, 
$
		\eta(\alpha, \beta)
		=
		1- 
		\left[\bbE\left(\bfX_{\mcalA^{c}}\bfX_{\mcalA    }^{T}\right)\right]
		\left[\bbE\left(\bfX_{\mcalA    }\bfX_{\mcalA    }^{T}\right)\right]^{-1}
		\sign(\beta_{\mcalA}).
$
\end{corollary}
\begin{proof}
	Let $\frac{\partial^{2} L(\bfY, \alpha+\bfX^{T}\beta)}{(\partial v)^{2}}$ denote the second derivative of $L$ with respect to its second argument.
	Since $\bfY \perp \bfX | \alpha+\beta^{T}\bfX$, we get
	\begin{eqnarray*}
		H(\alpha, \beta)
		& = &
		\bbE
		\left[
		\bbE
		\left[
		Q(\bfX)
		\cdot
		\frac{\partial^{2} L(\bfY, \alpha+\bfX^{T}\beta)}{(\partial v)^{2}}
		\left|
		\vphantom{\frac{\partial^{2} L(\bfY, \alpha+\bfX^{T}\beta)}{(\partial v)^{2}}}
		\alpha + \bfX^{T}\beta
		\right.
		\right]
		\right]
		\\
		& = & 
		\bbE
		\left[
		\bbE
		\left[
		Q(\bfX)
		\left|
		\vphantom{\frac{\partial^{2} L(\bfY, \alpha+\bfX^{T}\beta)}{(\partial v)^{2}}}
		\alpha + \bfX^{T}\beta
		\right.
		\right]
		\cdot
		\bbE
		\left[
		\frac{\partial^{2} L(\bfY, \alpha+\bfX^{T}\beta)}{(\partial v)^{2}}
		\left|
		\vphantom{\frac{\partial^{2} L(\bfY, \alpha+\bfX^{T}\beta)}{(\partial v)^{2}}}
		\alpha + \bfX^{T}\beta
		\right.
		\right]
		\right].
	\end{eqnarray*}
	Condition (b) in Theorem \ref{result:simplified_gi_condition} is thus satisfied with $w(\alpha+\beta^{T}\bfX) = \bbE\left[\frac{\partial^{2} L(\bfY, \alpha+\bfX^{T}\beta)}{(\partial v)^{2}}\left|\vphantom{\frac{\partial^{2} L(\bfY, \alpha+\bfX^{T}\beta)}{(\partial v)^{2}}}\alpha + \bfX^{T}\beta\right.\right]$.
\end{proof}

Corollary \ref{result:simplified_gi_condition_smooth_linear_models} shows that, if the predictors are Gaussian and the response $\bfX$ only depends on $\bfX$ through an affine transform, the conditions for model selection consistency of many Generalized Linear Models \citep[GLMs][]{nelder:1972:generalized-linear-models} only depends on the covariance between relevant and irrelevant predictors even if the model is not correctly specified.
For canonical GLMs, condition (d) can be relaxed as the weight function can be shown not to depend on the response $\bfY$.

As we will see in the case of the hinge loss, twice differentiability of the loss with respect to its second argument is not essential.
For condition (b) in Theorem \ref{result:simplified_gi_condition} to be satisfied, what seems to be essential is that the loss has the form shown in \eqref{equation:definition_penalized_estimate_linear_model} and that $\bfY$ is conditionally independent of $\bfX$ given $\alpha+\bfX^{T}\beta$.


\section{Application to SVM and logistic regression classifiers}
\label{section:linear_classifiers}

We now obtain the limiting behavior of some linear classifiers to study the model selection consistency of their $\ell_{1}$-penalized estimates.
We will use these results along with Theorem \ref{result:generalized_irrepresentability_condition} to study the model selection consistency of $\ell_{1}$-penalized SVM and logistic regression classifiers.
The response variable $\bfY\in\{-1,1\}$ is modeled in terms of a linear transformation of a set of predictors $\bfX\in \bbR^{p}$.
Setting some of the coefficients on the estimates of the $\beta$ parameter to zero corresponds to eliminating some effects from the model thus leading to more interpretable models.

In what follows, we will characterize the asymptotic behavior of the loss functions associated to logistic regression and support vector machines.
Logistic regressions are a particular case of Generalized Linear Models \citep{nelder:1972:generalized-linear-models, mccullagh:1989:generalized-linear-models} and are widely used by statisticians when modeling the outcome of binomial variables.
Support vector machines \citep{cortes:1995:support-vector-networks} are amply used for obtaining linear classification rules and is based on the hinge-loss function.
For both the logistic regression and support vector machines, the corresponding loss functions are often interpreted as convex surrogates for the $0-1$ classification loss \citep{zhang:2004:statistical-behavior-and-consistency-of-classification-methods-based-on-convex-risk-minimization, bartlett:2005:convexity}.
Efficient algorithms exist for obtaining both the $\ell_{1}$-norm penalized SVM \citep{zhu:2004:l1_penalized_svms} and logistic  \citep{park:2006:l1-regularization-path-algorithms-for-generalized-linear-models} classifiers.
Both SVM classification and logistic regression have been used to select relevant predictors in areas as diverse as genomics \citep[see, for instance][]{guyon:2002:gene-selection-for-cancer-classification-using-support-vector-machines, meier:2006:logistic-glasso} and text categorization \citep[see, for instance][]{joachims:1998:text-categorization-with-support-vector-machines:-learning-with-many-relevant-features, genkin:2007:logistic_lasso_text}.

We now set up terminology and notation we will use in connection with the SVM and logistic classifiers for the remainder of the paper.
Given a value for the parameters in the linear classification model $t=(a, b)\in \bbR^{1+p}$, a \textit{linear classification rule} is defined as
\begin{eqnarray}
	\label{equation:linear_classification_rule}
	\hat{\bfY}\left(\bfX|t\right) = \mbox{sign}\left(a+\bfX^{T}b\right).
\end{eqnarray}
The separating hyperplane $\mcalH(t)$ associated to a linear classification rule as in \eqref{equation:linear_classification_rule} is defined as
\begin{eqnarray}
	\label{equation:definition_separating_hyperplane}	
	\mcalH(t) & := & \{\bfx\in \bbR^{p}: a + \bfx b = 0\}, \mbox{ for } t=(a, b).
\end{eqnarray}
The set $\mcalH(t)$ defines the boundary in the predictor space between the points where, for the linear classification rule based in $t$, the response variable is predicted to be $1$ (the set $\{\bfx: \hat{\bfY}(\bfx|t)=1\} = \{\bfx: a+\bfx^{t}b>0\}$) from the points where $\bfY$ is predicted to be $-1$ (the set $\{\bfx: \hat{\bfY}(\bfx|t)=1\} = \{\bfx: a+\bfx^{t}b<0\}$).
We call \textit{optimal linear classification rule} the classification rule corresponding to setting $t=\theta$ and the \textit{estimated linear classification rules} the classification rule formed by setting $t=\hat{\theta}_{n}(\lambda_{n})$ with $\hat{\theta}_{n}(\lambda)$ as defined in 
\eqref{equation:definition_penalized_estimate}.
We define the \textit{linear predictor variable}:
\begin{eqnarray}
	\label{equation:definition_margin_variable}
	\bfM & := & \alpha+\bfX^{T}\beta,
\end{eqnarray}
which measures the distance from point $\bfX$ to the separating hyper-plane defined by the optimal linear classifier.
If the distribution of $\bfY$ only depends on $\bfX$ through a linear combination, both the linear SVM and logistic regression are known to recover the optimal Bayes classifier.
We also define the true conditional distribution of $\bfY$ given $\bfX$ as:
\begin{eqnarray}
	\label{equation:probability_function}
	p(\bfX) = \bbP\left(\bfY=1|\bfX\right).
\end{eqnarray}

\subsection{Regularity conditions and model selection consistency for SVM and logistic classifiers}

Before we can use the results from Section \ref{section:application} to study and compare the $\ell_{1}$-penalized SVM and logistic linear classifiers, we must obtain a set of conditions on the joint distribution of $(\bfX, \bfY)$ such that the hinge and logistic regression losses satisfy the requirements on loss functions laid out in Assumption Set 1.
Conditions C1-C3 -- C a mnemonic for classification -- gives one such a set of sufficient conditions in terms of the marginal distribution of the predictors $\bfX$ and the conditional distribution of $\bfY$ given $\bfX$.

\begin{assumption}{Classification Assumptions (CA)}
\begin{enumerate} 
	\item [C1.] $\var\left[\bfX\left|\bfY\right.\right]\in \bbR^{p\times p}$ is a positive definite matrix for $\bfY\in\{1,-1\}$, 
	\item [C2.] The distribution of $\bfX$ has a density $f_{\bfX}(\bfx)>0$, for all $\bfx\in \bbR^{p}$, and 
	\item [C3.] $p(\bfX)\in (0,1)$ for almost every $\bfX$, that is, for all values $\bfX$ in the support of the distribution of $\bfX$, $\bfY$ can assume any of its two possible values;
\end{enumerate}
\end{assumption}

Condition C1 rules out the case of perfectly correlated predictors and is required to ensure uniqueness of the minimizer $\theta$ as defined in \eqref{equation:definition_risk_minimizer}.
Assumptions C2 and C3 are used to ensure the SVM and logistic regression loss functions satisfy the assumptions in Lemma \ref{result:uniform_convergence_for_loss}, but can be relaxed.

The remainder of this section describes linear SVM and logistic classification, shows how Conditions C1-C3 ensure their corresponding loss functions are amenable to the theory laid out in Section \ref{section:application} and provide expressions for the covariance matrix of scores $J(\theta)$ and the Hessian $H(\theta)$ for the risk functions associated to the linear SVM and logistic regression classifiers.

\subsubsection{Logistic Regression}
\label{section:logistic_risk}

The canonical logistic regression is one instance of Generalized Linear Model \citep{nelder:1972:generalized-linear-models} where the probability of $\bfY=1$ is modeled as:
\begin{eqnarray}
	\label{equation:logistic_conditional_probability_model}
	\bbP\left(\bfY=1|\bfX, a, b\right)
	& = &
	\frac{\exp(a+\bfX^{T}b)}{1+\exp(a+\bfX^{T}b)},
\end{eqnarray}
where $a\in\bbR$ and $b\in\bbR^{p}$ are parameters to be determined.
The population parameters $\alpha$ and $\beta$ are defined as the minimizers of the Kullbach-Leibler divergence between the true conditional distribution of $\bfY$ given $\bfX$ and the Bernoulli distribution with parameter given by \eqref{equation:logistic_conditional_probability_model}.
The corresponding loss function is:
\begin{eqnarray}
	\label{equation:logistic_loss_definition}
	L \left(\bfY, a+b^{T}\bfX\right) 
	& = & 
	-a\cdot \bbI\left(\bfY=1\right) - \bbI\left(\bfY=1\right)\cdot \bfX^{T}\cdot b  + \log\left[1-\exp\left(a+\bfX^{T}\cdot b\right)\right],
\end{eqnarray}
where $\bbI\left(\bfY=1\right)$ is the indicator of $\bfY=1$.
An estimate for $\theta=\left(\alpha, \beta\right)$ is obtained by minimizing the empirical risk with respect to $t=(a, b).$

\begin{lemma}
	\label{result:variance_and_hessian_for_logistic_loss}
	Suppose that the conditions in Assumption Set 3 are observed.
	Then, the logistic regression loss function \eqref{equation:logistic_loss_definition} satisfies the conditions in Assumption Set 1 with:
	\begin{eqnarray*}
		J(\theta) 
		& = & 
		\bbE
		\left[
		Q(\bfX)
		\cdot 
		\left[
		p\left(\bfX\right)
		-
		2 
		\cdot
		p\left(\bfX\right)
		\cdot
		\frac
		{\exp\left(\alpha+\bfX\beta\right)}
		{1+\exp\left(\alpha+\bfX\beta\right)}
		+
		\left(
		\frac
		{\exp\left(\alpha+\bfX\beta\right)}
		{1+\exp\left(\alpha+\bfX\beta\right)}
		\right)
		^{2}
		\right]
		\right], \mbox{ and }
		\\
		H(\theta) 
		& = & 
		\bbE
		\left[
		\vphantom{\int}
		Q(\bfX)
		\cdot 
		\frac
		{\exp\left(\alpha+\bfX\beta\right)}
		{\left(1+\exp\left(\alpha+\bfX\beta\right)\vphantom{\int}\right)^{2}}
		\right].
	\end{eqnarray*}
\end{lemma}

A proof is given in Appendix \ref{section:proof}. 
%
The expression for the Hessian of the logistic loss can be rewritten as
\begin{eqnarray}
	\label{equation:logistic_hessian_in_conditional_form}
		H(\theta) 
		& = & 
		\bbE
		\left[
		\vphantom{\int}
		\bbE
		\left[
		Q(\bfX)
		\left|
		\vphantom{\int_{0}^{1}}
		\alpha+\bfX\beta
		\right.
		\right]	
		\cdot 
		\frac
		{\exp\left(\alpha+\bfX\beta\right)}
		{\left(1+\exp\left(\alpha+\bfX\beta\right)\vphantom{\int}\right)^{2}}
		\right],
\end{eqnarray}
and hence satisfies the conditions of Theorem \ref{result:simplified_gi_condition} even if the model is not correctly specified.
Indeed, the Hessian for the logistic risk does not depend on the distribution of $\bfY$ at all.

In addition, equation \eqref{equation:logistic_hessian_in_conditional_form} tells us that the Hessian for the logistic regression risk function is a weighted average of second moment matrices conditional on the linear predictor variable $\alpha+\bfX\beta$.
Because $\frac{\exp\left(\alpha+\bfX\beta\right)}{\left(1+\exp\left(\alpha+\bfX\beta\right)\right)^{2}}$ is an even function of the linear predictor variable, the matrices of conditional second moments at predictor variables that are equally distant from the separating hyperplane are equally weighted.
In addition, the higher weight is given to $\bbE\left[Q(\bfX)\left|\vphantom{\int_{0}^{1}}\alpha+\bfX\beta=0\right.\right]$ and the weighting is decreasing on the absolute value of the linear predictor variable.
As a result, in what concerns asymptotic model selection consistency of $\ell_{1}$-norm penalized logistic coefficient estimates, the correlation structure of the predictors on regions closer to the separating hyperplane have the most importance confirming the margin phenomenon observed earlier in non-parametric works by \citet{audibert:2007:fast-learning-rates-for-plug-in-classifiers} and \citet{steinwart:2007:fast-rates-for-support-vector-machines-using-gaussian-kernels}.

\subsubsection{The parametric SVM: linear classification with the Hinge loss function}
\label{section:svm_risk}

Classification by means of Support Vector Machines with linear kernel was first introduced in the case where it is possible to perfectly separate the space of predictors $\bfX$ according to the the binomial variable $\bfY$.
In that setting, the SVM parameters define a hyper-plane (characterized by the parameters $\alpha, \beta$) that maximizes the gap between the classes:
\begin{eqnarray*}
	\begin{array}{cccccl}
		(\hat{\alpha}, \hat{\beta}) & = & \argmin\limits_{a, b} & & \|b\|_{2} \\
                                    &   & \mbox{s.t.}              & & \bfY_{i}\cdot\left(a-\bfX_{i}^{T}b\right) \ge 1, & \mbox{ for all } i = 1, \ldots, n.
	\end{array}
\end{eqnarray*}
To adapt this method to the ``no perfect-separation'' case, non-negative slack variables $\xi_{i}$ are introduced and the optimization problem becomes
\begin{eqnarray*}
	\begin{array}{cccccl}
		(\hat{\alpha}, \hat{\beta}) & = & \argmin\limits_{a, b} & & \|\beta\|_{2} + C\cdot\sum_{i=1}\xi_{i}\\
                                    &   & \mbox{s.t.}                      & & \bfY_{i}\cdot\left(a-\bfX_{i}^{T}b\right) \ge 1-\xi_{i}, 
                                                                           & \mbox{ for all } i = 1, \ldots, n, \mbox{ and } \\
		                            &   &                                  & & \xi_{i}\ge 0, 
		                                                                   & \mbox{ for all } i = 1, \ldots, n,
	\end{array}
\end{eqnarray*}
where $C$ is a constant controlling the trade-off between margin maximization and total amount of slack.
The ``lack of fit'' in SVM is measured by the total distance of the misclassified points to the classification boundary, represented as the sum of the slack variables.
The Euclidean norm acts as a penalization term: in the perfect separation case it unsureness uniqueness of the solution.
More consistently with the form in \eqref{equation:definition_penalized_estimate}, the empirical SVM parameter estimates can then be rewritten as:
\begin{eqnarray*}
	\label{equation:hinge_loss_definition}
		(\hat{\alpha}, \hat{\beta}) & = & \argmin\limits_{a, b} \sum_{i=1}L(\bfY_{i}, a+b^{T}\bfX) + \lambda\cdot\|\beta\|_{2}, \mbox{ with }\\ 
	L(\bfY_{i}, a+b^{T}\bfX) 
	& = & 
	\left[\vphantom{\int}1-\bfY_{i}\left(a-\bfX_{i}^{T} b\right)\right]\cdot\bbI\left[1-\bfY_{i}\cdot\left(a-\bfX_{i}^{T} b\ge 0\right)\right], 
	\mbox{the \textit{hinge-loss function.}}
	\end{eqnarray*}
Here, we will consider the hinge loss on its own, in the spirit of the ``assembling'' Lemma \ref{result:assembling_theorem}.
The next result establishes that under the conditions of Assumption Set 3, the hinge loss satisfies the assumptions in Theorem \ref{result:uniform_convergence_for_loss}.

\begin{lemma}
	\label{result:variance_and_hessian_for_hinge_loss}
	Suppose the conditions in Assumption Set 3 hold.
	If in addition $\beta\neq0$, then the hinge loss function \eqref{equation:hinge_loss_definition} satisfies the conditions in Assumption Set 1 with:
	\begin{eqnarray*}
		J(\theta) 
		& = & 
		\bbE
		\left[
		\vphantom{\frac{\int_{0}^{1}}{\int_{0}^{1}}}
		\left[
		\vphantom{\int_{0}^{1}}
		p(\bfX)\cdot\bbI(1-\alpha-\bfX^{T} \beta\ge0)
		+ 
		\left(1-p(\bfX)\right)\cdot\bbI(1+\alpha+\bfX^{T} \beta\ge0)
		\right]
		\cdot
		Q(\bfX)
		\right], \mbox{ and }
		\\
		H(\theta) 
		& = & 
		\bbE
		\left[
		\vphantom{\frac{\int_{0}^{1}}{\int_{0}^{1}}}
		\left[
		\vphantom{\int_{0}^{1}}
		p(\bfX)
		\cdot
		\delta\left(1-\alpha-\bfX^{T}\beta\right)
		+
		\left(1-p(\bfX)\right)
		\cdot
		\delta\left(1+\alpha+\bfX^{T}\beta\right)
		\right]
		\cdot
		Q(\bfX)
		\right], 
	\end{eqnarray*}
	where $\delta$ denotes Dirac delta function.
\end{lemma}

The expressions for $J(\theta)$ and $H(\theta)$ in Lemma \ref{result:variance_and_hessian_for_hinge_loss} closely parallel results by \citet{koo:2008:a-bahadur-representation-of-the-linear-support-vector-machine} concerning the Bahadur representation of the linear support vector machines.
In Appendix \ref{section:proof}, we present an alternative proof similar in spirit to the construction by \citet{phillips:1991:a-shortcut-to-lad-estimator-asymptotics}.
In \citet{koo:2008:a-bahadur-representation-of-the-linear-support-vector-machine} conditions ensuring $\beta\neq0$ are also obtained.

Borrowing from the terminology for support vector regression, we call the set where $\alpha+\bfX^{T}\beta=-1$ the negative ``elbow'' of the SVM risk.
Similarly, the positive ``elbow'' of the SVM risk is the set where $\alpha+\bfX^{T}\beta=1$.
Assuming that $\bfY$ is independent of $\bfX$ given $\alpha+\bfX^{T}\beta$, the expression for the Hessian in Lemma \ref{result:variance_and_hessian_for_hinge_loss} can be rewritten in a more revealing form in terms of conditional expectations at these elbows of the SVM risk:
\begin{eqnarray}
	\label{equation:svm_hessian_in_conditional_form}
	\begin{array}{ccl}
	H(\theta) 
	& = & 
	\bbE
	\left[
	\bbE
	\left[
	\vphantom{\frac{\int_{0}^{1}}{\int_{0}^{1}}}
	Q(\bfX)
	\left|
	\vphantom{{\int_{0}^{1}}}
	\alpha+\bfX^{T}\beta
	\right.
	\right]
	\cdot 
	\bbP\left(\vphantom{\int_{0}^{1}}\bfY=1\left|\vphantom{\int}\alpha+\bfX^{T}\beta\right.\right)
	\cdot
	\delta(1-\alpha-\bfX^{T}\beta)
	\right]
	\\
	& &
	+
	\bbE
	\left[
	\bbE
	\left[
	\vphantom{\frac{\int_{0}^{1}}{\int_{0}^{1}}}
	Q(\bfX)
	\left|
	\vphantom{{\int_{0}^{1}}}
	\alpha+\bfX^{T}\beta
	\right.
	\right]
	\cdot 
	\left[
	1-\bbP\left(\vphantom{\int_{0}^{1}}\bfY=1\left|\vphantom{\int}\alpha+\bfX^{T}\beta\right.\right)
	\right]
	\cdot
	\delta(-1-\alpha-\bfX^{T}\beta)
	\right]
	\\
	& = & 
	\bbE
	\left[
	\vphantom{\frac{\int_{0}^{1}}{\int_{0}^{1}}}
	Q(\bfX)
	\left|
	\vphantom{{\int_{0}^{1}}}
	\alpha+\bfX^{T}\beta=1
	\right.
	\right]
	\cdot 
	\bbP\left(\vphantom{\int_{0}^{1}}\bfY=1\left|\vphantom{\int}\alpha+\bfX^{T}\beta=1\right.\right)
	\cdot
	\tilde{f}(1)
	\\
	& &
	+
	\bbE
	\left[
	\vphantom{\frac{\int_{0}^{1}}{\int_{0}^{1}}}
	Q(\bfX)
	\left|
	\vphantom{{\int_{0}^{1}}}
	\alpha+\bfX^{T}\beta=-1
	\right.
	\right]
	\cdot 
	\bbP\left(\vphantom{\int_{0}^{1}}\bfY=-1\left|\vphantom{\int}\alpha+\bfX^{T}\beta=-1\right.\right)
	\cdot
	\tilde{f}(-1),
	\end{array}
\end{eqnarray}
where $\tilde{f}$ denotes the density of the linear predictor variable $\alpha+\bfX^{T}\beta$.
This representation for the Hessian of the linear SVM risk (expected value of the hinge loss over $\bfY$ and $\bfX$) shows that if $\bfY$ is independent of $\bfX$ given $\alpha+\bfX^{T}\beta$ the hinge loss function is amenable to the results in Theorem \ref{result:simplified_gi_condition}.
It also provides many insights into the behavior of the linear SVM classifier.

Equation \eqref{equation:svm_hessian_in_conditional_form} tells us that the Hessian of the SVM risk is a weighted sum of the conditional second moments of the predictors given that the linear predictor variable $\alpha+\bfX^{T}\beta$ is at the elbows of the SVM risk.
According to Theorem \ref{result:generalized_irrepresentability_condition}, the generalized irrepresentability condition is not affected if the Hessian matrix is multiplied by a constant.
It follows that, with respect to model selection consistency of $\ell_{1}$-norm penalized linear SVM classifiers, the scalar factors
$\bbP\left(\vphantom{\int_{0}^{1}}\bfY=-1\left|\vphantom{\int}\alpha+\bfX^{T}\beta=-1\right.\right)\cdot\tilde{f}(-1)$
and 
$\bbP\left(\vphantom{\int_{0}^{1}}\bfY=-1\left|\vphantom{\int}\alpha+\bfX^{T}\beta=-1\right.\right)\cdot\tilde{f}(-1)$
only determine the relative importance of the two conditional second moment matrices, $\bbE\left[\vphantom{\frac{\int_{0}^{1}}{\int_{0}^{1}}}Q(\bfX)\left|\vphantom{{\int_{0}^{1}}}\alpha+\bfX^{T}\beta=1\right.\right]$ and $\bbE\left[\vphantom{\frac{\int_{0}^{1}}{\int_{0}^{1}}}Q(\bfX)\left|\vphantom{{\int_{0}^{1}}}\alpha+\bfX^{T}\beta=-1\right.\right]$, in the composition of the Hessian.
If the two conditional moment matrices happen to be equal, the scalar factors have no bearings in whether the generalized irrepresentable condition is met or not.
If the two conditional moment matrices are different, the relative importance of the conditional second moments at the two elbows depends on the density of the linear predictor variable $\alpha+\bfX^{T}\beta$ and how well defined a class is at each of the elbows.
For example, if $\tilde{f}(1)\gg \tilde{f}(-1)$ and $\bbP\left(\vphantom{\int_{0}^{1}}\bfY=1\left|\vphantom{\int}\alpha+\bfX^{T}\beta=1\right.\right)\gg \bbP\left(\vphantom{\int_{0}^{1}}\bfY=-1\left|\vphantom{\int}\alpha+\bfX^{T}\beta=-1\right.\right)$, the SVM Hessian will be largely determined by the second moment of the predictor at the positive elbow $\bbE\left[\vphantom{\frac{\int_{0}^{1}}{\int_{0}^{1}}}Q(\bfX)\left|\vphantom{{\int_{0}^{1}}}\alpha+\bfX^{T}\beta=1\right.\right]$, which in turn will have the most influence in determining whether $\ell_{1}$-norm penalized SVM classifier is model selection consistent.

In addition to determining the weighting between the conditional covariances, the density of the predictors and the probabilities of $\bfY$ belonging to each class on the positive and negative can inflate or deflate the covariance matrix of $\hat{\theta}_{n}$.
Standard results concerning parametric M-estimators \citep[see, for instance][]{bickel:2001:mathematical-statistics, casella:2001:statistical-inference}
yield that $\lim_{n\rightarrow\infty}\var\left[\sqrt{n}\cdot\left(\hat{\theta}_{n}-\theta\right)\right]=H^{-1}(\theta)J(\theta)H^{-1}(\theta)$. As a result, the higher the density of the predictors and the easier the separation of the classes at the elbows, the larger the Hessian and the less variable the coefficients in the SVM classifier.





\section{Simulations}
\label{section:simulations}

We now present a series of simulation results which give empirical evidence supporting the theory for model selection consistency for $\ell_{1}$-penalized linear SVM and logistic regression classifiers.
In addition, we use the simulations to compare the model selection performance of $\ell_{1}$-penalized linear SVM and logistic regression classifiers asymptotically and in finite samples.
To avoid a simulation set-up that is biased in favor of either linear SVMs or logistic regression, we base our conclusions on randomly selected joint distributions for $(\bfY, \bfX)$, where $\bfY$ is the binomial response variable and $\bfX$ is the predictor.
We start off by detailing how the designs used throughout this section are sampled.

\subsection{Randomly constructing joint distributions $(\bfY, \bfX)$}
\label{section:sampling_joint_distributions}

Throughout our simulation experiments, we will call a design the joint distribution of $(\bfY, \bfX)$ characterized by the parameters of the conditional distribution of $\bfY$ given $\bfX$ and the the distribution of the predictors $\bfX$.

The conditional distribution of the binomial random variable $\bfY\in\{-1,1\}$ given $\bfX\in\bbR^{p}$ is characterized by a \textit{probability profile function} $g:\bbR\rightarrow (0,1)$, an intercept $\zeta$ and a normal direction to the separating hyperplane $\nu \in \bbR^{p}$.
Given these elements, we set $\bbP(\bfY=1|\bfX) = g(\zeta+\bfX^{T}\nu)$, so $\bfY$ is independent of $\bfX$ given any one-to-one transformation of $\zeta+\bfX^{T}\nu$, in particular $\alpha+\bfX^{T}\beta$.
In all designs, we set $\zeta=0$.
Given a number of non-zero terms $q$, we partition the normal direction according to $\nu = \left[\begin{array}{cc}\nu_{\mcalA} & \bfzero\end{array}\right] \in \bbR^{q}\times \bbR^{p-q}$.
The non-zero component of the normal direction to the separating hyper-plane $\bfv_{\mcalA}$ is sampled uniformly on the unit sphere on $\bbR^{q}$.
One problem with this sampling scheme is that it may result in tiny coefficients which are hard to detect in finite samples, thus complicating the comparison between asymptotic and experimental results.
To avoid such tiny coefficients, we discard directions $\nu_{\mcalA}$ having $\max_{1\le j\le q} |\nu_{j}|/\min_{1\le j\le q} |\nu_{j}|>5$.
To provide stronger evidence in favor of Theorem \ref{result:generalized_irrepresentability_condition}, we will consider two different probability profile functions $g$:
	\begin{eqnarray*}
		\begin{array}{cccll}			
		\mbox{the logistic function,} & g_{1}(r)  & := & \frac{\exp(r)}{1+\exp(r)}, & \mbox{ and }	
		\\
		\mbox{the ``blip'' function,} & g_{2}(r) & := & \frac{1}{2}\left(1+r\cdot\exp\left(\frac{1-r^{2}}{2}\right)\right).
		\end{array}
	\end{eqnarray*}
The logistic function ($g_{1}$) is the canonical link for Bernoulli GLM models.
The ``blip'' function ($g_{2}$) concentrates all the action close to the separation boundary between the classes and is thus expected to favor SVM classifiers.

For the distribution of the predictors, we consider two families of distributions: Gaussian and mixture of Gaussian distributions.
For the Gaussian predictors, the mean is fixed at $\bfzero\in\bbR^{p}$ and a covariance matrix $\Sigma\in \bbR^{p}$ is sampled as follows.
First, $\tilde{\Sigma}\in \bbR^{p}$ is sampled from a Wishart$(\bfI_{p}, p, p)$ distribution, where $\bfI_{p}$ is the identity matrix. Then, $\tilde{\Sigma}\in \bbR^{p}$ is normalized to have unit diagonal and $\Sigma=\gamma\cdot\tilde{\Sigma}$ with the scalar $\gamma>0$ chosen so that $\nu^{T}\Sigma\nu=\sigma^{2}$, where $\sigma^{2}$ is a parameter controlling the variance of $\bfX$.
The mixed Gaussian predictors are a mixture of two Gaussian distributions with equal proportions, common variance $\Sigma$ and symmetric means $\mu$ and $-\mu$.
The parameter $\mu$ is randomly selected as $\mu = \frac{4}{5}\cdot \tilde{\mu}\cdot \sigma$, where $\tilde{\mu} = |\chi|\cdot \nu+ \mbox{w}$, with $\chi\sim N(0,1)$ and $\mbox{w}\sim N(0, \bfI_{p})$.
The common variance matrix of the components of the mixture of Gaussian is sampled similarly as the covariance matrix for the Gaussian case, with the difference that $\gamma$ is chosen so $\nu^{T}\Sigma\nu=\frac{9}{25}\cdot \sigma^{2}$.
The factors $\frac{4}{5}$ for $\mu$ and $\frac{9}{25}$ for $\Sigma$ are used to ensure that the contribution of the mean and variance for the second moment $\bbE \bfX\bfX^{T} = \mu\mu^{T} + \Sigma$ is somewhat balanced.

To obtain the population parameter $\theta=(\alpha, \beta)$ as defined in \eqref{equation:definition_risk_minimizer} for each of the sample designs, we first notice that the probability profile functions satisfy $g_{j}(z) = 1-g_{j}(-z)$, for $j=1,2$, $z\in \bbR$ and the distribution of the predictors are symmetric about zero.
It thus follows that the optimization problem defining $\theta$ is symmetric about $0\in \bbR^{p}$ and we have $\alpha=0$ for all designs.
Then, because $\bbP(\bfY=1|\bfX)$ only depends on $\bfX$ through $\bfX^{T}\nu$, $\beta$ has the form $\beta = c^{*}\cdot \nu$, for some scalar $c^{*}\in \bbR$.
The value of $c^{*}$ that minimizes the risk is obtained by numerically minimizing the average of the risk function conditional on $\bfX$ for a large sample ($10^6$) from the predictor distribution.
For any given design, the value of $c^{*}$ differs depending on the risk function being used.

\subsection{Model selection consistency and the GI condition for linear classifiers}

We now provide empirical evidence of the validity of Theorem \ref{result:generalized_irrepresentability_condition} for $\ell_{1}$-norm penalized linear SVM and logistic regression classifiers.
According to Theorem \ref{result:generalized_irrepresentability_condition}, the proportion of paths containing sign-correct estimates should approach $1$ as $n\rightarrow\infty$ if the GI index $\eta(\theta)$ is positive.

To estimate the probability that a sample regularization path contains a sign-consistent estimate for a given design, we used replicates of the regularization path by sampling from the joint distribution of $(\bfY, \bfX)$ and computing the regularization path for $\ell_{1}$-norm penalized linear SVM \citep{li:2006:the-l1-norm-quantile-regression} and logistic regression \citep{park:2006:l1-regularization-path-algorithms-for-generalized-linear-models}.
To compute the GI index for a given design, we can use the expressions in equations \eqref{equation:logistic_hessian_in_conditional_form} and \eqref{equation:svm_hessian_in_conditional_form} in conjunction with the expressions for the conditional second moments of Gaussian and mixed Gaussian random variables shown in Appendix \ref{section:analytical_hessians}.

Figures~\ref{figure:classification_rnorm_model_selection_phase_transition}~and~\ref{figure:classification_rmixednorm_model_selection_phase_transition} show plots of the proportion of sample regularization paths containing a sign-correct solution against the GI index $\eta(\theta)$ under various conditions.
In all cases considered, the proportion of times the $\ell_{1}$-penalized classifier contains a sign-correct estimate in its regularization path increases as $n$ increases if $\eta(\theta)>0$.

Figures~\ref{figure:classification_rnorm_model_selection_phase_transition}~and~\ref{figure:classification_rmixednorm_model_selection_phase_transition} also show that  that, in most cases, correct recovery of the signs of $\theta$ is harder if $\eta(\theta)<0$.
One notable exception occurs for mixed Gaussian predictors under the ``blip'' conditional probability profile.
In that case, it is possible to have a high probability of correct sign recovery even under $\eta(\theta)<0$.
Notice that this result does not contradict Theorem \ref{result:generalized_irrepresentability_condition}. 
Even though there the probability that the signs will not be recovered correctly is never zero if $\eta(\theta)<0$, it can be quite small.
A more careful analysis of the probability of correct sign recovery must take into account the variance of the estimates $\hat{\beta}^{(n)}_{j}(\lambda_{n})$ with indices in $\mcalA^{c} = \{j\in 1, \ldots, p: \beta_{j}=0\}$.

It also possible to notice that, given the asymptotic nature of the results, the probability of correct sign-recovery can still be small for smaller sample sizes $n$ and for larger number of predictors $p$ especially under a fainter signal (``blip'' conditional probability profile).
The extension to the theory in Section \ref{section:application} to the non-parametric case $p=p_{n}\rightarrow\infty$ can potentially offer more precise answers on the how the total number of predictors affects the chance that the regularization path contains a sign-correct model.

\setlength{\tabcolsep}{2pt}
\renewcommand{\figurewidth}{0.2}
\begin{figure}[p]
\begin{center}
\begin{tabular}{|c|cc|cc|}
\cline{2-5}
\multicolumn{1}{c|}{}
&
\multicolumn{2}{|c|}{SVM}
&
\multicolumn{2}{|c|}{Logistic}
\\
\multicolumn{1}{c|}{}
&
\begin{tabular}{c}
$p=08$
\\
$q=04$
\end{tabular}
&
\begin{tabular}{c}
$p=16$
\\
$q=04$
\end{tabular}
&
\begin{tabular}{c}
$p=08$
\\
$q=04$
\end{tabular}
&
\begin{tabular}{c}
$p=16$
\\
$q=04$
\end{tabular}
\\
\hline
\begin{tabular}{c}
\begin{sideways}
	{
	\begin{small}
	\begin{tabular}{c} Gaussian \\ $g_{1}$ profile \\ $n=100$, $\sigma=0.5$ \end{tabular}
	\end{small}
	}
\end{sideways}
\end{tabular}
&
\begin{tabular}{c}
\includegraphics[width=\figurewidth\textwidth, angle=270]
{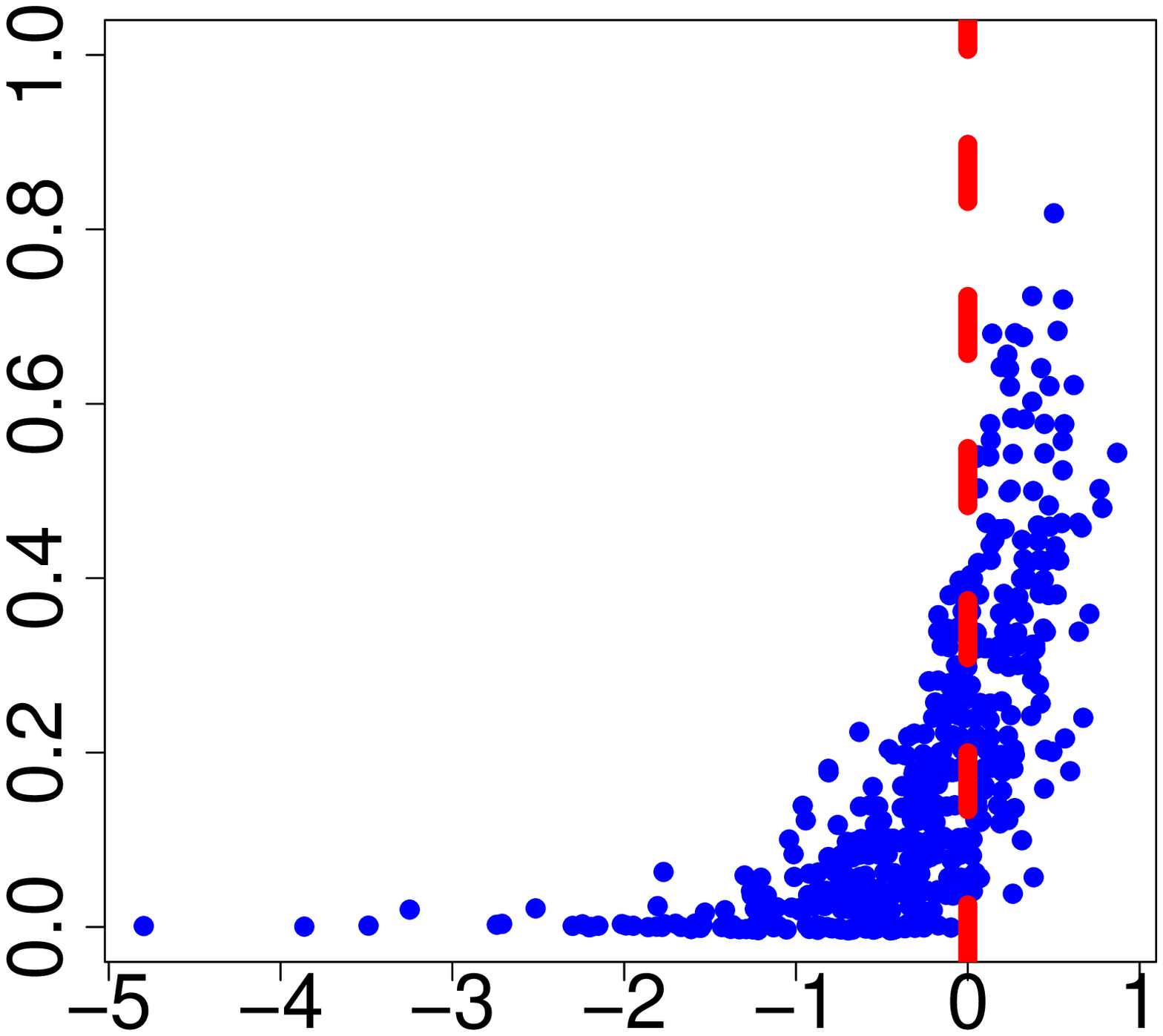}
\end{tabular}
&
\begin{tabular}{c}
\includegraphics[width=\figurewidth\textwidth, angle=270]
{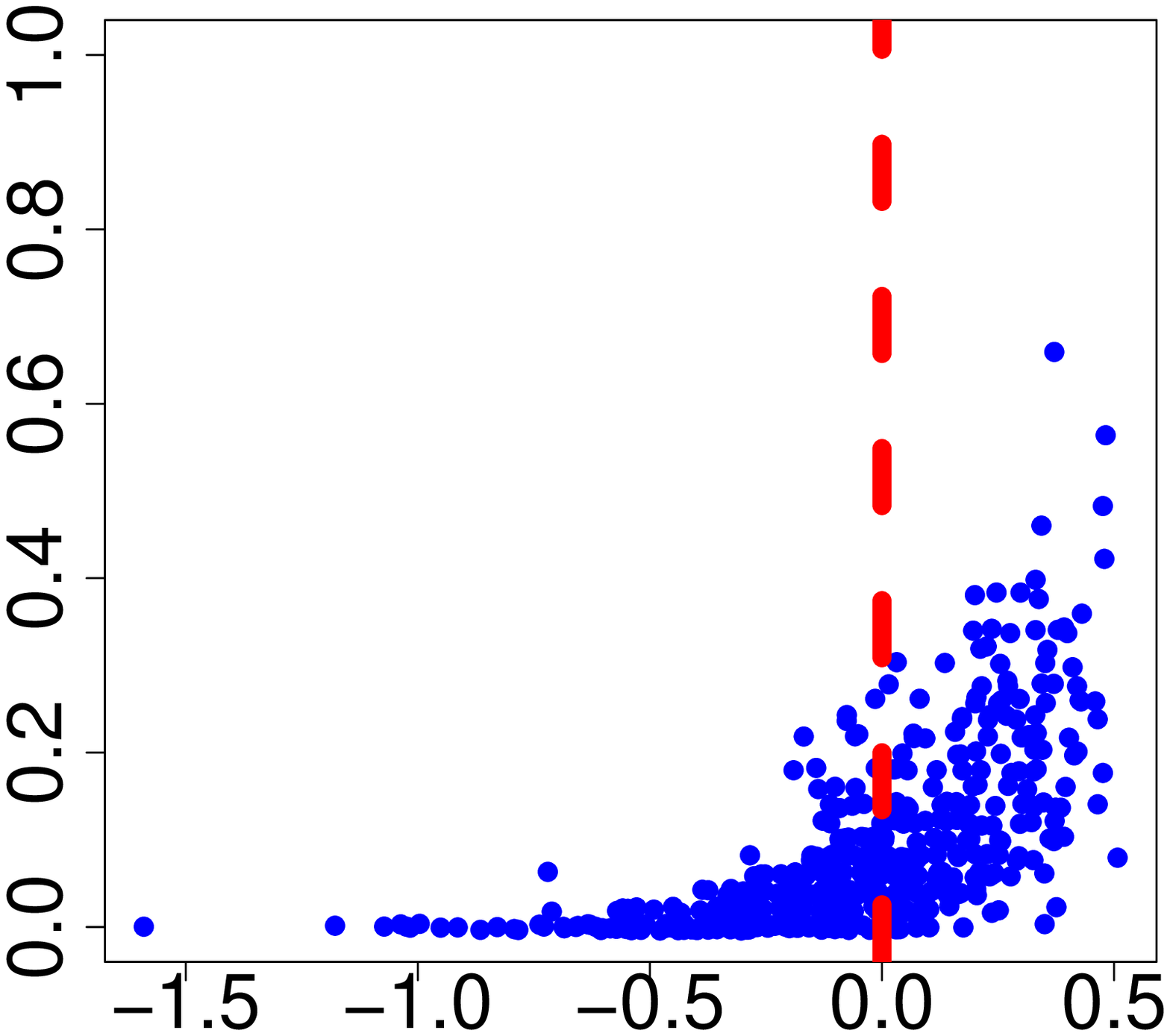}
\end{tabular}
&
\begin{tabular}{c}
\includegraphics[width=\figurewidth\textwidth, angle=270]
{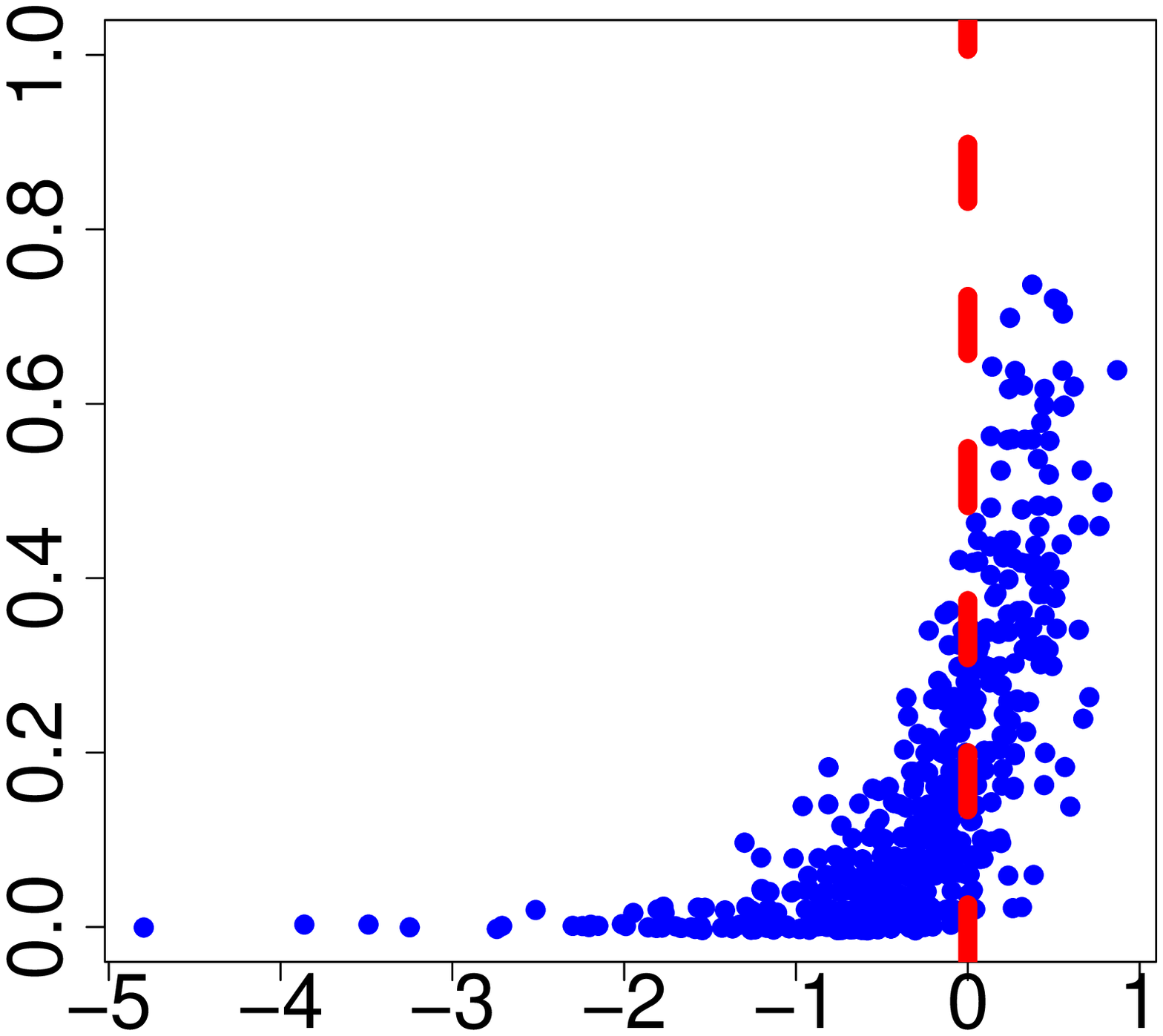}
\end{tabular}
&
\begin{tabular}{c}
\includegraphics[width=\figurewidth\textwidth, angle=270]
{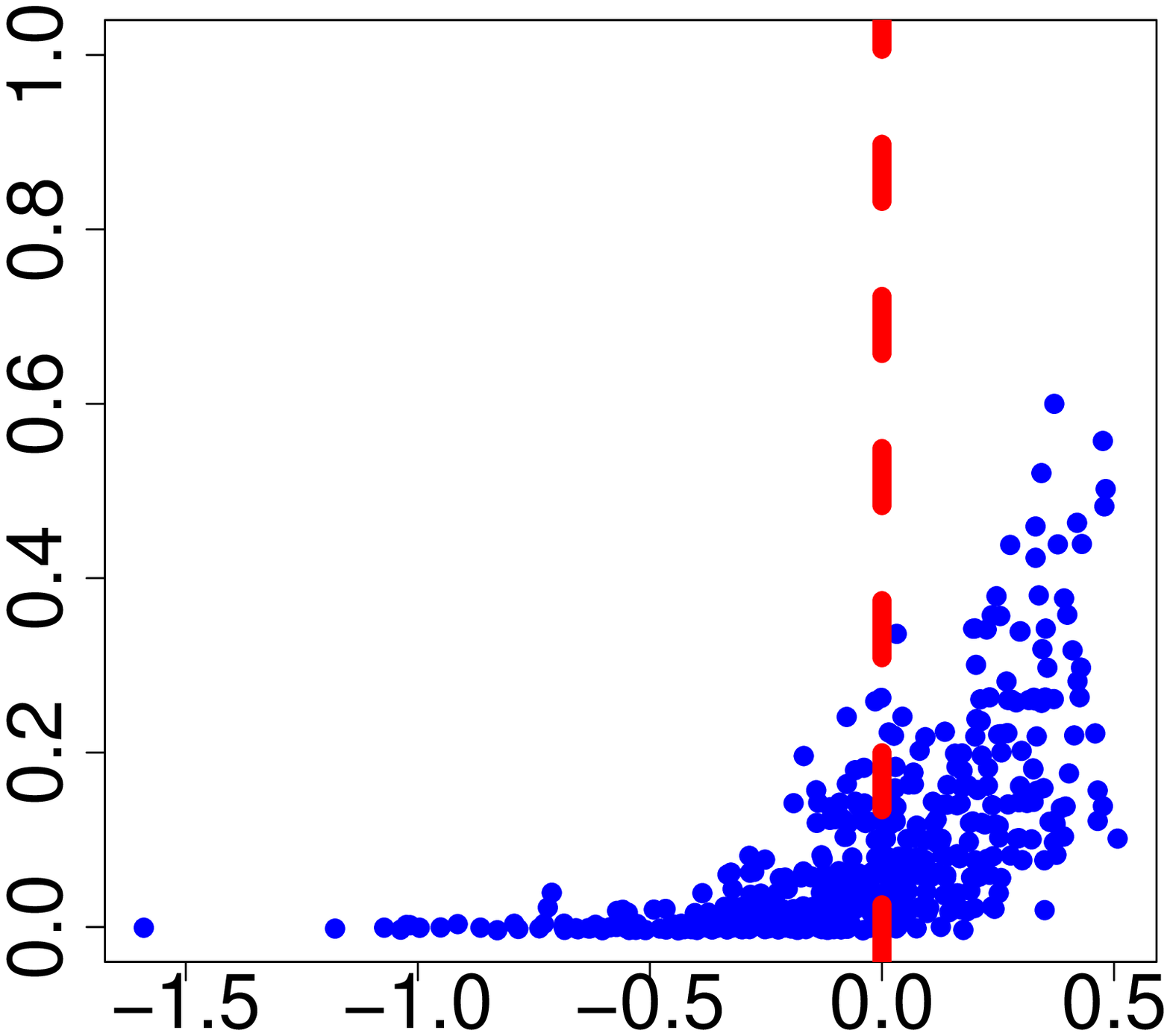}
\end{tabular}
\\
\hline
\begin{tabular}{c}
\begin{sideways}
	{
	\begin{small}
	\begin{tabular}{c} Gaussian \\ $g_{1}$ profile \\ $n=500$, $\sigma=0.5$ \end{tabular}
	\end{small}
	}
\end{sideways}
\end{tabular}
&
\begin{tabular}{c}
\includegraphics[width=\figurewidth\textwidth, angle=270]
{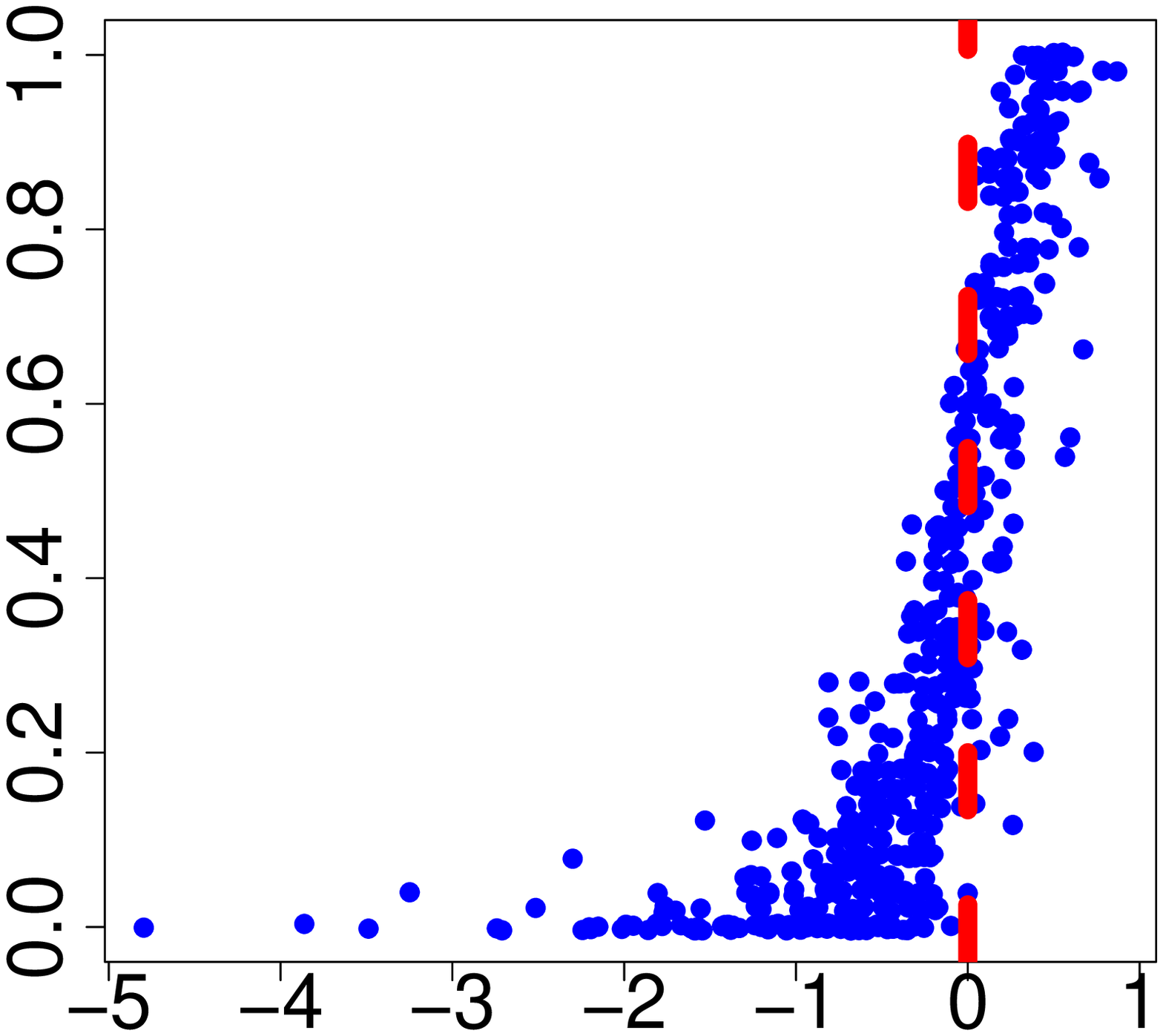}
\end{tabular}
&
\begin{tabular}{c}
\includegraphics[width=\figurewidth\textwidth, angle=270]
{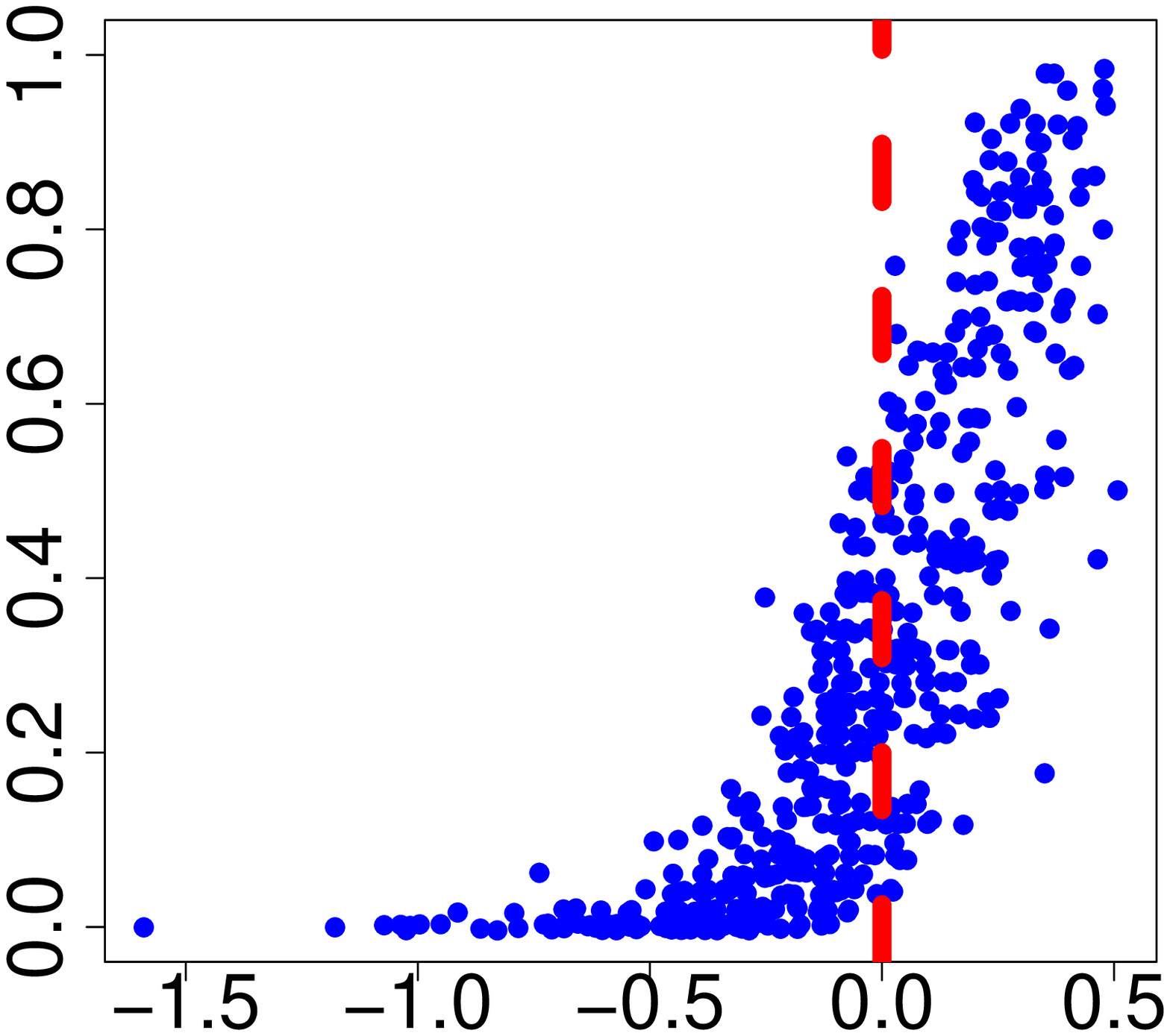}
\end{tabular}
&
\begin{tabular}{c}
\includegraphics[width=\figurewidth\textwidth, angle=270]
{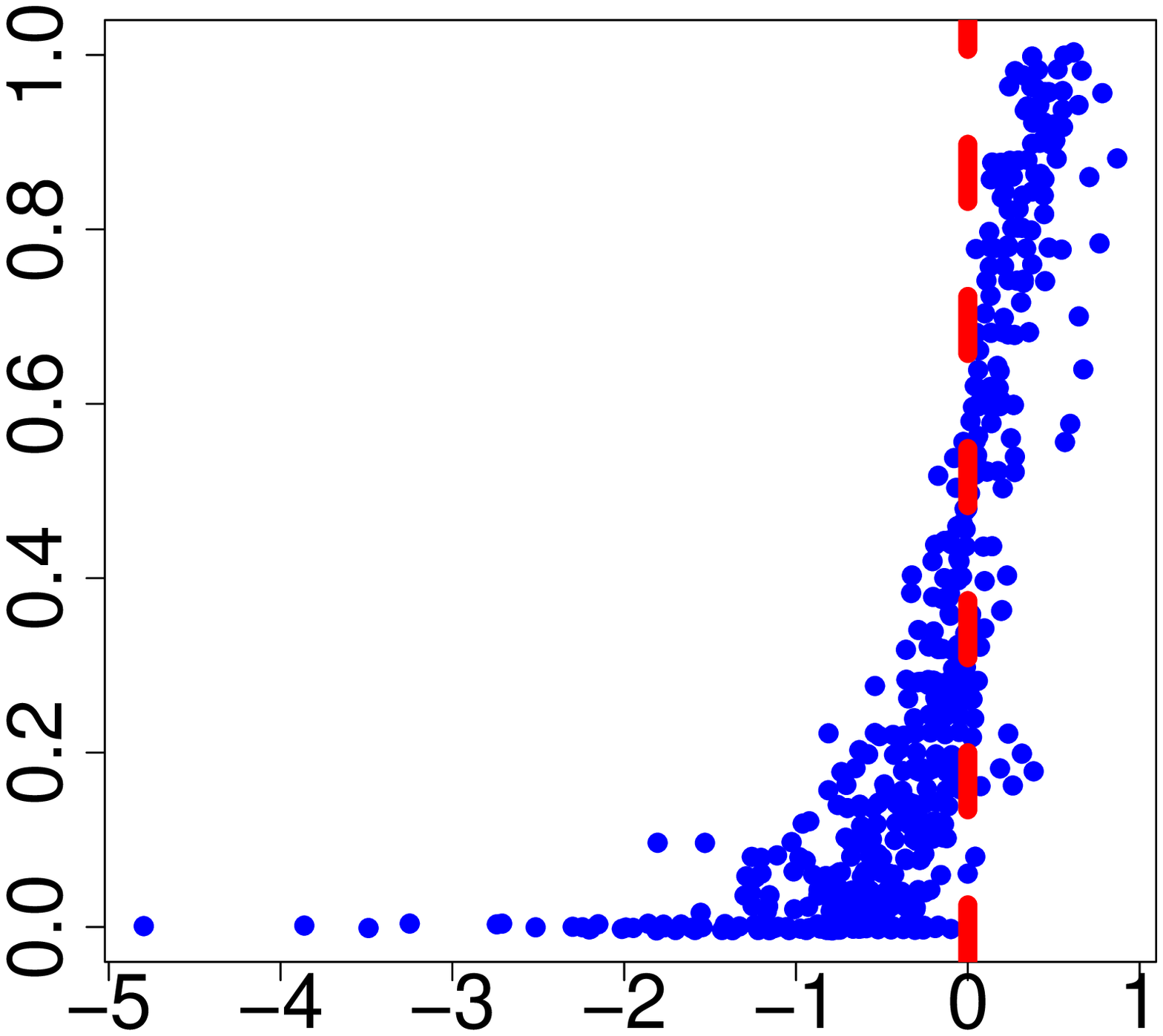}
\end{tabular}
&
\begin{tabular}{c}
\includegraphics[width=\figurewidth\textwidth, angle=270]
{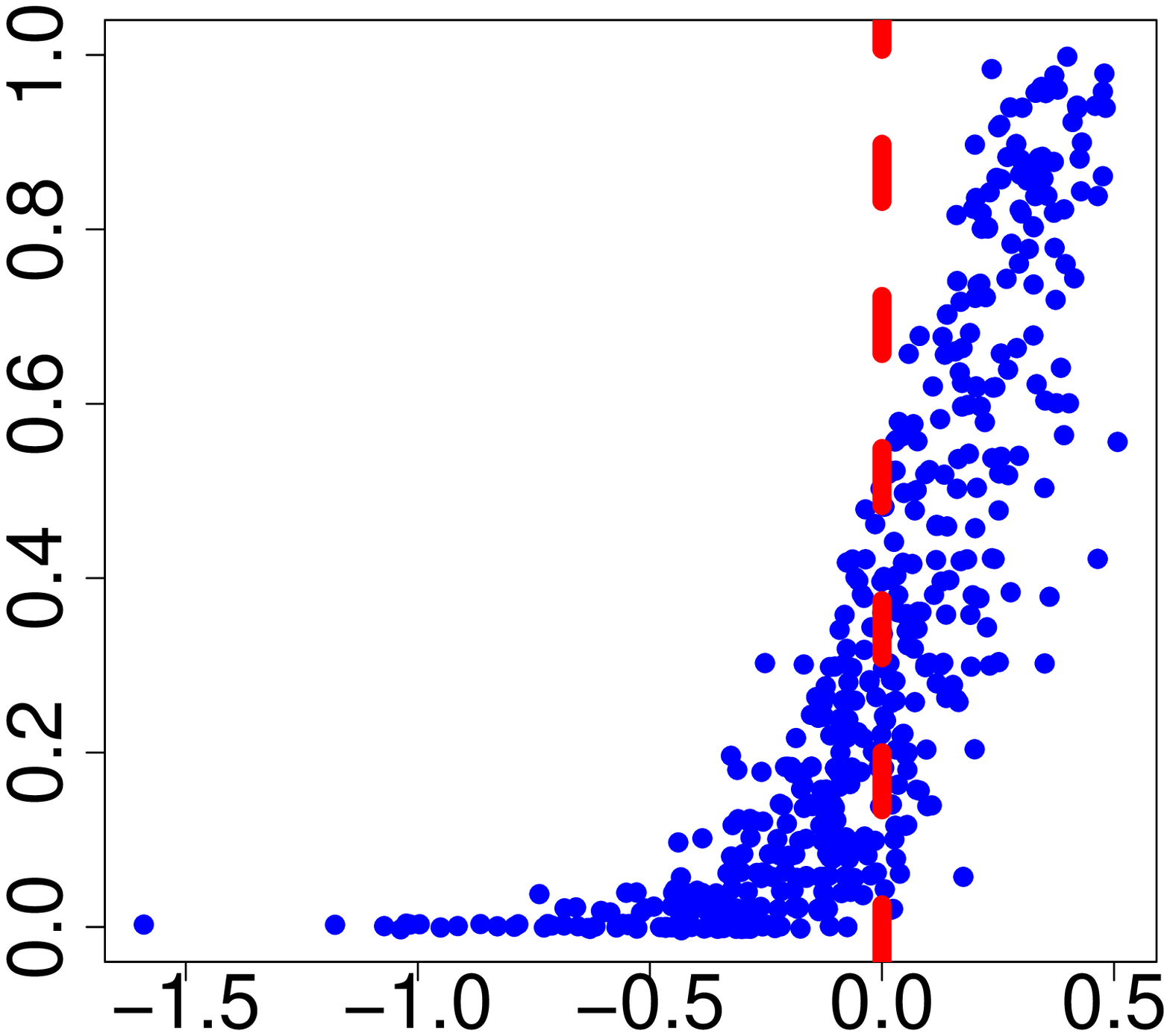}
\end{tabular}
\\
\hline
\begin{tabular}{c}
\begin{sideways}
	{
	\begin{small}
	\begin{tabular}{c} Gaussian \\ $g_{2}$ profile \\ $n=500$, $\sigma=0.5$ \end{tabular}
	\end{small}
	}
\end{sideways}
\end{tabular}
&
\begin{tabular}{c}
\includegraphics[width=\figurewidth\textwidth, angle=270]
{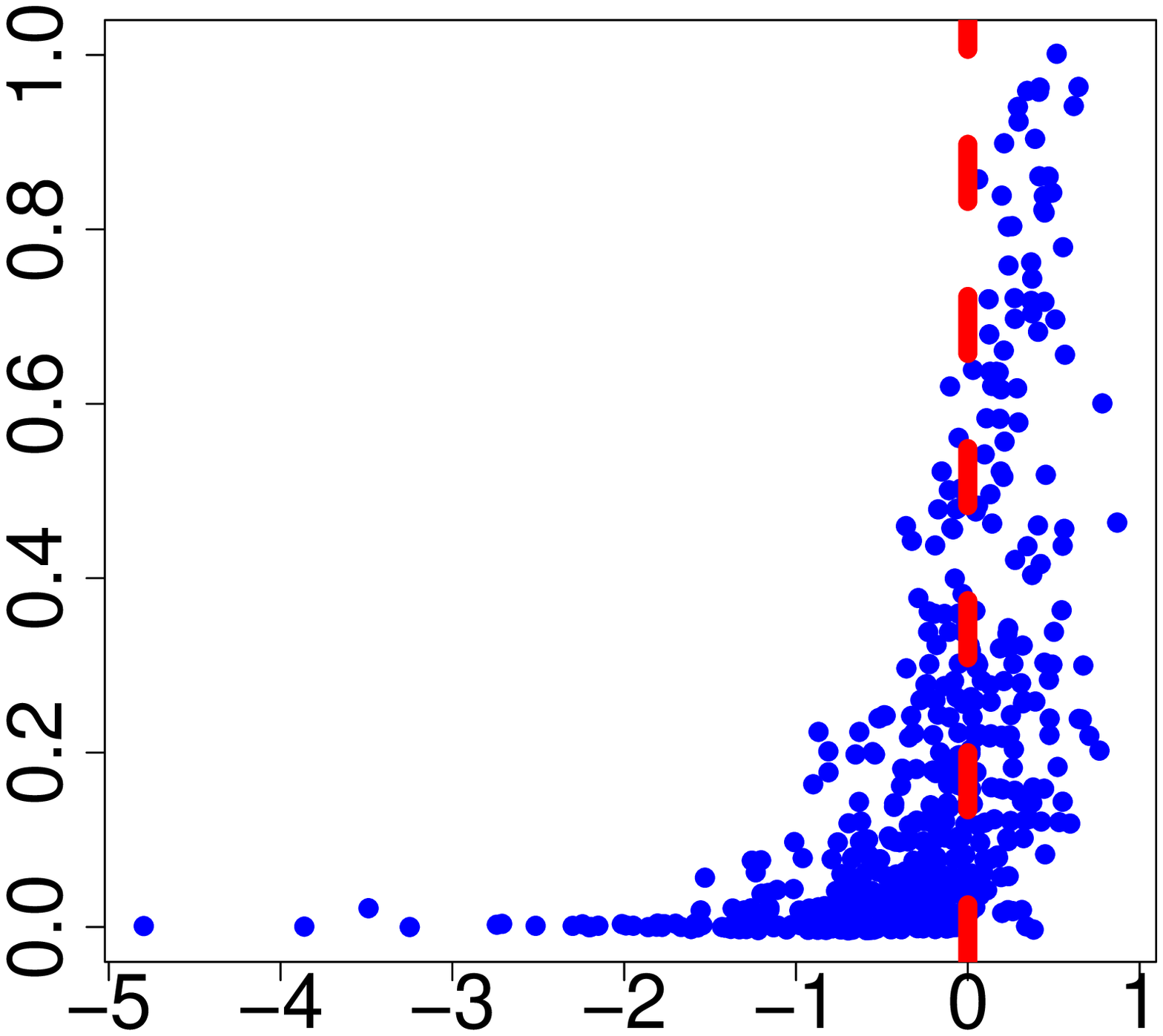}
\end{tabular}
&
\begin{tabular}{c}
\includegraphics[width=\figurewidth\textwidth, angle=270]
{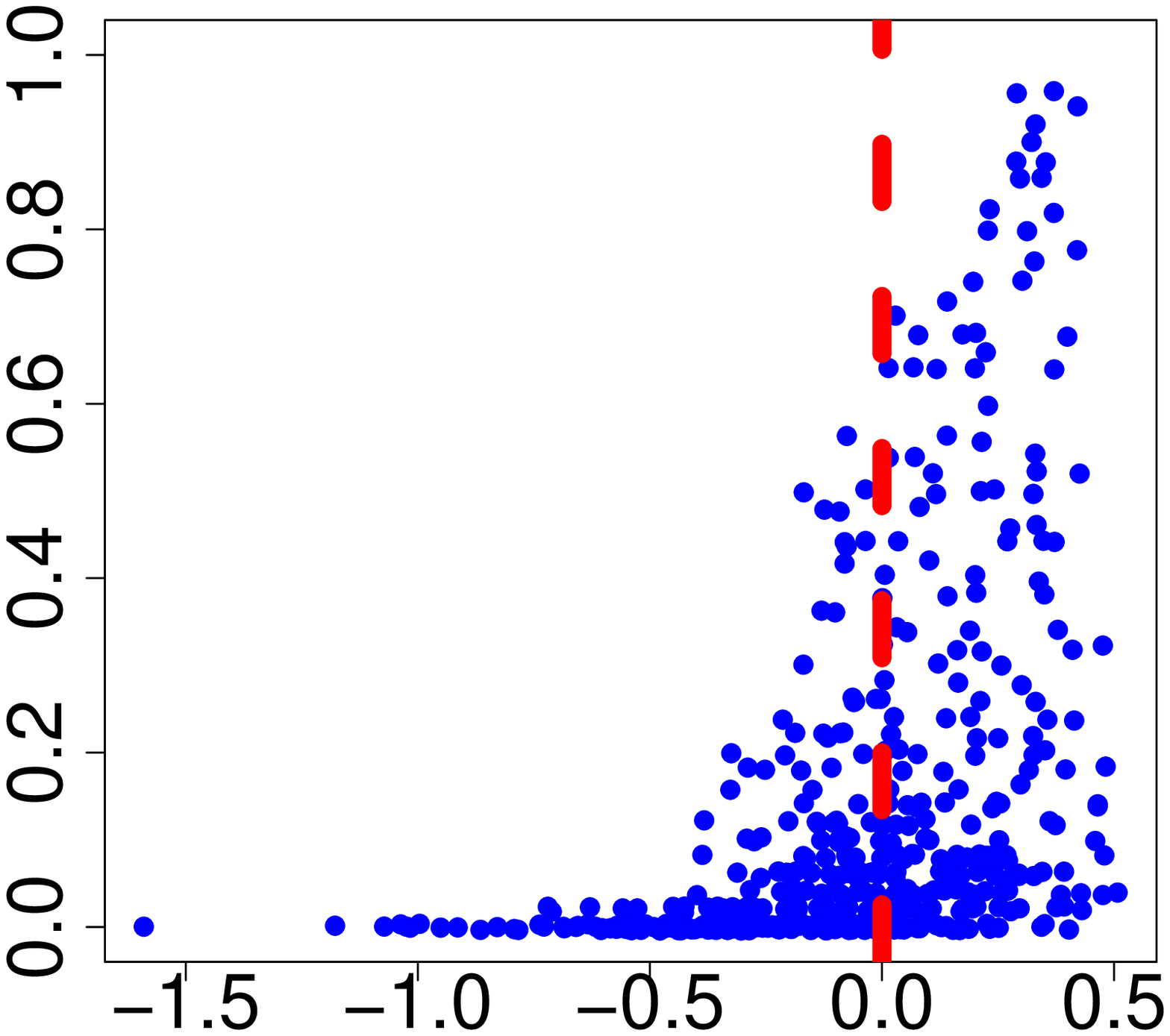}
\end{tabular}
&
\begin{tabular}{c}
\includegraphics[width=\figurewidth\textwidth, angle=270]
{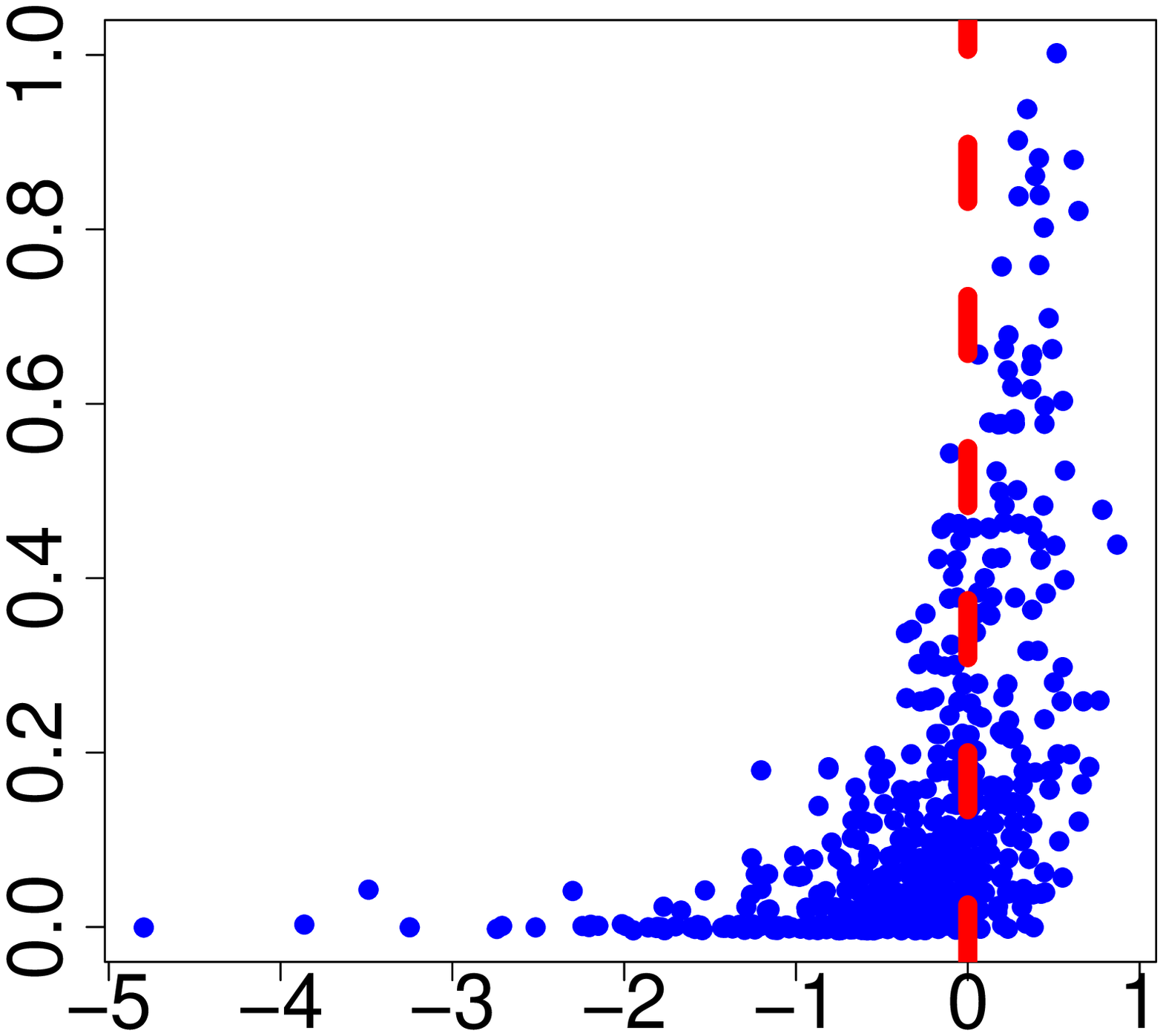}
\end{tabular}
&
\begin{tabular}{c}
\includegraphics[width=\figurewidth\textwidth, angle=270]
{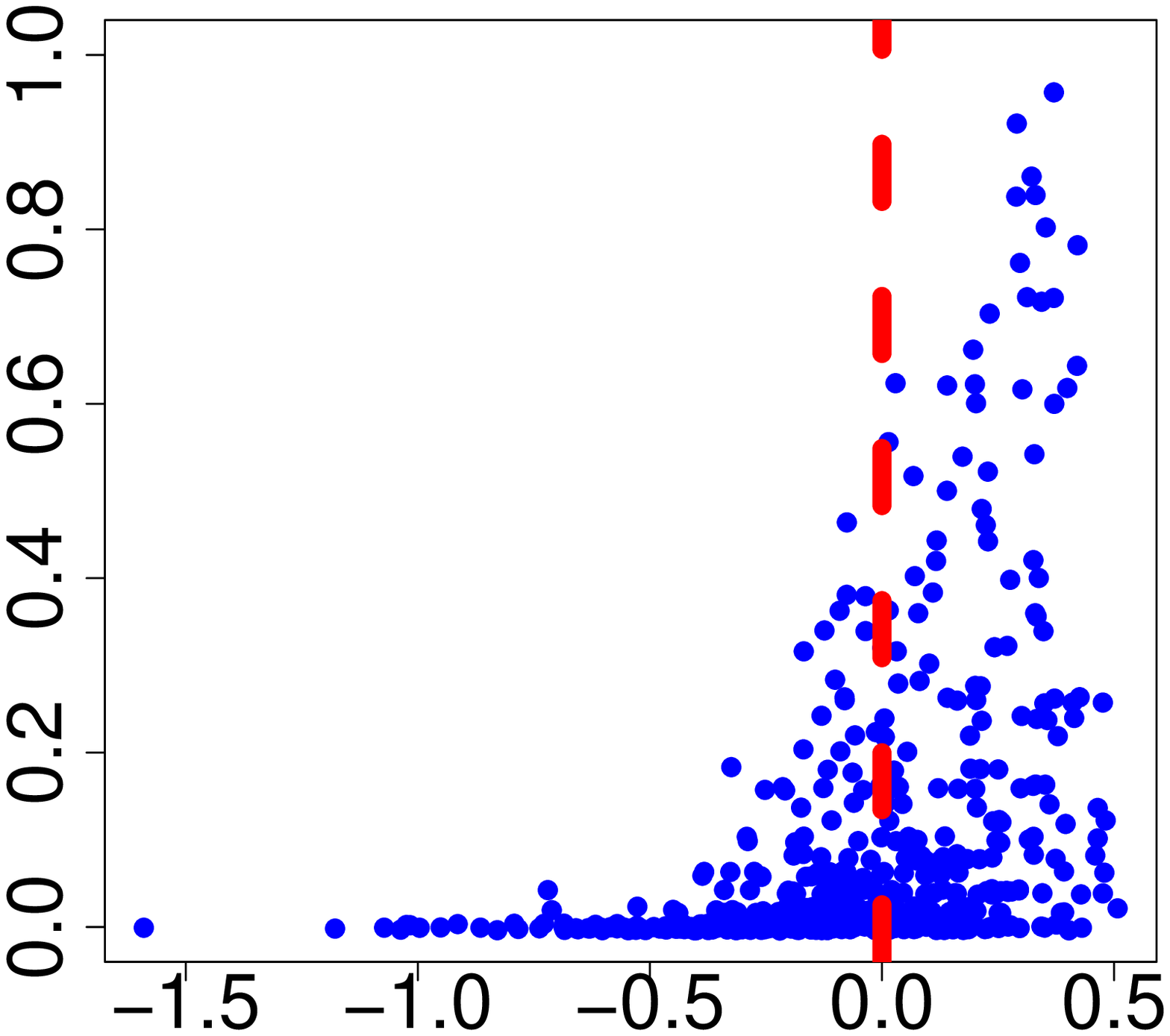}
\end{tabular}
\\
\hline
\begin{tabular}{c}
\begin{sideways}
	{
	\begin{small}
	\begin{tabular}{c} Gaussian \\ $g_{2}$ profile \\ $n=1,000$, $\sigma=0.5$ \end{tabular}
	\end{small}
	}
\end{sideways}
\end{tabular}
&
\begin{tabular}{c}
\includegraphics[width=\figurewidth\textwidth, angle=270]
{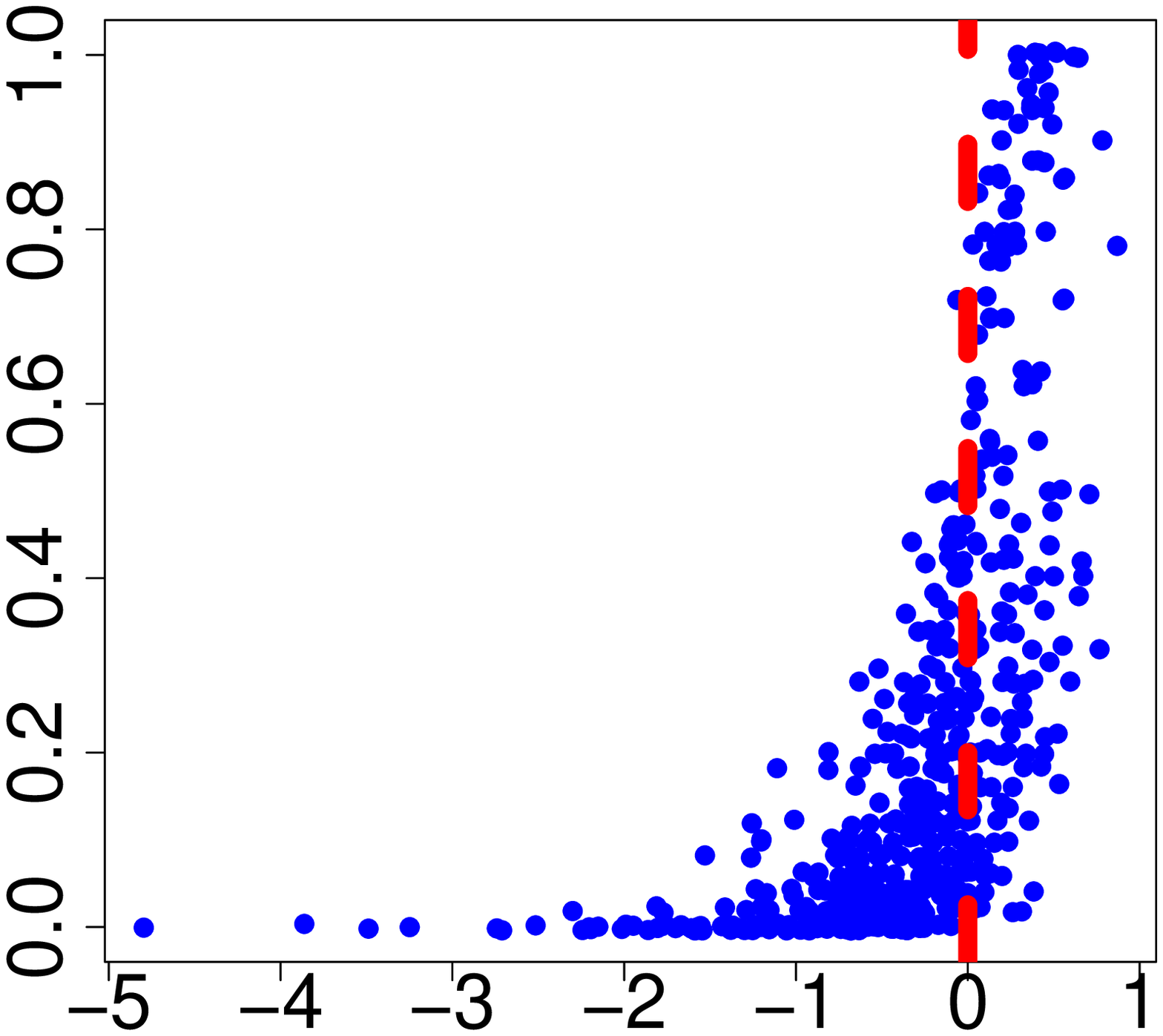}
\end{tabular}
&
\begin{tabular}{c}
\includegraphics[width=\figurewidth\textwidth, angle=270]
{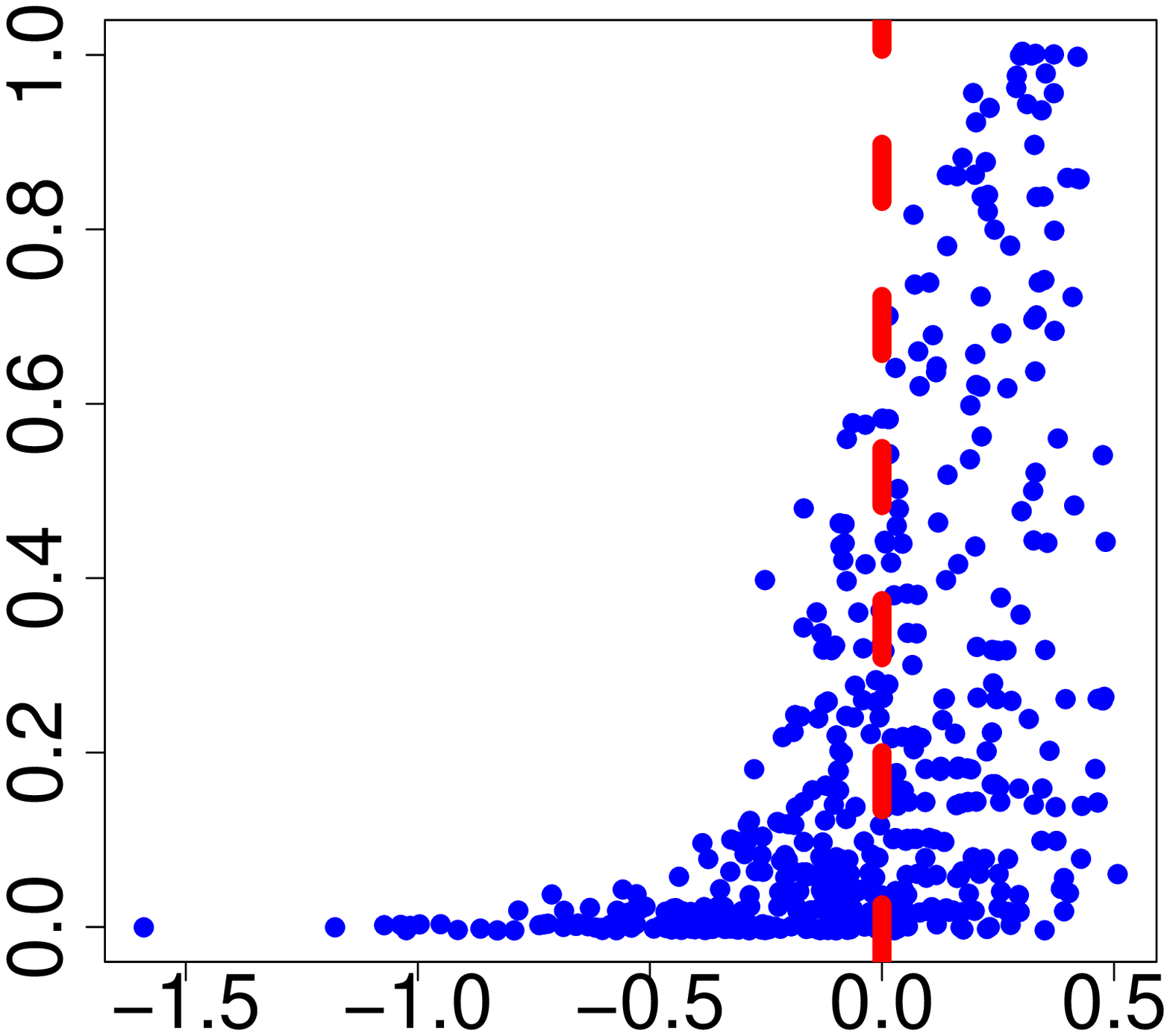}
\end{tabular}
&
\begin{tabular}{c}
\includegraphics[width=\figurewidth\textwidth, angle=270]
{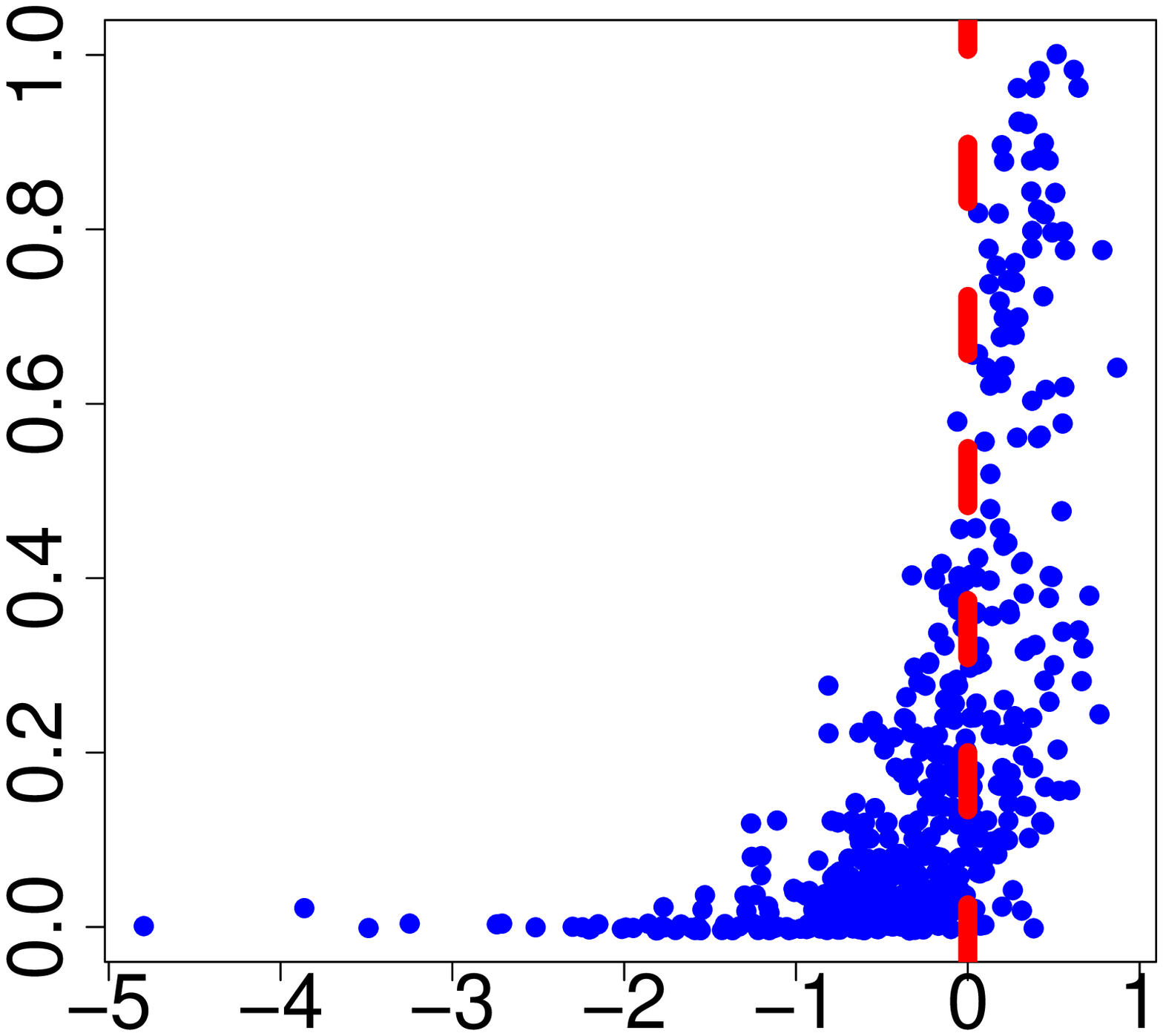}
\end{tabular}
&
\begin{tabular}{c}
\includegraphics[width=\figurewidth\textwidth, angle=270]
{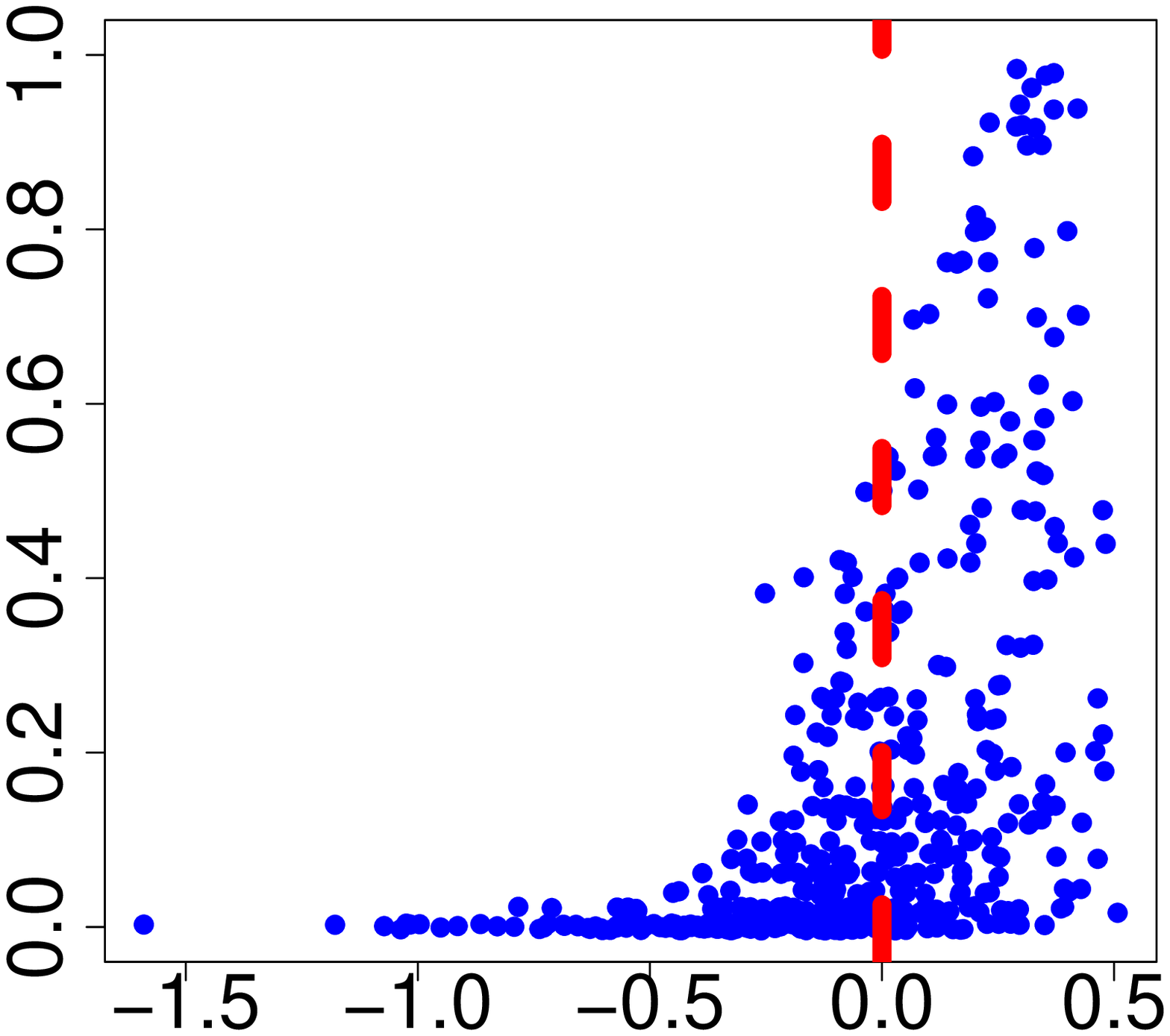}
\end{tabular}
\\
\hline
\end{tabular}
\caption{
\label{figure:classification_rnorm_model_selection_phase_transition}
\textbf{Proportion of sample regularization paths containing a sign-model  vs. GI index $\eta(\theta)$ under Gaussian predictors:}
The proportion at each point is based on 50 replicates of the sample regularization path for the corresponding design.
The results displayed in these panels show good agreement with the theory for sign consistency of general $\ell_{1}$-penalized M-estimators developed in Section \ref{section:application}: for increasing sample sizes, the proportion of paths containing sign correct model approaches one as the sample size increases whenever $\eta(\theta)>0$.
For $\eta(\theta)<0$, the chance of correct sign recovery are low throughout.
Not surprisingly, the asymptotic approximation works better for smaller $p$.
Also notice that the fainter signal of the ``blip'' profile makes the recovery of the correct signs harder.
}
\end{center}	
\end{figure}
\afterpage{\clearpage}

\renewcommand{\figurewidth}{0.20}

\begin{figure}[p]
\begin{center}
\begin{tabular}{|c|cc|cc|}
\cline{2-5}
\multicolumn{1}{c|}{}
&&&&
\\
\multicolumn{1}{c|}{}
&
\multicolumn{2}{|c|}{SVM}
&
\multicolumn{2}{|c|}{Logistic}
\\
\multicolumn{1}{c|}{}
&
\begin{tabular}{c}
$p=08$
\\
$q=04$
\end{tabular}
&
\begin{tabular}{c}
$p=16$
\\
$q=04$
\end{tabular}
&
\begin{tabular}{c}
$p=08$
\\
$q=04$
\end{tabular}
&
\begin{tabular}{c}
$p=16$
\\
$q=04$
\end{tabular}
\\
\hline
\begin{tabular}{c}
\begin{sideways}
	\begin{tabular}{c} 
	Mixed Gaussian \\ $g_{1}$ profile \\ $n=500$, $\sigma=1.0$ 
	\end{tabular}
\end{sideways}
\end{tabular}
&
\begin{minipage}[l]{\figurewidth\textwidth}
\includegraphics[width=\textwidth, angle=270]
{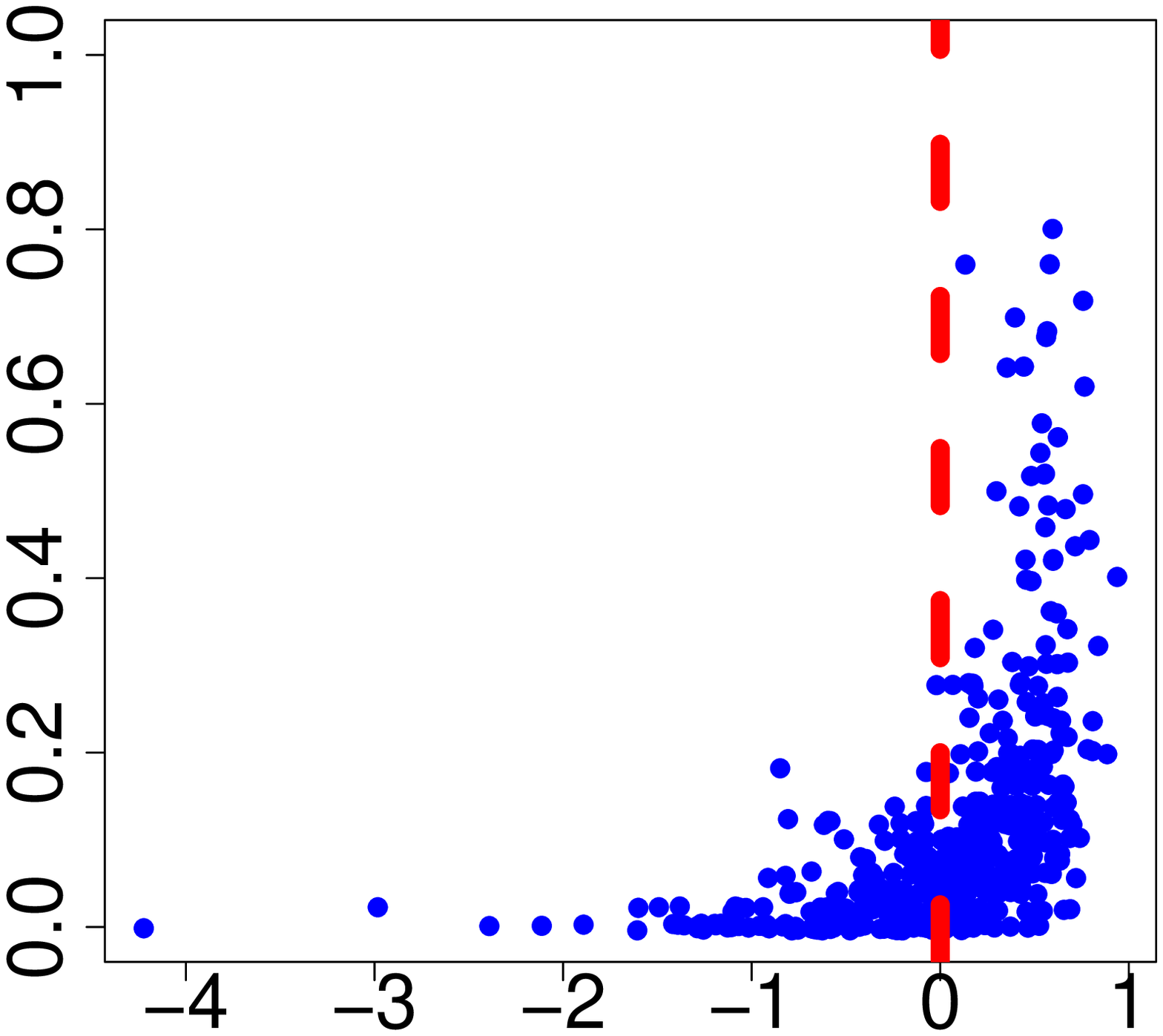}
\end{minipage}
&
\begin{minipage}[l]{\figurewidth\textwidth}
\includegraphics[width=\textwidth, angle=270]
{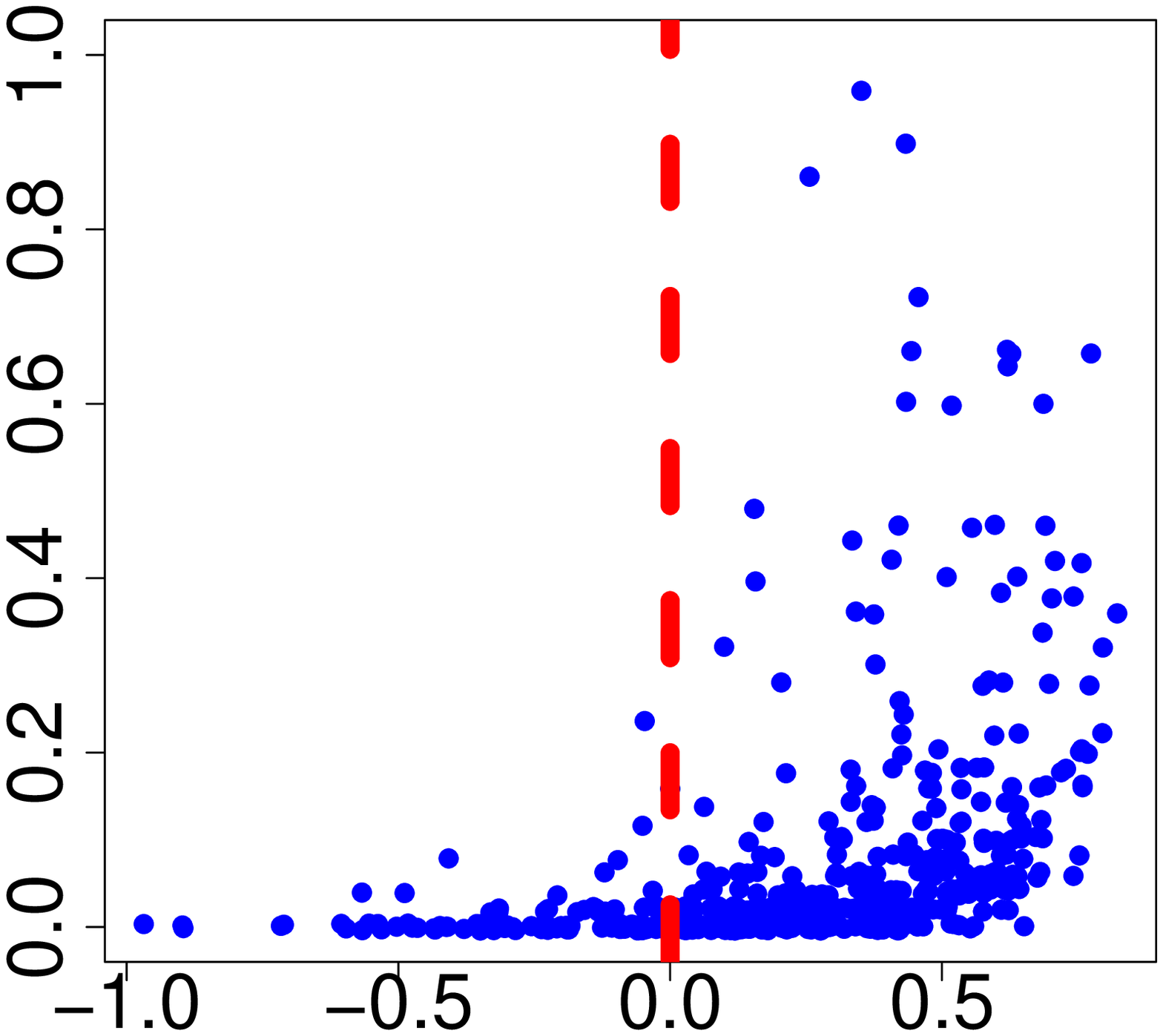}
\end{minipage}
&
\begin{minipage}[l]{\figurewidth\textwidth}
\includegraphics[width=\textwidth, angle=270]
{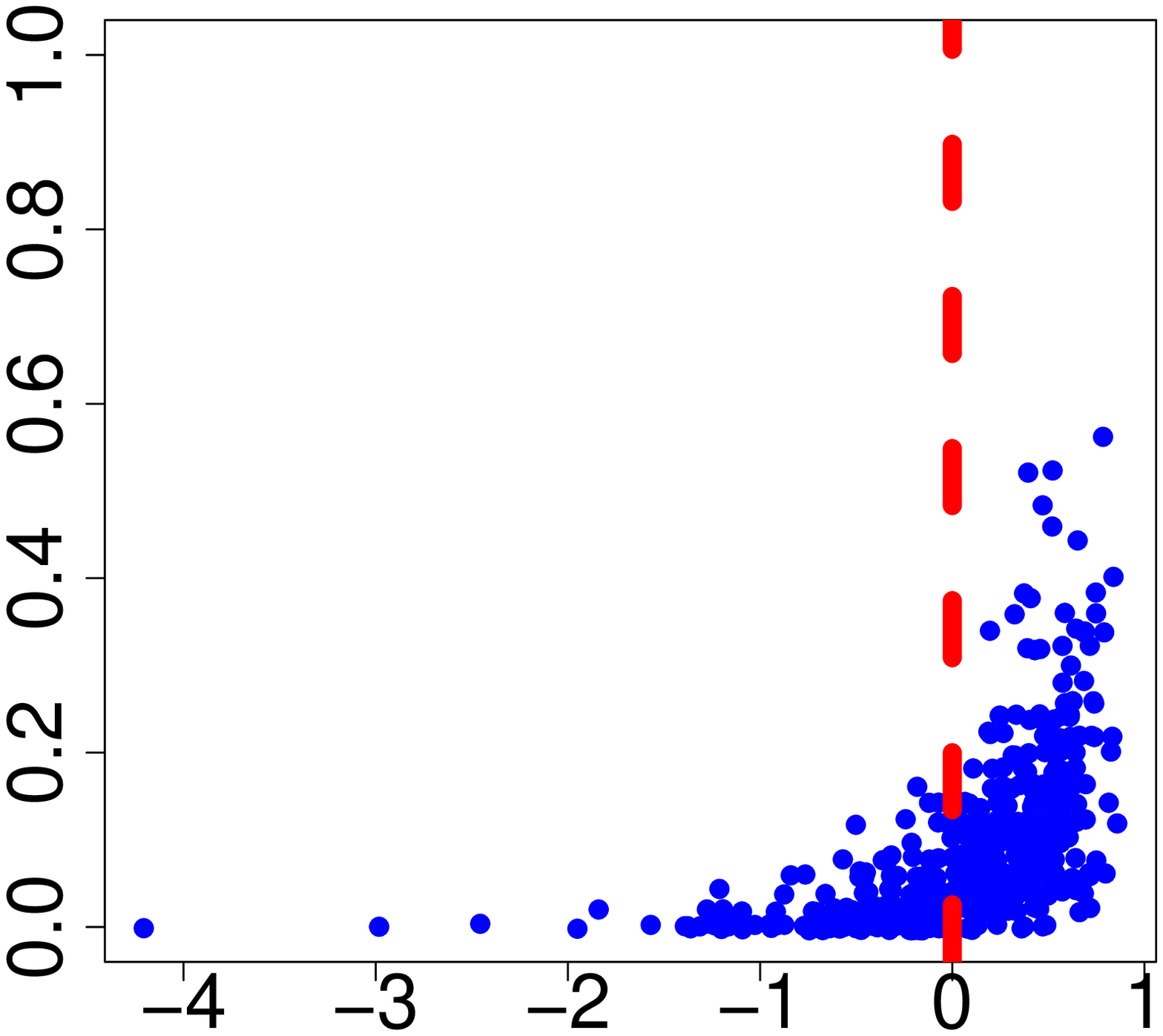}
\end{minipage}
&
\begin{minipage}[l]{\figurewidth\textwidth}
\includegraphics[width=\textwidth, angle=270]
{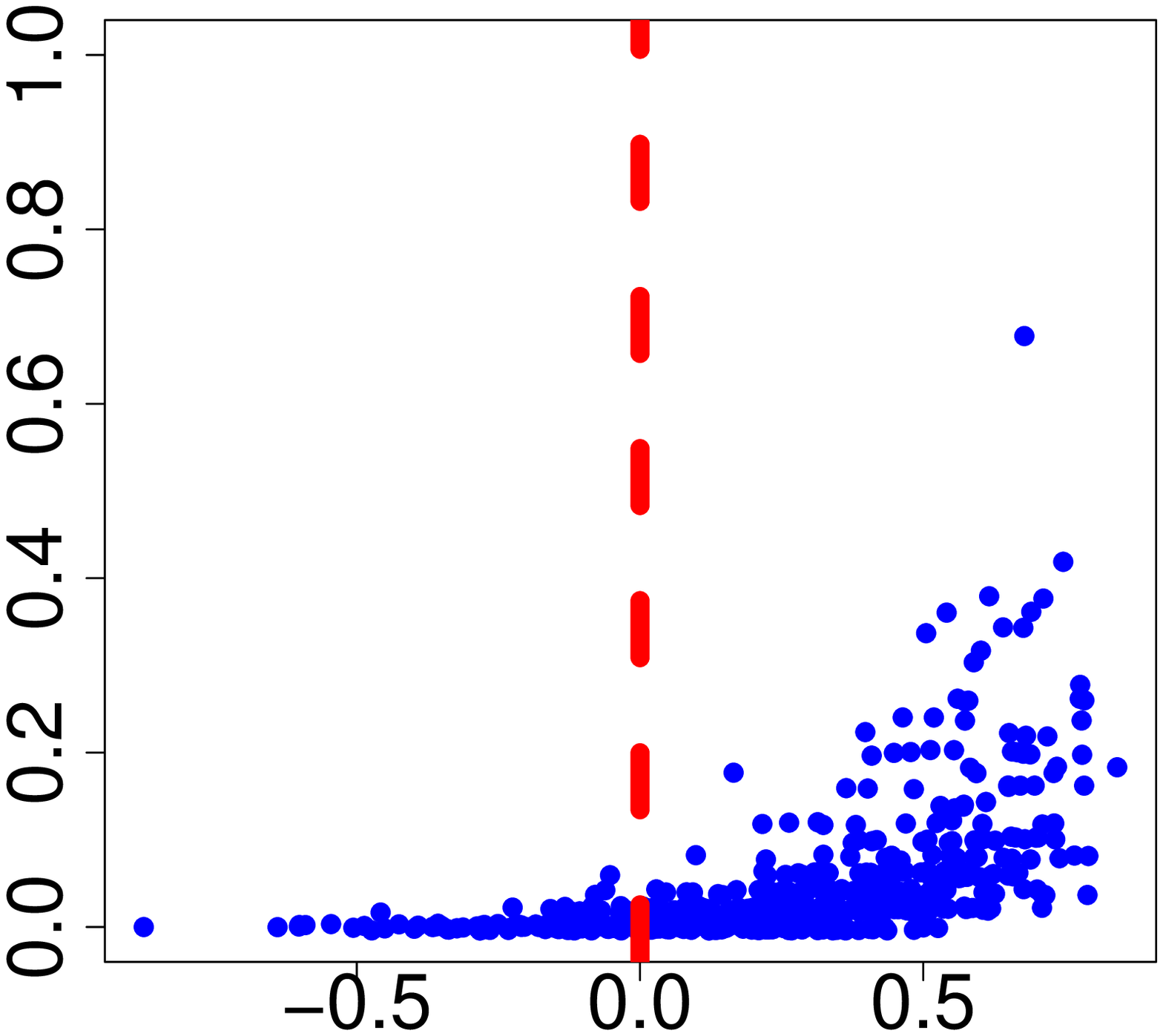}
\end{minipage}
\\
\hline
\begin{tabular}{c}
\begin{sideways}
	\begin{tabular}{c} 
	Mixed Gaussian \\ $g_{1}$ profile \\ $n=500$, $\sigma=1.0$ 
	\end{tabular}
\end{sideways}
\end{tabular}
&
\begin{minipage}[l]{\figurewidth\textwidth}
\includegraphics[width=\textwidth, angle=270]
{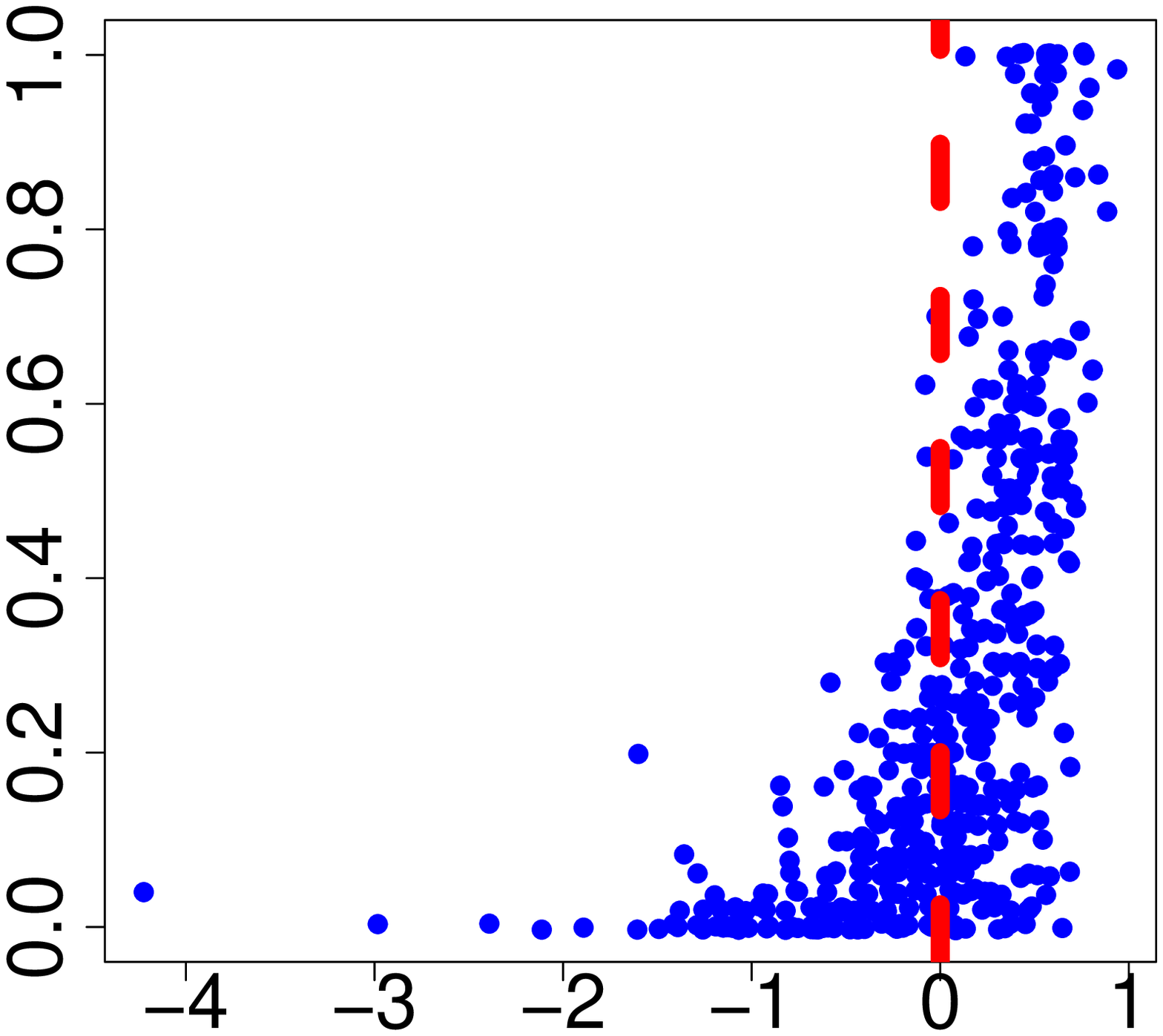}
\end{minipage}
&
\begin{minipage}[l]{\figurewidth\textwidth}
\includegraphics[width=\textwidth, angle=270]
{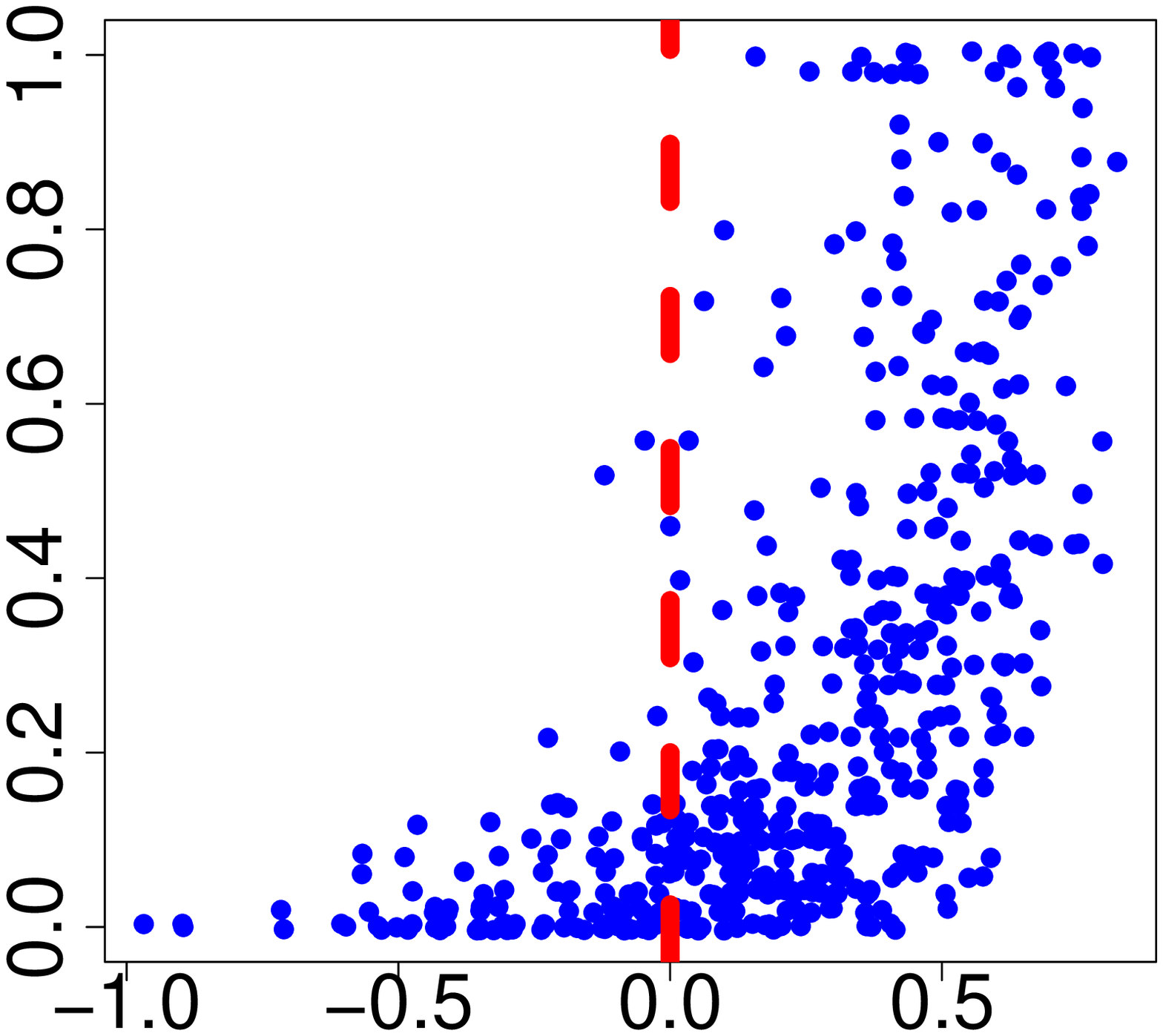}
\end{minipage}
&
\begin{minipage}[l]{\figurewidth\textwidth}
\includegraphics[width=\textwidth, angle=270]
{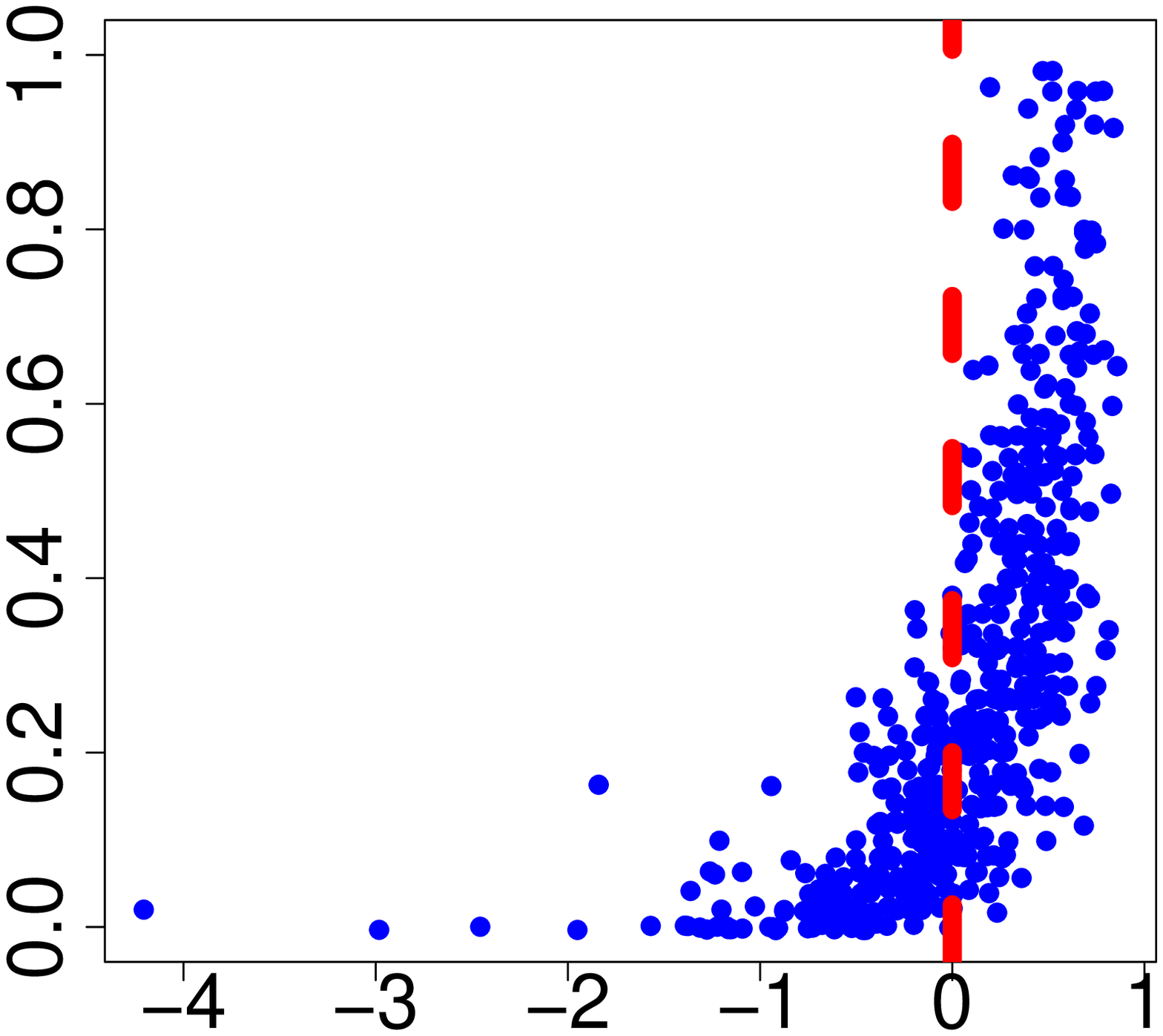}
\end{minipage}
&
\begin{minipage}[l]{\figurewidth\textwidth}
\includegraphics[width=\textwidth, angle=270]
{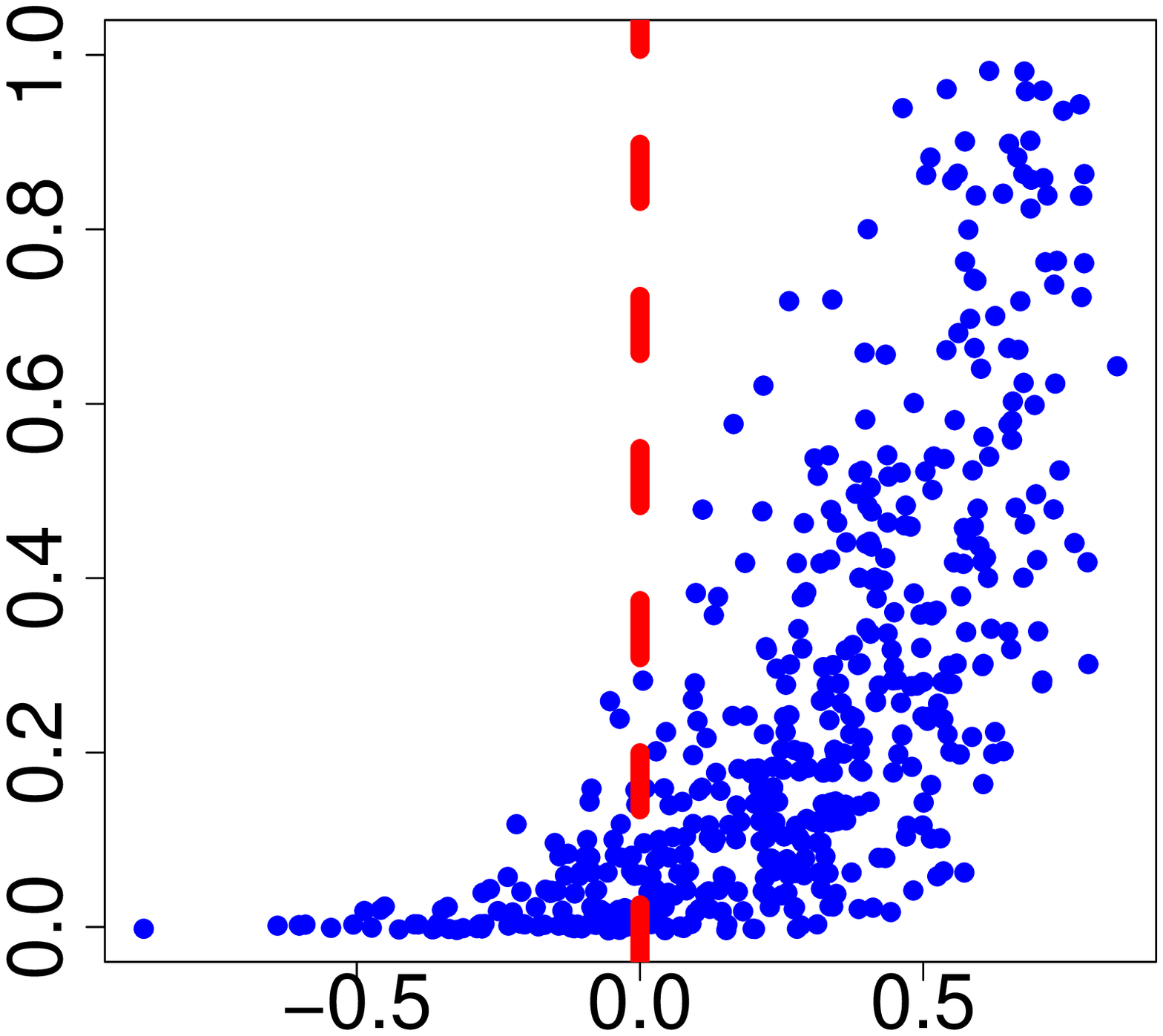}
\end{minipage}
\\
\hline
\begin{tabular}{c}
\begin{sideways}
	\begin{tabular}{c} 
	Mixed Gaussian \\ $g_{2}$ profile \\ $n=500$, $\sigma=1.0$ 
	\end{tabular}
\end{sideways}
\end{tabular}
&
\begin{minipage}[l]{\figurewidth\textwidth}
\includegraphics[width=\textwidth, angle=270]
{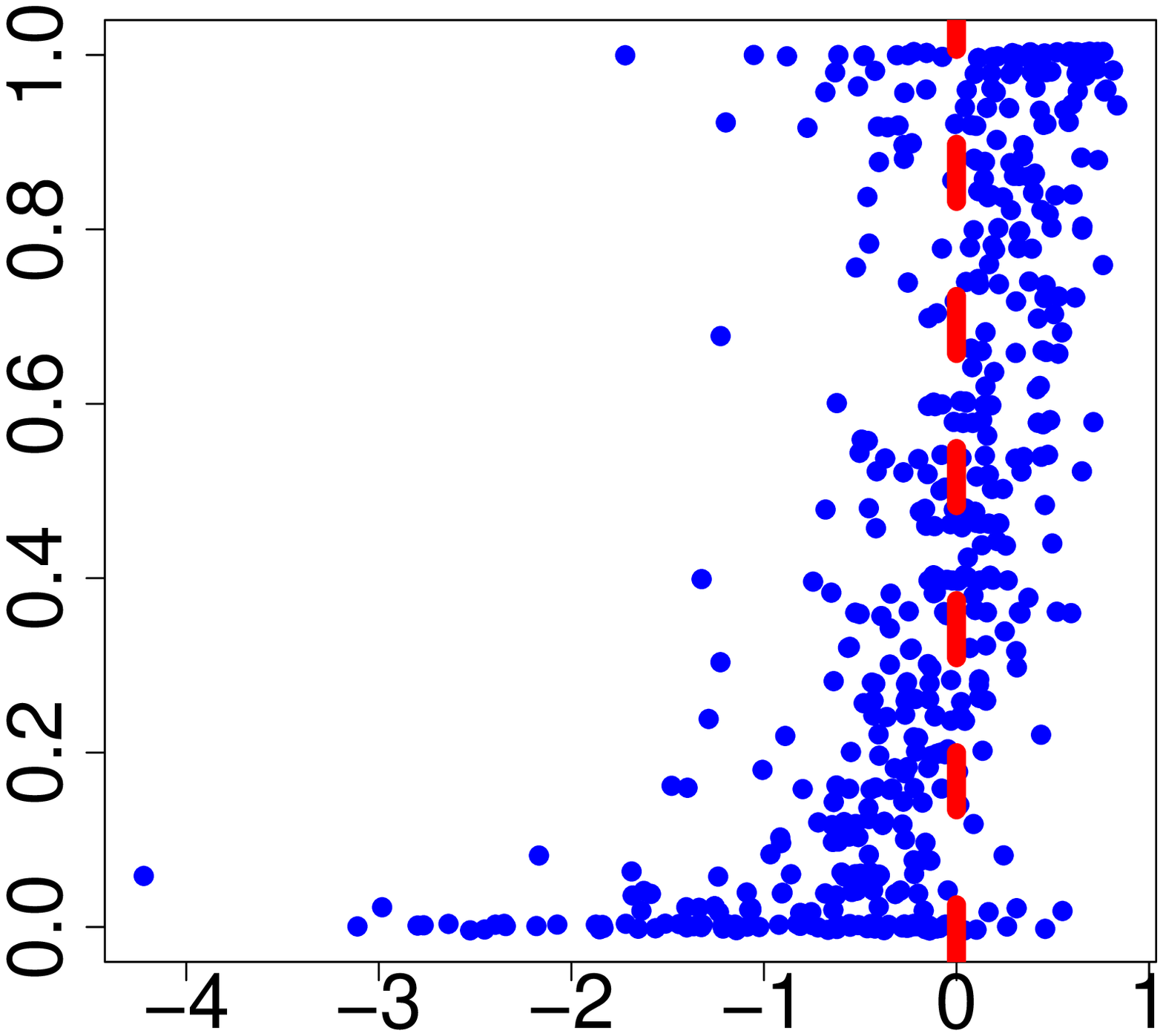}
\end{minipage}
&
\begin{minipage}[l]{\figurewidth\textwidth}
\includegraphics[width=\textwidth, angle=270]
{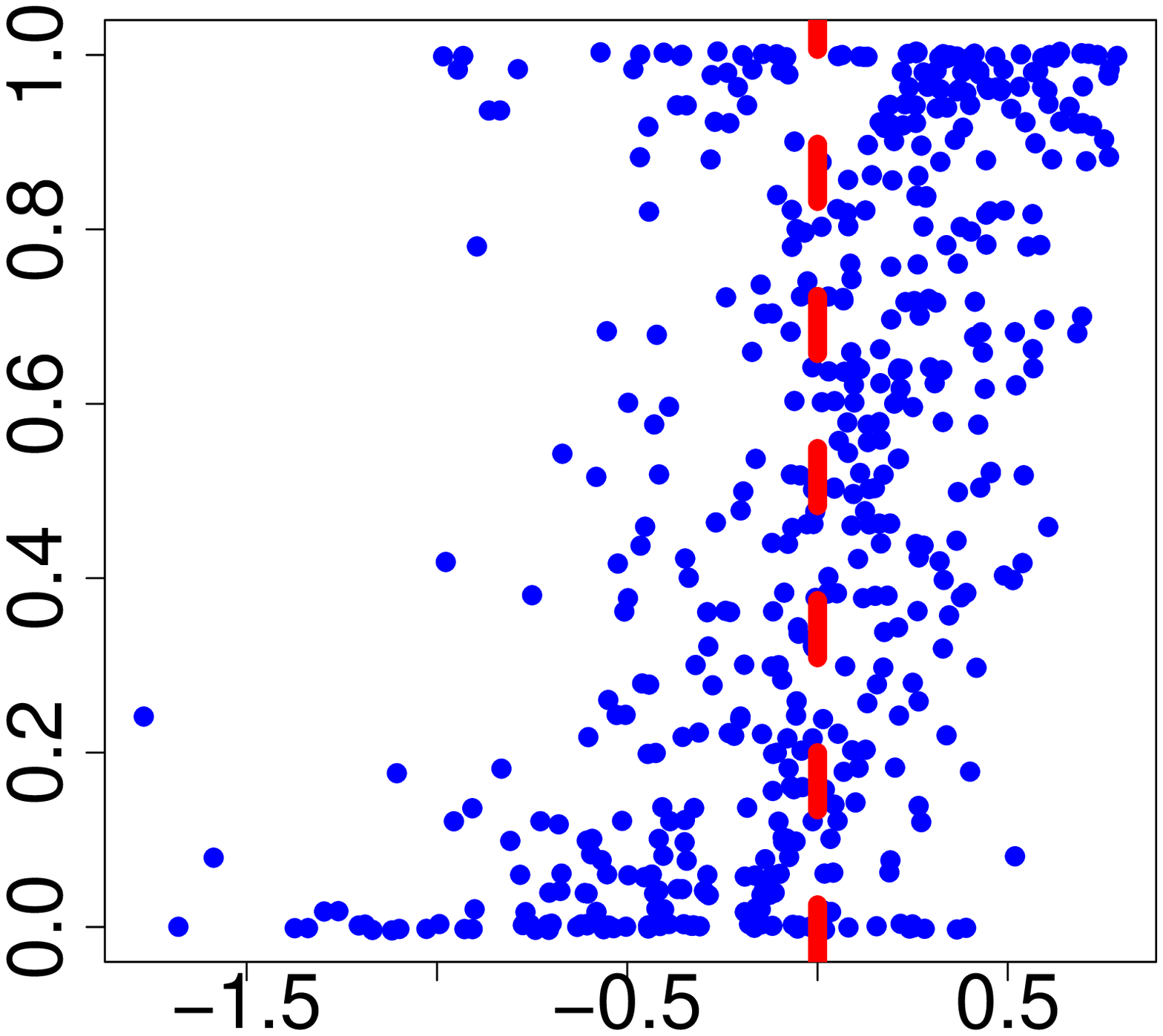}
\end{minipage}
&
\begin{minipage}[l]{\figurewidth\textwidth}
\includegraphics[width=\textwidth, angle=270]
{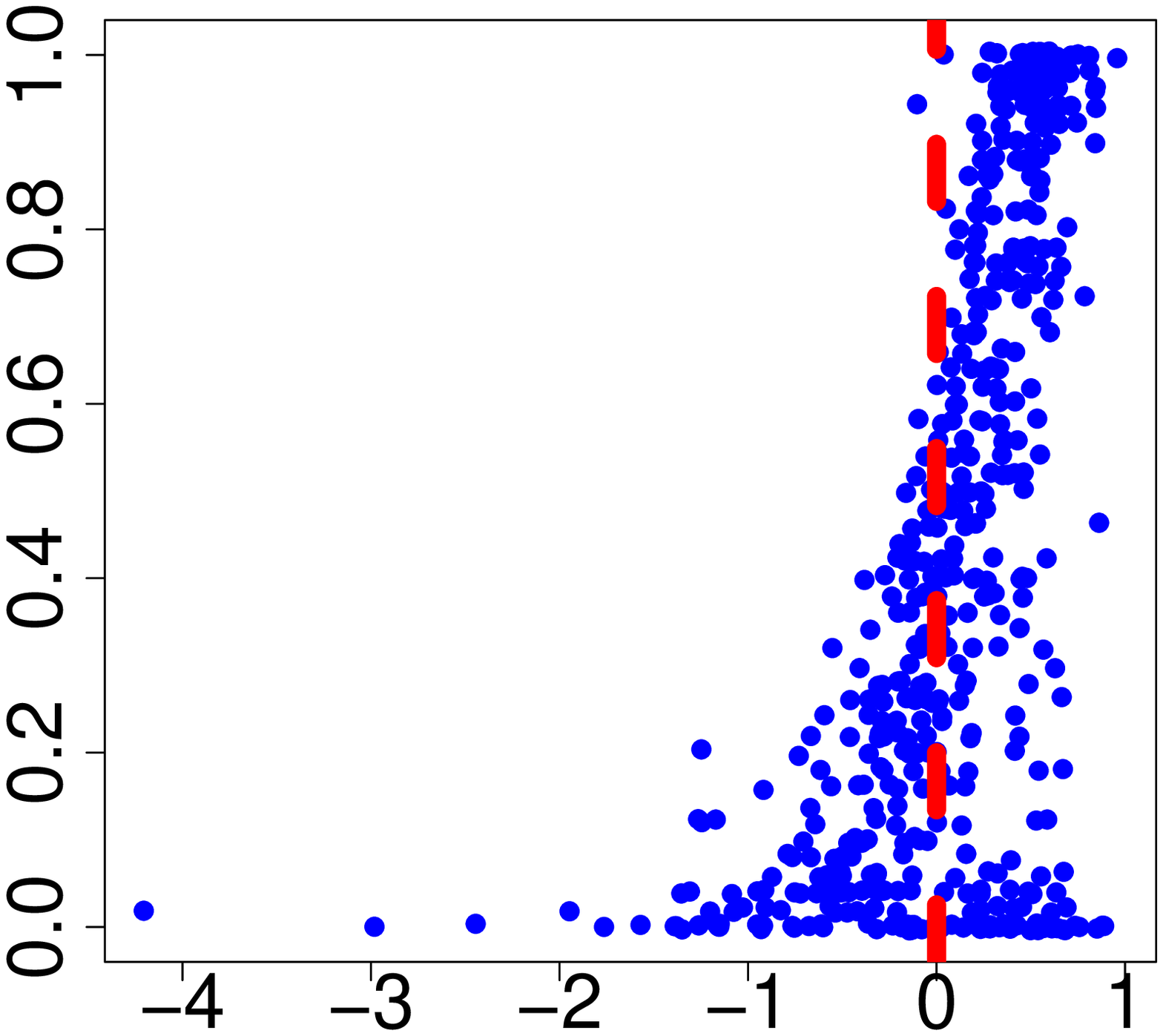}
\end{minipage}
&
\begin{minipage}[l]{\figurewidth\textwidth}
\includegraphics[width=\textwidth, angle=270]
{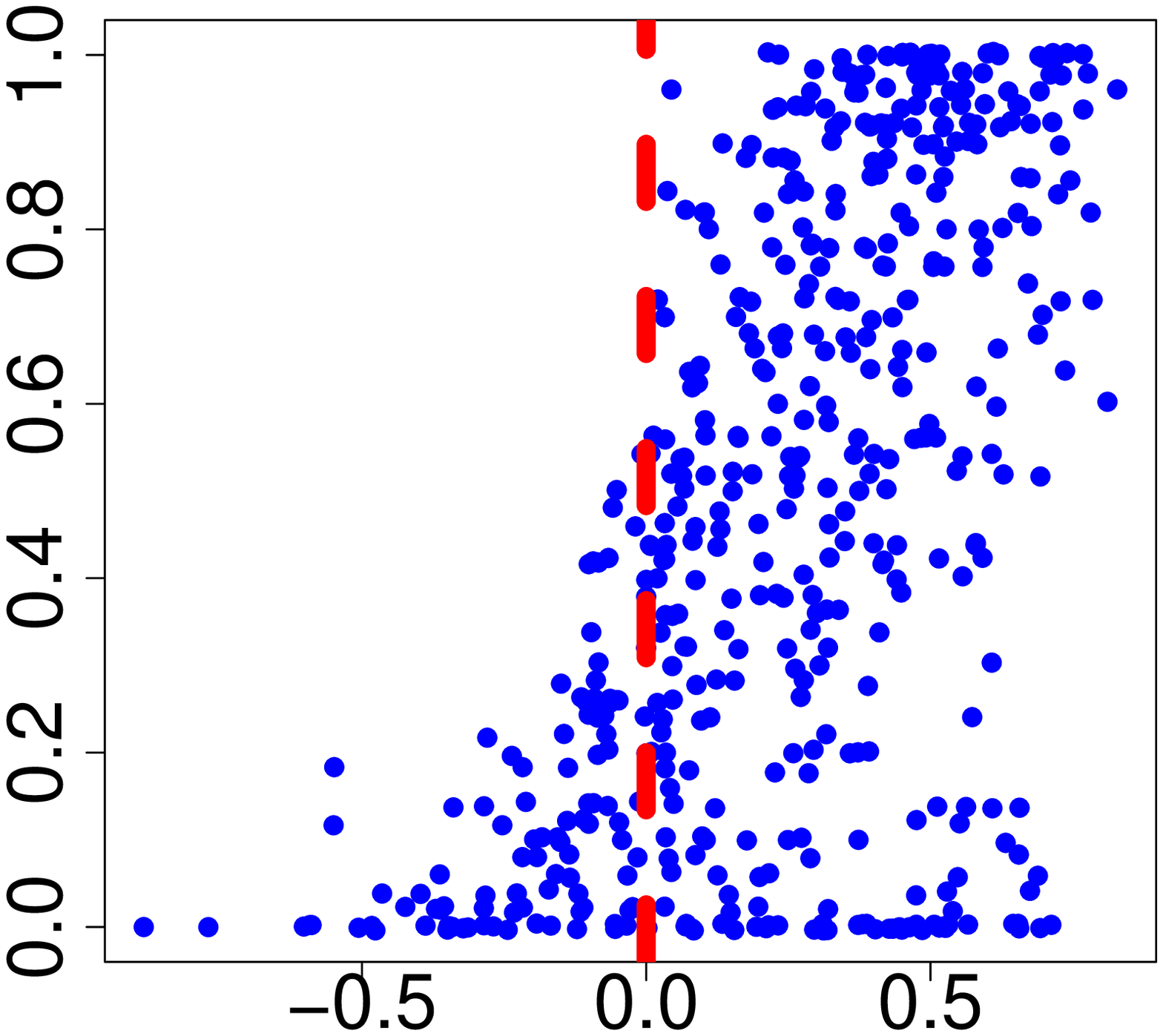}
\end{minipage}
\\
\hline
\begin{tabular}{c}
\begin{sideways}
	\begin{tabular}{c} 
	Mixed Gaussian \\ $g_{2}$ profile \\ $n=1,000$, $\sigma=1.0$ 
	\end{tabular}
\end{sideways}
\end{tabular}
&
\begin{minipage}[l]{\figurewidth\textwidth}
\includegraphics[width=\textwidth, angle=270]
{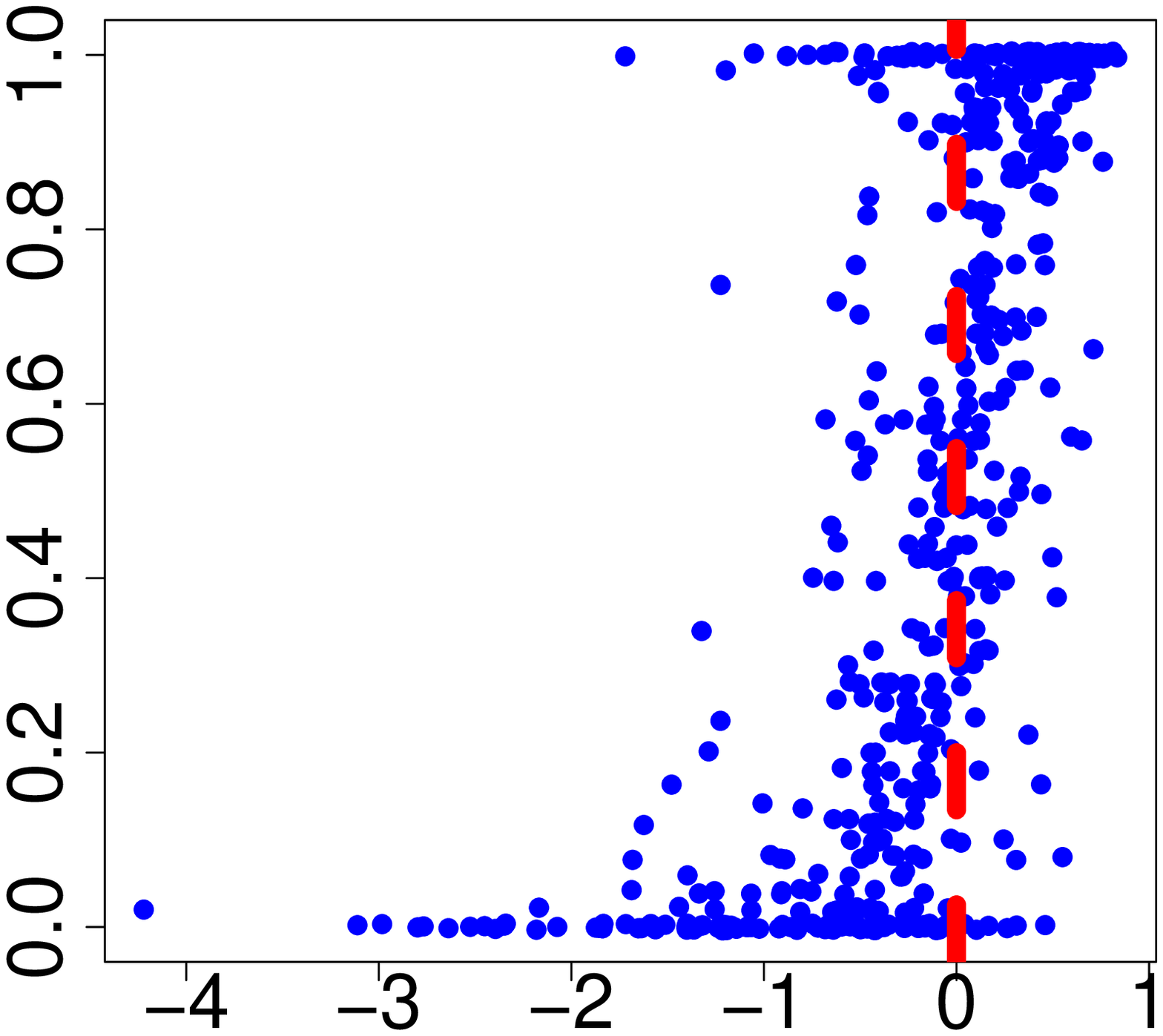}
\end{minipage}
&
\begin{minipage}[l]{\figurewidth\textwidth}
\includegraphics[width=\textwidth, angle=270]
{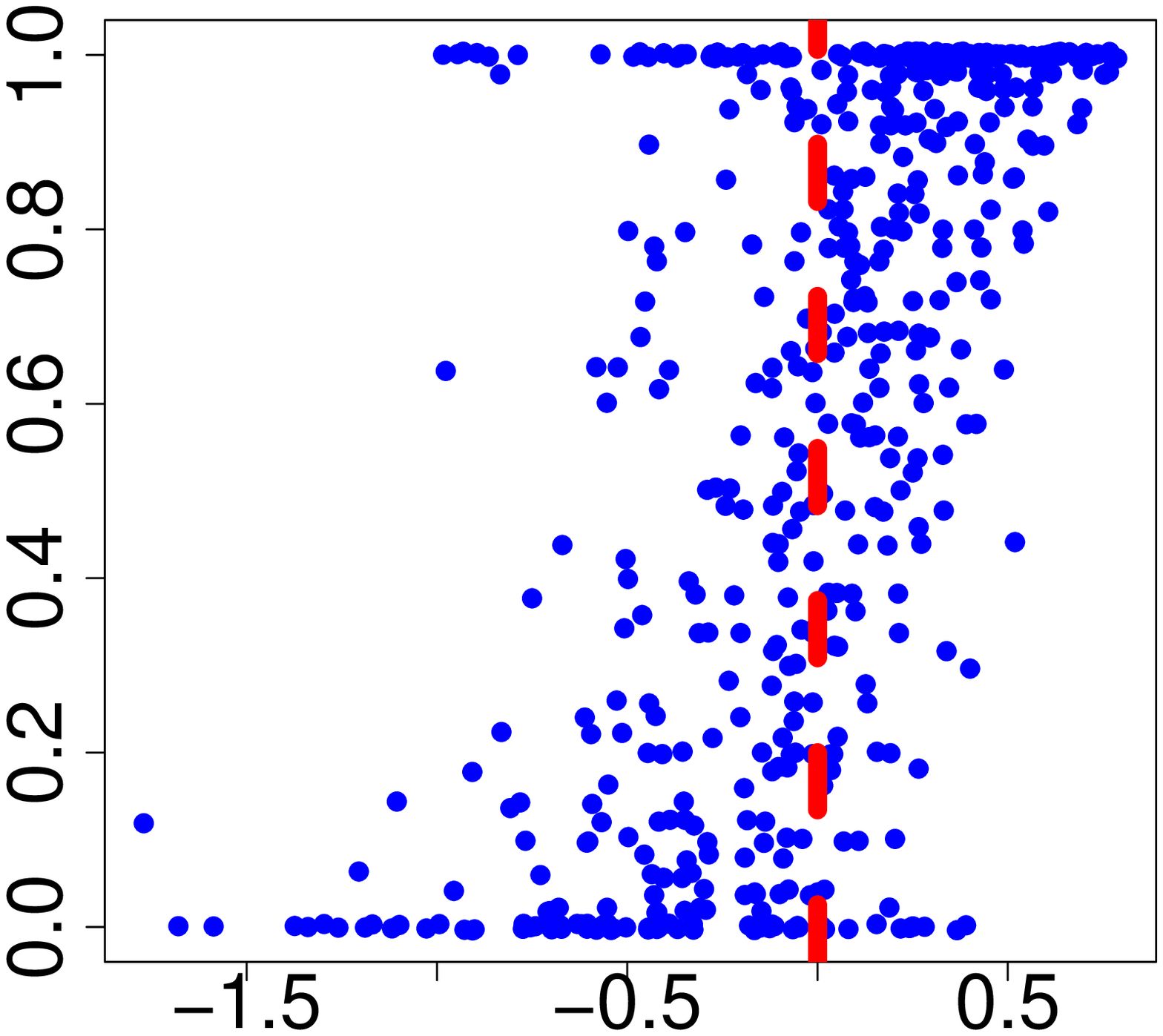}
\end{minipage}
&
\begin{minipage}[l]{\figurewidth\textwidth}
\includegraphics[width=\textwidth, angle=270]
{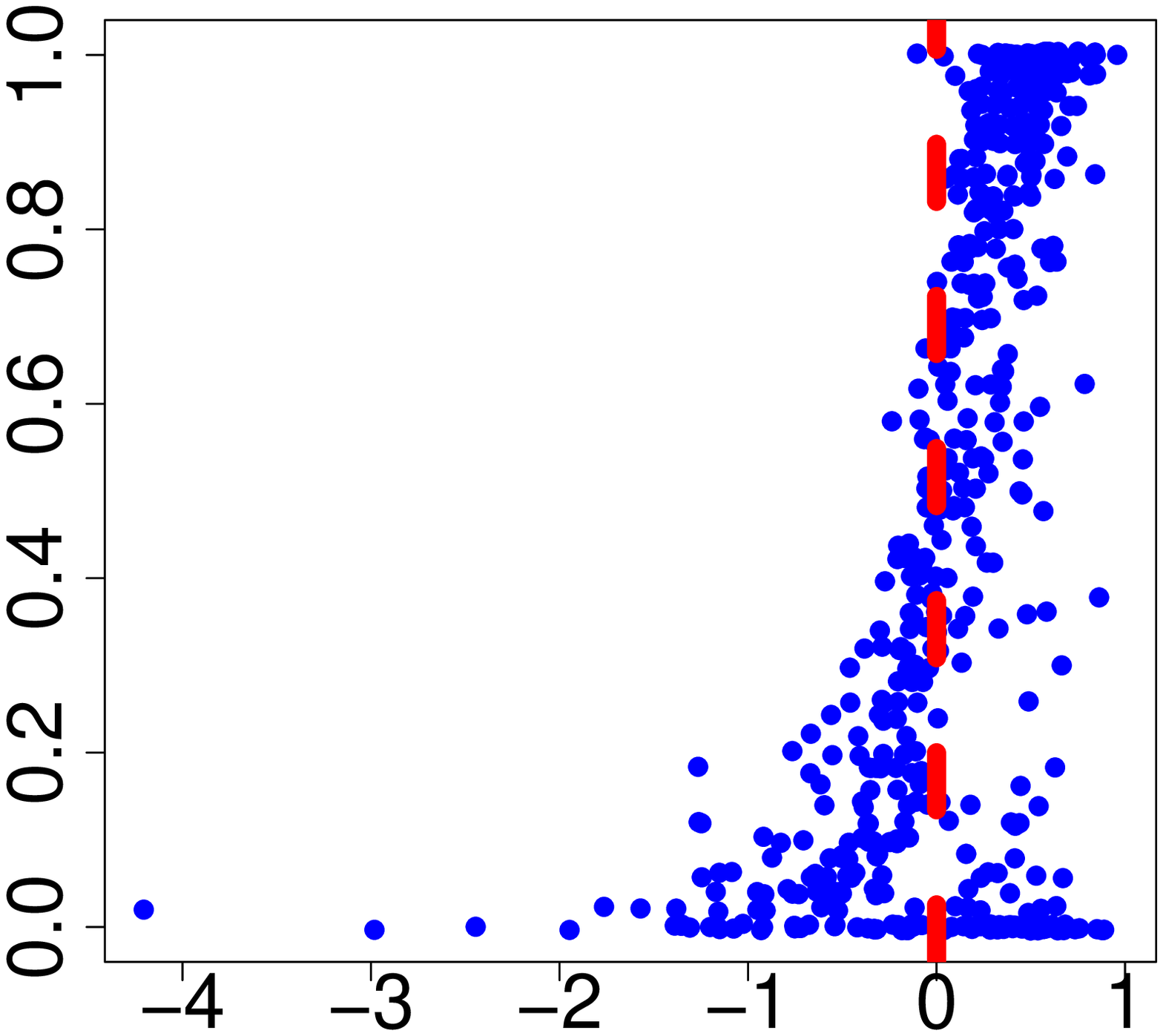}
\end{minipage}
&
\begin{minipage}[l]{\figurewidth\textwidth}
\includegraphics[width=\textwidth, angle=270]
{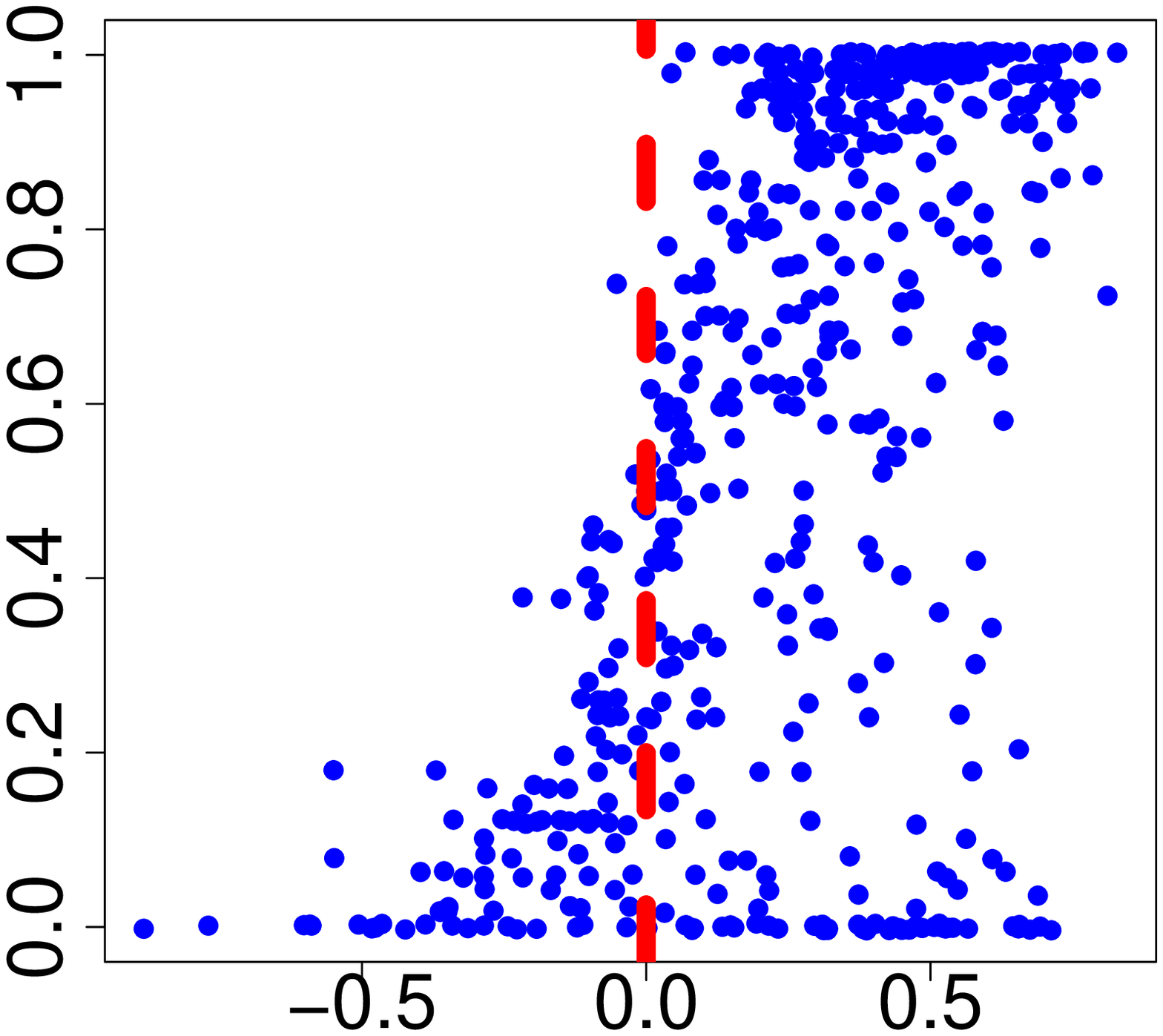}
\end{minipage}
\\
\hline
\end{tabular}
\caption{
\label{figure:classification_rmixednorm_model_selection_phase_transition}
\textbf{Proportion of sample paths containing a sign-correct model vs. GI index for mixed Gaussian predictors:}
The proportion at each point is based on 50 replicates of the sample regularization path for the corresponding design.
As in Figure~\ref{figure:classification_rnorm_model_selection_phase_transition}, the results displayed in these panels show good agreement with the theory for sign consistency of general $\ell_{1}$-penalized M-estimators developed in Section \ref{section:application}: for increasing sample sizes, the proportion of paths containing sign correct model approaches one as the sample size increases whenever $\eta(\theta)>0$.
It is interesting to notice that a high probability of correct sign recovery is possible even if  $\eta(\theta)<0$ (see the SVM estimates under the ``blip'' profile) but it does not approach one asymptotically.
Also notice that the fainter signal of the ``blip'' profile makes the recovery of the correct signs harder when $\eta(\theta)>0$, especially for the logistic classifiers.
}
\end{center}	
\end{figure}
\afterpage{\clearpage}

\subsection{Comparison of $\ell_{1}$-penalized SVM and logistic regression classifiers}

In addition to allowing us to study the model selection consistency of SVM and logistic classifiers, Theorem \ref{result:generalized_irrepresentability_condition} along with Lemmas \ref{result:variance_and_hessian_for_logistic_loss} and \ref{result:variance_and_hessian_for_hinge_loss} lets us to shed some light onto a question often asked by practitioners: which of SVM and logistic regression classifiers should be used for variable selection?
Our theoretical and experimental results suggest that, if variable selection is made through $\ell_{1}$-penalization, the answer depends critically on the sample size available.

\subsubsection{Large sample (asymptotic) comparison}

If a large enough sample size is available, Theorem~\ref{result:simplified_gi_condition} suggests that in terms of variable selection by means of $\ell_{1}$-norm penalized estimates logistic and SVM are equally likely to be model selection consistent for the designs sampled as described in \ref{section:sampling_joint_distributions}.
For non-Gaussian predictors, a comparison of the GI indices $\eta(\theta)$ shows that model selection consistency can be theoretically guaranteed for logistic regression classifiers in more designs than SVM.
The results are shown in Figure~\ref{figure:logistic_gi_vs_svm_gi} and Table~\ref{table:logistic_gi_vs_svm_gi}.
Interestingly, for the distribution of designs considered, logistic was more likely to be model selection consistent even under the ``blip'' conditional probability profile function -- thought to favor SVM by concentrating must of the class discrimination information on a band around the optimal separating hyperplane.


\subsubsection{Finite sample (asymptotic) comparison}

Figure~\ref{figure:finite_sample_sign_correct_proportion_comparison} shows a comparison of the proportion of times the $\ell_{1}$-penalized logistic and SVM regularization paths contained a model with correctly selected variables.
In each plot, each point is obtained by plotting the proportion of paths containing a sign-correct model for logistic (vertical axis) against the same proportion for SVM for a given design.
Thus, the further a point sits to the lower right corner, the better was the performance of SVM in comparison to logistic for that specific design.
The proportions are obtained from 50 replications of one of the designs sampled as described in Section~\ref{section:sampling_joint_distributions}.

The results shown in Figure~\ref{figure:finite_sample_sign_correct_proportion_comparison} suggest that the comparison of logistic and SVM classification based solely on the GI condition should be taken with a grain of salt.
Two factors are involved here: first, the results are based on asymptotic approximations and, second, a negative GI index does not necessarily imply a low probability of correct sign recovery (though such probability is known not to approach one).
While in most cases the two methods are comparable in their ability to contain a correct model in their regularization path, SVM does seem to have some advantage over logistic under Gaussian predictors and the ``blip'' conditional profile even at large sample sizes ($n=1,000$).
For smaller sample sizes ($n=100$), SVM did perform markedly better than logistic regression under mixed Gaussian predictors and the logistic profile.

\renewcommand{\figurewidth}{0.20}
\begin{figure}[p]
\begin{center}
\begin{tabular}{|c|cc|cc|}
\cline{2-5}
\multicolumn{1}{c|}{}
&
\multicolumn{2}{|c|}{Gaussian}
&
\multicolumn{2}{c|}{Mix. Gaussian}
\\
\multicolumn{1}{c|}{}
&
$p=08$
&
$p=16$
&
$p=08$
&
$p=16$
\\
\hline
&&&&
\\
\begin{tabular}{c}
\begin{sideways}
	{
	\begin{small}
	\begin{tabular}{c} $g_{1}$ profile \end{tabular}
	\end{small}
	}
\end{sideways}
\end{tabular}
&
\begin{tabular}{c}
\includegraphics[width=\figurewidth\textwidth, angle=270]
{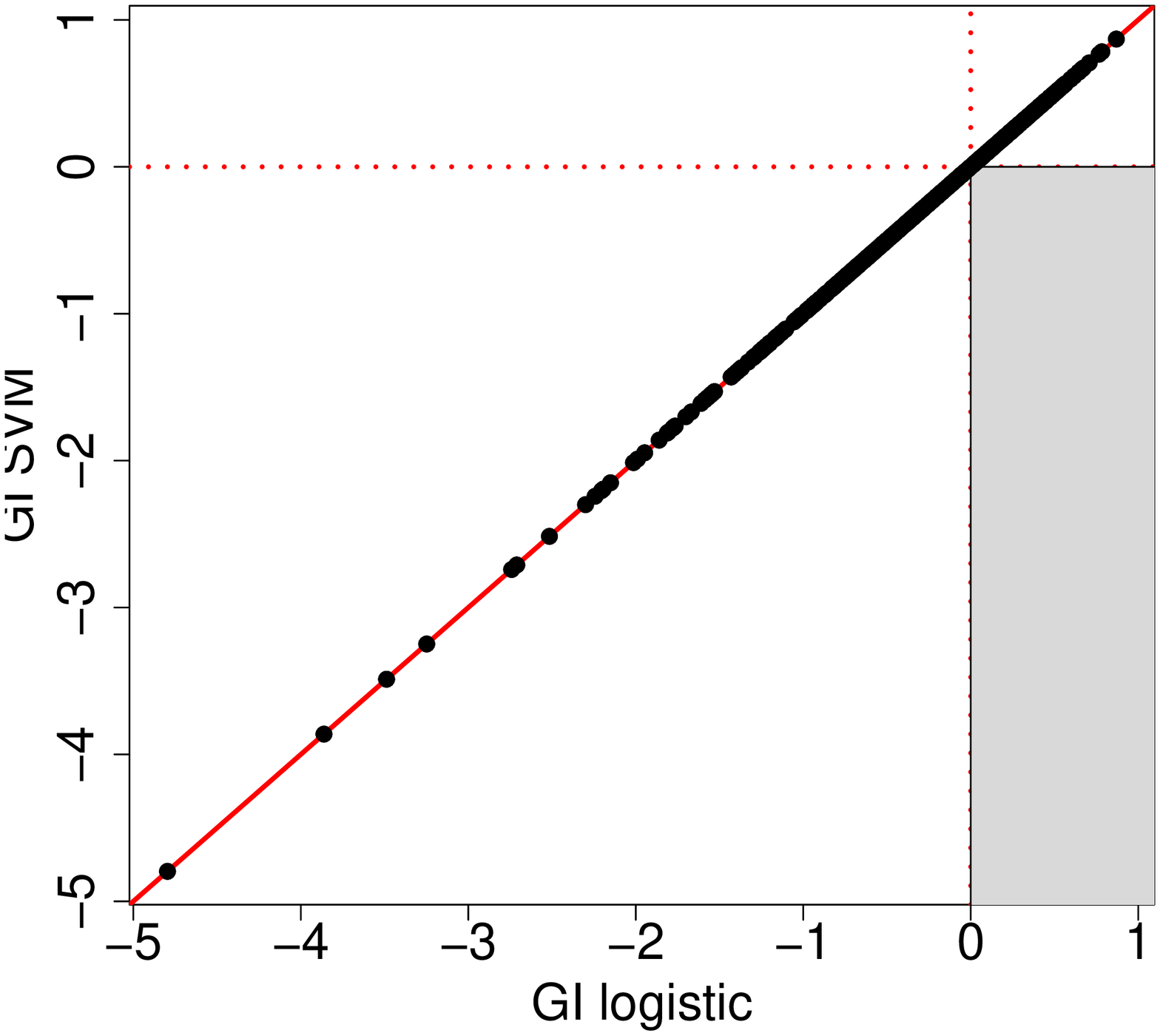}
\end{tabular}
&
\begin{tabular}{c}
\includegraphics[width=\figurewidth\textwidth, angle=270]
{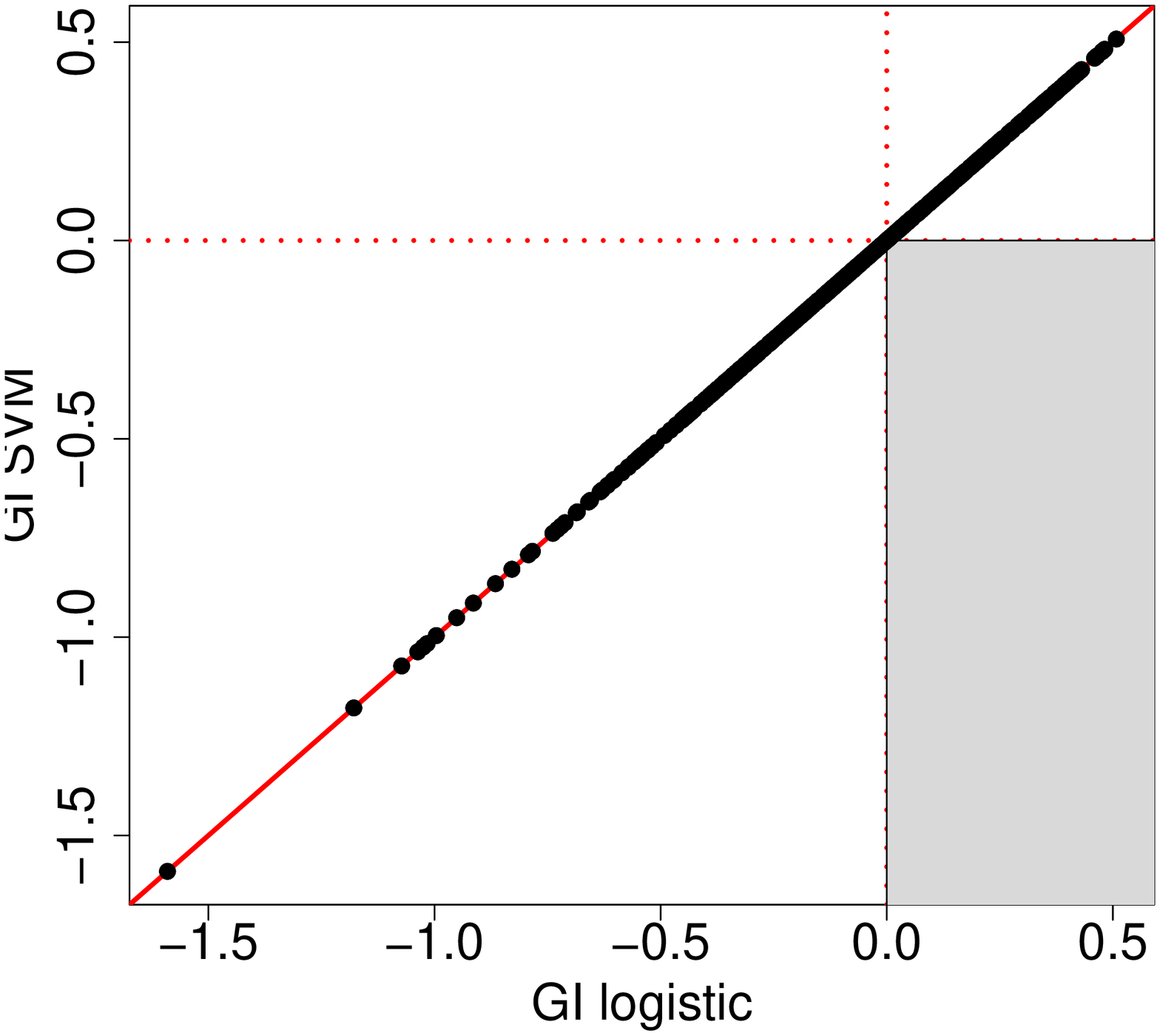}
\end{tabular}
&
\begin{tabular}{c}
\includegraphics[width=\figurewidth\textwidth, angle=270]
{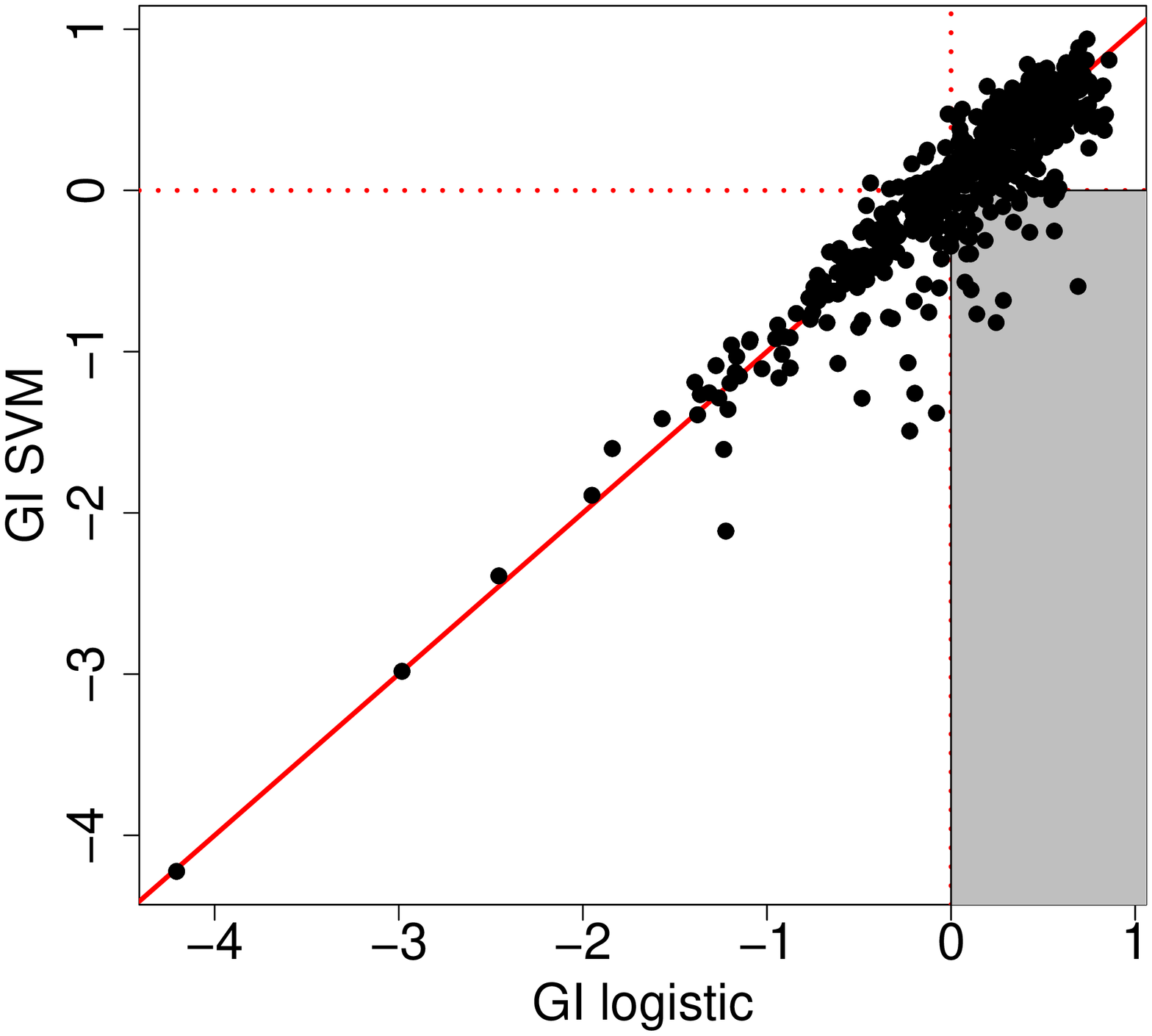}
\end{tabular}
&
\begin{tabular}{c}
\includegraphics[width=\figurewidth\textwidth, angle=270]
{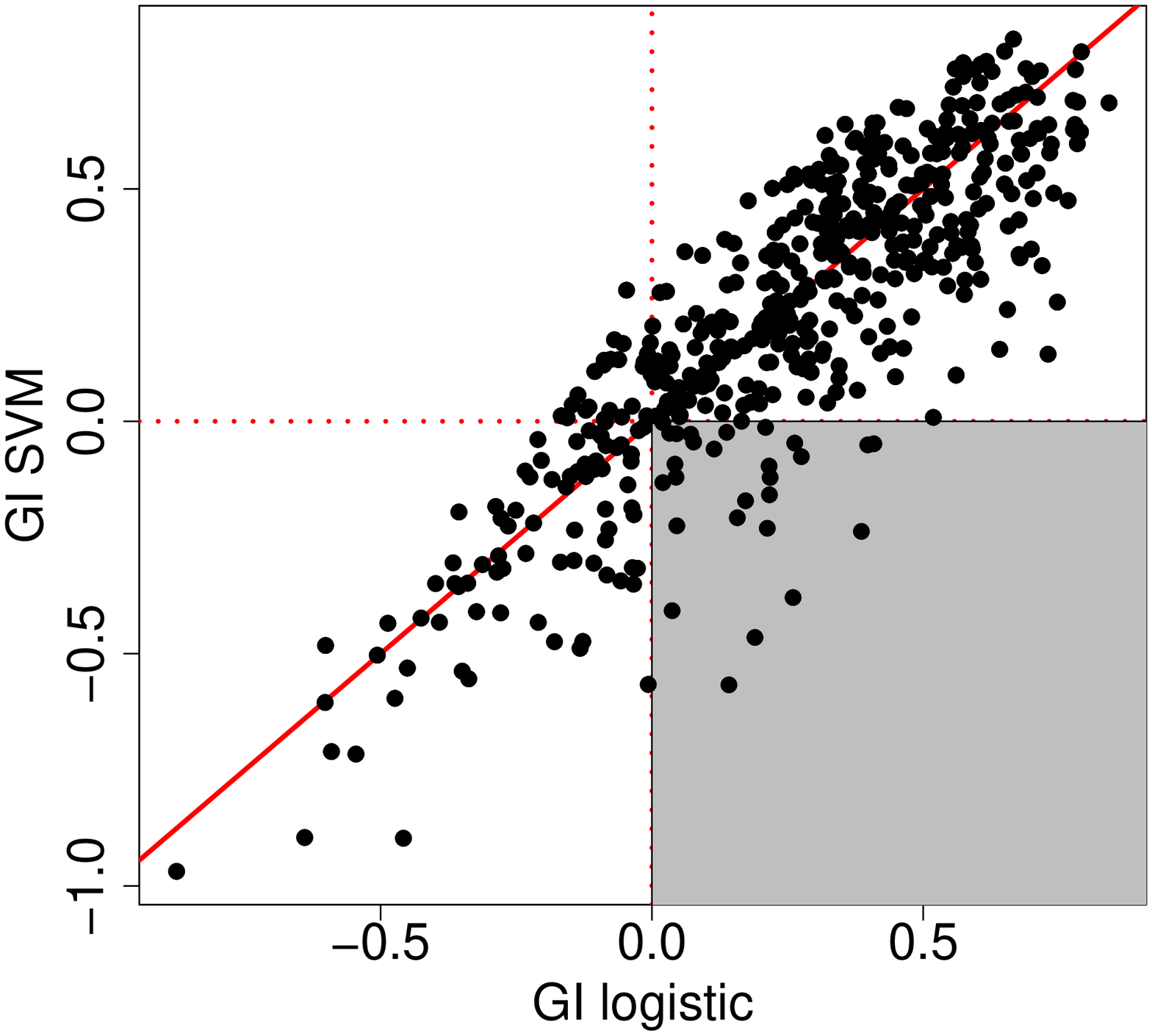}
\end{tabular}
\\
&&&&
\\
\hline
&&&&
\\
\begin{tabular}{c}
\begin{sideways}
	{
	\begin{small}
	\begin{tabular}{c} $g_{2}$ profile \end{tabular}
	\end{small}
	}
\end{sideways}
\end{tabular}
&
\begin{tabular}{c}
\includegraphics[width=\figurewidth\textwidth, angle=270]
{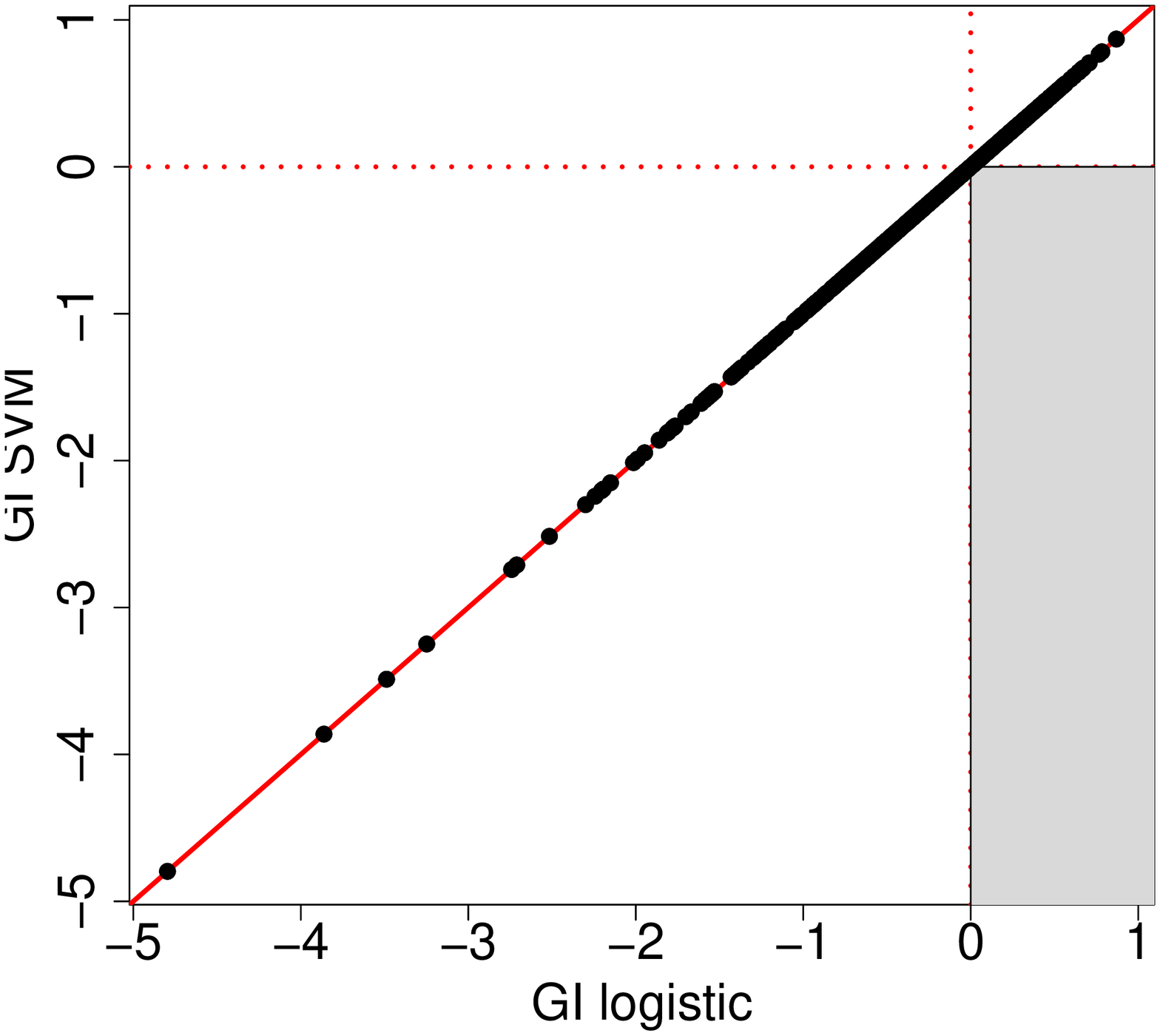}
\end{tabular}
&
\begin{tabular}{c}
\includegraphics[width=\figurewidth\textwidth, angle=270]
{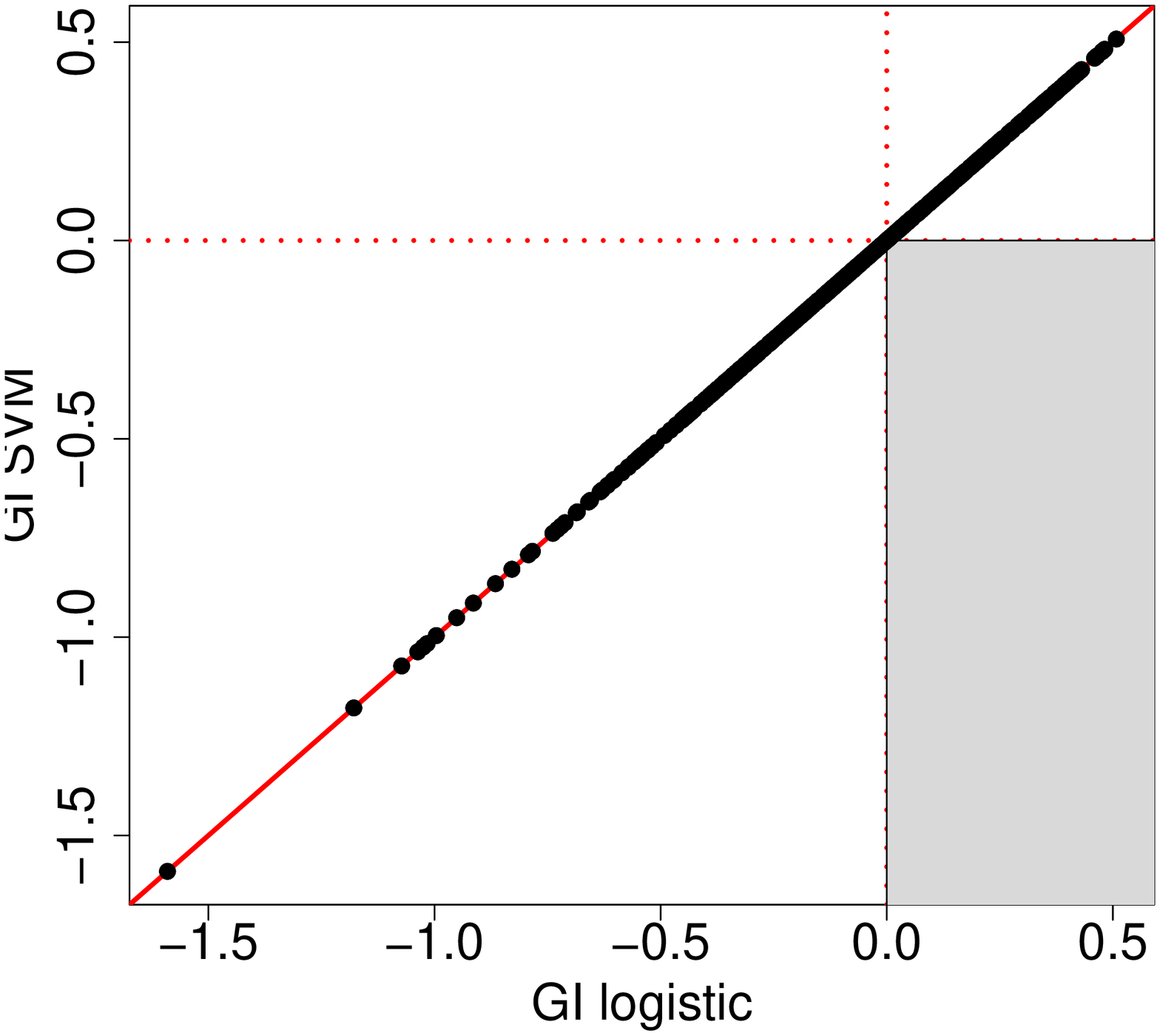}
\end{tabular}
&
\begin{tabular}{c}
\includegraphics[width=\figurewidth\textwidth, angle=270]
{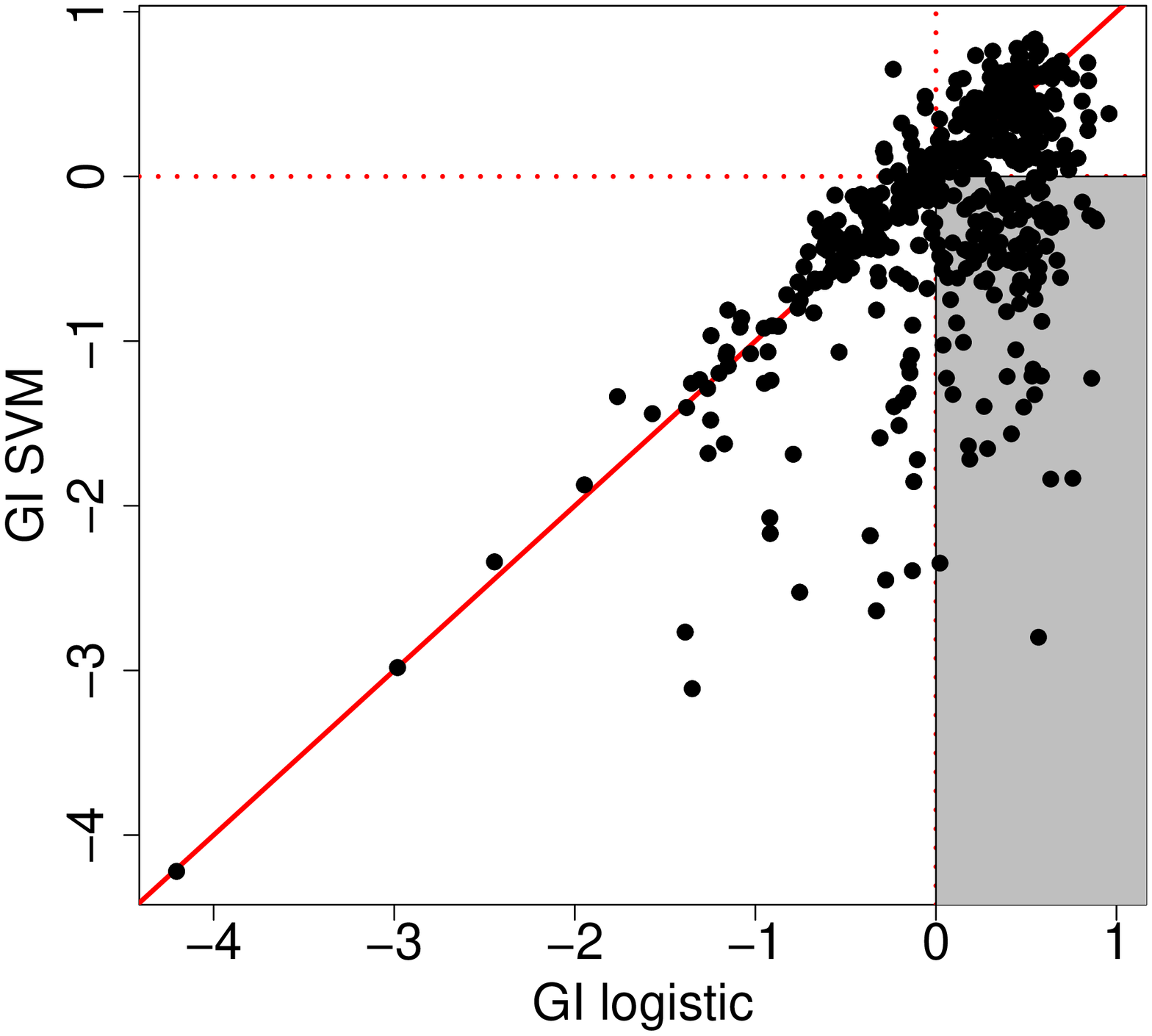}
\end{tabular}
&
\begin{tabular}{c}
\includegraphics[width=\figurewidth\textwidth, angle=270]
{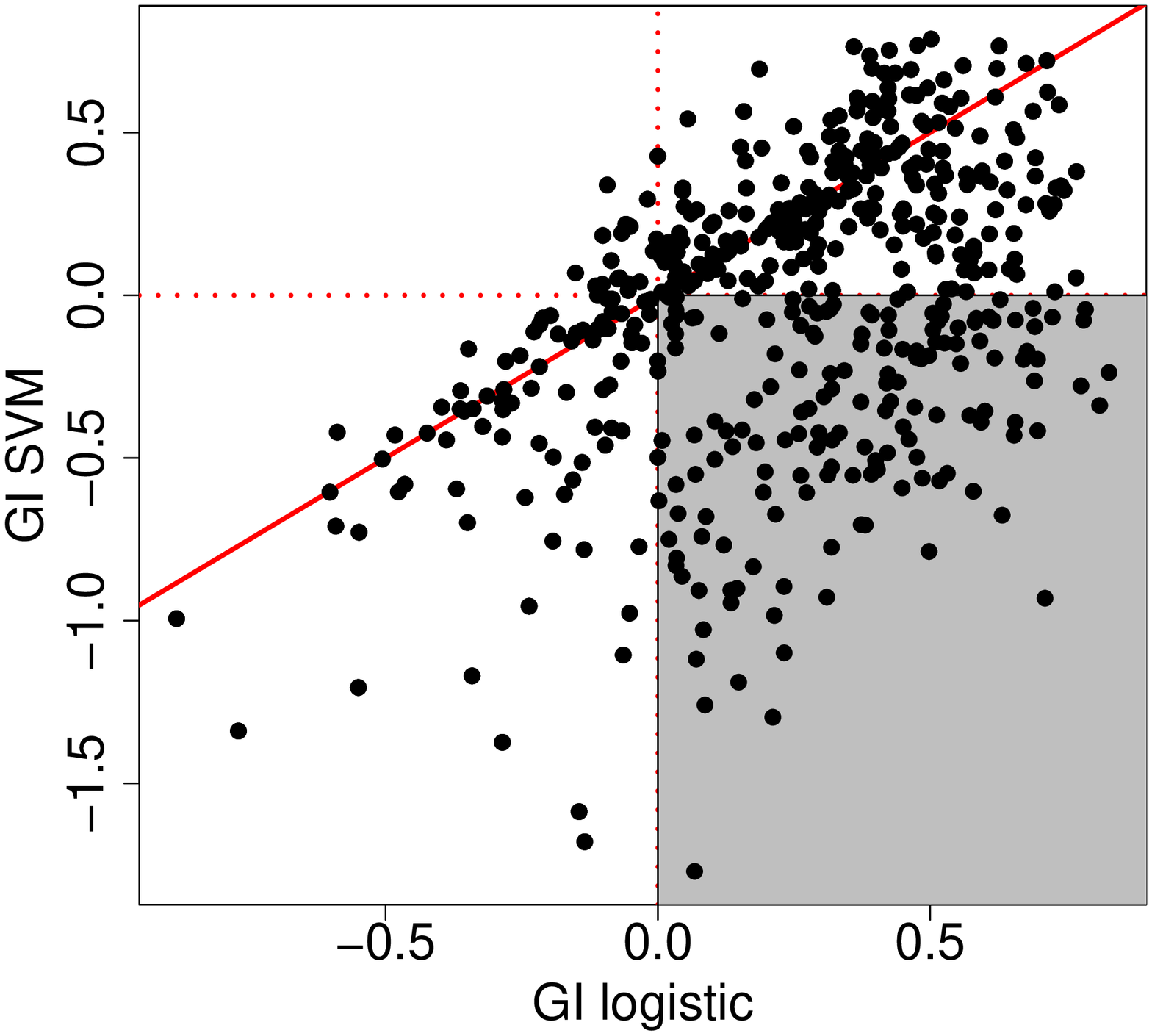}
\end{tabular}
\\
&&&&
\\
\hline
\end{tabular}
\caption{
\label{figure:logistic_gi_vs_svm_gi}
\textbf{Logistic GI vs. SVM GI for 500 designs:}
The shaded area shows where logistic regression is model selection consistent and SVM is not.
Under Gaussian predictors (the four leftmost panels), the GI indices are exactly the same for the SVM and logistic classifiers, as expected in view of Theorem \ref{result:simplified_gi_condition} and Lemmas \ref{result:variance_and_hessian_for_logistic_loss} and \ref{result:variance_and_hessian_for_hinge_loss}.
For mixed Gaussian predictors, the logistic regression classifier is model selection consistent slightly more often than SVM under the logistic conditional probability profile and, surprisingly, much more often under the ``blip'' design.
Recall, however, that SVM was shown to have high probability of correct sign recovery even in cases with $\eta(\theta)<0$.
}
\end{center}	
\end{figure}

\begin{figure}[p]
\begin{center}	
\begin{tabular}{|c|cc|cc|}
\cline{2-5}
\multicolumn{1}{c|}{}
&
\multicolumn{2}{|c|}{Gaussian}
&
\multicolumn{2}{|c|}{Mix. Gaussian}
\\
\multicolumn{1}{c|}{}
&
$p=08$
&
$p=16$
&
$p=08$
&
$p=16$
\\
\hline
\begin{tabular}{c}
\begin{sideways}
	{
	\begin{small}
	\begin{tabular}{c} $g_{1}$ profile \\ $n=100$ \end{tabular}
	\end{small}
	}
\end{sideways}
\end{tabular}
&
\begin{tabular}{c}
\includegraphics[width=\figurewidth\textwidth, angle=270]
{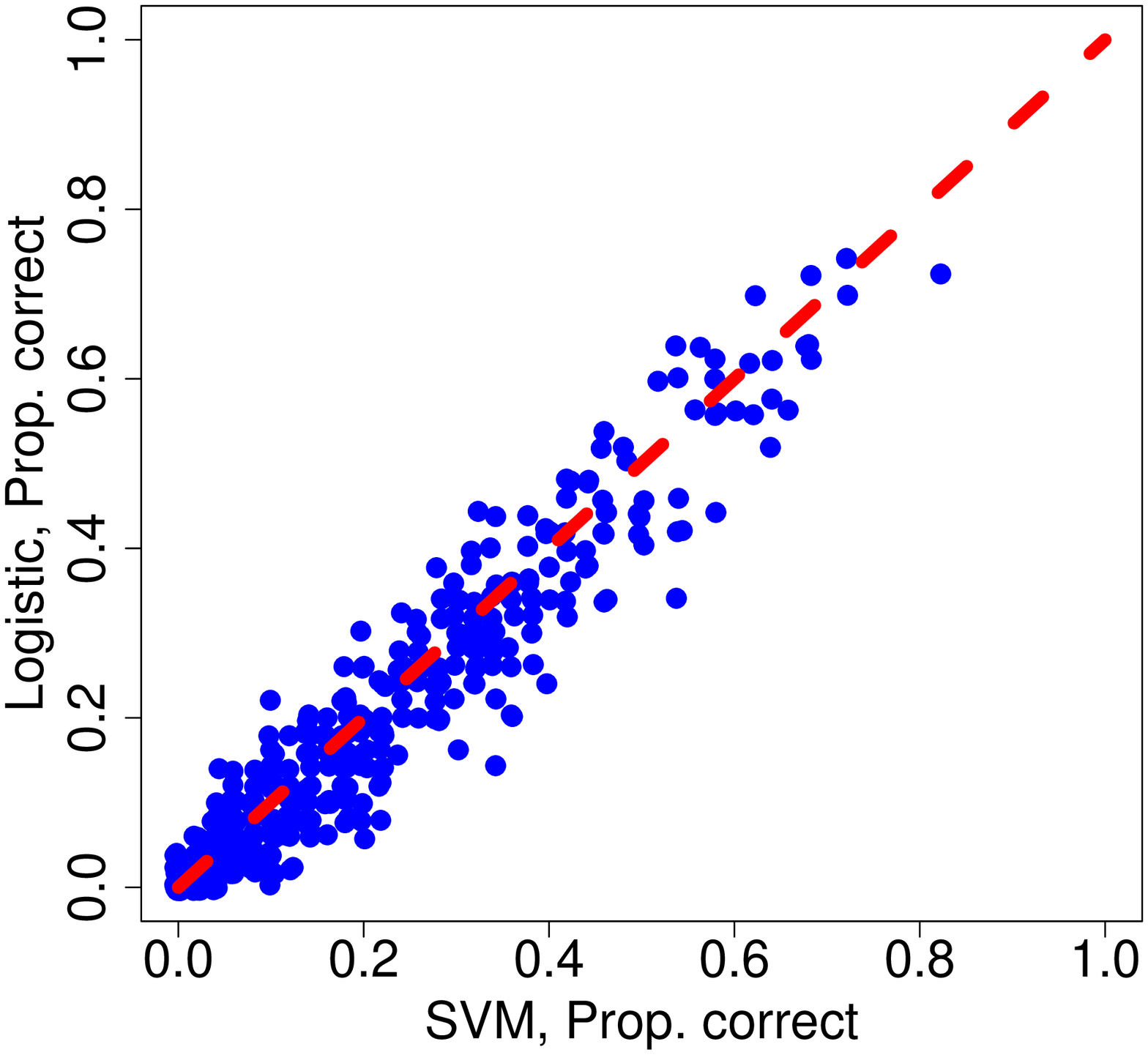}
\end{tabular}
&
\begin{tabular}{c}
\includegraphics[width=\figurewidth\textwidth, angle=270]
{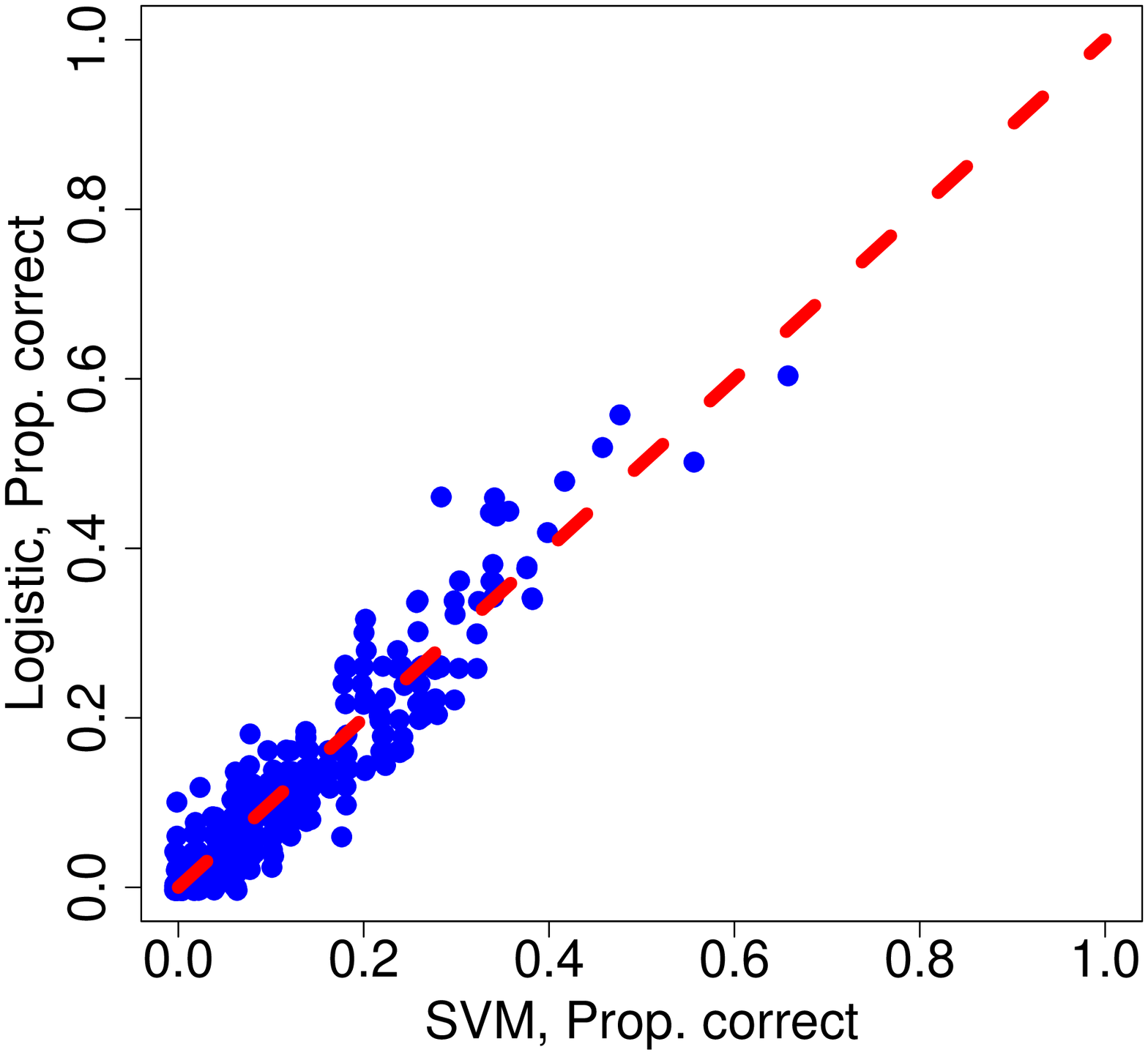}
\end{tabular}
&
\begin{tabular}{c}
\includegraphics[width=\figurewidth\textwidth, angle=270]
{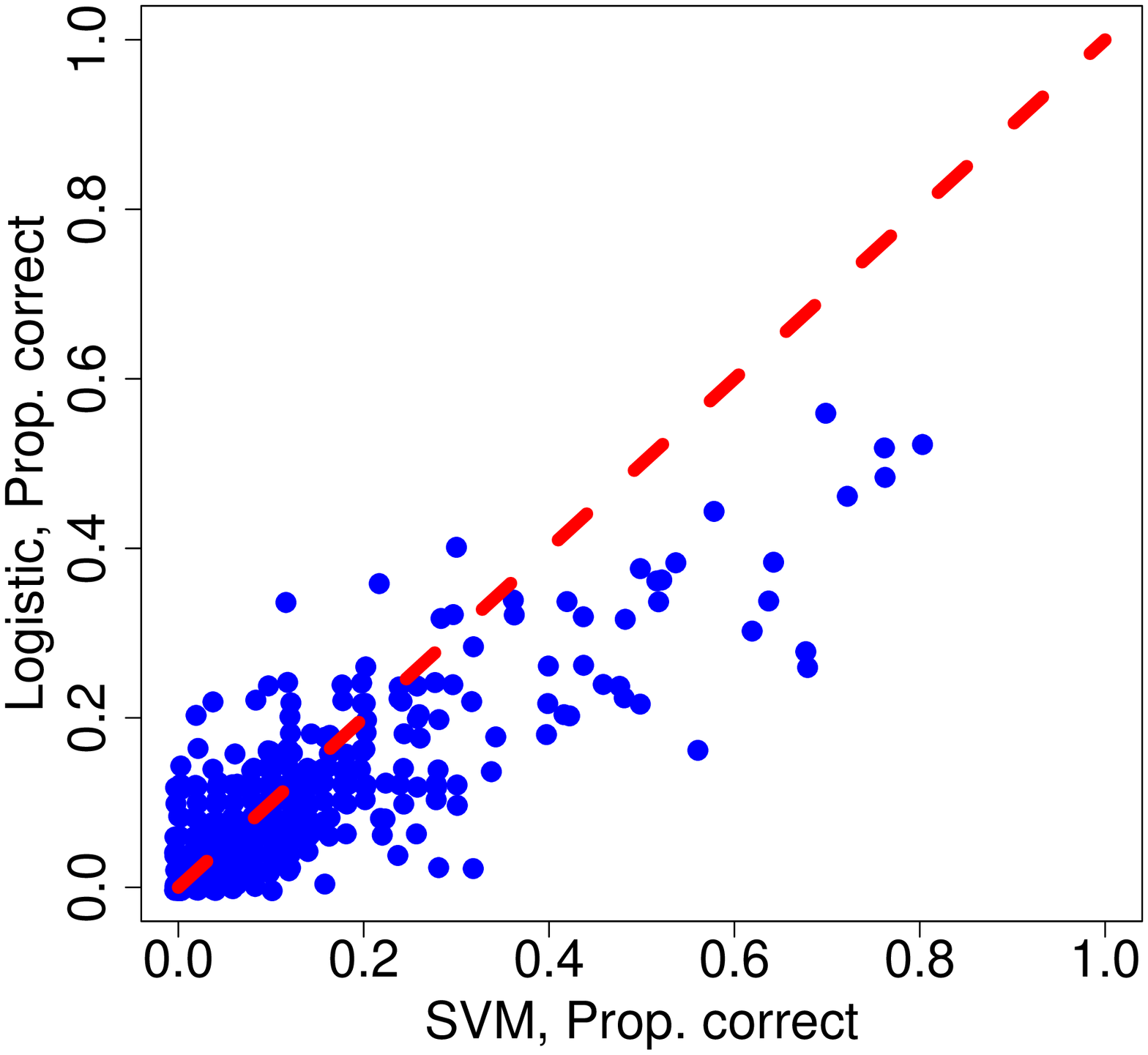}
\end{tabular}
&
\begin{tabular}{c}
\includegraphics[width=\figurewidth\textwidth, angle=270]
{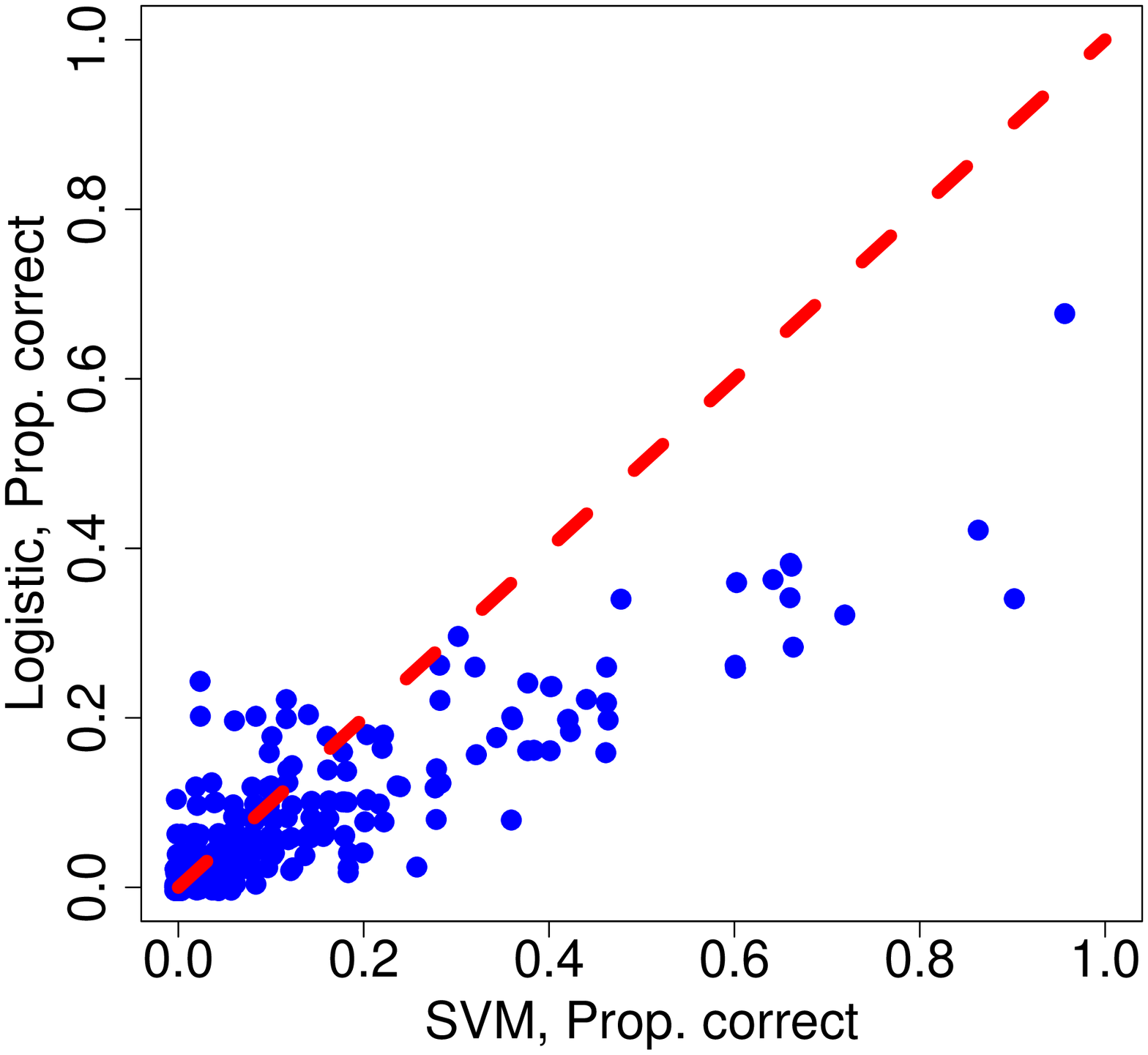}
\end{tabular}
\\
\hline
\begin{tabular}{c}
\begin{sideways}
	{
	\begin{small}
	\begin{tabular}{c} $g_{1}$ profile \\ $n=500$ \end{tabular}
	\end{small}
	}
\end{sideways}
\end{tabular}
&
\begin{tabular}{c}
\includegraphics[width=\figurewidth\textwidth, angle=270]
{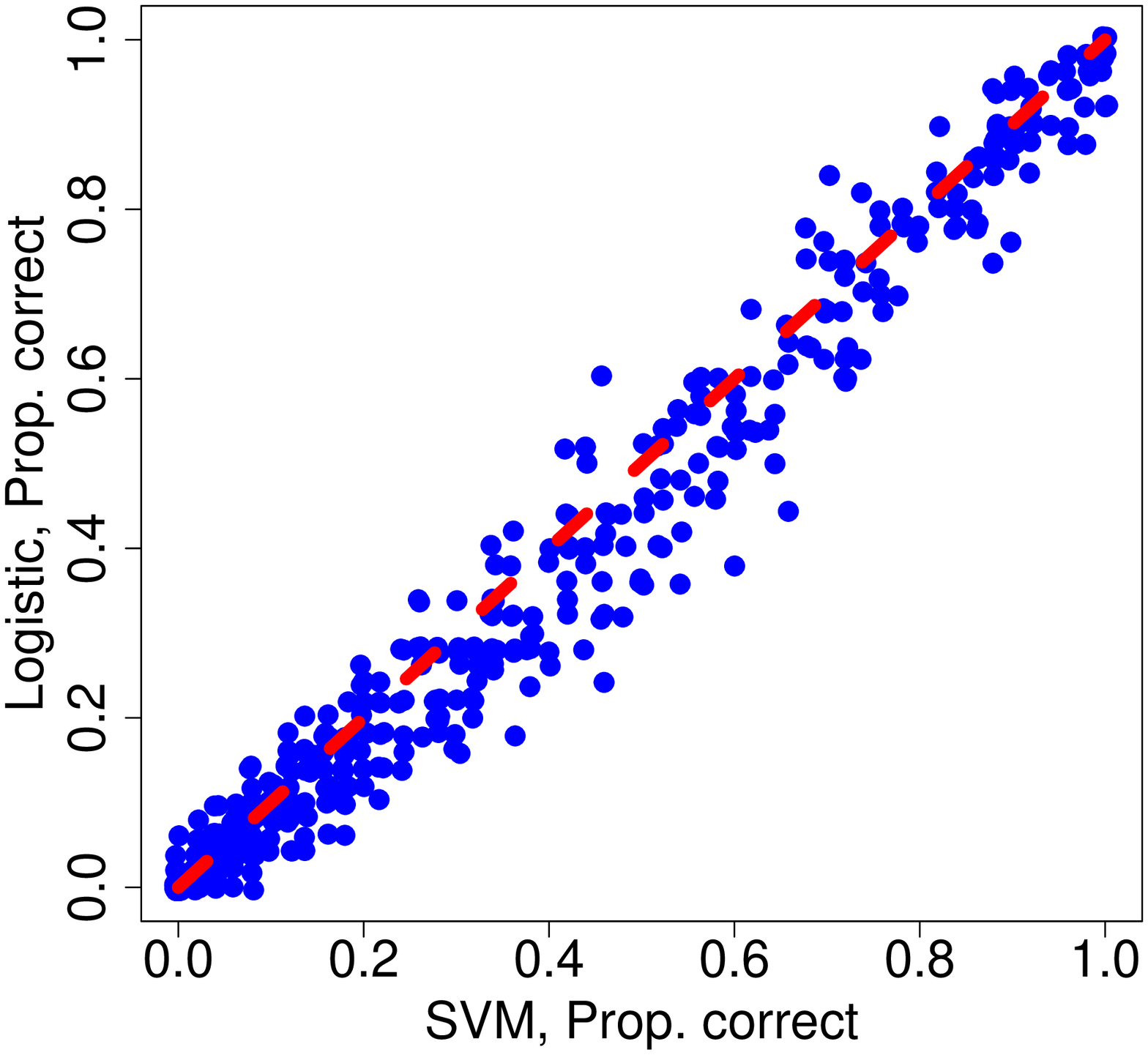}
\end{tabular}
&
\begin{tabular}{c}
\includegraphics[width=\figurewidth\textwidth, angle=270]
{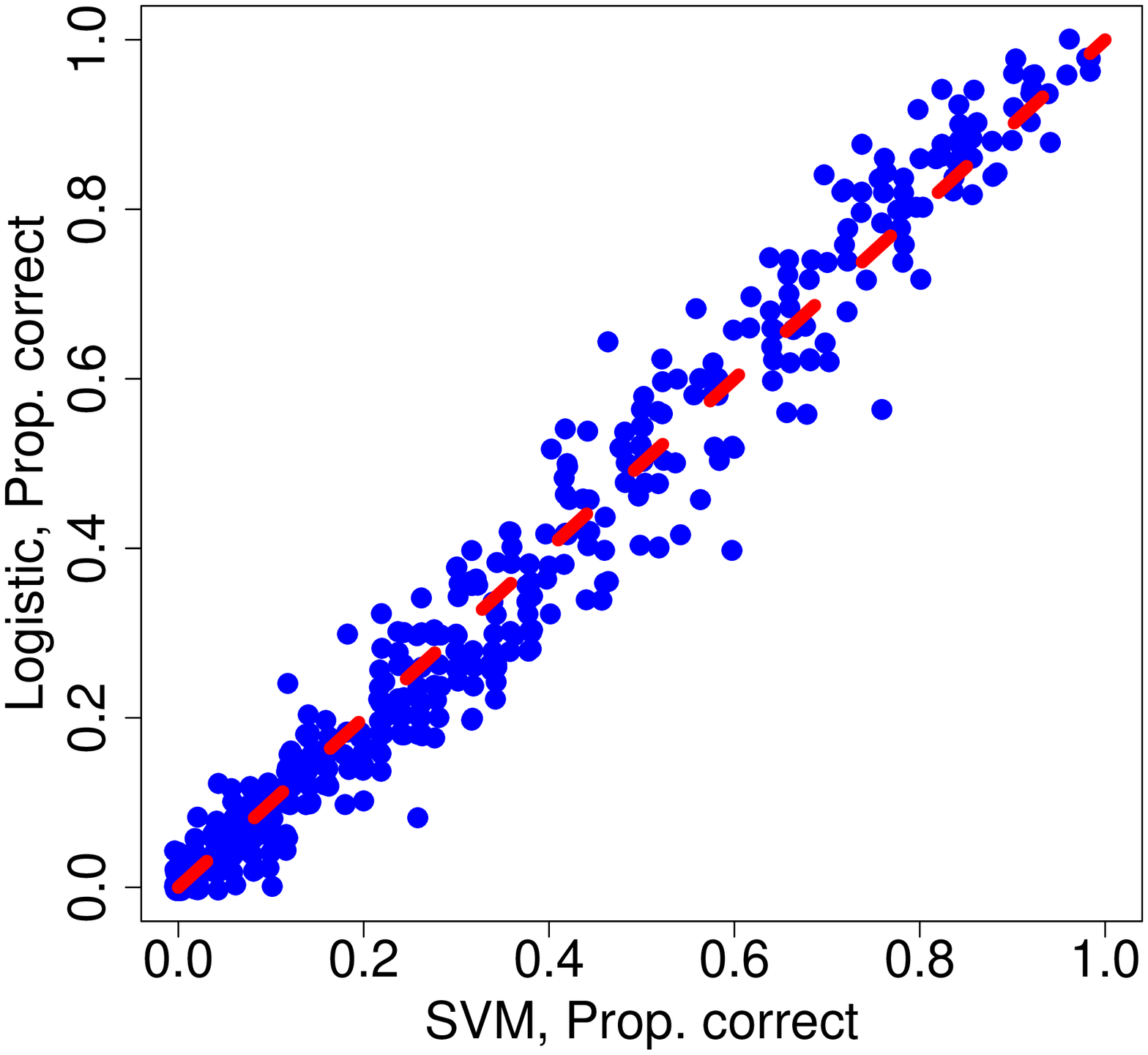}
\end{tabular}
&
\begin{tabular}{c}
\includegraphics[width=\figurewidth\textwidth, angle=270]
{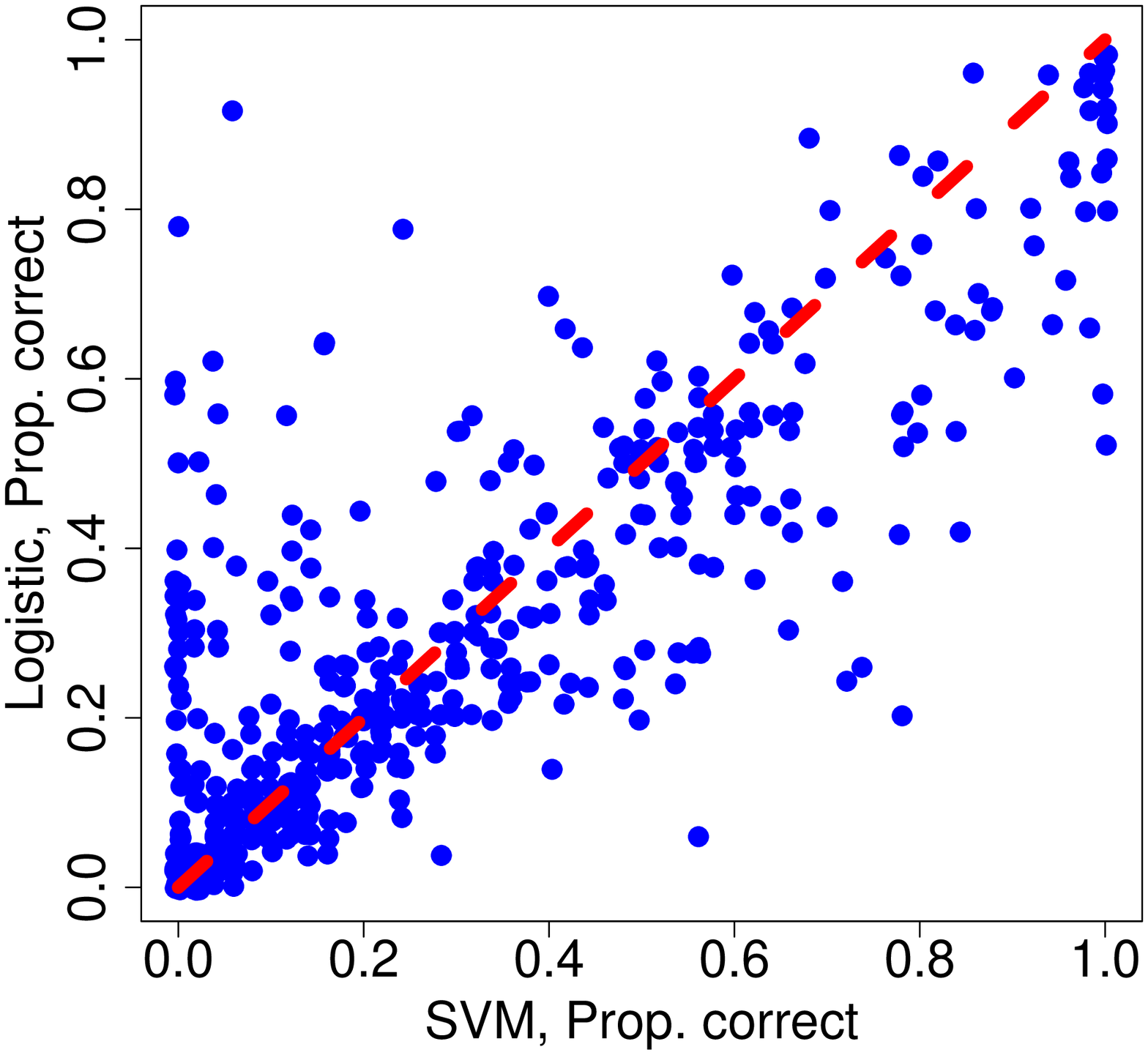}
\end{tabular}
&
\begin{tabular}{c}
\includegraphics[width=\figurewidth\textwidth, angle=270]
{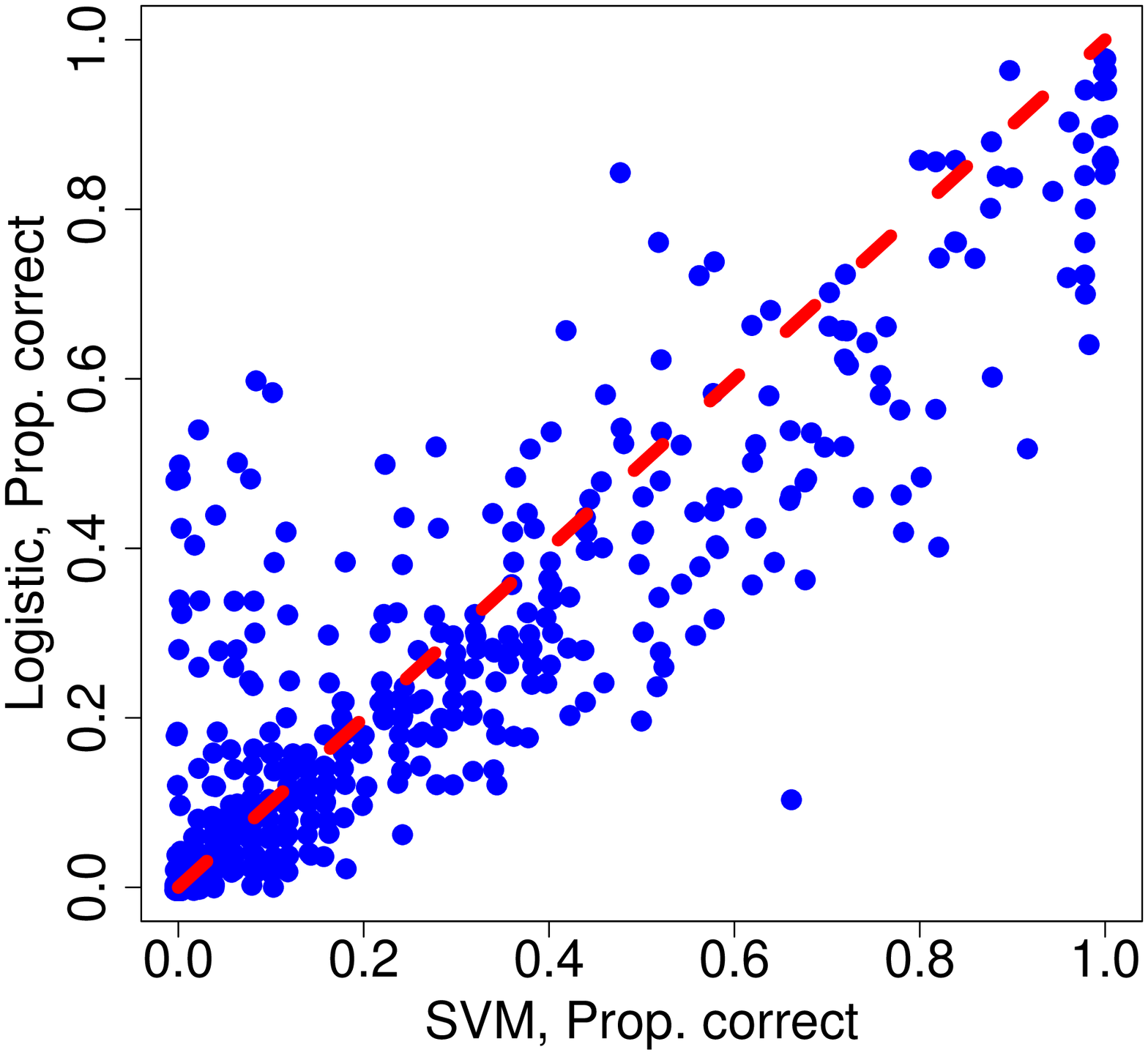}
\end{tabular}
\\
\hline
\begin{tabular}{c}
\begin{sideways}
	{
	\begin{small}
	\begin{tabular}{c} $g_{2}$ profile \\ $n=500$ \end{tabular}
	\end{small}
	}
\end{sideways}
\end{tabular}
&
\begin{tabular}{c}
\includegraphics[width=\figurewidth\textwidth, angle=270]
{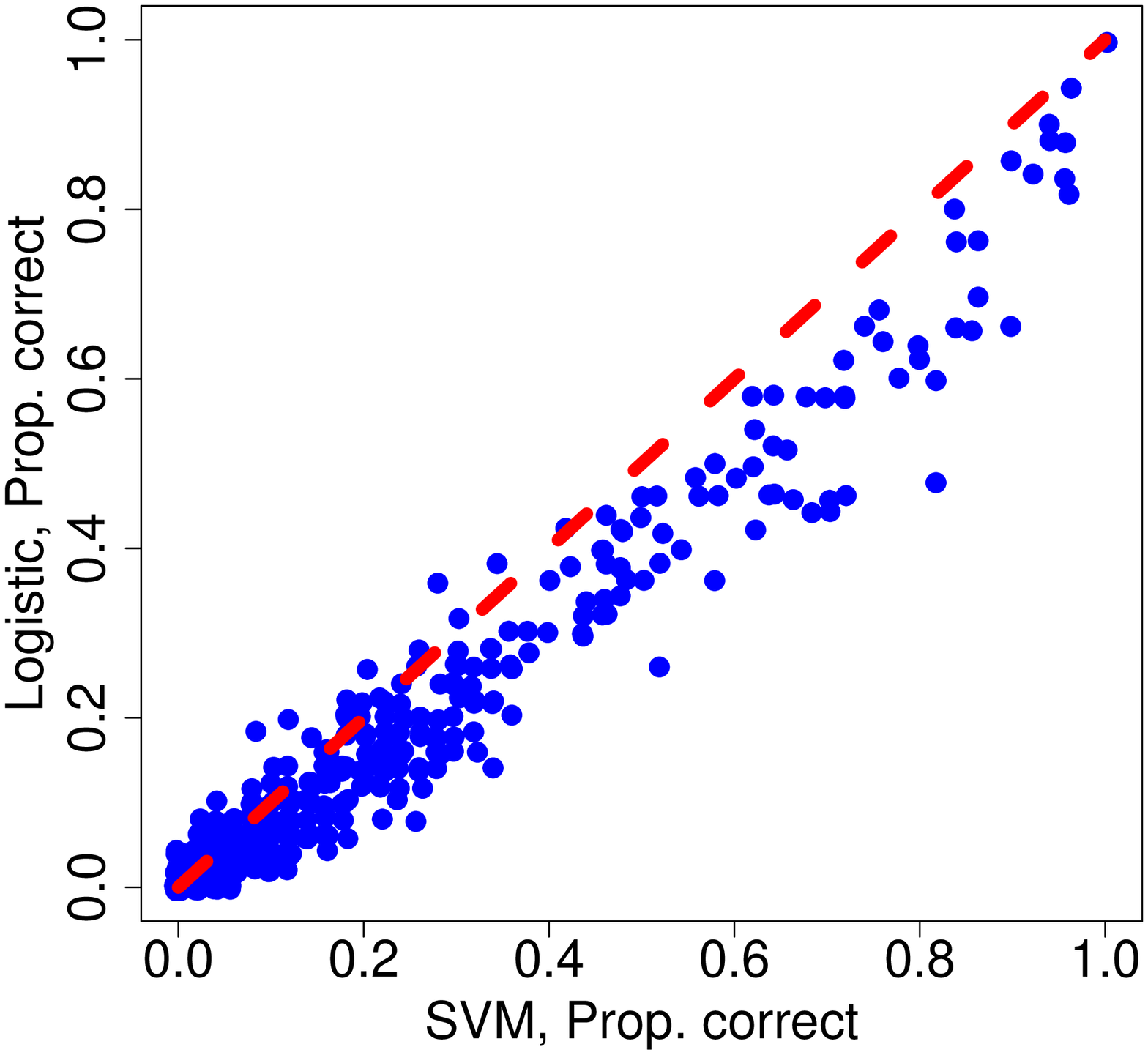}
\end{tabular}
&
\begin{tabular}{c}
\includegraphics[width=\figurewidth\textwidth, angle=270]
{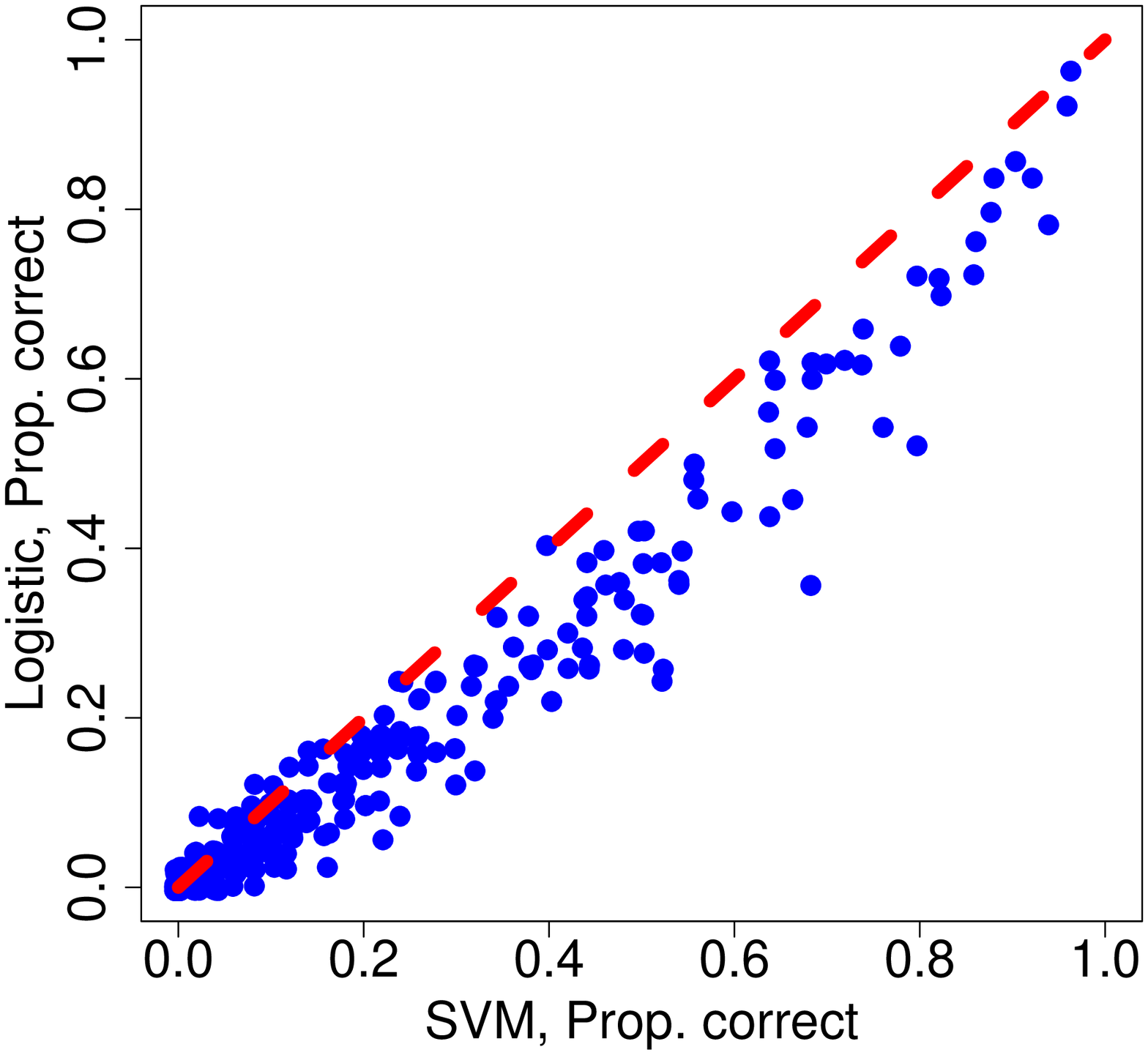}
\end{tabular}
&
\begin{tabular}{c}
\includegraphics[width=\figurewidth\textwidth, angle=270]
{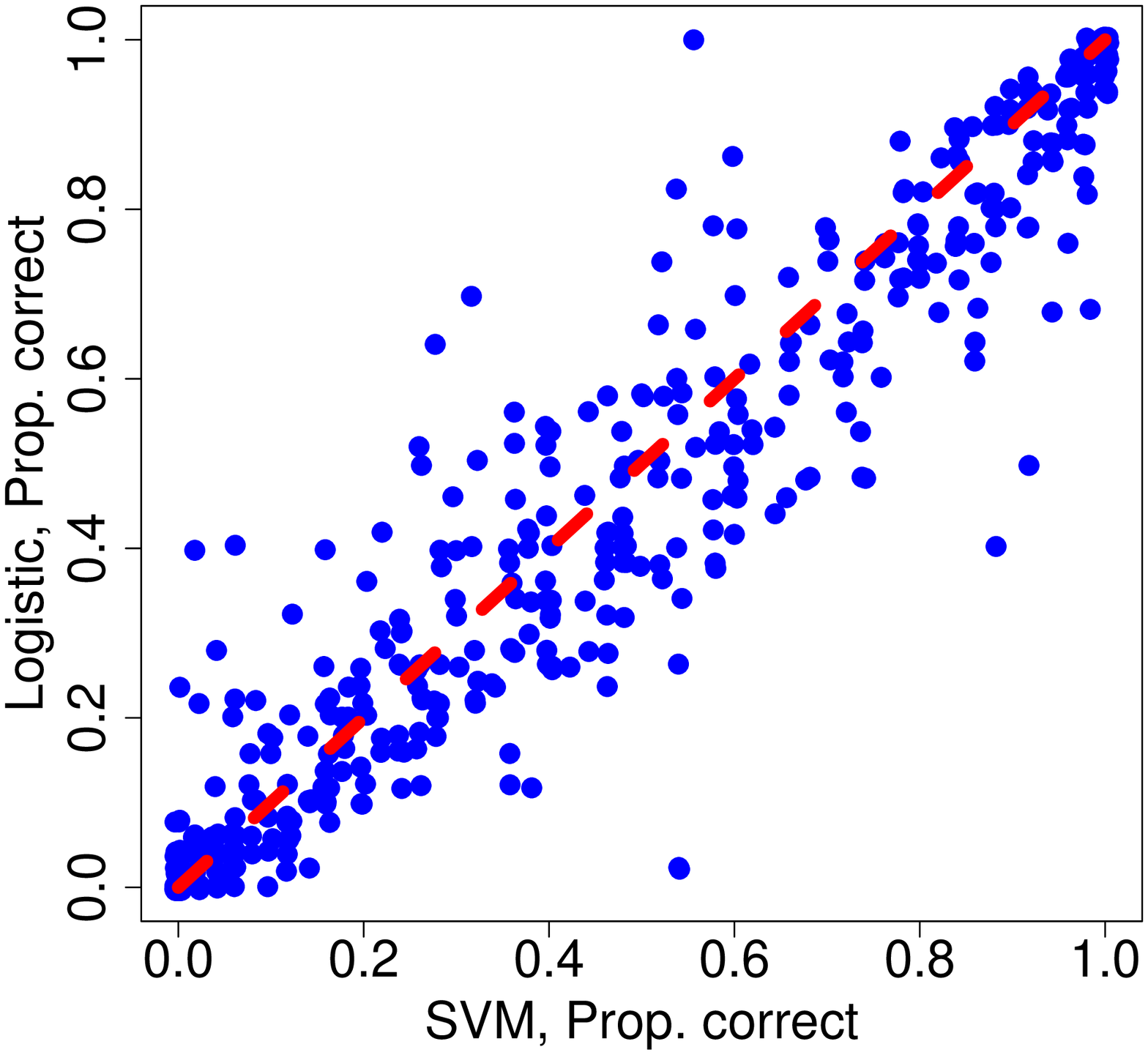}
\end{tabular}
&
\begin{tabular}{c}
\includegraphics[width=\figurewidth\textwidth, angle=270]
{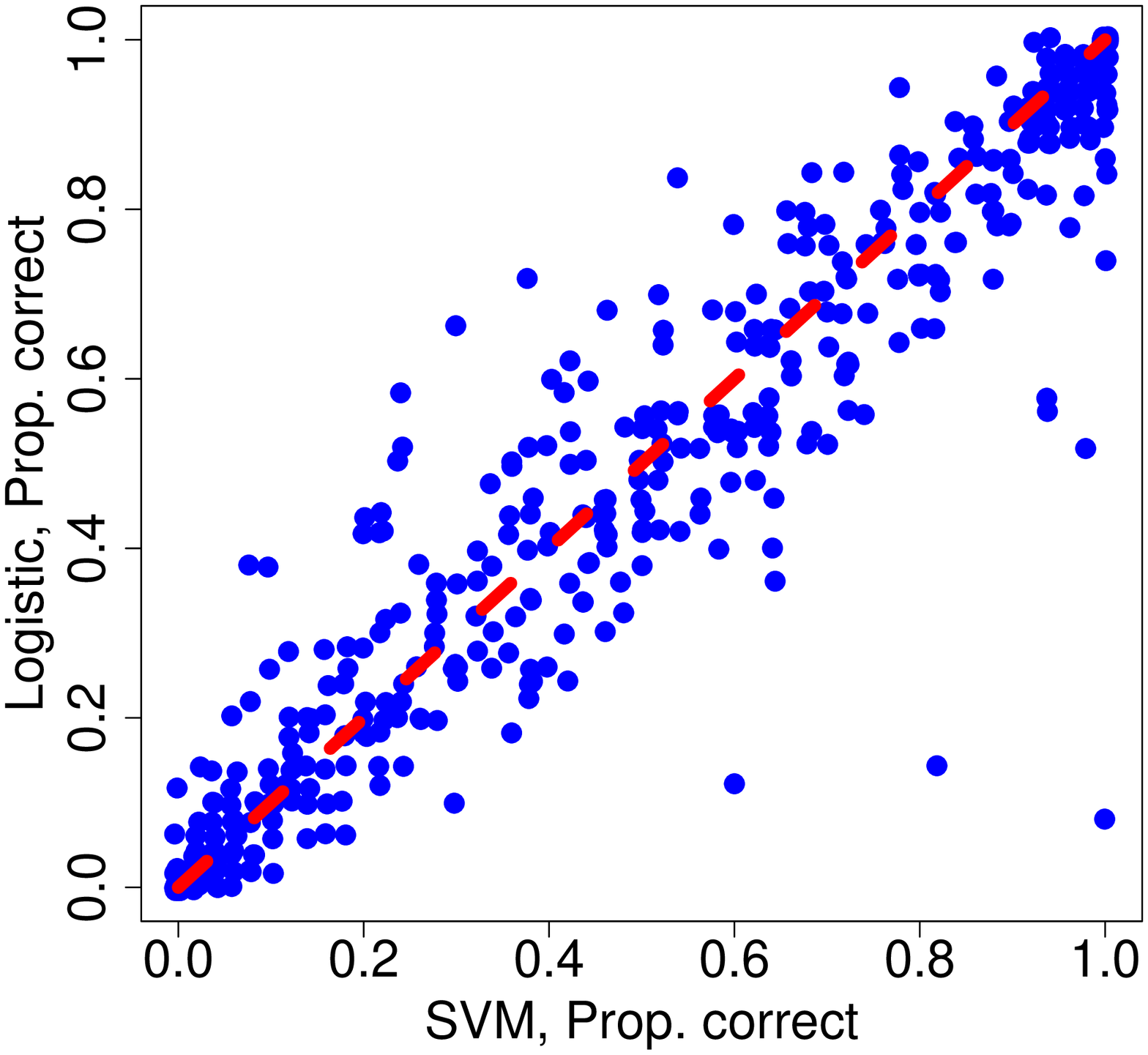}
\end{tabular}
\\
\hline
\begin{tabular}{c}
\begin{sideways}
	{
	\begin{small}
	\begin{tabular}{c} $g_{2}$ profile \\ $n=1,000$ \end{tabular}
	\end{small}
	}
\end{sideways}
\end{tabular}
&
\begin{tabular}{c}
\includegraphics[width=\figurewidth\textwidth, angle=270]
{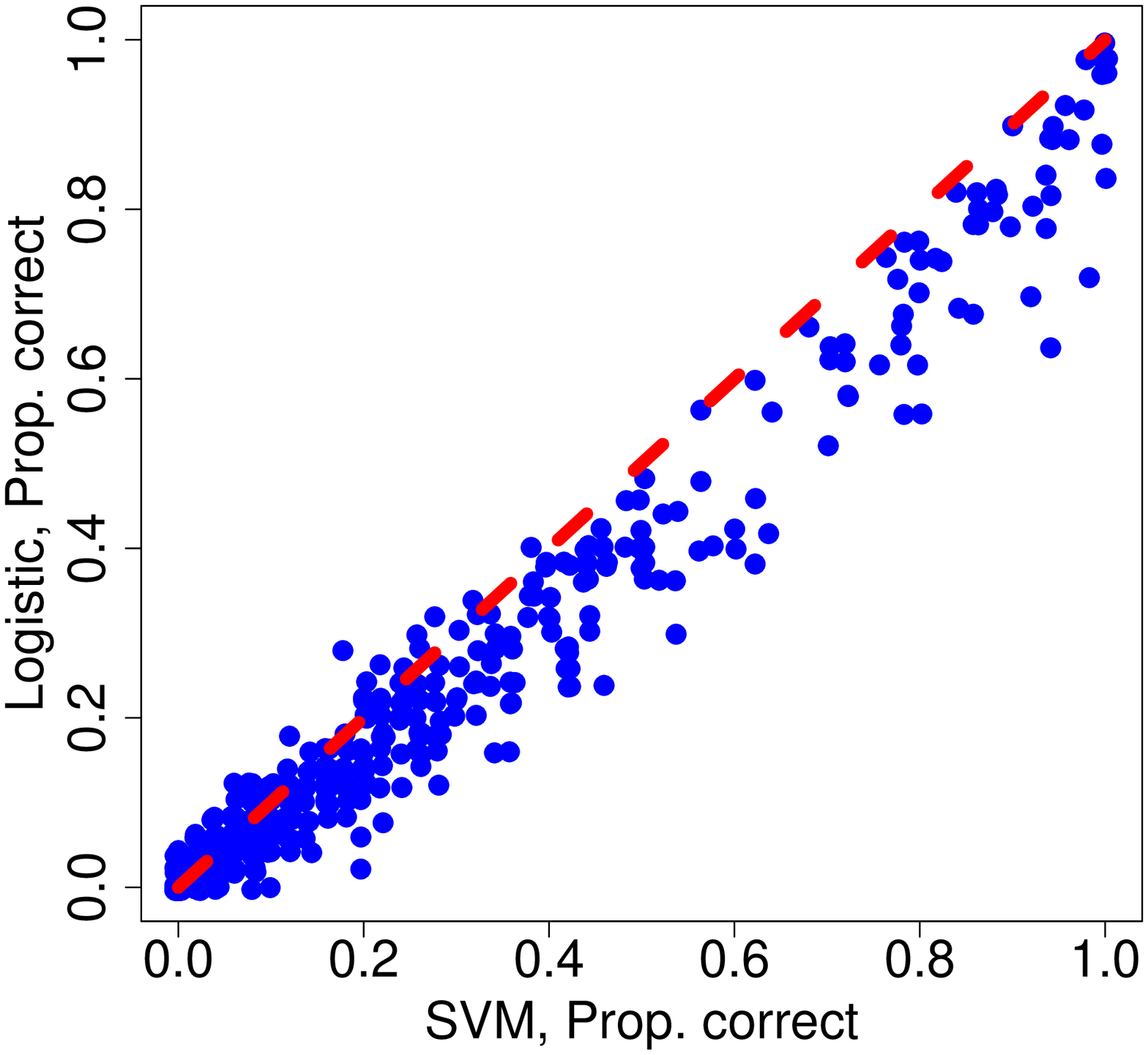}
\end{tabular}
&
\begin{tabular}{c}
\includegraphics[width=\figurewidth\textwidth, angle=270]
{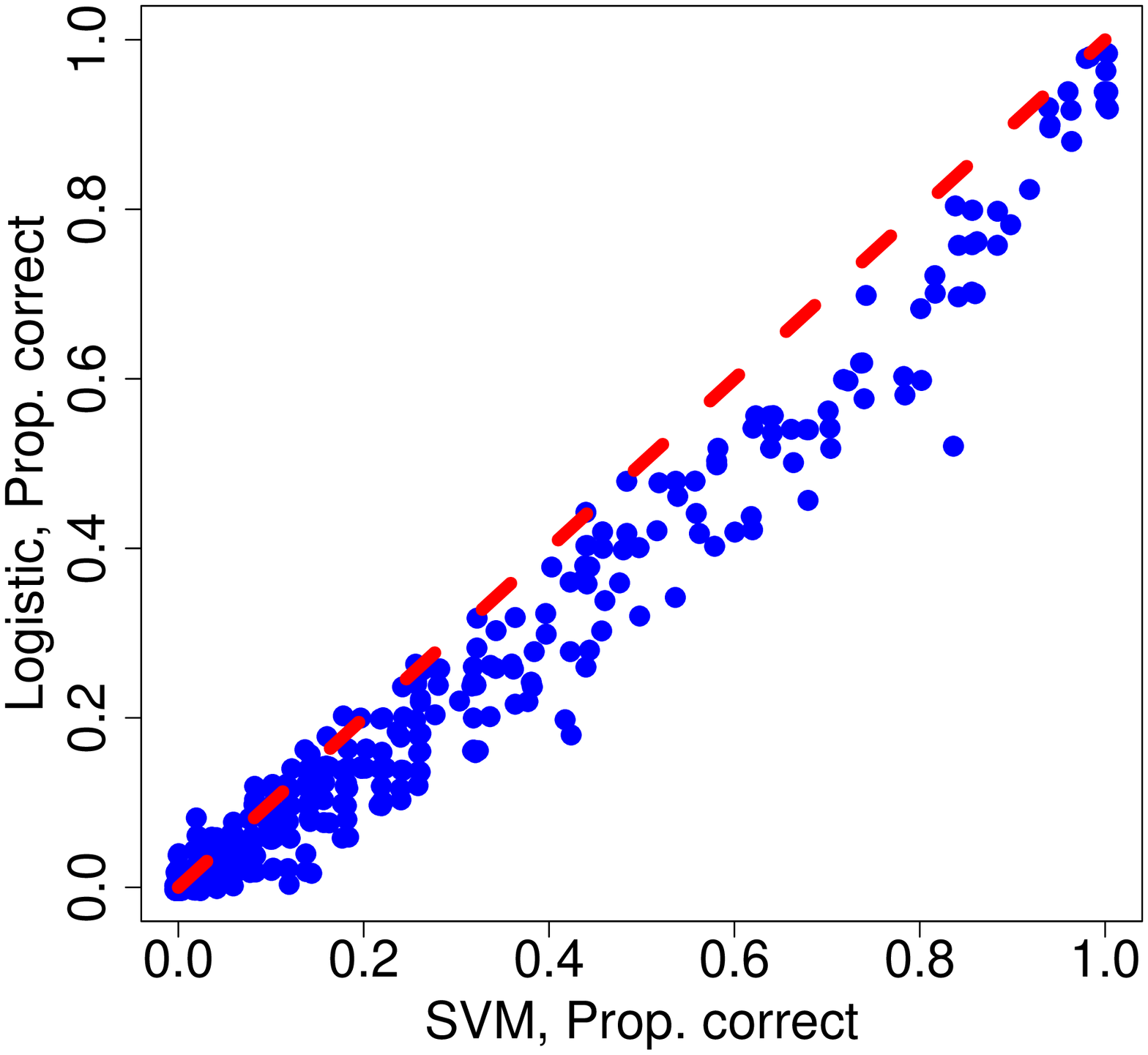}
\end{tabular}
&
\begin{tabular}{c}
\includegraphics[width=\figurewidth\textwidth, angle=270]
{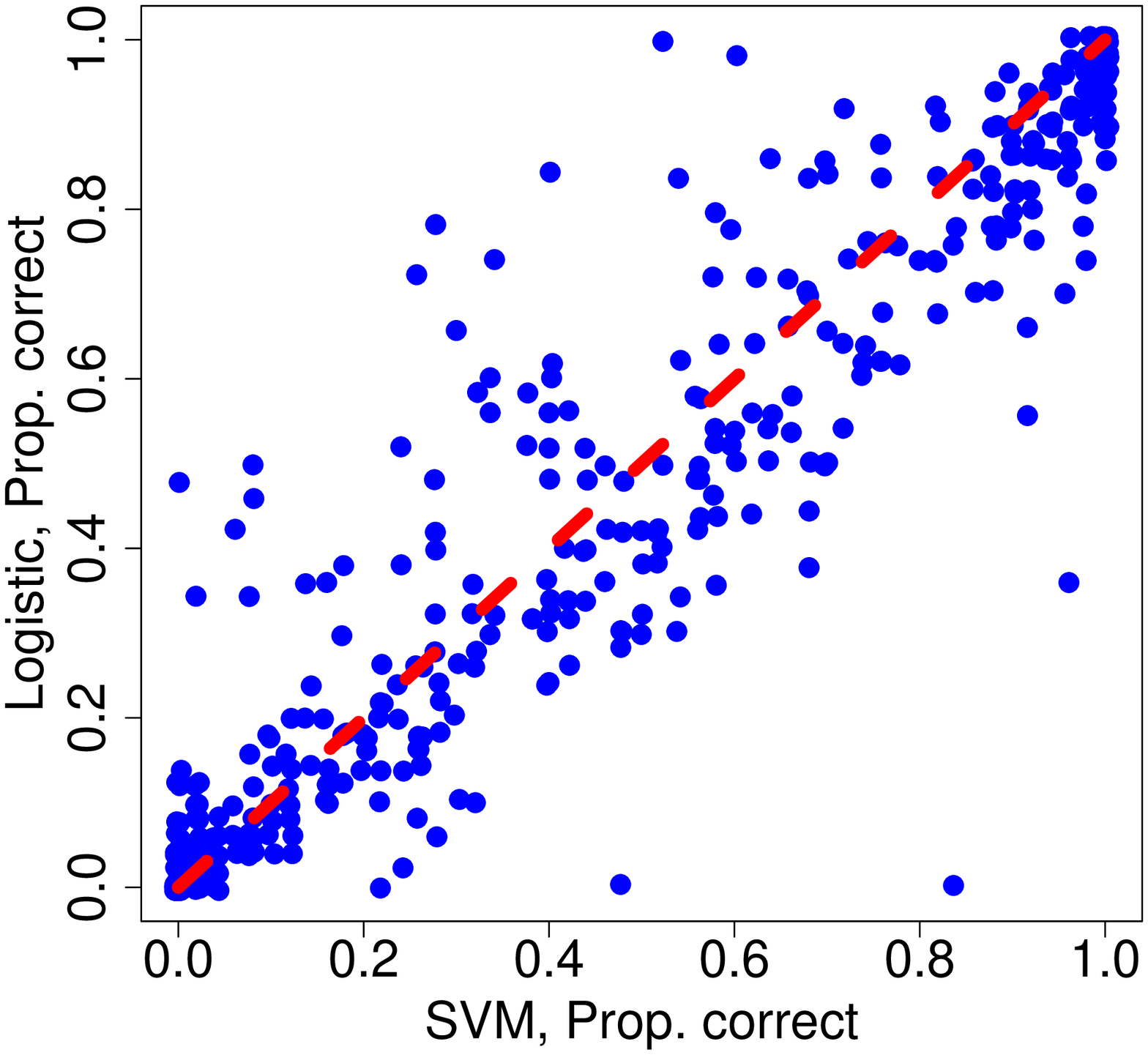}
\end{tabular}
&
\begin{tabular}{c}
\includegraphics[width=\figurewidth\textwidth, angle=270]
{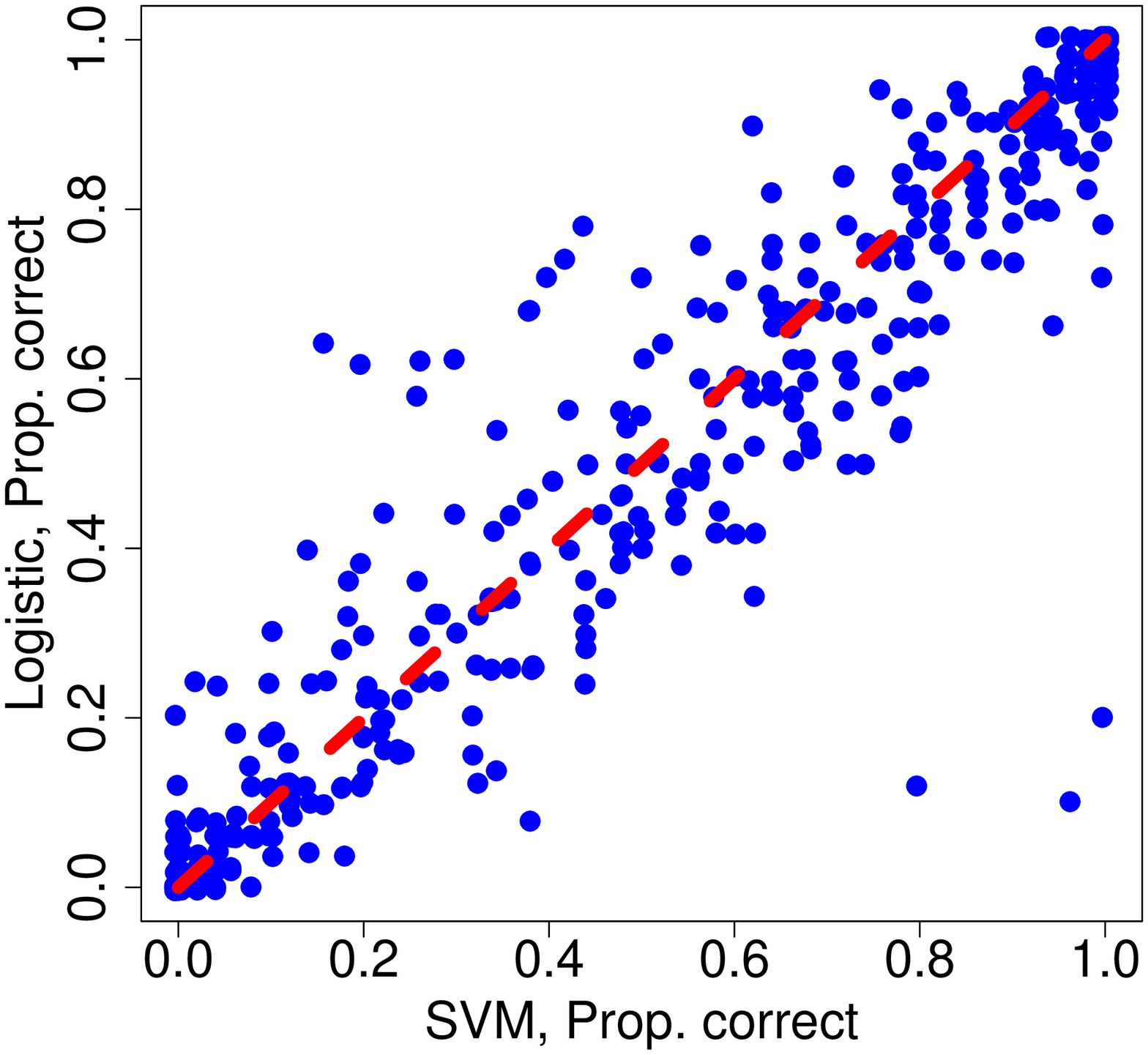}
\end{tabular}
\\
\hline
\end{tabular}
\caption{
\label{figure:finite_sample_sign_correct_proportion_comparison}
\textbf{Comparison of the proportion of sample paths containing sign correct estimates in finite samples:}
The GI condition (Theorem~\ref{result:generalized_irrepresentability_condition}) concerns an asymptotic guarantee and does not ensure the probability of correct sign recovery to be low if $\eta(\theta)$.
In these plots, we compare SVM and logistic classifiers in terms of probability of correct sign recovery in finite samples.
The SVM classifier seems to perform better in terms of the probability of correct sign recovery under Gaussian predictors and the ``blip'' conditional probability profile.
The SVM classifier also performs better in smaller sample sizes under mixed Gaussian predictors and the logistic conditional probability profile.
}
\end{center}	
\end{figure}
\afterpage{\clearpage}

\begin{table}[p]
\begin{center}
\begin{tabular}{|c|cc|}
\cline{2-3}
\multicolumn{1}{c|}{}
&
$p=08$
&
$p=16$
\\
\hline
&&
\\
\begin{tabular}{c}
\begin{sideways}
\begin{tabular}{c}
$g_{1}$ profile
\end{tabular}
\end{sideways}
\end{tabular}
&
\begin{tabular}{c}
\begin{tabular}{|c|cc|}
\hline
                  & \multicolumn{2}{|c|}{SVM}\\
                  & not MSC & MSC \\
\hline
Logistic not MSC & 31.8\% &  6.0\% \\
Logistic MSC     &  7.0\% & 55.2\%\\
\hline
\end{tabular}
\end{tabular}
&
\begin{tabular}{c}
\begin{tabular}{|c|cc|}
\hline
                  & \multicolumn{2}{|c|}{SVM}\\
                  & not MSC & MSC \\
\hline
Logistic not MSC & 15.8\% &  4.8\% \\
Logistic MSC     &  5.4\% & 74.0\% \\
\hline
\end{tabular}
\end{tabular}
\\
&&
\\
\begin{tabular}{c}
\begin{sideways}
\begin{tabular}{c}
$g_{2}$ profile
\end{tabular}
\end{sideways}
\end{tabular}
&
\begin{tabular}{c}
\begin{tabular}{|c|cc|}
\hline
                  & \multicolumn{2}{|c|}{SVM}\\
                  & not MSC & MSC \\
\hline
Logistic not MSC & 31.8\% &  5.2\% \\
Logistic MSC     & 26.0\% & 37.0\% \\
\hline
\end{tabular}
\end{tabular}
&
\begin{tabular}{c}
\begin{tabular}{|c|cc|}
\hline
                  & \multicolumn{2}{|c|}{SVM}\\
                  & not MSC & MSC \\
\hline
Logistic not MSC & 16.6\% &  3.6 \%\\
Logistic MSC     & 32.2\% & 47.6 \%\\
\hline
\end{tabular}
\end{tabular}
\\
&&
\\
\hline
\end{tabular}
\caption{
\label{table:logistic_gi_vs_svm_gi}
\textbf{Frequency at which SVM and logistic are model selection consistent:}
Each table shows the proportion out of 500 designs with mixed Gaussian predictors in which the $\ell_{1}$-norm penalized SVM and logistic classifiers are model selection consistent (MSC).
For most designs, both SVM and logistic would asymptotically contain estimates with all signs correct in their regularization paths.
Among the cases where only one of the two classifiers had would asymptotically contain a sign correct estimate in its path, the logistic classifier would be the correct one in most cases.
}
\end{center}	
\end{table}

\section{Discussion and concluding remarks}
\label{section:discussion}

In this paper, we have 
extended the asymptotic characterization of the distribution of LASSO estimates ($\ell_{1}$ penalized least squares) given by \citet{knight:2000:asymptotics} to more general loss and penalty functions in the parametric case.
The key to our extension consists of finding conditions under which it is possible to obtain a local quadratic approximation that is uniformly valid on a neighborhood of the risk minimizer.
Given the widespread use of convex loss functions use in the literature, the Convexity Lemma by \citet{pollard:1991:asymptotics-for-least-absolute-deviation-regression-estimators} was our tool of choice.
As we restrict attention to the parametric case, we have been able to keep the study of loss and penalty functions separate. 
To the possible extent, we have state our results in a modular fashion so they can be applied to various combinations of loss and penalty functions.

We have used the asymptotic characterization of the distribution of $\ell_{1}$ penalized parametric M-estimates to obtain sufficient conditions ensuring the existence of a model selection consistent estimate for some appropriate value of the regularization parameter.
Interestingly, the condition involves the Hessian but not the variance of the score function evaluated at the risk minimizer.
\citet{ravikumar:2008:high-dimensional-covariance-estimation-minimizing-ell1-penalized-log-determinant-divergence} have obtained a similar condition in the non-parametric case ($p\gg n$) for the penalized maximum likelihood estimate of Gaussian covariance matrices.
That suggests the results we present in this paper can be extended to the non-parametric setting under appropriate conditions, which will be the theme of future research.
We also show (Theorem~\ref{result:simplified_gi_condition}) that, under appropriate assumptions, the condition for sign-consistency of $\ell_{1}$ penalized parametric M-estimates can be expressed solely in terms of the matrix of second moments of the predictors. 

Our simulations provide ample empirical evidence to the theory we have presented in the context of SVM and logistic regression classification.
For Gaussian predictors and a given design, one of the two can happen: both logistic regression and linear SVM classifiers will are sparsistent and sign-consistent or neither of them is.
In finite samples, SVM seems to enjoy a slight advantage in picking the correct signs in the cases we simulated.
For a set of randomly selected designs with non-Gaussian predictors, logistic regression classifiers were sparsistent and sign-consistent more frequently than SVM classifiers.
In finite samples, however, the evidence in favor of either SVM or logistic regression classifiers was mixed.

\section*{Acknowledgments}

The authors would like to thank Youjuan Li, Wang Li and Ji Zhu for providing code for the $\ell_{1}$-norm penalized estimation of SVMs and to thankfully acknowledge support from grants NSF DMS-0605165 (06-08), NSFC (60628102), a grant from MSRA and a CDI award from NSF.
Guilherme Rocha would like to acknowledge helpful discussions with Karen Kafadar, Nicolai Meinshausen and Ram Rajagopal.

\renewcommand{\thesection}{\Alph{section}}
\setcounter{section}{0}

\renewcommand{\theequation}{A-\arabic{equation}}
\setcounter{equation}{0}

\onehalfspacing

\section{Proofs of theoretical results}
\label{section:proof}

\subsection{Proof of results in Section \ref{section:theory}}

We now state and prove the results in Section \ref{section:theory}.
Before that, we prove technical Lemma \ref{result:rescaled_recentered_objective_function} which is used in the proof of Theorem \ref{result:main_result}.

\begin{lemma}
\label{result:rescaled_recentered_objective_function}
Define:
\begin{eqnarray*}
	V_{\theta}^{(n)}(\hat{\bbP}_{n}, \lambda_{n},  u)
	:=
	\sum_{i=1}^{n}
	\left[
	L\left(Z_{i}, \theta+\frac{ u}{q_{n}}\right)
	-
	L\left(Z_{i}, \theta                 \right)
	\right]
	+
	\lambda_{n}
	\cdot
	\left[
	T\left(\theta+\frac{ u}{q_{n}}\right)
	-
	T\left(\theta\right)
	\right].
\end{eqnarray*}
Then:
\begin{eqnarray}
	q_{n}\left(\hat{\theta}_{n}(\lambda_{n})-\theta\right)
	& = & 
	\argmin_{ u\in \bbR^{p}}\left[V_{\theta}^{(n)}(\hat{\bbP}_{n}, \lambda_{n},  u)\right].
\end{eqnarray}
\begin{proof}[Proof of Lemma \ref{result:rescaled_recentered_objective_function}.]
From the definition of $\hat{\theta}_{n}(\lambda_{n})$, we know that:
\begin{small}
\begin{eqnarray*}
	\hat{\theta}_{n}(\lambda_{n})
	& = & 
	\argmin_{t\in \Theta}
	\left[
	\sum_{i=1}^{n}
	\left[
	L\left(Z_{i}, \theta + \frac{q_{n}\left(t-\theta\right)}{q_{n}}\right)
	\right]
	+
	\lambda_{n}\cdot
	T\left(\theta + \frac{q_{n}\left(t-\theta\right)}{q_{n}}\right)
	\right]
	\\
	& = & 	
	\argmin_{t\in \Theta}
	\left[
	\left[
	\sum_{i=1}^{n}
	\left[
	L\left(Z_{i}, \theta + \frac{q_{n}\left(t-\theta\right)}{q_{n}}\right)
	-
	L\left(Z_{i}, \theta \right)
	\right]
	\right]
	+
	\lambda_{n}\cdot
	\left[
	T\left(\theta + \frac{q_{n}\left(t-\theta\right)}{q_{n}}\right)
	-
	T\left(\theta\right)
	\right]
	\right].
\end{eqnarray*} 
\end{small}
The result follows from making a variable transformation $ u(t) = q_{n}\cdot(t-\theta)$ and letting $\hat{ u} =  u(\hat{\theta}_{n}(\lambda_{n}))$.
\end{proof}
\end{lemma}

\begin{proof}[Proof of Theorem \ref{result:assembling_theorem}]
\begin{enumerate}
	\item [a)] The conclusion in (a) follows easily from the triangular inequality, since for each compact set $K$:
\begin{eqnarray*}
\sup_{ u\in K}
\left|
V_{\theta}^{(n)}(\hat{\bbP}_{n}, \lambda_{n},  u)
-
V_{\theta}(\bfW,  u)
\right|
& \le & 
\sup_{ u\in K}
\left|
\left[
\sum_{i=1}^{n}
\left[
L\left(Z_{i}, \theta+\frac{ u}{q_{n}}\right)
-
L\left(Z_{i}, \theta\right)
\right]
\right]
-C_{\theta}(\bfW,  u)
\right|
\\
& + & 
\sup_{ u\in K}
\left|
\lambda_{n}
\cdot
\left[
T\left(\theta+\frac{ u}{q_{n}}\right)
-T(\theta)
\right]
-\lambda\cdot G_{\theta}( u)
\right|.
\end{eqnarray*}

\item [b)] Define:
\begin{eqnarray*}
	\hat{ u}_{n} & = & q_{n}\cdot\left(\hat{\theta}(\lambda_{n})-\theta\right),
	\\
	\hat{ u} & = & \arg\min \left[C_{\theta}(\bfW,  u) - \lambda\cdot G_{\theta}( u)\right].
\end{eqnarray*}

For any compact set $K_{\varepsilon}$ and each $n$, we know:
\begin{eqnarray*}
	\bbP\left(\left|\hat{ u}_{n}-\hat{ u}\right|>\delta\right) 
	& = & 
	\bbP\left(\left|\hat{ u}_{n}-\hat{ u}\right|>\delta\left|\vphantom{\int}\right. \hat{ u}_{n}\in K_{\varepsilon}\right) \cdot 
	\bbP\left(\hat{ u}_{n}\in K_{\varepsilon}\right)
	\\
	& & + 
	\bbP\left(\left|\hat{ u}_{n}-\hat{ u}\right|>\delta\left|\vphantom{\int}\right. \hat{ u}_{n}\not\in K_{\varepsilon}\right) \cdot 
	\bbP\left(\hat{ u}_{n}\not\in K_{\varepsilon}\right).
	\\
	& \le & 
	\bbP\left(\left|\hat{ u}_{n}-\hat{ u}\right|>\delta\left|\vphantom{\int}\right. \hat{ u}_{n}\in K_{\varepsilon}\right)
	+ 
	\bbP\left(\hat{ u}_{n}\not\in K_{\varepsilon}\right).
\end{eqnarray*}

Since $\hat{ u}_{n}$ is $O_{p}(1)$, there exists a compact set $K_{\varepsilon}$ such that $\lim\limits_{n\rightarrow\infty}\bbP\left(\hat{ u}_{n} \not\in K_{\varepsilon}\right) = 0$ and thus:
\begin{eqnarray*}
	\lim_{n\rightarrow\infty}
	\bbP\left(\left|\hat{ u}_{n}-\hat{ u}\right|>\delta\left|\vphantom{\int}\right. \hat{ u}_{n}\not\in K_{\varepsilon}\right) \cdot 
	\bbP\left(\hat{ u}_{n}\not\in K_{\varepsilon}\right)
	& \le &
	\lim_{n\rightarrow\infty}
	\bbP\left(\hat{ u}_{n}\not\in K_{\varepsilon}\right)
	= 
	0.	
\end{eqnarray*}
To show the second term vanish, the uniform convergence over compact sets gives that:
\begin{eqnarray*}
	\lim_{n\rightarrow\infty}\bbP\left(\left|\hat{ u}_{n}-\hat{ u}\right|>\delta\left|\vphantom{\int}\right. \hat{ u}_{n}\in K_{\varepsilon}\right) = 0,
\end{eqnarray*}
\end{enumerate}
which concludes the proof.
\end{proof}

\begin{proof}[Proof of Lemma \ref{result:uniform_convergence_for_loss}]
\textbf{Proof of a) pointwise convergence:}
To establish pointwise convergence, define:
\begin{eqnarray}
	\begin{array}{rcl}
		\delta_{i,n}( u) 
		& = & 
		L(Z_{i}, \theta+\frac{ u}{\sqrt{n}}) 
		- 
		L(Z_{i}, \theta)
		\\	
		D_{i} & = & \nabla_{b}L(Z_{i}, \theta)
		\\
		R_{i,n}( u) & = & \delta_{i,n}( u) - D_{i}\frac{ u}{\sqrt{n}}
		\\
		W_{n} & = & \frac{1}{\sqrt{n}}\sum_{i=1}^{n}D_{i}
		\\
		B_{n}( u) & = & \sum_{i=1}^{n}\delta_{i,n}( u).
	\end{array}
\end{eqnarray}
In terms of these definitions, we have:
\begin{eqnarray}
	B_{n}( u) 
	& = & 
	n\cdot\left[
	\frac{1}{n}
	\sum_{i=1}^{n}
	\left[
	L\left(Z_{i}, \theta+\frac{ u}{\sqrt{n}}\right)
	-
	L\left(Z_{i}, \theta\right)
	\right]
	\right]
	\nonumber
	\\
	& = &
	\label{equation:bn_expansion_step_1}
	W_{n}^{T} \cdot  u + \sum_{i=1}^{n}\left[R_{i,n}( u)\right].
\end{eqnarray}

Now, because $\theta$ is optimal we have $\bbE D_{i} = 0$ and, thus:
\begin{eqnarray*}
	\bbE\left[B_{n}( u)\right] & = & \sum_{i=1}^{n}\bbE\left[R_{i,n}( u)\right].
\end{eqnarray*}

Summing $\bbE\left[B_{n}( u)\right]$ and subtracting $\sum_{i=1}^{n}\bbE\left[R_{i,n}( u)\right]$ from the right hand side of \eqref{equation:bn_expansion_step_1}:
\begin{eqnarray*}
	B_{n}( u) 
	& = & 
	\bbE\left[B_{n}( u) \right]
	+
	W_{n}^{T} u 
	+ 
	\sum_{i=1}^{n}\left[R_{i,n}( u)-\bbE\left[ R_{i,n}( u)\right]\right].
\end{eqnarray*}

Pointwise convergence for each $ u$ follows from obtaining a quadratic approximation to $\bbE\left[B_{n}( u)\right]$, a weak convergence for $W_{n}^{T} u$ and proving that the last term is $o_{p}(1)$.
These facts are established next.

\textbf{i) Quadratic approximation to $\bbE\left[B_{n}( u)\right]$:}
	
	First notice that this term is just the difference of the risk function evaluated at  $\theta$ and $\theta+\frac{ u}{\sqrt{n}}$:
	\begin{eqnarray*}
		\bbE\left[B_{n}( u)\right]
		& = & 
		n\cdot
		\bbE\left[
		\frac{1}{n}
		\sum_{i=1}^{n}
		\left[
		L\left(Z_{i}, \theta+\frac{ u}{\sqrt{n}}\right)
		-
		L\left(Z_{i}, \theta\right)
		\right]
		\right]
		\\
		& = & 
		n\cdot\left[
		\bbE\left[L\left(\bfZ, \theta+\frac{ u}{\sqrt{n}}\right)\right]
		-
		\bbE\left[L(\bfZ, \theta)\right]		
		\right]
		\\
		& = & 
		n\cdot\left[
		R\left(\theta+\frac{ u}{\sqrt{n}}\right)
		-
		R\left(\theta\right)
		\right]
	\end{eqnarray*}
	
	Since the risk function $R$ is twice differentiable (L2.c) and $\theta$ is optimal, the gradient of the risk with respect to its argument $t$ must be zero at $t=\theta$.
	In addition, for $H(\theta)$ as defined in assumption L2.c, we can write the approximation:
	\begin{eqnarray*}
		\bbE\left[B_{n}( u)\right]
		& = & 
		n\cdot \left[\frac{ u^{T}}{\sqrt{n}}\cdot\left[H(\theta)\vphantom{\frac{ u}{\sqrt{n}}}\right]\cdot\frac{ u}{\sqrt{n}}
		+
		o\left(\left(\frac{|u|}{\sqrt{n}}\right)^{2}\right)\right].
		\\
		& = & 
		u^{T}\cdot H(\theta)\cdot u
		+
		o(1).
	\end{eqnarray*}
	
\textbf{ii)] Weak convergence of $W_{n}^{T} u$:}
	
	Optimality of $\theta$ and differentiability of the risk function imply that $\bbE\left[ D_{i} \right]= 0$, thus $\bbE\left[W_{n}\right]=0$.

	Since $\nabla_{b}L(Z, b)$ exists almost everywhere, all terms in the summation defining $W_{n}^{T}$ almost surely exist. Since the terms in the summation are i.i.d. and each has finite variance, the Central Limit applies and we can conclude that:
	\begin{eqnarray*}
		W_{n}\convinlaw N\left(0, J(\theta)\right), \mbox{ with } J(\theta) = \bbE\left[\nabla_{t}L(\bfZ, \theta)\nabla_{t}L(\bfZ, \theta)^{T}\right].
	\end{eqnarray*}
	
\textbf{iii) $\sum_{i=1}^{n}\left[R_{i,n}( u)-\bbE\left[ R_{i,n}( u)\right]\right]$ is $o_{p}(1)$:}
	
	Let $\xi_{i} = R_{i,n}( u)-\bbE\left[R_{i,n}( u)\right]$.
	Since convergence in quadratic mean implies convergence in probability, it is enough to prove that $\bbP \left|\sum_{i=1}^{n} \xi_{i}\right|^{2} = o(1)$.
	
	Clearly $\bbE\xi_{i} = 0$ for all $i$. That, along with independence across the observed samples, yields:
	\begin{eqnarray*}
		\bbE\left[ \left|\sum_{i=1}^{n} \xi_{i}\right|^{2}\right] 
		& = &  
		\var\left[ \sum_{i=1}^{n} \xi_{i}\right] 
		\\
		& = & 
		\sum_{i=1}^{n}\var\left[  \xi_{i}\right]
		\\
		& = & 
		\sum_{i=1}^{n}\left(\bbE\left[  \left|R_{i, n}( u)\right|^{2}\right] - \left[\bbE\left(R_{i, n}( u)\right)\right]^{2}\right)
		\\
		& \le & 
		\sum_{i=1}^{n}\left(\bbE\left[  \left|R_{i, n}( u)\right|^{2}\right]\right)
		.
	\end{eqnarray*}
	
	Because $L(Z_{i}, t)$ is differentiable at $t = \theta$ for almost every $Z_{i}$, we have that:
	\begin{eqnarray*}
		\left|R_{i,n}( u)\right|^{2} 
		& = & 
		\left|L\left(Z_{i}, \theta+\frac{ u}{\sqrt{n}}\right)
		-
		L(\bfZ_{i}, \theta)
		-\nabla_{b}L(Z_{i}, \theta)\cdot\frac{ u}{\sqrt{n}}\right|^{2}
		= 
		o\left(\frac{1}{n}\right),
		\mbox{ for almost all } Z_{i}.
	\end{eqnarray*}
	We conclude that $\bbE\left[\sum_{i=1}^{n} \left|R_{i, n}( u)\right|^{2}\right] = o(1)$.
	
\textbf{Proof of b.1) uniform convergence over compact sets:}

Uniform convergence of $B_{n}( u)$ over compact sets follows from the pointwise convergence just proven and the Convexity Lemma due to \citet{pollard:1991:asymptotics-for-least-absolute-deviation-regression-estimators}.

\paragraph{Proof of b.2) boundedness of $\sqrt{n}\cdot \hat{\theta}(0)$:}
	
	Our proof of $\sqrt{n}$-boundedness of the un-penalized estimate is an adaptation of an argument due to \citet{pollard:1991:asymptotics-for-least-absolute-deviation-regression-estimators}.
	As a first step, we ``complete the squares'' in the quadratic approximation by letting $C$ be a decomposition of the (non-singular) Hessian matrix, i.e. $C^{T}C = H(\theta)$. 
	We then write $B_{n}( u)$ as:
	\begin{eqnarray*}
		B_{n}( u) 
		& = & 
		\left\|C u + \frac{1}{2}(C^{-1})^{T}\bfW_{n}\right\|^{2}
		-
		\left\|\frac{1}{2}(C^{-1})^{T}\bfW_{n}\right\|^{2}
		+
		r_{n}( u),
	\end{eqnarray*}
	Let $A_{n}$ denote the ball with center $-\frac{1}{2}(C^{-1})^{T}\bfW_{n}$ and radius $\delta>0$. Since $\bfW_{n}$ converges in distribution, it is stochastically bounded and, hence, a compact set $K^{*}$ with probability arbitrarily close to one can be chosen to contain $A_{n}$. Thus:
	\begin{eqnarray*}
		\Delta_{n} & := & \sup_{ u \in A_{n}}|r_{n}( u)| \convinprob 0.
	\end{eqnarray*}
	We now study the behavior of $B_{n}$ outside of $A_{n}$ to conclude that $\hat{\theta}_{n}(0)$ is consistent.
	To do that, let $z$ be a point outside the ball and define:
	\begin{eqnarray*}
		m 
		& = & 
		\|z - \frac{1}{2}(C^{-1})^{T}\bfW_{n}\|_{2}
		\\
		v 
		& = &  
		\frac
		{z - \frac{1}{2}(C^{-1})^{T}\bfW_{n}}
		{m}
	\end{eqnarray*}
	Because of convexity, we have that for $ u^{*}=-\frac{1}{2}C^{-1}(C^{-1})^{T}\bfW_{n}+\delta \cdot v$ on the boundary of the $A_{n}$ ball:
	\begin{eqnarray*}
		\left(\frac{\delta}{m}\right)B_{n}(z)
		+
		\left(1-\frac{\delta}{m}\right)
		B_{n}\left(\frac{1}{2}(C^{-1})^{T}\bfW_{n}\right)
		& \ge &
		B_{n}( u^{*})
		\\
		& \ge & \inf\limits_{|v|\le1}\left(v^{T}H(\theta)v\right)
		-
		\frac{1}{4}\bfW_{n}^{T}\left[H(\theta)\right]^{-1}\bfW_{n}
		-
		\Delta_{n}
		\\
		& \ge & \delta^{2}\inf\limits_{|v|\le1}\left(v^{T}H(\theta)v\right)
		-
		\frac{1}{4}\bfW_{n}^{T}\left[H^{-1}(\theta)\right]\bfW_{n}
		-
		\Delta_{n}		
		\\
		& \ge & 
		\delta^{2}\Lambda_{p}(H(\theta))
		-
		B_{n}\left(-\frac{1}{2}C^{-1}(C^{-1})^{T}\bfW_{n}\right)
		-
		2\Delta_{n},		
	\end{eqnarray*}
	where $\Lambda_{p}(H(\theta))$ is the smallest eigenvalue of $H(\theta)$.
	We then conclude that:
	\begin{eqnarray*}
		\inf\limits_{\left| u+\frac{1}{2}(C^{-1})^{T}\bfW_{n}\right|>\delta}
		B_{n}( u)
		\ge
		B_{n}(-\frac{1}{2}(C^{-1})^{T}\bfW_{n})
		+
		\frac{m}{\delta}\left[\delta^{2}\Lambda_{p}(H(\theta))-2 \Delta_{n}\right].
	\end{eqnarray*}
	Since $\Delta_{n}\convinprob 0$, we have that with probability approaching one that $\left|\hat{ u}+\frac{1}{2}(C^{-1})^{T}\bfW_{n}\right|<\delta$ and the result follows from recalling that $\hat{ u} = \sqrt{n}(\hat{\theta}_{n}(0)-\theta)$.
\end{proof}

\begin{proof}[Proof of Lemma \ref{result:uniform_convergence_for_penalty}]
We first brake the problem into two easier to handle pieces:
\begin{eqnarray*}
	\sup_{ u\in K\subset \bbR^{p}}
	\left\|
	\lambda_{n}\cdot\left[T(\theta + \frac{ u}{q_{n}}) - T(\theta)\right]
	-
	\lambda\cdot G_{\theta}( u)
	\right\|
	& \le & 
	\sup_{ u\in K\subset \bbR^{p}}
	\left\|
	\left(\frac{\lambda_{n}}{q_{n}}-\lambda\right)\cdot\left[
	\frac{T(\theta +  u\cdot q^{-1}_{n}) - T(\theta)}{q_{n}^{-1}}
	\right]
	\right\|
	\\
	& & + 
	\sup_{ u\in K\subset \bbR^{p}}
	\left\|
	\lambda\cdot
	\left(
	\frac{T(\theta +  u\cdot q^{-1}_{n}) - T(\theta)}{q_{n}^{-1}}-G_{\theta}( u)
	\right)
	\right\|
\end{eqnarray*}

Since $T$ is continuous, for the compact set $K\subset\bbR^{p}$ there exists $0<M_{K}<\infty$ such that:
\begin{eqnarray*}
	\sup_{ u\in K\subset \bbR^{p}}
	\left\|
	\left(\frac{\lambda_{n}}{q_{n}}-\lambda\right)\cdot\left[
	\frac{T(\theta +  u\cdot q^{-1}_{n}) - T(\theta)}{q_{n}^{-1}}
	\right]
	\right\|
	& \le & 
	\left\|
	\left(\frac{\lambda_{n}}{q_{n}}-\lambda\right)
	\right\|
	\cdot M_{K}
	\convinprob 0, \mbox{ as } n \rightarrow \infty.
\end{eqnarray*}

For the second term, we know from condition P4 in Assumption Set 2:
\begin{eqnarray*}
	\lim_{n\rightarrow \infty}
	\frac{T(\theta +  u\cdot q_{n}^{-1}) - T(\theta)}{q_{n}^{-1}}
	& = & 
	\lim_{h \downarrow 0}
	\frac{T(\theta + h\cdot  u) - T(\theta)}{h} = G_{\theta}( u).
\end{eqnarray*}
Because $G_{\theta}$ is assumed continuous, the pointwise convergence can be strengthened to uniform convergence over compact sets.
\end{proof}

\begin{lemma}
	\label{result:bounded_unregularized_estimate_implies_bounded_penalized_estimate}
	Let $\hat{\theta}_{n}(\lambda_{n})$ be as defined in \eqref{equation:definition_penalized_estimate}, $q_{n}$ be a sequence such that $q_{n}\rightarrow \infty$ as $n\rightarrow\infty$ and $\lambda_{n}$ be a sequence of (potentially random) non-negative real numbers.
	Assume $T$ is a penalty function satisfying condition P4 in PA.
	If $q_{n}\cdot\hat\theta_{n}(0)=O_{p}(1)$, then  $q_{n}\cdot\hat\theta_{n}(\lambda_{n})=O_{p}(1)$.
\end{lemma}
\begin{proof}
First, we use a contradiction to prove that $T(\hat\theta_{n}(\lambda_{n}))\le T(\hat\theta_{n}(0))$.
From the definition of $\hat\theta_{n}(0)$, we have:
\begin{eqnarray*}
	\frac{1}{n}\cdot \sum_{i=1}^{n}L\left(Z_{i}, \hat{\theta}_{n}(0)\right)
	& \le & 
	\frac{1}{n}\cdot \sum_{i=1}^{n}L\left(Z_{i}, \hat{\theta}_{n}(\lambda_{n})\right). 
\end{eqnarray*}
Supposing that $T(\hat\theta_{n}(\lambda_{n})) > T(\hat\theta_{n}(0))$, we get
\begin{eqnarray*}
	\frac{1}{n}\cdot \sum_{i=1}^{n}L\left(Z_{i}, \hat{\theta}_{n}(0)\right)
	+ 
	\lambda_{n}\cdot T(\hat\theta_{n}(0))
	& < & 
	\frac{1}{n}\cdot \sum_{i=1}^{n}L\left(Z_{i}, \hat{\theta}_{n}(\lambda_{n})\right) 
	+ 
	\lambda_{n}\cdot T(\hat\theta_{n}(\lambda_{n})),
\end{eqnarray*}
a contradiction with the definition of $\hat{\theta}_{n}(\lambda_{n})$ as the minimizer of $f(t) = \left[\frac{1}{n}\sum_{i=1}^{n}L\left(Z_{i}, t\right)\right] + \lambda_{n}\cdot T(t)$.

Now, from $q_{n}\cdot\hat{\theta}_{n}(0)=O_{p}(1)$, we have that, for any $\delta>0$, there exists compact $K_{0}\subset \Theta$ such that $\bbP\left(q_{n}\cdot\hat{\theta}_{n}(0)\in K_{0}\right)>1-\delta$.
Let $U = \max_{t\in K_{0}}q_{n}\cdot T(t)$ and define $\tilde{K}_{0} = \{t\in \Theta: q_{n}\cdot T(t)\le U\}$.
Since $T(\hat{\theta}_{n}(\lambda_{n}))<T(\hat{\theta}_{n}(0))$, it follows that $\bbP\left(\hat{\theta}(\lambda_{n})\in \tilde{K}_{0}\right) \ge \bbP\left(\hat{\theta}(0)\in \tilde{K}_{0}\right)>1-\delta$.
\end{proof}

\subsection{Proof of results in Section \ref{section:application}}

\begin{proof}[Proof of Theorem \ref{result:generalized_irrepresentability_condition}]
For the $\ell_{1}$-penalty, the difference between 
$	\lambda_{n}\left(\|\beta+\frac{u}{\sqrt{n}}\|_{1}-\|\beta\|_{1}\right)
    -
	\frac{\lambda_{n}}{\sqrt{n}}\left(u_{\mcalA}-\|u_{\mcalA^{c}}\|_{1}\right)
	\rightarrow
	0,$
uniformly as $n\rightarrow \infty$.
Using Theorem \ref{result:assembling_theorem} and Lemma \ref{result:uniform_convergence_for_loss},
\begin{eqnarray*}
	\sqrt{n}\left(\hat{\theta}_{n}(\lambda_{n})-\theta\right)
	\convinlaw 
	\hat{ u}
	:=
	\argmin_{ u} 
	\left[
	 u^{T} H(\theta) u + \bfW^{T} u + 
	\frac{\lambda_{n}}{\sqrt{n}}\left(u_{\mcalA} + \left\| u_{\mcalA^{c}}\right\|_{1}\right)
	\right],
\end{eqnarray*}
with $\bfW\sim N\left(0, J(\theta)\right)$.
%
We assume, without loss of generality that $\beta_{\mcalA}>0$ with the inequality holding element-wise.
In that case:
\begin{eqnarray*}
	\begin{array}{lllll}
		\sign(\hat{\beta}_{j}(\lambda_{n})) = \sign(\beta_{j}) 
		& \Leftrightarrow & 
		\frac{\hat{ u}_{j}}{\sqrt{n}} \ge -\beta_{j}, 
		&
		\mbox{ for } j\in \mcalA,
		\\
		\sign(\hat{\beta}_{j}(\lambda_{n})) = \sign(\beta_{j}) 
		& \Leftrightarrow & 
		\hat{ u}_{j} = 0, 
		&
		\mbox{ for } j\in \mcalA^{c}.
	\end{array}
\end{eqnarray*}
For the remainder of this proof, we drop denote $H(\theta)$ by $H$.
In terms of the $\alpha$, $\mcalA$, $\mcalA^{c}$ partition, the Karush-Kuhn-Tucker (KKT) conditions for optimization defining $\hat{u}$ above are
\begin{eqnarray*}
	\begin{array}{rllll}
		H_{\alpha    , \mcalA    }\cdot\hat{ u}_{\mcalA}
		+
		H_{\alpha    , \mcalA^{c}}\cdot\hat{ u}_{\mcalA^{c}}
		+
		H_{\alpha    , \alpha    }\cdot\hat{ u}_{\alpha}
		+
		\bfW_{\alpha   }
		& = & 
		0,
		\\
		H_{\mcalA    , \mcalA    }\cdot\hat{ u}_{\mcalA}
		+
		H_{\mcalA    , \mcalA^{c}}\cdot\hat{ u}_{\mcalA^{c}}
		+
		H_{\mcalA    , \alpha    }\cdot\hat{ u}_{\alpha}
		+
		\bfW_{\mcalA   }
		-
		\frac{\lambda_{n}}{\sqrt{n}}
		& = & 
		0,
		\\
		H_{j         , \mcalA    }\cdot\hat{ u}_{\mcalA}
		+
		H_{j         , \mcalA^{c}}\cdot\hat{ u}_{\mcalA^{c}}
		+
		H_{j         , \alpha    }\cdot\hat{ u}_{\alpha}
		+
		\bfW_{j}
		-
		\frac{\lambda_{n}}{\sqrt{n}}\cdot
		\sign(\hat{ u}_{j})
		& = & 
		0,
		&
		\mbox{ for } j \in \mcalA^{c} \mbox{ s.t. } \hat{u}_{j}\neq 0  
		\\
		\left|
		H_{j         , \mcalA    }\cdot\hat{ u}_{\mcalA    }
		+
		H_{j         , \mcalA^{c}}\cdot\hat{ u}_{\mcalA^{c}}
		+
		H_{j         , \alpha    }\cdot\hat{ u}_{\alpha}
		+
		\bfW_{j}
		\right|
		& \le & 
		\frac{\lambda_{n}}{\sqrt{n}},
		&
		\mbox{ for } j \in \mcalA^{c} \mbox{ s.t. } \hat{u}_{j}= 0.
	\end{array}
\end{eqnarray*}
To select the zero terms in $\beta$ correctly, we must have $\hat{ u}_{\mcalA^{c}=0}$.
In that case,
\begin{eqnarray*}
	\left[
	\begin{array}{c}
	\hat{ u}_{\alpha}
	\\
	\hat{ u}_{\mcalA}		
	\end{array}
	\right]
	& = & 
	\left[
	\begin{array}{cc}
	H_{\alpha, \alpha} & H_{\mcalA, \alpha}
	\\
	H_{\mcalA, \alpha} & H_{\mcalA, \mcalA}
	\end{array}
	\right]^{-1}
	\cdot
	\left[
	\begin{array}{c}
	-\bfW_{\alpha}
	\\
	\frac{\lambda_{n}}{\sqrt{n}}\cdot\bfone_{q} - \bfW_{\mcalA}
	\end{array}
	\right].
\end{eqnarray*}
Using Schur's inversion formula for partitioned matrices, we get:
\begin{eqnarray*}
	\hat{ u}_{\mcalA}
	& = & 
	\left[H_{\mcalA, \mcalA}-H_{\mcalA, \alpha}H_{\alpha, \alpha}^{-1}H_{\alpha, \mcalA}\right]^{-1}
	\cdot 
	\left[ 
	\frac{\lambda_{n}}{\sqrt{n}}\cdot\bfone_{q} 
	- 
	\bfW_{\mcalA}
	-
	\left[H_{\mcalA, \alpha}H_{\alpha, \alpha}^{-1}\right]\cdot \bfW_{\alpha}
	\right].
\end{eqnarray*}

Define a zero mean Gaussian random vector $\tilde{\bfW} = \left[\begin{array}{cc}\tilde{\bfW}_{\mcalA}, & \tilde{\bfW}_{\mcalA^{c}}\end{array}\right]$ :
\begin{small}
\begin{eqnarray*}
	\tilde{\bfW}_{\mcalA}     & := & 
	-
	\mcalH
	\cdot
	\left[\bfW_{\mcalA} + H_{\mcalA, \alpha}\cdot H_{\alpha, \alpha}^{-1} \cdot \bfW_{\alpha}\right], \mbox{ and }
	\\
	\tilde{\bfW}_{\mcalA^{c}} & := & 
	\bfW_{\mcalA^{c}} 
	-
	\left[
	\begin{array}{c}
	H_{\mcalA^{c}, \alpha}
	\\
	H_{\mcalA^{c}, \mcalA}
	\end{array}
	\right]
	^{T}
	\left[
	\begin{array}{cc}
	H_{\alpha, \alpha}	&	H_{\alpha, \mcalA}\\
	H_{\mcalA, \alpha}	&	H_{\mcalA, \mcalA}
	\end{array}
	\right]
	^{-1}
	\left[
	\begin{array}{c}
	\bfW_{\alpha}
	\\
	\bfW_{\mcalA}
	\end{array}
	\right], \mbox{ with }\\
	\mcalH & := & \left[H_{\mcalA, \mcalA}-H_{\mcalA, \alpha}H_{\alpha, \alpha}^{-1}H_{\alpha, \mcalA}\right]^{-1}
\end{eqnarray*}
\end{small}
The M-estimated parameter fails to have the correct signs if:
\begin{small}
\begin{eqnarray*}
	\begin{array}{lllll}
	\tilde{\bfW}_{j}	
	& \ge &
	\sqrt{n}\cdot \beta_{j}
	-
	\frac{\lambda_{n}}{\sqrt{n}}
	\cdot
	\sum_{k\in \mcalA}\mcalH_{jk},
	& 
	\mbox{ for some $j\in \mcalA$, } & \mbox{ OR, }
	\\
	\tilde{\bfW}_{j}
	& > & 
	\frac{\lambda_{n}}{\sqrt{n}}
	\left(
	1
	-
	H_{\mcalA^{c}, \mcalA}
	\left[H_{\mcalA, \mcalA}-H_{\mcalA, \alpha}H_{\alpha, \alpha}^{-1}H_{\alpha, \mcalA}\right]^{-1}
	\bfone_{q}
	\right),
	&
	\mbox{ for some $j\in \mcalA^{c}$,}
	&
	\mbox{ OR, }
	\\
	\tilde{\bfW}_{j}
	& < & 
	\frac{\lambda_{n}}{\sqrt{n}}
	\left(
	-1
	-
	H_{\mcalA^{c}, \mcalA}
	\left[H_{\mcalA, \mcalA}-H_{\mcalA, \alpha}H_{\alpha, \alpha}^{-1}H_{\alpha, \mcalA}\right]^{-1}
	\bfone_{q}
	\right),
	&
	\mbox{ for some $j\in \mcalA^{c}$.}
	\end{array}
\end{eqnarray*}
\end{small}
For the remainder of the proof, let $\Phi$ denote the standard normal cumulative distribution function, and define $\varsigma_{j}^{2}:=\var(\tilde{\bfW}_{j})$.

\textbf{Proof of a):}\\
We prove that if $\|H_{\mcalA^{c}, \mcalA}\left[H_{\mcalA, \mcalA}-H_{\mcalA, \alpha}H_{\alpha, \alpha}^{-1}H_{\alpha, \mcalA}\right]^{-1}\bfone_{q}\|_{\infty}<1$ and there exists $c>0$ and $\lambda\in \bbR$ sucg that $\frac{\lambda_{n}}{n^{\frac{1+c}{2}}}\convinprob \lambda$ and $\frac{\lambda_{n}}{n}\convinprob 0$, then the probability of each of these three events decreases to zero exponentially fast.

For the first event, use the union bound and the inequality $1-\Phi(r)\le\exp\left[-\frac{r^{2}}{2}\right]$ for large enough $r$, to get
\begin{small}
\begin{eqnarray*}
	\bbP\left(
	\exists j\in \mcalA:
	\tilde{\bfW}_{j}
	> 
	\sqrt{n}\cdot\left(\beta_{j} - \frac{\lambda_{n}}{n}\right)
	\right)
	& \le &
	\sum_{j\in \mcalA} 
	\bbP\left(
	\tilde{\bfW}_{j}
	> 
	\sqrt{n}\left(\beta_{j} - \frac{\lambda_{n}}{n}\right)
	\right)	 
	\\
	& \le &
	\sum_{j\in \mcalA} 
	\bbP\left(
	\frac{\tilde{\bfW}_{j}}{\varsigma_{j}}
	> 
	\sqrt{n}\left(\frac{\beta_{j}}{\varsigma_{j}} - \frac{\lambda_{n}}{n\cdot \varsigma_{j}}\right)
	\right)	 
	\\
	& \le & 
	q
	\cdot
	\left[1
	-
	\Phi\left(
	\sqrt{n}\cdot \left(\min\limits_{j\in\mcalA}\left(\frac{\beta_{j}}{\varsigma_{j}}\right) - \frac{\lambda_{n}}{n\cdot\max\limits_{j\in\mcalA}{\varsigma_{j}}}\right)
	\right)
	\right]
	\\
	& \le &
	q\cdot
	\frac{
	\exp\left[
	-\frac{n}{2}
	\cdot 
	\left(
	\min\limits_{j\in\mcalA}\left(\frac{\beta_{j}}{\varsigma_{j}}\right)^{2} - \frac{\lambda_{n}^{2}}{n^{2}\cdot\max\limits_{j\in\mcalA}{\varsigma_{j}^{2}}}
	\right)
	\right]
	}
	{
	\sqrt{n}\cdot\left(\min\limits_{j\in\mcalA}\left(\frac{\beta_{j}}{\varsigma_{j}}\right) - \frac{\lambda_{n}}{n\cdot\max\limits_{j\in\mcalA}{\varsigma_{j}}}\right)
	}
	\\
	& \sim &
	q\cdot
	\frac
	{\exp\left[-\frac{n}{2}\min\limits_{j\in\mcalA}\left(\frac{\beta_{j}}{\varsigma_{j}}\right)^{2}\right]}
	{\sqrt{n}\cdot \min\limits_{j\in\mcalA}\left(\frac{\beta_{j}}{\varsigma_{j}}\right)}.
\end{eqnarray*}	
\end{small}
To tackle the second and third events, define $\eta = 1-\|H_{\mcalA^{c}, \mcalA}
	\left[H_{\mcalA, \mcalA}-H_{\mcalA, \alpha}H_{\alpha, \alpha}^{-1}H_{\alpha, \mcalA}\right]^{-1}
	\bfone_{q}
	\|_{\infty}$ and notice that
\begin{small}
\begin{eqnarray*}
	\bbP
	\left(
	\tilde{W}_{j}
	>
	\frac{\lambda_{n}}{\sqrt{n}}
	\left(
	1
	-
	H_{\mcalA^{c}, \mcalA}
	\left[H_{\mcalA, \mcalA}-H_{\mcalA, \alpha}H_{\alpha, \alpha}^{-1}H_{\alpha, \mcalA}\right]^{-1}
	\bfone_{q}
	\right)
	\right)
	& \le & 
	\bbP
	\left(
	\tilde{W}_{j}
	>
	\eta
	\cdot
	\frac{\lambda_{n}}{\sqrt{n}}
	\right), 
	\mbox{ for all } j \in \mcalA^{c}, \mbox{ AND }
	\\
	\bbP
	\left(
	\tilde{W}_{j}
	<
	-
	\frac{\lambda_{n}}{\sqrt{n}}
	\left(
	1
	-
	H_{\mcalA^{c}, \mcalA}
	\left[H_{\mcalA, \mcalA}-H_{\mcalA, \alpha}H_{\alpha, \alpha}^{-1}H_{\alpha, \mcalA}\right]^{-1}
	\bfone_{q}
	\right)
	\right)
	& \le & 
	\bbP
	\left(
	-\tilde{W}_{j}
	>
	-
	\eta
	\cdot
	\frac{\lambda_{n}}{\sqrt{n}}
	\right)
	.
\end{eqnarray*}
\end{small}
As a result, using the union bound gives that the probability of the second or third event happening is bounded above by $\sum_{j\in \mcalA^{c}} \bbP\left(\left|\tilde{W}_{j}\right|>\eta\cdot \frac{\lambda_{n}}{\sqrt{n}}\right)$. To prove this probability vanishes exponentially fast, we use the same inequality as above:
\begin{small}	
\begin{eqnarray*}
\sum_{j\in \mcalA^{c}} \bbP\left(\left|\tilde{W}_{j}\right|>\eta\cdot \frac{\lambda_{n}}{\sqrt{n}}\right)
& \le & 
\sum_{j\in \mcalA^{c}} \bbP\left(\left|\frac{\tilde{W}_{j}}{\varsigma_{j}}\right|>\frac{\eta}{\varsigma_{j}}\cdot \frac{\lambda_{n}}{\sqrt{n}}\right)
\\
& \le & 
2\cdot
\sum_{j\in \mcalA^{c}} 
\left[1-\Phi\left(\frac{\eta}{\varsigma_{j}}\cdot \frac{\lambda_{n}}{\sqrt{n}}\right)\right]
\\
& \le & 
2(p-q)
\cdot 
\left[1-\Phi
\left(\frac{\eta}{\max\limits_{j\in\mcalA^{c}}\varsigma_{j}}\cdot \frac{\lambda_{n}}{\sqrt{n}}\right)
\right]
\\
& \le & 
2(p-q)
\cdot 
\exp
\left[
-
\frac{1}{2}
\left(\frac{\eta}{\max\limits_{j\in\mcalA^{c}}\varsigma_{j}}\cdot \frac{\lambda_{n}}{\sqrt{n}}\right)^{2}
\right]
\\
& \sim &
2(p-q)
\cdot 
\exp
\left[
-
\frac{1}{2}
\cdot
n^{c}
\cdot
\left(\min\limits_{j\in \mcalA^{c}}\frac{\eta}{\varsigma_{j}}\right)^{2}
\right].
\end{eqnarray*}
\end{small}

\noindent\textbf{Proof of b):}\\
To prove the converse in part (b), first notice that $\mcalH_{\mcalA, \mcalA}$ is a positive definite matrix.
It follows that
	$\sum_{j\in \mcalA}\sum_{k\in \mcalA}\mcalH_{jk}>0$, 
and thus that there must exists $j\in \mcalA$ with $\sum_{k\in \mcalA}\mcalH_{jk}>0$.
Thus, if $\frac{\lambda_{n}}{n}\rightarrow \infty$, the first event takes place with probability approaching one (exponentially fast) as long as $\mcalA$ is non-empty.
On the other hand, if $\frac{\lambda_{n}}{\sqrt{n}}\rightarrow 0$, then the union of the second and third event occurs with probability approaching one (exponentially fast).
Thus, we only need to consider the case $\frac{\lambda_{n}}{n^{\frac{1+c}{2}}}\rightarrow \lambda$, for some finite $\lambda\in \bbR$ and $c\in[0,1)$.

As before, let $\eta=1-\|H_{\mcalA^{c}, \mcalA}\left[H_{\mcalA, \mcalA}-H_{\mcalA, \alpha}H_{\alpha, \alpha}^{-1}H_{\alpha, \mcalA}\right]^{-1}\bfone_{q}\|_{\infty}$.
If $c>0$ and $\eta<0$, the probability of the second or third events converges to one (exponentially fast).
If $\eta=0$, the second or third events have a positive probability of taking place regardless of $\lambda_{n}$.
Likewise, if $c=0$, $\eta<0$, the second or third events happen with strictly positive probability.
\end{proof}

\begin{proof}[Proof of Theorem \ref{result:simplified_gi_condition}]
Throughout this proof we denote $\nu:=\frac{\beta}{\|\beta\|}\in \bbR^{p}$, the unit vector in the direction of $\beta$.
Using the properties of Gaussian distributions and the condition $\nu^{T}\mu=0$, we get
\begin{small}
\begin{eqnarray*}
	\begin{array}{rcl}
	\bbE\left[\bfX\bfX^{T}\left|\alpha+\bfX^{T}\beta\right.\right]
	& = &
	\mu\mu^{T}
	+
	\Sigma
	+
	\frac{\Sigma\nu\nu^{T}\Sigma}{(\nu^{T}\Sigma\nu)^{2}}
	\cdot
	\left[
	\left(
	\frac
	{\left(\alpha+\bfX^{T}\beta\right)-\left(\alpha+\beta^{T}\mu\right)}
	{\|\beta\|}
	\right)^{2}-1
	\right]
	.
	\end{array}
\end{eqnarray*}
\end{small}
For details, we refer the reader to Appendix \ref{section:conditional_moments_of_gaussian_distributions}.
Letting $f_{\bfM}$ denote the density of the random variable $\bfM=\alpha+\bfX^{T}\beta$ and defining
\begin{eqnarray*}
	\kappa 
	& := &
	\int
	\left[
	\left(
	\frac
	{m-\left(\alpha+\beta^{T}\mu\right)}
	{\|\beta\|}
	\right)^{2}-1
	\right]
	\cdot 
	w(m)
	\cdot 
	\tilde{f}_{\bfM}(m)
	\cdot 
	\mbox{d}m,
\end{eqnarray*}
the Hessian of the risk function becomes:
\begin{eqnarray*}
	H(\theta)
	& = &
	\mu\mu^{T}
	+
	\Sigma 
	+
	\kappa
	\cdot
	\frac{\Sigma \nu\nu^{T} \Sigma}{\left(\nu^{T}\Sigma\nu\right)^{2}}.
\end{eqnarray*}

Partition the vectors $\nu$, $\mu$, and the matrix $\Sigma$ according to the sparsity pattern in $\nu$:
\begin{eqnarray*}
	\nu & = & \left[\begin{array}{cc}\nu_{\mcalA}^{T} & \bfzero^{T}\end{array}\right]^{T},
	\\
	\mu & = & \left[\begin{array}{cc}\mu_{\mcalA}^{T} & \mu_{\mcalA^{c}}^{T}\end{array}\right]^{T}, \mbox{ and }
	\\
	\Sigma 
	& = & 
	\left[
	\begin{array}{cc}
		\Sigma_{\mcalA    ,\mcalA    } & \Sigma_{\mcalA    ,\mcalA^{c}} \\
		\Sigma_{\mcalA^{c},\mcalA    } & \Sigma_{\mcalA^{c},\mcalA^{c}} \\
	\end{array}
	\right].
\end{eqnarray*}

The partitioned Hessian becomes
\begin{small}
\begin{eqnarray*}
	H(\theta)
	& = & 
	\left[
	\begin{array}{cc}
	\left(\mu_{\mcalA}\mu_{\mcalA}^{T} + \Sigma_{\mcalA    , \mcalA    }\right)
	\cdot
	\left(\bfI_{q}+\kappa\cdot\nu_{\mcalA}\nu_{\mcalA}^{T}\cdot\Sigma_{\mcalA, \mcalA}\right)
	&
	\left(\bfI_{q}+\kappa\cdot\Sigma_{\mcalA, \mcalA}\cdot\nu_{\mcalA}\nu_{\mcalA}^{T}\right)
	\cdot
	\left(\mu_{\mcalA}\mu_{\mcalA^{c}}^{T} + \Sigma_{\mcalA    , \mcalA^{c}}\right)
	\\
	\left(\mu_{\mcalA^{c}}\mu_{\mcalA}^{T} + \Sigma_{\mcalA^{c}, \mcalA    }\right)
	\cdot
	\left(\bfI_{q}+\kappa\cdot\nu_{\mcalA}\nu_{\mcalA}^{T}\cdot\Sigma_{\mcalA, \mcalA}\right)
	&
	\mu_{\mcalA^{c}}\mu_{\mcalA^{c}}^{T}
	+
	\Sigma_{\mcalA^{c}, \mcalA^{c}} 
	+
	\kappa
	\cdot
	\Sigma_{\mcalA^{c}, \mcalA}\cdot\nu_{\mcalA}\nu_{\mcalA}^{T}\cdot\Sigma_{\mcalA    , \mcalA^{c}}
	\end{array}
	\right].
\end{eqnarray*}
\end{small}
Defining $\bfA := \left(\bfI_{q}+\kappa\cdot\nu_{\mcalA}\nu_{\mcalA}^{T}\cdot\Sigma_{\mcalA, \mcalA}\right)$, we get
\begin{small}
\begin{eqnarray*}
	H_{\mcalA^{c}, \mcalA}(\theta)\left[H_{\mcalA, \mcalA}(\theta)\right]^{-1}
	& = &
	\left(\mu_{\mcalA^{c}}\mu_{\mcalA}^{T} + \Sigma_{\mcalA^{c}, \mcalA    }\right)
	\cdot
	\bfA
	\times
	\left[
	\left(\mu_{\mcalA}\mu_{\mcalA}^{T} + \Sigma_{\mcalA    , \mcalA    }\right)
	\cdot
	\bfA
	\right]^{-1}
	\\
	& = &
	\left(\mu_{\mcalA^{c}}\mu_{\mcalA}^{T} + \Sigma_{\mcalA^{c}, \mcalA    }\right)
	\bfA\bfA^{-1}
	\left(\mu_{\mcalA}\mu_{\mcalA}^{T} + \Sigma_{\mcalA    , \mcalA    }\right)^{-1}
	\\
	& = &
	\left[\bbE\left(\bfX_{\mcalA^{c}, \mcalA    }\right)\right]
	\left[\bbE\left(\bfX_{\mcalA    , \mcalA    }\right)\right]^{-1}.
\end{eqnarray*}
\end{small}
The result follows from post-multiplying both sides of this last equation by $\sign(\beta_{\mcalA})$.
\end{proof}

\begin{proof}[Proof of Lemma \ref{result:variance_and_hessian_for_logistic_loss}]
Throughout the proof of Lemma \ref{result:variance_and_hessian_for_logistic_loss}, we define $\tilde{\bfX} = \left[\begin{array}{cc}1 & \bfX^{T}\end{array}\right]^{T}$.
\begin{enumerate}
	\item [L1)] 
	We first prove the existence of a minimizer.
	Given that the risk function is continuous, it is enough to prove that the closed set $\mcalS(\bfM) = \left\{t\in \bbR^{p}: \bbE\left[-\bbI(\bfY=1)\cdot \tilde{\bfX}^{T}t + \log\left(1+\exp\left(\tilde{\bfX}^{T}t\right)\right)\right]\le \bfM\right\}$ for large enough $\bfM$ is bounded.
	We establish boundedness of $\mcalS(\bfM)$ by proving that $\mcalS(\bfM)$ is contained on a finite sphere around the origin which can be established by proving  that, for any $\bfM$ there exists $\gamma$ such that:
	\begin{eqnarray}
		\label{equation:logistic_out_of_box_inclusion}
		\left|\langle t, t\rangle\right| > \gamma 
		& \Rightarrow & 
		\bbE\left[-\bbI(\bfY=1)\cdot \tilde{\bfX}^{T}t + \log\left(1+\exp\left(\tilde{\bfX}^{T}t\right)\right)\right]>\bfM.
	\end{eqnarray}
	To prove the assertion in \eqref{equation:logistic_out_of_box_inclusion}, let $t$ be a non-zero vector with $u = \sqrt{\langle t, t\rangle} \neq 0$, so $t$ can be written as $t = u v$, for some $v\neq 0$. For any $u_{0}\in \bbR$, we can write:
	\begin{small}
	\begin{eqnarray*}
		\begin{array}{ccc}
		\bbE\left[-u\cdot\bbI(\bfY=1)\cdot \tilde{\bfX}^{T}v + \log\left(1+\exp\left(u\cdot \tilde{\bfX}^{T}v\right)\right)\right]
		& = & 
		\\
		-u\cdot\bbE\left[p(\bfX)\cdot \tilde{\bfX}^{T}v\right]
		+ 
		\bbE\left[\log\left(1+\exp\left(u \cdot\tilde{\bfX}^{T}v\right)\right)\right]
		& \ge & 
		\\
		-u\cdot\bbE\left[p(\bfX)\cdot \tilde{\bfX}^{T}v\right]
		+ 
		\bbE\left[\log\left(1+\exp\left(u_{0} \cdot\tilde{\bfX}^{T}v\right)\right)\right]
		+ 
		\bbE
		\left[
		\frac
		{\exp\left(u_{0}\cdot \tilde{\bfX}^{T}v\right)}
		{1+\exp\left(u_{0}\cdot \tilde{\bfX}^{T}v\right)}
		\cdot 
		\tilde{\bfX}^{T}v
		\right]
		\cdot
		(u-u_{0})
		& = & 
		\\
		c(u_{0}, v)
		+
		\bbE\left[
		\left(
		\frac
		{\exp\left(u_{0}\cdot \tilde{\bfX}^{T}v\right)}
		{1+\exp\left(u_{0}\cdot \tilde{\bfX}^{T}v\right)}
		-
		p(\bfX)
		\right)
		\cdot 
		\tilde{\bfX}^{T}v\right]
		\cdot
		u,
		\end{array}
	\end{eqnarray*}
	\end{small}
	where the inequality follows from convexity of the mapping $s\mapsto \log\left(1+\exp\left(s \cdot\tilde{\bfX}^{T}v\right)\right)$ and $c(u_{0}, v) = 
		\bbE\left[\log\left(1+\exp\left(u_{0} \cdot\tilde{\bfX}^{T}v\right)\right)\right]
		-	
		\bbE
		\left[
		\frac
		{\exp\left(u_{0}\cdot \tilde{\bfX}^{T}v\right)}
		{1+\exp\left(u_{0}\cdot \tilde{\bfX}^{T}v\right)}
		\cdot 
		\tilde{\bfX}^{T}v
		\right]
		\cdot
		u_{0}$, which does not involve $u$.
	
	Under the assumptions made, we can use the dominated convergence theorem to get:
	\begin{eqnarray*}
		\lim_{u_{0}\rightarrow\infty}
		\bbE
		\left[
		\left(
		\frac
		{\exp\left(u_{0}\cdot \tilde{\bfX}^{T}v\right)}
		{1+\exp\left(u_{0}\cdot \tilde{\bfX}^{T}v\right)}
		-
		p(\bfX)
		\right)
		\cdot 
		\tilde{\bfX}^{T}v
		\right]
		& = & 
		\bbE
		\left[
		\left(
		1		
		-
		p(\bfX)
		\right)
		\cdot 
		\tilde{\bfX}^{T}v
		\right], \mbox{ and }
		\\
		\lim_{u_{0}\rightarrow-\infty}
		\bbE
		\left[
		\left(
		\frac
		{\exp\left(u_{0}\cdot \tilde{\bfX}^{T}v\right)}
		{1+\exp\left(u_{0}\cdot \tilde{\bfX}^{T}v\right)}
		-
		p(\bfX)
		\right)
		\cdot 
		\tilde{\bfX}^{T}v
		\right]
		& = & 
		-
		\bbE
		\left[
		p(\bfX)
		\cdot 
		\tilde{\bfX}^{T}v
		\right].
	\end{eqnarray*}
	Since the density is everywhere positive, the hyperplane $\{\bfs\in \bbR^{p+1}:\bfs^{T}v=0\}$ has probability zero for any $v\neq0$ and thus we have either $\bbE\left(\tilde{\bfX}v\right)>0$ or $\bbE\left(\tilde{\bfX}v\right)<0$.

	If $\bbE\left(\tilde{\bfX}v\right)>0$, pick $u_{0}$ large enough so $\bbE\left[\left(1-p(\bfX)\right)\cdot \tilde{\bfX}^{T}v \right]>0$ to conclude that 			
	\begin{eqnarray*}
		\lim_{u\rightarrow\infty}\bbE\left[-\bbI(\bfY=1)\cdot \tilde{\bfX}^{T}t + \log\left(1+\exp\left(\tilde{\bfX}^{T}t\right)\right)\right]=\infty.
	\end{eqnarray*}

	If, on the other hand, $\bbE\left(\tilde{\bfX}v\right)<0$, pick $u_{0}$ small enough so $\bbE\left[p(\bfX)\cdot \tilde{\bfX}^{T}v \right]<0$ to conclude:
	\begin{eqnarray*}
		\lim_{u\rightarrow\infty}\bbE\left[-\bbI(\bfY=1)\cdot \tilde{\bfX}^{T}t + \log\left(1+\exp\left(\tilde{\bfX}^{T}t\right)\right)\right]=\infty.
	\end{eqnarray*}

	This establishes that for any $\bfM$ there exists $\gamma$ such that the risk function exceeds $\bfM$ and completes the proof of existence.
	
	The proof of uniqueness follows from strict convexity of the risk function.
	We prove that below by showing that the Hessian matrix of the risk function is everywhere strictly positive definite under the assumptions made.
	
	\item [L2)]	For the canonical logistic regression loss function, we have:
	\begin{eqnarray*}
		\bbE\left[\left|L(\bfY, \bfX, t)\right|\right]
		& = & 
		\bbE\left[\left|\bbI(\bfY=1)\cdot \tilde{\bfX}^{T}t - \log\left(1+\exp\left(\tilde{\bfX}^{T}t\right)\right)\right|\right]
		\\
		& \le & 
		\bbE\left[\left|\tilde{\bfX}\right|\right]^{T}t
		+
		\bbE\left[\left|\log\left(1+\exp\left(\tilde{\bfX}^{T}t\right)\right)\right|\right]
		\\
		& = & 
		\bbE\left[\left|\tilde{\bfX}\right|\right]^{T}t
		+
		\bbE\left[\log\left(1+\exp\left(\tilde{\bfX}^{T}t\right)\right)\right],
	\end{eqnarray*}
	where the equality follows from $\exp(\tilde{\bfX}^{T}t)\ge0$ for all $t$.
	Because $\bbE\left[\bfX\bfX^{T}\right]<\infty$, there exists $C$ such that $\bbE\left[\left|\bfX_{j}\right|\right]< C$ for all $j=1, \ldots, p$ and the first term of the sum is bounded above.
	To bound the second term, write:
	\begin{eqnarray*}
		\bbE\left[\log\left(1+\exp\left(\tilde{\bfX}^{T}t\right)\right)\right]
		& \le & 
		\bbE\left[\log\left(1+\exp\left(\left|\tilde{\bfX}^{T}t\right|\right)\right)\right]
		\\
		& \le & 
		\log\left(2\right)
		+
		2\cdot\bbE\left[|\tilde{\bfX}^{T}t|\right],
	\end{eqnarray*}
	where the first inequality follows from $h(u) := \log\left(1+\exp\left(u\right)\right)$ being non-decreasing and the second stems from $h$ having derivatives bounded above by $1$.
	The result now follows from $\bbE\left|\bfX\right|$ being bounded.

	\item [L3)]
	The canonical logistic regression loss function is twice differentiable everywhere, with:
	\begin{eqnarray*}
		\nabla_{t}L(\bfY, \bfX, t) 
		& = &
		\left[\bbI\left(\bfY=1\right) - \frac{\exp\left(\tilde{\bfX}^{T}t\right)}{1+\exp\left(\tilde{\bfX}^{T}t\right)}\right]\cdot\tilde{\bfX}, \mbox{ and }
		\\
		\nabla_{t}^{2}L(\bfY, \bfX, t) 
		& = &
		\frac{\exp\left(\tilde{\bfX}^{T}t\right)}{\left(1+\exp\left(\tilde{\bfX}^{T}t\right)\right)^{2}}\cdot\tilde{\bfX}\tilde{\bfX}^{T}.
	\end{eqnarray*}
	
	For all $\bfY$, $\bfX$ and $t\in \bbR^{p+1}$, we have:
	\begin{eqnarray*}
		\left[\left(2\bfY-1\right) - \frac{\exp\left(\tilde{\bfX}^{T}t\right)}{1+\exp\left(\tilde{\bfX}^{T}t\right)}\right]
		& \le & 
		2, \mbox{ and }
		\\
		\frac{\exp\left(\tilde{\bfX}^{T}t\right)}{\left(1+\exp\left(\tilde{\bfX}^{T}t\right)\right)^{2}}
		& \le & 1,
	\end{eqnarray*}
	so, using the assumptions on the moments of $\bfX$, we know that:
	\begin{eqnarray*}
		\bbE\left[\left|\nabla_{t}L(\bfY, \bfX, t)\right|\right]
		& \le &
		2\max\left[1, \max_{1\le j\le p}\bbE \left|\bfX_{j}\right|\right] < \infty, \mbox{ and}
		\\
		\bbE\left[\left|\nabla_{t}^{2}L(\bfY, \bfX, t)\right|\right]
		& \le &
		\bbE \left[\bfQ(\bfX)\right] < \infty.
	\end{eqnarray*}
	Using the Dominated Convergence Theorem we get that:
	\begin{eqnarray*}
		\nabla_{t}R(t) 
		& = & 
		\bbE\left[\nabla_{t}L(\bfY, \bfX, t)\right] 
		= 
		\bbE\left[\left(\bbE\left(\bbI(\bfY=1)|\bfX\right)-\frac{\exp(\tilde{\bfX}^{T}t)}{1+\exp(\tilde{\bfX}^{T}t)}\right)\cdot\tilde{\bfX}\right]
		\\
		& = & 
		\bbE\left[\left(p(\bfX)-\frac{\exp(\tilde{\bfX}^{T}t)}{1+\exp(\tilde{\bfX}^{T}t)}\right)\cdot\tilde{\bfX}\right], \mbox{ and }
		\\
		\nabla^{2}_{t}R(t) 
		& = & 
		\bbE\left[\nabla^{2}_{t}L(\bfY, \bfX, t)\right] 
		=
		\bbE
		\left[
		\frac
		{\exp\left(\tilde{\bfX}^{T}t\right)}
		{\left(1+\exp\left(\tilde{\bfX}^{T}t\right)\right)^{2}}
		\cdot
		\tilde{\bfX}\tilde{\bfX}^{T}\right].
	\end{eqnarray*}
	
	We now prove that the population risk minimizer $\theta$ for the logistic regression is unique under the conditions of Assumption Set 3, by proving that the Hessian $\nabla^{2}_{t}R(\theta)$ is a strictly positive definite matrix.
	
	From the assumption that $\bbE\left[Q(\bfX)\right]$ is strictly positive definite and bounded, we get that:
	\begin{eqnarray*}
		\bbE\left[Q(\bfX)\right] 
		& = & 
		\lim_{s\rightarrow\infty}\bbE\left[Q(\bfX)\cdot\bbI(\|\bfX\|\le s)\right],
	\end{eqnarray*}
	and thus, there must exist large enough $S$ such that $\bbE\left[Q(\bfX)\cdot\bbI(\|\bfX\|\le S)\right]$ is strictly positive definite.
	Let $\xi = \inf_{\|\bfX\|\le S} \frac{\exp(\alpha+\bfX^{T}\beta)}{\left[1+\exp(\alpha+\bfX^{T}\beta)\vphantom{\int_{0}^{1}}\right]^{2}}$.
	Because $\|\bfX\|\le S$ is a compact set and $\frac{\exp(\alpha+\bfX^{T}\beta)}{\left[1+\exp(\alpha+\bfX^{T}\beta)\vphantom{\int_{0}^{1}}\right]^{2}}>0$ for all $\bfX$, we get that $\xi>0$.
	
	In what follows, the binary relationship between matrices $A$ and $B$ indicated by $A\succeq B$ means $A-B$ is positive semi-definite and its strict version $A\succ B$ means $A-B$ is strictly positive definite.
	Now,
	\begin{eqnarray*}
	\begin{array}{rcl}
	\nabla^{2}_{t}R(\theta)
	& = &
	\bbE
	\left[
	\frac
	{\exp\left(\tilde{\bfX}^{T}t\right)}
	{\left(1+\exp\left(\tilde{\bfX}^{T}t\right)\right)^{2}}
	\cdot
	Q(\bfX)
	\right]
	\\
	& = & 
	\bbE
	\left[
	\frac
	{\exp\left(\tilde{\bfX}^{T}t\right)}
	{\left(1+\exp\left(\tilde{\bfX}^{T}t\right)\right)^{2}}
	\cdot
	Q(\bfX)
	\cdot
	\bbI\left(\|\bfX\| \le S\right)
	\right]
	+
	\bbE
	\left[
	\frac
	{\exp\left(\tilde{\bfX}^{T}t\right)}
	{\left(1+\exp\left(\tilde{\bfX}^{T}t\right)\right)^{2}}
	\cdot
	Q(\bfX)
	\cdot
	\bbI\left(\|\bfX\| > S\right)
	\right]
	\\
	& \succeq &
	\xi
	\cdot
	\bbE
	\left[
	Q(\bfX)
	\cdot
	\bbI\left(\|\bfX\| \le S\right)
	\right]
	+
	\bbE
	\left[
	\frac
	{\exp\left(\tilde{\bfX}^{T}t\right)}
	{\left(1+\exp\left(\tilde{\bfX}^{T}t\right)\right)^{2}}
	\cdot
	Q(\bfX)
	\cdot
	\bbI\left(\|\bfX\| > S\right)
	\right]
	\succ 0, 
	\end{array}
	\end{eqnarray*}
	where the last generalized inequality follows from $\bbE\left[Q(\bfX)\cdot	\frac{\exp\left(\tilde{\bfX}^{T}t\right)}{\left(1+\exp\left(\tilde{\bfX}^{T}t\right)\right)^{2}}\cdot\bbI\left(\|\bfX\| > S\right)\right]\succeq 0$, $\xi>0$, and $\bbE\left[Q(\bfX)\cdot\bbI\left(\|\bfX\| \le S\right)\right]\succ 0$.
	\item [L4)]	
	The loss function corresponds to the neg-loglikelihood function of a canonical exponential family and is thus convex.
	As the risk is an expected value of convex functions, it is also convex.
\end{enumerate}

\end{proof}	

\begin{proof}[Proof of Lemma \ref{result:variance_and_hessian_for_hinge_loss}]
Throughout the proof of Lemma \ref{result:variance_and_hessian_for_hinge_loss}, we define $\tilde{\bfX} = \left[\begin{array}{cc}1 & \bfX^{T}\end{array}\right]^{T}$.
\begin{enumerate}
	\item [L1)]
	We first prove that a minimizer exist. 
	Given that the risk function is continuous, it is enough to prove that for large enough $M$ the closed set $\mcalS(M) = \left\{t\in \bbR^{p}: \bbE\left[\left|1-\bfY\tilde{\bfX}^{T}t\right|_{+}\right]\le M\right\}$ is bounded.
	We establish boundedness of $\mcalS(M)$ by proving that $\mcalS(M)$ is contained on a finite box around the origin.
	Letting $e_{j}$ be a unit vector with a $1$ in its $j$-th entry and zeroes in all other components, it is sufficient to prove that, for any $M$ and each $j=1, \ldots, p+1$, there exist $\gamma_{j, M}$ such that:
	\begin{eqnarray}
		\label{equation:out_of_box_inclusion}
		\left|\langle t, e_{j}\rangle\right| > \gamma_{j, M} & \Rightarrow & \bbE\left[\left|1-\bfY\tilde{\bfX}^{T}t\right|_{+}\right]>M.
	\end{eqnarray}
	To prove the assertion in $\eqref{equation:out_of_box_inclusion}$, let $t$ be a non-zero vector with $u = \langle t, e_{j}\rangle \neq 0$, so $t$ can be written as $t = u e_{j} + v$, for some $v$ with $\langle v, e_{j}\rangle = 0$.
	The risk function at $t$ becomes:
	\begin{eqnarray*}
		\bbE\left[\left|1-\bfY\tilde{\bfX}^{T}t\right|_{+}\right]
		& = &		\bbE\left[\left|1-u\cdot\tilde{\bfX}^{T}e_{j}-\tilde{\bfX}^{T}v\right|\bbI\left(\bfY=1\right)\right]
		+		\bbE\left[\left|1+u\cdot\tilde{\bfX}^{T}e_{j}-\tilde{\bfX}^{T}v\right|\bbI\left(\bfY=-1\right)\right]
		\\
		& \ge &
		\bbE\left[\left|u\cdot\tilde{\bfX}^{T}e_{j}+\tilde{\bfX}^{T}v\right|\right]-1
		\\
		& \ge &
		\inf_{v\in \mcalO(e_{j})}
		\bbE\left[\left|u\cdot\tilde{\bfX}^{T}e_{j}+\tilde{\bfX}^{T}v\right|\right]-1
		\\
		& = &
		|u|\cdot
		\inf_{v\in \mcalO(e_{j})}
		\bbE\left[\left|\tilde{\bfX}^{T}e_{j}+\tilde{\bfX}^{T}v\right|\right]-1,
	\end{eqnarray*}
	where $\mcalO(e_{j})$ is the set of all vectors orthogonal to $e_{j}$. Because $e_{j}$ has unit norm, $\|e_{j}+v\|\ge 1$ for all $v\in \mcalO(e_{j})$ and it follows that $\{e_{j}+v:v\in\mcalO(e_{j}) \}\subset \{v:\|v\|\ge 1\}$, yielding:
	\begin{eqnarray*}
		\bbE\left[\left|1-\bfY\tilde{\bfX}^{T}t\right|_{+}\right]
		& \ge &
		|u|\cdot
		\inf_{v:\|v\|\ge 1}
		\bbE\left[\left|\tilde{\bfX}^{T}v\right|\right]-1
		=
		|u|\cdot
		\inf_{v:\|v\| = 1}
		\bbE\left[\left|\tilde{\bfX}^{T}v\right|\right]-1,
	\end{eqnarray*}
	where the equality follows from noticing that $\left|\tilde{\bfX}^{T}v\right|$ is increasing in $\left|v\right|$.
	If we can find $c>0$, such that  $\inf_{v:\|v\|= 1}\bbE\left[\left|\tilde{\bfX}^{T}v\right|\right]>c$, it is possible to find the $\gamma_{j,M}$ we want.
	To find such a positive lower bound,  define the compact set $\bfK = \{\tilde{x}\in \bbR^{p+1}:\|\tilde{x}\|_{2}\le\bfC\}$ for some constant $\bfC$.
	Since it is assumed that $f_{\bfX}(\bfx)>0$ is continuous for all $\bfx\in\bbR^{p}$, we get that $f^{*}=\min_{\bfx\in \bfK}>0$.
	Now, letting $\mu(\bfA)$ denote the Lebesgue measure of a set $\bfA\subset\bbR^{p+1}$, we have:
	\begin{eqnarray*}
		\inf_{v:\|v\|= 1}
		\bbE\left[\left|\tilde{\bfX}^{T}v\right|\right]
		& \ge &
		\eta
		\cdot
		\inf_{v:\|v\|\ge 1}
		\bbP\left[\left|\tilde{\bfX}^{T}v\right|>\eta\right]
		\\
		& \ge &
		\inf_{v:\|v\|\ge 1}
		\bbP\left[\tilde{\bfX}^{T}v>\eta, \mbox{ and } \tilde{\bfX}\in \bfK \right]
		\\
		& \ge &
		f^{*}
		\cdot
		\inf_{v:\|v\|\ge 1}
		\mu\left(\{\bfx:\tilde{\bfx}^{T}v>\eta\}\right).
		\\
		& = &
		f^{*}
		\cdot
		\mu\left(\{\bfx:\tilde{\bfx}^{T}e_{1}>\eta\}\right)=:c_{\eta}>0,
	\end{eqnarray*}
	where the last equality follows from noticing that, because of symmetry:
	\begin{eqnarray*}
		\mu\left(\{\bfx:\tilde{\bfx}^{T}v>\eta\}\right) 
		= 
		\mu\left(\{\bfx:\tilde{\bfx}^{T}e_{1}>\eta\}\right), \mbox{ for all }v\in \bbR^{p+1}: \|v\|_{2}=1.
	\end{eqnarray*}
	
	Using the strictly positive lower bound afforded by $c_{\eta}$, we get:
	\begin{eqnarray*}
		|u| = \left|\langle t, e_{j}\rangle\right| > \gamma_{j, M}:= \frac{M+1}{f^{*}c_{\eta}} , \mbox{ for some } j=1, \ldots, p 
		&\Rightarrow& 
		\bbE\left[\left|1-\bfY\tilde{\bfX}^{T}t\right|_{+}\right]>M.
	\end{eqnarray*}
	
	Uniqueness of the minimizer follows from strict convexity of the risk function under the assumptions made.
	Strict convexity of the risk function in its turn is proved below, by showing the Hessian matrix for the risk is everywhere strictly positive definite.
	
	\item [L2)] For all $t\in \bbR^{p+1}$:
	\begin{eqnarray*}
		\bbE\left[\left|L(\bfZ, t)\right|\right]
		\le
		\bbE\left[\max\left\{\left|1-\tilde{\bfX}t\right|, \left|1+\tilde{\bfX}t\right|\right\}\right]
		\le
		1+\bbE\left[\left|\tilde{\bfX}t\right|\right]
		\le
		1+\sqrt{t^{t}\cdot\bbE\left[\tilde{\bfX}^{T}\tilde{\bfX}\right]\cdot t},
	\end{eqnarray*}
	which is bounded given the assumptions on the distribution of $\bfX$.
	
	\item [L3)]
	
		The hinge loss is not differentiable on the set $\{\bfX: \tilde{\bfX}t=1 \mbox{ or } \tilde{\bfX}t=-1\}$, which under the assumed conditions has zero probability.
		At all other points, hinge loss function has derivative with respect to $t$
		\begin{eqnarray*}
			\nabla_{t}L(\bfZ, t)
			& = & 
			\tilde{\bfX}
			\cdot
			\left[
			\bbI(\bfY=1)\cdot\bbI(\tilde{\bfX}t-1<0)
			-
			\bbI(\bfY=-1)\cdot\bbI(\tilde{\bfX}t+1>0)
			\right]
		\end{eqnarray*}
		
		To obtain the Hessian, write the SVM risk as $R(t) = R_{1}(t) + R_{2}(t)$, with
		\begin{eqnarray*}
		R_{1}\left(t\right)	& := & \bbE
		                               \left[
		                               p\left(\bfX\right)
		                               \cdot
			                           \left(1-\tilde{\bfX}t\right)
		                               \cdot
		                               \bbI\left(1-\tilde{\bfX}t>0\right)
		                               \right],\mbox{ and }
		\\
		R_{2}\left(t\right)	& := & \bbE
		                               \left[
								      \left(1-p\left(\bfX\right)\right)
								      \cdot
								      \left(\tilde{\bfX}t-1\right)
								      \cdot
								      \bbI\left(\tilde{\bfX}t-1>0\right)
		                               \right].
		\end{eqnarray*}
		Let
			$
			\nabla_{t}R_{1}(t) 
			:= 
			\bbE\left[
			p(\tilde{\bfX})
			\cdot 
			\bbI\left(1-\tilde{\bfX}t>0\right) 
			\cdot 
			\tilde{\bfX}
			\right]
			$ 
			and 
			$
			\nabla^{2}_{t}R_{1}(t) 
			:=
			\bbE\left[
			p(\tilde{\bfX})
			\cdot
			\delta(1-\tilde{\bfX}t) 
			\cdot 
			\tilde{\bfX}^{T}
			\tilde{\bfX}
			\right].
			$
		We first show that $r_{1}(t) := R_{1}(t+\Delta t) - R_{1}(t) = \nabla_{t}R_{1}(t)\cdot\Delta t = o(\|\Delta t\|)$. 
		To do that, let $\bfd_{k}=\frac{\Delta t_{k}}{|\Delta t_{k}|}$ and write
		\begin{small}
		\begin{eqnarray*}
		\frac{R_{1}(t+\Delta t_{k}) - R_{1}(t)}{\Delta t_{k}}
		& = &
		\bbE
		\left[
		\frac
		{
		p(\tilde{\bfX})
		\left(1-\tilde{\bfX}t\right)
		\left[
		  \bbI\left(1-\tilde{\bfX}t>\tilde{\bfX}\Delta t\right) 
		  - 
		  \bbI\left(1-\tilde{\bfX}t>0\right) 
		\right]
		}{\tilde{\bfX}\Delta t}
		\cdot
		\tilde{\bfX}
		\right]
		\bfd_{k}
		\\
		& &
		-
		\bbE
		\left[
		p(\tilde{\bfX})
		\cdot 
		\bbI\left(1-\tilde{\bfX}t>\tilde{\bfX}\Delta t\right) 
		\cdot 
		\tilde{\bfX}
		\right]
		\bfd_{k}.
		\end{eqnarray*}
		\end{small}
		Using the Dominated Convergence Theorem to take the limit as $\Delta t\downarrow 0$ and collecting the limit of the multiplier of $\bfd_{k}$ yields
		\begin{small}
		\begin{eqnarray*}
		\nabla_{t}R_{1}(t)
		& = &
		\lim_{|\Delta t|\rightarrow 0}
		\bbE
		\left[
		\frac
		{
		p(\tilde{\bfX})
		\left(1-\tilde{\bfX}t\right)
		\left[
		  \bbI\left(1-\tilde{\bfX}t>\tilde{\bfX}\Delta t\right) 
		  - 
		  \bbI\left(1-\tilde{\bfX}t>0\right) 
		\right]
		}{\tilde{\bfX}\Delta t}
		\cdot
		\tilde{\bfX}
		\right]
		\\
		& &
		-
		\lim_{|\Delta t|\rightarrow 0}
		\bbE
		\left[
		p(\tilde{\bfX})
		\cdot 
		\bbI\left(1-\tilde{\bfX}t>\tilde{\bfX}\Delta t\right) 
		\cdot 
		\tilde{\bfX}
		\right]
		\\
		& = &
		\bbE
		\left[
		p(\tilde{\bfX})
		\left(1-\tilde{\bfX}t\right)
		\delta(1-\tilde{\bfX}^{T}t)
		\cdot
		\tilde{\bfX}
		\right]
		-
		\bbE
		\left[
		p(\tilde{\bfX})
		\cdot 
		\bbI\left(1-\tilde{\bfX}t>0\right) 
		\cdot 
		\tilde{\bfX}
		\right]
		\\
		& = & 
		-
		\bbE
		\left[
		p(\tilde{\bfX})
		\cdot 
		\bbI\left(1-\tilde{\bfX}t>0\right) 
		\cdot 
		\tilde{\bfX}
		\right].
		\end{eqnarray*}
		\end{small}
		
		To obtain the second differential for $R_{1}(t)$, write the residuals from the approximation from the first differential:
		\begin{small}
		\begin{eqnarray*}
		\frac
		{
		r_{1}(\Delta t)
		}
		{|\Delta t_{k}|^{2}}
		& = &
		\bfd_{k}
		\bbE
		\left[
		\frac{
		p(\tilde{\bfX})
		\cdot 
		\tilde{\bfX}^{T}
		\cdot
		\left[
		  \bbI\left(1-\tilde{\bfX}t>\tilde{\bfX}\Delta t\right) 
		  - 
		  \bbI\left(1-\tilde{\bfX}t>0\right) 
		\right]
		\cdot 
		\left(1-\tilde{\bfX}t\right)
		\cdot 
		\tilde{\bfX}
		}
		{\Delta t^{T}\tilde{\bfX}\tilde{\bfX}^{T}\Delta t}
		\right]
		\bfd_{k}
		\\
		& &
		-
		\bfd_{k}
		\cdot
		\bbE
		\left[
		\frac{
		p(\tilde{\bfX})
		\cdot 
		\tilde{\bfX}^{T}
		\cdot
		\left[
		\bbI\left(1-\tilde{\bfX}t>\tilde{\bfX}\Delta t\right) 
		-
		\bbI\left(1-\tilde{\bfX}t>0\right) 
		\right]
		\cdot 
		\tilde{\bfX}
		}
		{\tilde{\bfX}\Delta t}
		\right]
		\cdot
		\bfd_{k}
		\end{eqnarray*}
		\end{small}
	The second derivative is obtained using the Dominated Convergence Theorem to compute the limits of the terms in the sum.
	For the second term, the limit follows directly from pointwise convergence to a Dirac delta function:
	\begin{small}
	\begin{eqnarray*}
	\begin{array}{ccl}
	\lim\limits_{|\Delta t|\rightarrow 0}
	\bbE
	\left[
	\frac{
	p(\tilde{\bfX})
	\cdot 
	\tilde{\bfX}
	\cdot
	\left[
	\bbI\left(1-\tilde{\bfX}^{T}t>\tilde{\bfX}^{T}\Delta t\right) 
	-
	\bbI\left(1-\tilde{\bfX}^{T}t>0\right) 
	\right]
	\cdot 
	\tilde{\bfX}^{T}
	}
	{\tilde{\bfX}^{T}\Delta t}
	\right]
	=
	-
	\bbE
	\left[
	p(\tilde{\bfX})
	\cdot
	\delta(1-\tilde{\bfX}^{T}t) 
	\cdot
	\tilde{\bfX}
	\tilde{\bfX}^{T}
	\right].
	&&
	\end{array}
	\end{eqnarray*}
	\end{small}
	
	To obtain the limit for the other term, let $W = R\tilde{\bfx} = \left[\begin{array}{cc}\bfw_{1} & \bfw_{2}\end{array}\right] \in \bbR\times \bbR^{p-1}$ be a linear rotation of $\tilde{x}$ such that $\bfw_{1} = \tilde{\bfx}^{T}t$, and let $F_{\tilde{\bfx}}$, $F_{\bfw_{2}}$, and  $f_{\bfw_{1}|\bfw_{2}}$ denote the distributions of $\tilde{X}$, $W_{2}$ and the conditional distribution of $W_{1}$ given $W_{2}$ respectively.
	Then write
	\begin{small}
	\begin{eqnarray*}
	\begin{array}{lc}
	\bbE
	\left[
	\frac{
	p(\tilde{\bfX})
	\cdot 
	\tilde{\bfX}^{T}
	\cdot
	\left[
	  \bbI\left(1-\tilde{\bfX}t>\tilde{\bfX}\Delta t_{k}\right) 
	  - 
	  \bbI\left(1-\tilde{\bfX}t>0\right) 
	\right]
	\cdot 
	\left(1-\tilde{\bfX}t\right)
	\cdot 
	\tilde{\bfX}
	}
	{\Delta t_{k}^{T}\tilde{\bfX}\tilde{\bfX}^{T}\Delta t_{k}}
	\right]
	& 
	= 
	\\
	\int
	\left[
	\frac{
	\left(1-\tilde{\bfx}t\right)
	\cdot
	\left[
	\bbI\left(1-\tilde{\bfx}t>\tilde{\bfx}\Delta t_{k}\right) 
	- 
	\bbI\left(1-\tilde{\bfx}t>0\right) 
	\right]
	}
	{\Delta t_{k}^{{T}}\tilde{\bfx}^{T}\tilde{\bfx}^{T}\Delta t_{k}}
	\cdot	
	p(\tilde{\bfx})
	\cdot	
	\tilde{\bfx}
	\tilde{\bfx}^{T}
	\right]
	\mbox{d}F_{\tilde{\bfx}}\left(\tilde{\bfx}\right)
	& 
	= 
	\\
	\int
	\int
	\frac{
	\left(1-\bfw_{1}\right)
	\cdot
	\left[
	\bbI\left(1-\bfw_{1}>\tilde{\bfx}(\bfw)\Delta t_{k}\right) 
	- 
	\bbI\left(1-\bfw_{1}>0\right) 
	\right]
	}
	{\Delta t_{k}^{T}\tilde{\bfx}(\bfw)\tilde{\bfx}^{T}(\bfw)\Delta t_{k}}
	s(\bfw_{1}, \bfw_{2})
	dF_{\bfw_{1}|\bfw_{2}}(\bfw_{1}\left|\bfw_{2}\right.)
	dF_{\bfw_{2}}(\bfw_{2}),
	&
	\end{array}
	\end{eqnarray*}
	\end{small}
	where we used the notation $s(\bfw_{1}, \bfw_{2}) := p(\tilde\bfx(\bfw_{1}, \bfw_{2}))\bfx(\bfw_{1}, \bfw_{2})\bfx(\bfw_{1}, \bfw_{2})^{T}$.

	To obtain the limit, write the inner integral as:
	\begin{small}
	\begin{eqnarray*}
		\begin{array}{ll}
		\int
		\left[
		\frac{
		\left(1-\bfw_{1}\right)
		\cdot
		\left[
		\bbI\left(1-\bfw_{1}>\tilde{\bfx}(\bfw_{1}, \bfw_{2})\Delta t_{k}\right) 
		- 
		\bbI\left(1-\bfw_{1}>0\right) 
		\right]
		}
		{\Delta t_{k}^{T}\tilde{\bfx}(\bfw_{1}, \bfw_{2})\tilde{\bfx}^{T}(\bfw)\Delta t_{k}}
		\cdot	
		s(\bfw_{1}, \bfw_{2})
		\right]
		dF_{\bfw_{1}|\bfw_{2}}(\bfw_{1}\left|\bfw_{2}\right)
		&
		=
		\\
		\int
		\left[
		\frac{
		v
		\cdot
		\left[
		\bbI\left(v>\tilde{\bfx}(\bfw_{1}, \bfw_{2})\Delta t_{k}\right) 
		- 
		\bbI\left(v>0\right) 
		\right]
		}
		{\Delta t_{k}^{T}\tilde{\bfx}(1-v, \bfw_{2})\tilde{\bfx}^{T}(\bfw)\Delta t_{k}}
		\cdot	
		s(1-v, \bfw_{2})
		\right]
		dF_{\bfw_{1}|\bfw_{2}}(1-v\left|\bfw_{2}\right)
		&
		=
		\\
		-
		\int
		\left[
		\frac{1}{2}
		\cdot	
		s(\bfw_{1}, \bfw_{2})
		\cdot
		\delta(1-\bfw_{1})
		\right]
		dF_{\bfw_{1}|\bfw_{2}}(\bfw_{1}\left|\bfw_{2}\right)
		+
		o(\|\Delta t\|).
		\end{array}
	\end{eqnarray*}
	\end{small}
	Plugging that back into the expression for the expected value, we get:
	\begin{small}
	\begin{eqnarray*}
		\begin{array}{ll}
		\lim_{|\Delta t_{k}|\rightarrow 0}
		\bbE
		\left[
		\frac{
		p(\tilde{\bfX})
		\cdot 
		\tilde{\bfX}^{T}
		\cdot
		\left[
		  \bbI\left(1-\tilde{\bfX}t>\tilde{\bfX}\Delta t_{k}\right) 
		  - 
		  \bbI\left(1-\tilde{\bfX}t>0\right) 
		\right]
		\cdot 
		\left(1-\tilde{\bfX}t\right)
		\cdot 
		\tilde{\bfX}
		}
		{\Delta t_{k}^{T}\tilde{\bfX}\tilde{\bfX}^{T}\Delta t_{k}}
		\right]
		& = 
		\\
		-
		\int
		\left[
		\int
		\left[
		\frac{1}{2}
		\cdot	
		s(\bfw_{1}, \bfw_{2})
		\cdot
		\delta(1-\bfw_{1})
		\right]
		dF_{\bfw_{1}|\bfw_{2}}(\bfw_{1}\left|\bfw_{2}\right)
		\right]
		dF{\bfw_{2}}(\bfw_{2})
		& = 
		\\
		-\frac{1}{2}
		\cdot
		\bbE
		\left[
		p(\tilde{\bfX})
		\cdot 
		\tilde{\bfX}
		\tilde{\bfX}^{T}
		\cdot
		\delta(1-\tilde{\bfX}^{T}t)
		\right].
		&
	\end{array}
	\end{eqnarray*}
	\end{small}
	
	Summing the two terms (and taking into account the factor $\frac{1}{2}$ in the Taylor expansion) yield 
	\begin{eqnarray*}
	\nabla_{t}^{2}R_{1}(t) = 
	\bbE
	\left[
	p(\tilde{\bfX})
	\cdot
	\delta(1-\tilde{\bfX}^{T}t) 
	\cdot
	\tilde{\bfX}
	\tilde{\bfX}^{T}
	\right].
	\end{eqnarray*}
	
	For $R_{2}$, analogous steps yield
	\begin{eqnarray*}
	\nabla_{t}R_{2}(t)
		& = &
		\bbE
		\left[
		\left(1-p(\tilde{\bfX})\right)
		\cdot 
		\bbI\left(1+\tilde{\bfX}t>0\right) 
		\cdot 
		\tilde{\bfX}
		\right], \mbox{ and }
		\\
		\nabla_{t}^{2}R_{2}(t)
		& = & 
		\bbE
		\left[
		\left(1-p(\tilde{\bfX})\right)
		\cdot
		\delta(1+\tilde{\bfX}^{T}t) 
		\cdot
		\tilde{\bfX}
		\tilde{\bfX}^{T}
		\right].
	\end{eqnarray*}
		The result follows from summing the differentials for $R_{1}(t)$ and $R_{2}(t)$.
	\begin{eqnarray*}
	\nabla_{t}R(t)
	& = & 
	\bbE
	\left[
		\left(
		\left(1-p(\tilde{\bfX})\right)
		\cdot 
		\bbI\left(1+\tilde{\bfX}t>0\right) 
		-
		p(\tilde{\bfX})
		\cdot 
		\bbI\left(1+\tilde{\bfX}t>0\right) 
		\right)
		\cdot 
		\tilde{\bfX}
	\right], \mbox{ and }
	\\
	\nabla_{t}^{2}R(t)
	& = & 
	\bbE
	\left[
	\left(
	\left(1-p(\tilde{\bfX})\right)
	\cdot
	\delta(1+\tilde{\bfX}^{T}t) 
	+
	p(\tilde{\bfX})
	\cdot
	\delta(1-\tilde{\bfX}^{T}t) 
	\right)
	\cdot
	\tilde{\bfX}
	\tilde{\bfX}^{T}
	\right].
	\end{eqnarray*}
	
	Finally, we prove that the minimizer of the SVM risk is unique by establishing that $\nabla_{t}^{2}R(\theta)$ is strictly positive definite.
	To do that, first write:
	\begin{eqnarray*}
	\nabla_{t}^{2}R(t)
	& = & 
	\bbE
	\left[
	Q(\bfX)
	\cdot 
	\bbI\left(\bfY=1\right)
	\left|
	\alpha+\bfX^{T}\beta = 1
	\right.
	\right]
	\cdot 
	\tilde{f}(1)
	\\
	&&
	+
	\bbE
	\left[
	Q(\bfX)
	\cdot 
	\bbI\left(\bfY=-1\right)
	\left|
	\alpha+\bfX^{T}\beta = -1
	\right.
	\right]
	\cdot 
	\tilde{f}(-1)
	\\
	& = & 
	\bbE
	\left[
	Q(\bfX)
	\left|
	\bfY=1, 
	\alpha+\bfX^{T}\beta = 1
	\right.
	\right]
	\cdot 
	\bbP\left(\bfY=1\left|\vphantom{\int_{0}^{1}}\alpha+\bfX^{T}\beta = 1\right.\right)
	\cdot 
	\tilde{f}(1)
	\\
	&&
	+
	\bbE
	\left[
	Q(\bfX)
	\left|
	\bfY=-1,
	\alpha+\bfX^{T}\beta = -1
	\right.
	\right]
	\cdot 
	\bbP\left(\bfY=-1\left|\vphantom{\int_{0}^{1}}\alpha+\bfX^{T}\beta = -1\right.\right)
	\cdot 
	\tilde{f}(-1)
	,
	\end{eqnarray*}
	where $\tilde{f}$ denote the density of the random variable $\alpha+\bfX^{T}\beta$.
	Given assumption C2, $\tilde{f}(-1)>0$ and $\tilde{f}(1)>0$.
	In addition, assumption C3 gives that $\bbP\left(\bfY=-1\left|\alpha+\bfX^{T}\beta=1\right.\right)>0$ and $\bbP\left(\bfY=1\left|\alpha+\bfX^{T}\beta=1\right.\right)>0$.
	It is thus, enough to prove that either $\bbE\left[Q(\bfX)\left|\bfY=1, \alpha+\bfX^{T}\beta = -1\right.\right]$ or $\bbE\left[Q(\bfX)\left|\bfY=-1, \alpha+\bfX^{T}\beta = -1\right.\right]$ is strictly positive definite (or both).
	
	Define $\bfv_{\beta} = \frac{\beta}{\|\beta\|}$, the unit vector in the direction of $\beta$.
	The condition $\alpha+\bfX^{T}\beta = \kappa$ is equivalent to $\bfX^{T}\bfv_{\beta} = \frac{\kappa-\alpha}{\|\beta\|}$, so for any scalar $\kappa$:
	\begin{eqnarray*}
		\bbE\left[Q(\bfX)\left|\bfY, \alpha+\bfX^{T}\beta = \kappa\right.\right]
		& = &
		\bbE\left[Q(\bfX)\left|\bfY, \bfX^{T}\bfv_{\beta} = \frac{\kappa-\alpha}{\|\beta\|}\right.\right].
	\end{eqnarray*}
	Then notice that:
	\begin{eqnarray*}
		\bbE\left[Q(\bfX)\left|\bfY, \bfX^{T}\bfv_{\beta}=\frac{\kappa-\alpha}{\|\beta\|}\right.\right]
		\succeq
		\left(\frac{\kappa-\alpha}{\|\beta\|}\right)^{2}\left(\bfv_{\beta}\bfv_{\beta}^{T}\right)
		+
		\var\left(\bfX\left|\bfY, \bfX^{T}\bfv_{\beta}=\frac{\kappa-\alpha}{\|\beta\|}\right.\right)			
		,
	\end{eqnarray*}
	where $A\succeq B$ denotes that $A-B$ is positive semi definite.
	Because $\var\left[\bfX\left|\bfY\right.\right]$ is assumed to be non-singular, $\var\left(\bfX\left|\bfY, \bfX^{T}\bfv_{\beta}=\frac{\kappa-\alpha}{\|\beta\|}\right.\right)$ has rank $p-1$ and $\var\left(\bfX\left|\bfY, \bfX^{T}\bfv_{\beta}=\frac{\kappa-\alpha}{\|\beta\|}\right.\right)\bfv_{\beta}=0$.
	Thus, as long as $\kappa\neq \alpha$, $\bbE\left[Q(\bfX)\left|\bfY, \bfX^{T}\bfv_{\beta} = \frac{\kappa-\alpha}{\|\beta\|}\right.\right]$ is strictly positive definite.
	
	If $\alpha\not\in\{-1,1\}$, both terms in the sum defining $\nabla_{t}^{2}R(\theta)$ are strictly positive definite.
	If $\alpha\in\{-1,1\}$, one of the terms in the sum defining $\nabla_{t}^{2}R(\theta)$ is singular, but the other is necessarily strictly positive definite.
	Thus, it follows that $\nabla_{t}^{2}R(\theta)$ is strictly positive definite as stated.
	
	\item [L4)] 
	We can write $\|1-\tilde{\bfX}t\|_{+}$ as the maximum between the constant function $0$ and the function $\bfY-\tilde{\bfX}t$ which is linear -- thus, convex -- on $t$. Since it is the maximum between two convex functions on $t$, $\|\bfY-\tilde{\bfX}t\|_{+}$ is convex on $t$. A similar argument yields that $\|\tilde{\bfX}t-1\|_{-}$ is convex on $t$.
	
	The loss function $L(\bfZ, t)$ is written as the sum (with positive weights) of convex functions, which proves that AL.IV holds for the SVM loss function.

\end{enumerate}

\end{proof}

\section{Calculations for SVM and logistic risk Hessians in selected cases}
\label{section:analytical_hessians}

In this section, we first obtain expressions of the second moment of the predictors given the value of the margin variable $\alpha+\bfX^{T}\beta$ in the case of predictors $\bfX$ having a Gaussian and a mixture of Gaussian distributions.
Given the characterization of the SVM and logistic Hessians as a ``weighted average'' of such conditional second moments in Equations \eqref{equation:logistic_hessian_in_conditional_form} and \eqref{equation:svm_hessian_in_conditional_form}, the expressions for such conditional moments are useful in analytically comparing $\ell_{1}$-penalized SVM and logistic classifiers with respect to their model selection properties.
We then give explicit analytical expressions for the Hessian and logistic regression risk functions in the case of Gaussian and mixed Gaussian predictors.

For the duration of this section, $\bfZ$ denotes a rotated version of $\bfX$ whose first component is the projection of $\bfX$ along the direction normal to the optimal separating hyperplane $\mcalH(\theta)$.

\subsection{Conditional moments of Gaussian predictors given the value of one of its projections}
\label{section:conditional_moments_of_gaussian_distributions}

To obtain the conditional second moments used in the expressions for Hessians of the SVM and logistic regression risk functions, we first construct an orthogonal matrix $S$ according to
\begin{eqnarray*}
	S
	& := & 
	\left[
	\begin{array}{cc}
		\nu & U
	\end{array}
	\right],
\end{eqnarray*}
with $U$ a $p\times(p-1)$ matrix constructed using a Gram-Schmidt orthogonalization (as long as $\beta\neq \bfzero$).
By construction, $U^{T}\nu=\bfzero$ and $U^{T}U=\bfI_{p-1}$. 
The random vector $\bfZ = S^{T}\bfX\in\bbR^{p}$ is partitioned into a random scalar $\bfZ_{1} = \nu^{T}\bfX$ in the direction of $\nu$ and a $p-1$ dimensional random vector $\bfZ_{2} = U^{T}\bfX$ orthogonal to $\nu$,
\begin{eqnarray*}
	\bfZ^{T}
	& := & 
	\left[
	\begin{array}{cc}
	\bfZ_{1} &
	\bfZ_{2}^{T}
	\end{array}
	\right]
	=
	\left[
	\begin{array}{cc}
	\bfX^{T}\nu
	&
	\bfX^{T}U
	\end{array}
	\right]^{T}
	.
\end{eqnarray*}
Conditioning on the \textit{margin variable} $\bfM=\alpha+\bfX^{T}\beta$ -- defined in \eqref{equation:definition_margin_variable} -- is equivalent to conditioning on the  $\bfZ_{1}=\nu^{T}\bfX$ since:
\begin{eqnarray*}
	\alpha + \bfX^{T}\beta = \bfM
	\Leftrightarrow 
	\nu^{T}\bfX = \frac{\bfM-\alpha}{{\|\beta\|}}
	\Leftrightarrow 
	\bfZ_{1} = \frac{\bfM-\alpha}{{\|\beta\|}}.
\end{eqnarray*}
Since $S$ is orthogonal, $\bfX = S\bfZ$ and
\begin{eqnarray*}
	\begin{array}{ccccc}
	\bbE\left[\bfX\bfX^{T}|\nu^{T}\bfX\right]
	& = & 
	S\bbE\left[\bfZ\bfZ^{T}|\bfZ_{1}\right]S^{T}
	& = &
	\left[S\;\bbE\left[\bfZ|\bfZ_{1}\right]\vphantom{\int}\right]\cdot\left[S\;\bbE\left[\bfZ|\bfZ_{1}\right]\vphantom{\int}\right]^{T}
	+
	S\cdot\var\left[\bfZ\left|\bfZ_{1}\right.\right]\cdot S^{T}.
	\end{array}
\end{eqnarray*}
For $\bfX\sim\normal(\mu, \Sigma)$, $\bfZ$ is also Gaussian with expected value $S^{T}\mu$ and variance $S^{T} \Sigma S$. Partitioning the expressions for the expected value and variance of $\bfZ$ we get
\begin{eqnarray*}
	\begin{array}{ccccccc}
	\bbE\bfZ
	& = &
	\left[
	\begin{array}{c}
	\nu^{T}\mu 
	\\
	U^{T}\mu 
	\end{array}
	\right], & 
	\mbox{ and }
	& 
	\var\left[\bfZ\right]
	& = &
	\left[
	\begin{array}{cc}
	\nu^{T}\Sigma \nu & \nu^{T}\Sigma U
	\\
	U^{T}\Sigma\nu & U^{T}\Sigma U
	\end{array}
	\right].	
	\end{array}
\end{eqnarray*}
Based on these expressions and standard results on multivariate Gaussian distributions, we get:
\begin{eqnarray*}
	\begin{array}{cclcclcccc}
	\bbE\left[\bfZ_{1}\left|\bfZ_{1}\right.\right]
	&
	=  
	&
	\bfZ_{1},
	&
	\bbE\left[\bfZ_{2}\left|\bfZ_{1}\right.\right]
	&
	= 
	&
	U^{T}
	\left[
	\mu 
	+
	\frac{\Sigma\nu}
	{\nu^{T}\Sigma\nu}
	\left(\bfZ_{1}-\nu^{T}\mu\right)\right],
	\\
	\var\left[\bfZ_{1}\left|\bfZ_{1}\right.\right]
	&
	=
	&  
	0,
	&
	\var\left[\bfZ_{2}\left|\bfZ_{1}\right.\right]
	&
	=
	& 
	U^{T}
	\left[ 
	\Sigma 
	- 
	\frac
	{\Sigma\nu\nu^{T}\Sigma}
	{\nu^{T}\Sigma\nu}
	\right]
	U,
	&
	\mbox{ and }
	&
	\cov\left[\bfZ_{1}, \bfZ_{2}\left|\bfZ_{1}\right.\right]
	&
	=
	& 
	0.
	\end{array}
\end{eqnarray*}
It thus follows that:
\begin{eqnarray*}
	\bbE\left[\bfX\left|\nu^{T}\bfX\right.\right]
	& = & 
	S\cdot\bbE\left[\bfZ\left|\bfZ_{1} = \nu^{T}\bfX\right.\right]
	= 
	\nu\cdot \bbE\left[\bfZ_{1}\left|\bfZ_{1}\right.\right] +
	U \cdot\bbE\left[\bfZ_{2}\left|\bfZ_{1}\right.\right]
	\\
	& = & 
	\left(\nu^{T}\bfX\right)
	\cdot
	\left[\bfI_{p} + UU^{T}\frac{\Sigma}{\nu^{T}\Sigma\nu}\right]
	\cdot
	\nu
	+ 
	UU^{T}
	\left[
	\bfI_{p}
	-
	\frac{\Sigma}{\nu^{T}\Sigma\nu}\nu\nu^{T}
	\right]
	\cdot
	\mu
	, \mbox{and}
	\\
	\var\left[\bfX\left|\nu^{T}\bfX\right.\right]
	& = & 
	S\cdot\var\left[\bfZ\left|\bfZ_{1} = \nu^{T}\bfX\right.\right]\cdot S^{T}
	\\
	& = & 
	\left(UU^{T}\right)
	\cdot
	\left(
	\Sigma 
	- 
	\frac
	{\Sigma\nu\nu^{T}\Sigma}
	{\nu^{T}\Sigma\nu}
	\right)
	\cdot
	\left(UU^{T}\right).
\end{eqnarray*}
By noticing that $UU^{T} = U(U^{T}U)^{-1}U^{T}$ is a projection matrix on the orthogonal complement of the space spanned by $\nu$, 
$UU^{T}$ can be rewritten as $UU^{T} = \bfI_{p} - \nu\left(\nu^{T}\nu\right)^{-1}\nu^{T} = \bfI_{p} - \nu\nu^{T}$.
Using this expression for $UU^{T}$ and some algebra,
\begin{small}
\begin{eqnarray}
	\begin{array}{ccl}
	\bbE\left[\bfX\left|\nu^{T}\bfX\right.\right]
	& = &
	\mu
	+
	\left[
	\bfI_{p}
	+
	\left(\bfI_{p}-\nu\nu^{T}\right)\frac{\Sigma\nu}{\nu^{T}\Sigma\nu}
	\right]
	\nu\nu^{T}
	\left(\bfX-\mu\right)
	\\
	& = &
	\mu
	+
	\frac{\Sigma\nu}{\nu^{T}\Sigma\nu}
	\left(\nu^{T}\bfX-\nu^{T}\mu\right)
	, \mbox{and}
	\\
	\var\left[\bfX\left|\nu^{T}\bfX\right.\right]
	& = &
	\Sigma 
	- 
	\frac{\Sigma \nu\nu^{T}\Sigma}{\nu^{T}\Sigma\nu}.
	\end{array}
	\label{equation:conditional_mean_and_variance_gaussian}
\end{eqnarray}
\end{small}
From \eqref{equation:conditional_mean_and_variance_gaussian}, the second moment of $\bfX$ given $\nu^{T}\bfX$ becomes
\begin{small}
\begin{eqnarray}
	\begin{array}{rcl}
	\bbE\left[\bfX\bfX^{T}\left|\nu^{T}\bfX\right.\right]
	& = & 
	\bbE\left[\bfX\left|\nu^{T}\bfX\right.\right]
	\bbE\left[\bfX\left|\nu^{T}\bfX\right.\right]^{T}
	+
	\var\left[\bfX\left|\nu^{T}\bfX\right.\right]
	\\
	& = &
	\left(\mu\mu^{T}+\Sigma\right)
	+
	\frac{\Sigma\nu\nu^{T}\Sigma}{(\nu^{T}\Sigma\nu)^{2}}
	\cdot
	\left[\left(\nu^{T}\bfX-\nu^{T}\mu\right)^{2}-1\right]
	+
	\left[
	\frac{\Sigma\nu\mu^{T}}{\nu^{T}\Sigma\nu}
	+
	\frac{\mu\nu^{T}\Sigma}{\nu^{T}\Sigma\nu}
	\right]
	\cdot
	\left[\nu^{T}\bfX-\nu^{T}\mu\right].
	\end{array}
	\label{equation:conditional_second_moment_gaussian}
\end{eqnarray}
\end{small}
Given the linear predictor variable as defined in \eqref{equation:definition_margin_variable}, the conditional first and second moments of $\bfX$ are
\begin{eqnarray*}
	\begin{array}{rcl}
	\bbE\left[\bfX\left|\bfM\right.\right]
	& = &
	\mu
	+
	\left(\frac{\bfM-\alpha-\mu^{T}\beta}{\|\beta\|}\right)
	\cdot
	\frac{\Sigma}{\nu^{T}\Sigma\nu}
	\cdot
	\nu
	, \mbox{and}
	\\
	\bbE\left[\bfX\bfX^{T}\left|\bfM\right.\right]
	& = &
	\bbE\left(\bfX\bfX^{T}\right)
	+
	\frac{\Sigma\nu\nu^{T}\Sigma}{(\nu^{T}\Sigma\nu)^{2}}
	\cdot
	\left[\left(\frac{\bfM-\alpha-\beta^{T}\mu}{\|\beta\|}\right)^{2}-1\right]
	+
	\left[
	\frac{\Sigma\nu\mu^{T}}{\nu^{T}\Sigma\nu}
	+
	\frac{\mu\nu^{T}\Sigma}{\nu^{T}\Sigma\nu}
	\right]
	\cdot
	\left[\frac{\bfM-\alpha-\beta^{T}\mu}{\|\beta\|}\right].
	\end{array}
	\label{equation:conditional_second_moment_gaussian_in_terms_of_margin}
\end{eqnarray*}

\subsection{Hessians for SVM and Logistic regression risk functions}

With the expression for the conditional second moments of a multivariate Gaussian variable given the value of one of its projections along the direction $\nu=\frac{\beta}{\|\beta\|}$, equations \eqref{equation:svm_hessian_in_conditional_form} and \eqref{equation:logistic_hessian_in_conditional_form} give expressions for the Hessian of SVM and logistic regression risk functions.

\subsubsection{Hessians for Gaussian predictors}

To simplify the expressions, we partition the Hessian according to the intercept and the predictors $\bfX$ as
\begin{eqnarray*}
	H_{\theta}(t)
	& = &
	\left[
	\begin{array}{cc}
		\left[H_{\theta}(t)\right]_{\alpha, \alpha} 
		& 
		\left[H_{\theta}(t)\right]_{\alpha, \beta} 
		\\ 
		\left[H_{\theta}(t)\right]_{\beta, \alpha} 
		& 
		\left[H_{\theta}(t)\right]_{\beta, \beta} 
	\end{array}
	\right], \mbox{ with }
\end{eqnarray*}
$\left[H_{\theta}(t)\right]_{\alpha, \beta} =\left[H_{\theta}(t)\right]_{\beta, \alpha}^{T}$.
Throughout this section $\tilde{f}$ denotes the density of the linear predictor variable $\bfM$.

\paragraph{Hessian for the Logistic regression classifier:}

Using the expressions derived in Section \ref{section:logistic_risk}, we get
\begin{eqnarray}
	\label{equation:logistic_hessian}
	\begin{array}{rcll}
	\left[H_{\theta}(\theta)\right]_{\alpha, \alpha}
	& = &
	\kappa_{0}
	\\
	\left[H_{\theta}(\theta)\right]_{\beta, \alpha}
	& = &
	\kappa_{0}\cdot\mu + \kappa_{1}\cdot \frac{\Sigma}{\nu^{T}\Sigma\nu}\cdot \nu
	\\
	\left[H_{\theta}(\theta)\right]_{\beta, \beta}
	& = &
	\left(\mu\mu^{T}+\Sigma\right)\cdot \kappa_{0} 
	+ 
	\left[\frac{\Sigma\nu\mu^{T}+\mu\nu^{T}\Sigma}{\nu^{T}\Sigma\nu}\right] \cdot \kappa_{1}
	+
	\left[\frac{\Sigma \nu \nu^{T}\Sigma}{\nu^{T}\Sigma \nu}\right] \cdot \kappa_{2}
	\end{array}
\end{eqnarray}
where $\kappa_{0}$, $\kappa_{1}$ and $\kappa_{2}$ are scalars given by
\begin{eqnarray}
	\label{equation:logistic_hessian_constants}
	\begin{array}{rcll}
	\kappa_{0} 
	& = & 
	\int\left.
	\left[\frac{\exp(m)}{\left(1+\exp(m)\vphantom{\int}\right)^{2}}\right]
	\cdot
	\tilde{f}(m)
	\cdot
	\mbox{d}m\right.
	, 
	\\
	\kappa_{1} 
	& = & 
	\int\left.
	\left(\frac{m-\alpha-\mu^{T}\beta}{\|\beta\|}\right)
	\cdot
	\left[\frac{\exp(m)}{\left(1+\exp(m)\vphantom{\int}\right)^{2}}\right]
	\cdot
	\tilde{f}(m)
	\cdot
	\mbox{d}m\right.
	, \mbox{ and }
	\\
	\kappa_{2} 
	& = & 
	\int\left.
	\left[\left(\frac{m-\alpha-\mu^{T}\beta}{\|\beta\|}\right)^{2}-1\right]
	\cdot
	\left[\frac{\exp(m)}{\left(1+\exp(m)\vphantom{\int}\right)^{2}}\right]
	\cdot
	\tilde{f}(m)
	\cdot
	\mbox{d}m\right.
	. 
	\end{array}
\end{eqnarray}

\paragraph{Hessian for the SVM classifier:}
Using the expressions derived in Section \ref{section:svm_risk}, we get
\begin{eqnarray}
	\label{equation:svm_hessian}
	\begin{array}{rcll}
	\left[H_{\theta}(\theta)\right]_{\alpha, \alpha}
	& = &
	\kappa_{0}
	\\
	\left[H_{\theta}(\theta)\right]_{\beta, \alpha}
	& = &
	\kappa_{0}\cdot\mu + \kappa_{1}\cdot \frac{\Sigma}{\nu^{T}\Sigma\nu}\cdot \nu
	\\
	\left[H_{\theta}(\theta)\right]_{\beta, \beta}
	& = &
	\left(\mu\mu^{T}+\Sigma\right)\cdot \kappa_{0} 
	+ 
	\left[\frac{\Sigma\nu\mu^{T}+\mu\nu^{T}\Sigma}{\nu^{T}\Sigma\nu}\right] \cdot \kappa_{1}
	+
	\left[\frac{\Sigma \nu \nu^{T}\Sigma}{\nu^{T}\Sigma \nu}\right] \cdot \kappa_{2}
	\end{array}
\end{eqnarray}
where $\kappa_{0}$, $\kappa_{1}$ and $\kappa_{2}$ are scalars given by
\begin{eqnarray}
	\label{equation:svm_hessian_constants}
	\begin{array}{rcll}
	\kappa_{0} 
	& = & 
	\tilde{f}(1) 
	\cdot 
	\bbP\left(\bfY=1\left|\bfM=1\right.\right)
	+
	\tilde{f}(-1) 
	\cdot 
	\bbP\left(\bfY=-1\left|\bfM=-1\right.\right)
	, 
	\\
	\kappa_{1} 
	& = & 
	\left(\frac{1-\alpha-\mu^{T}\beta}{\|\beta\|}\right)
	\cdot
	\tilde{f}(1) 
	\cdot 
	\bbP\left(\bfY=1\left|\bfM=1\right.\right)
	\\
	&&
	+
	\left(\frac{-1-\alpha-\mu^{T}\beta}{\|\beta\|}\right)
	\cdot
	\tilde{f}(-1) 
	\cdot 
	\bbP\left(\bfY=-1\left|\bfM=-1\right.\right)
	, \mbox{ and }
	\\
	\kappa_{2} 
	& = & 
	\left[\left(\frac{1-\alpha-\mu^{T}\beta}{\|\beta\|}\right)^{2}-1\right]
	\cdot
	\tilde{f}(1) 
	\cdot 
	\bbP\left(\bfY=1\left|\bfM=1\right.\right)
	\\
	&&
	+
	\left[\left(\frac{-1-\alpha-\mu^{T}\beta}{\|\beta\|}\right)^{2}-1\right]
	\cdot
	\tilde{f}(-1) 
	\cdot 
	\bbP\left(\bfY=-1\left|\bfM=-1\right.\right)
	. 
	\end{array}
\end{eqnarray}

\subsubsection{Hessians for mixed Gaussian predictors}

When $\bfX$ is distributed according to a mixture of $K$ multivariate Gaussians, the conditional moments of $\bfX$ involved in the expression for the risk Hessian can be written as a weighted sum of the corresponding conditional moments for each of the individual Gaussian components as detailed next.
Letting $\pi_{k}$ denote the proportion of the mixture sampled from a multivariate Gaussian with mean $\mu_{k}$ and covariance matrix $\Sigma_{k}$, for $k=1, \ldots, K$, the density function of $\bfX$ is:
\begin{eqnarray*}
	f\left(\bfx\right) 
	& = & 
	\sum_{k=1}^{K}\pi_{k}\cdot\left(\frac{1}{2\pi|\Sigma_{k}|}\right)^{\frac{p}{2}}\exp\left[-\frac{1}{2}\cdot(\bfx-\mu_{k})\Sigma_{k}^{-1}(\bfx-\mu_{k})^{T}\right].
\end{eqnarray*}
The conditional second moment $\bbE\left[\bfX\bfX^{T}\left|\bfM, \mu_{k}, \Sigma_{k}\right.\right]$ given the margin variable $\bfM$ \underline{and} that $\bfX$ was sampled from the component with mean $\mu_{k}$ and covariance $\Sigma_{k}$ follows from \eqref{equation:conditional_second_moment_gaussian} above.
The first and second moment conditional solely on the margin variable $\bfM$ can then be computed as:
\begin{eqnarray*}
	\begin{array}{rcll}		
	\bbE\left[\bfX\left|\bfM\right.\right]
	& = & 
	\sum_{k=1}^{K}
	\bbE\left[\bfX\left|\bfM, \mu_{k}, \Sigma_{k}\right.\right]\cdot\bbP(\mu_{k}, \Sigma_{k}\left|\bfM\right.), 
	&
	\mbox{ and }
	\\
	\bbE\left[\bfX\bfX^{T}\left|\bfM\right.\right]
	& = & 
	\sum_{k=1}^{K}
	\bbE\left[\bfX\bfX^{T}\left|\bfM, \mu_{k}, \Sigma_{k}\right.\right]\cdot\bbP(\mu_{k}, \Sigma_{k}\left|\bfM\right.),
	\end{array}
\end{eqnarray*}
where $\bbP(\mu_{k}, \Sigma_{k}\left|\bfM\right)$ denotes the probability of a point having been sampled from the Gaussian component with center $\mu_{k}$ and variance $\Sigma_{k}$ given the margin variable $\bfM$.
The distribution of $\bfM=\alpha+\bfX^{T}\beta$ is itself a mixture of Gaussians whose density $\tilde{f}$ is
\begin{eqnarray*}
	\tilde{f}\left(m\right) 
	& = & 
	\sum_{k=1}^{K}
	\pi_{k}
	\cdot
	\left(\frac{1}{2\pi|\beta^{T}\Sigma_{k}\beta|}\right)^{\frac{1}{2}}
	\exp\left[-\frac{1}{2}\cdot\frac{(m-\alpha-\beta^{T}\mu_{k})^{2}}{\beta^{T}\Sigma_{k}\beta}\right].
\end{eqnarray*}
An expression for $\bbP(\mu_{k}, \Sigma_{k}\left|\bfM\right.)$ then follows from using Bayes's theorem:
\begin{eqnarray*}
	\bbP\left(\mu_{j}, \Sigma_{j}\left|\bfM\right.\right) 
	& = & 	
	\frac
	{
	\pi_{k} 
	\cdot
	\left(\beta^{T}\Sigma_{k}\beta\right)^{-\frac{1}{2}}
	\cdot 
	\exp
	\left[
	-\frac{1}{2}
	\cdot
	\frac
	{\left(\bfM-\alpha-\beta^{T}\mu_{k}\right)^{2}}
	{\beta^{T}\Sigma_{k}\beta}\right]
	}
	{
	\sum\limits_{\tilde{k}=1}^{K}
	\pi_{\tilde{k}}
	\cdot
	\left(\beta^{T}\Sigma_{\tilde{k}}\beta\right)^{-\frac{1}{2}}
	\cdot 
	\exp\left[-\frac{1}{2}
	\cdot
	\frac
	{\left(\bfM-\alpha-\beta^{T}\mu_{\tilde{k}}\right)^{2}}
	{\beta^{T}\Sigma_{\tilde{k}}\beta}
	\right]
	}.
\end{eqnarray*}

Using \eqref{equation:logistic_hessian} and \eqref{equation:svm_hessian}, we have that the Hessian for SVM and logistic regression risks are given by:
\begin{eqnarray}
	\label{equation:gaussian_mixture_hessian}
	\begin{array}{rcll}
	\left[H_{\theta}(\theta)\right]_{\alpha, \alpha}
	& = &
	\sum\limits_{k=1}^{K}\kappa_{k,0}
	\\
	\left[H_{\theta}(\theta)\right]_{\beta, \alpha}
	& = &
	\sum\limits_{k=1}^{K}
	\left(
	\kappa_{k,0}\cdot\mu_{k} 
	+ 
	\kappa_{k,1}\cdot \frac{\Sigma_{k}\nu}{\nu^{T}\Sigma_{k}\nu}
	\right), 
	&
	\mbox{ and }
	\\
	\left[H_{\theta}(\theta)\right]_{\beta, \beta}
	& = &
	\sum\limits_{k=1}^{K}
	\left[
	\kappa_{k,0} \cdot \left(\mu_{k}\mu_{k}^{T}+\Sigma_{k}\right)
	+ 
	\kappa_{k,1}\cdot\left(\frac{\Sigma_{k}\nu\mu_{k}^{T}+\mu_{k}\nu^{T}\Sigma_{k}}{\nu^{T}\Sigma_{k}\nu}\right)
	+
	\kappa_{k,2}\cdot\left(\frac{\Sigma_{k}\nu \nu^{T}\Sigma_{k}}{\nu^{T}\Sigma_{k} \nu}\right)
	\right],
	\end{array}
\end{eqnarray}
where the scalars $\kappa_{k,0}$, $\kappa_{k,1}$ and $\kappa_{k,2}$ for $k=1, \ldots, K$ are computed according to the risk function.
For each Gaussian component, the $\kappa_{k,0}$, $\kappa_{k,1}$ and $\kappa_{k,2}$ correspond to the $\kappa_{0}$, $\kappa_{1}$ and $\kappa_{2}$ scalars in \eqref{equation:logistic_hessian_constants} and \eqref{equation:svm_hessian_constants} multiplied by the conditional probability of that component given the margin variable as indicated next.

For the logistic risk and mixed Gaussian predictors, the $\kappa$ scalars are:
\begin{eqnarray}
	\label{equation:logistic_hessian_constants_mixed_gaussians}
	\begin{array}{rcll}
	\kappa_{k, 0} 
	& = & 
	\int\left.
	\bbP\left(\mu_{k}, \Sigma_{k}\left|\bfM=m\right.\right)
	\cdot
	\left[\frac{\exp(m)}{\left(1+\exp(m)\vphantom{\int}\right)^{2}}\right]
	\cdot
	\tilde{f}_{k}(m)
	\cdot
	\mbox{d}m\right.
	, 
	\\
	\kappa_{k, 1} 
	& = & 
	\int\left.
	\bbP\left(\mu_{k}, \Sigma_{k}\left|\bfM=m\right.\right)
	\cdot
	\left(\frac{m-\alpha-\mu_{k}^{T}\beta}{\|\beta\|}\right)
	\cdot
	\left[\frac{\exp(m)}{\left(1+\exp(m)\vphantom{\int}\right)^{2}}\right]
	\cdot
	\tilde{f}_{k}(m)
	\cdot
	\mbox{d}m\right.
	, \mbox{ and }
	\\
	\kappa_{k, 2} 
	& = & 
	\int\left.
	\bbP\left(\mu_{k}, \Sigma_{k}\left|\bfM=m\right.\right)
	\cdot
	\left[\left(\frac{m-\alpha-\mu_{k}^{T}\beta}{\|\beta\|}\right)^{2}-1\right]
	\cdot
	\left[\frac{\exp(m)}{\left(1+\exp(m)\vphantom{\int}\right)^{2}}\right]
	\cdot
	\tilde{f}_{k}(m)
	\cdot
	\mbox{d}m\right.
	. 
	\end{array}
\end{eqnarray}

For the SVM risk and mixed Gaussian predictors, the $\kappa$ scalars are:
\begin{eqnarray}
	\label{equation:svm_hessian_constants_mixed_gaussians}
	\begin{array}{rcll}
	\kappa_{k, 0} 
	& = & 
	\bbP\left(\mu_{k}, \Sigma_{k}\left|\bfM=-1\right.\right)
	\cdot
	\tilde{f}_{k}(1) 
	\cdot 
	\bbP\left(\bfY=1\left|\bfM=1\right.\right)
	\\
	& &
	+
	\bbP\left(\mu_{k}, \Sigma_{k}\left|\bfM=1\right.\right)
	\cdot
	\tilde{f}_{k}(-1) 
	\cdot 
	\bbP\left(\bfY=-1\left|\bfM=-1\right.\right)
	, 
	\\
	\kappa_{k, 1} 
	& = & 
	\bbP\left(\mu_{k}, \Sigma_{k}\left|\bfM=1\right.\right)
	\cdot
	\left(\frac{1-\alpha-\mu_{k}^{T}\beta}{\|\beta\|}\right)
	\cdot
	\tilde{f}_{k}(1) 
	\cdot 
	\bbP\left(\bfY=1\left|\bfM=1\right.\right)
	\\
	&&
	+
	\bbP\left(\mu_{k}, \Sigma_{k}\left|\bfM=-1\right.\right)
	\cdot
	\left(\frac{-1-\alpha-\mu_{k}^{T}\beta}{\|\beta\|}\right)
	\cdot
	\tilde{f}_{k}(-1) 
	\cdot 
	\bbP\left(\bfY=-1\left|\bfM=-1\right.\right)
	, \mbox{ and }
	\\
	\kappa_{k, 2} 
	& = & 
	\bbP\left(\mu_{k}, \Sigma_{k}\left|\bfM=1\right.\right)
	\cdot
	\left[\left(\frac{1-\alpha-\mu_{k}^{T}\beta}{\|\beta\|}\right)^{2}-1\right]
	\cdot
	\tilde{f}_{k}(1) 
	\cdot 
	\bbP\left(\bfY=1\left|\bfM=1\right.\right)
	\\
	&&
	+
	\bbP\left(\mu_{k}, \Sigma_{k}\left|\bfM=-1\right.\right)
	\cdot
	\left[\left(\frac{-1-\alpha-\mu_{k}^{T}\beta}{\|\beta\|}\right)^{2}-1\right]
	\cdot
	\tilde{f}_{k}(-1) 
	\cdot 
	\bbP\left(\bfY=-1\left|\bfM=-1\right.\right).
	\end{array}
\end{eqnarray}

\nocite{meinshausen:2006:lasso-type-recovery-of-sparse-representations-for-high-dimensional-data}
\nocite{li:2008:variable-selection-in-semiparametric-regression-modeling}
\nocite{zhang:2009:some-sharp-performance-bounds-for-least-squares-regression-with-l1-regularization}

\begin{small}
\singlespacing
\bibliographystyle{acmtrans}
\bibliography{my_papers}
\end{small}

\end{document}